\documentclass[12pt]{amsart}

\usepackage{amsmath}
\usepackage{amsbsy}
\usepackage{amssymb}
\usepackage{amscd}
\usepackage{eq}
\usepackage[dvips]{graphicx, color}

\newcommand{\bbc}{{\mathbb C}}

\newcommand{\bbp}{{\mathbb P}}
\newcommand{\bbq}{{\mathbb Q}}
\newcommand{\bbr}{{\mathbb R}}

\newcommand{\bbz}{{\mathbb Z}}

\newcommand{\al}{{\alpha}}

\newcommand{\be}{{\beta}}

\newcommand{\gam}{{\gamma}}

\newcommand{\Gam}{{\Gamma}}
\newcommand{\del}{{\delta}}

\newcommand{\Del}{{\Delta}}

\newcommand{\vep}{{\varepsilon}}

\newcommand{\lam}{{\lambda}}
\newcommand{\Lam}{{\Lambda}}

\newcommand{\sig}{{\sigma}}

\newcommand{\om}{{\omega}}

\newcommand{\gB}{{\mathfrak B}}

\newcommand{\gC}{{\mathfrak C}}

\newcommand{\gI}{{\mathfrak I}}

\newcommand{\gS}{{\mathfrak S}}
\newcommand{\gt}{{\mathfrak t}}

\newcommand{\ch}{{\operatorname{ch}}\,}

\newcommand{\aff}{{\operatorname {Aff}}}
\newcommand{\kernel}{{\operatorname {Ker}}}

\newcommand{\sgn}{{\operatorname {sgn}}}

\newcommand{\m}{{\operatorname {M}}}
\newcommand{\h}{{\operatorname {H}}}

\newcommand{\im}{{\operatorname {Im}}}

\newcommand{\gal}{{\operatorname{Gal}}}
\newcommand{\gl}{{\operatorname{GL}}}

\newcommand{\spl}{{\operatorname{SL}}}

\newcommand{\sst}{{\operatorname{ss}}}
\newcommand{\pfaff}{{\operatorname{Pfaff}}}
\newcommand{\sep}{{\operatorname{sep}}}

\newcommand{\sym}{{\operatorname{Sym}}}
\newcommand{\zero}{{\operatorname{Zero}}}
\newcommand{\Hom}{{\operatorname{Hom}}}

\newcommand{\rep}{representation}

\newcommand{\pv}{prehomogeneous vector space}

\newcommand{\Z}{\bbz}
\newcommand{\Q}{\bbq}
\newcommand{\R}{\bbr}
\newcommand{\C}{\bbc}
\newcommand{\p}{\bbp}

\newcommand{\mk}{k^{\times}}

\newcommand{\rv}{V_{k}}

\newcommand{\kableadd}%
{Department of Mathematics\\ Cornell University\\
Ithaca NY 14853}

\newcommand{\sub}{\subset}
\newcommand{\bk}{\backslash}
\newcommand{\ti}{\widetilde}
\newcommand{\mzer}{\setminus \{ 0\}}
\newcommand{\ccd}{,\ldots,}

\newcommand{\lan}{\langle}
\newcommand{\ran}{\rangle}

\makeatletter
\def\varddots{\mathinner{
\mkern1mu%
 \raise\p@\hbox{.}\mkern2mu%
 \raise4\p@\hbox{.}\mkern2mu%
 \raise7\p@\vbox{\kern7\p@\hbox{.}}%
\mkern1mu}}
\makeatother

\theoremstyle{plain}
\newtheorem{thm}{Theorem}[section]
\newtheorem{lem}[thm]{Lemma}
\newtheorem{cor}[thm]{Corollary}
\newtheorem{prop}[thm]{Proposition}

\theoremstyle{definition}

\newtheorem{defn}[thm]{Definition}

\newtheorem{cond}[thm]{Condition}

\newtheorem{property}[thm]{Property}

\theoremstyle{remark}
\newtheorem{rem}[thm]{Remark}        

\usepackage{mathrsfs,bbm} 
\usepackage[all]{xy}

\newcommand{\weyl}{\mathbb W}
\newcommand{\Conv}{\operatorname{Conv}}
\newcommand{\Span}{\operatorname{Span}}
\newcommand{\semi}{\mathrm {ss}}
\newcommand{\coorde}{{\mathbbm e}}
\newcommand{\coordf}{{\mathbbm f}}
\newcommand{\coorda}{{\mathbbm a}}
\newcommand{\diag}{{\mathrm{diag}}}
\newcommand{\size}{\tiny}
\newcommand{\Ex}{\mathrm{Ex}}
\newcommand{\tallstrut}{\rule[-12pt]{-0.2cm}{28pt}}
\newcommand{\tinsert}{\quad\tallstrut}
\newcommand{\bbma
}{\mathbbm a}

\newcommand{\bbmp}{\mathbbm p}
\newcommand{\bbmq}{\mathbbm q}
\newcommand{\bbmr}{\mathbbm r}
\newcommand{\Prg}{\mathrm{Prg}}

\begin{document}

\address[K. Tajima]
{National Institute of Technology, Sendai College, Natori Campus, 
48 Nodayama, Medeshima-Shiote, Natori-shi, Miyagi, 981-1239, Japan}
\email{kazuaki.tajima.a8@tohoku.ac.jp}

\address[A. Yukie]
{Department of Mathematics, Graduate School of Science, 
Kyoto University, Kyoto 606-8502, Japan}
\email{yukie.akihiko.7x@kyoto-u.ac.jp}
\thanks{The second author was partially supported by 
Grant-in-Aid (C) (20K03512)\\}


\title{On the GIT stratification of prehomogeneous vector spaces III}
\author{Kazuaki Tajima}
\author{Akihiko Yukie}
\maketitle

\begin{abstract}
We determine all orbits of the \pv{} 
$(\gl_5\times \gl_4,\wedge^2\aff^5 \otimes \aff^4)$ 
rationally over an arbitrary 
perfect field whose characteristic is not $2$ in this paper.  
\end{abstract}

\section{Introduction}
\label{sec:introduction}

This is part three of a series of four papers. 
In Part I, we determined the set $\gB$ of vectors 
which parametrizes the GIT (geometric invariant theory) 
stratification \cite{tajima-yukie} 
of the four \pv s (1)--(4) in \cite{tajima-yukie-GIT1}. 
Even though the set $\gB$ was determined, 
there may be strata $S_{\be}$ ($\be\in\gB$) 
which are the empty set. In Part II, we determined 
which strata $S_{\be}$ are non-empty for the \pv s (1), (2) in 
\cite{tajima-yukie-GIT1}. 

In this part, we consider the following \pv:
\begin{equation}
\label{eq:PV}
G=\gl_5\times \gl_4,\;
V=\wedge^2 \aff^5\otimes \aff^4.
\end{equation}
This is the \pv{} (3) in \cite{tajima-yukie-GIT1}. 

For a general introduction to this series of papers, 
see the introduction of \cite{tajima-yukie-GIT1}. 

Throughout this paper, $k$ is a fixed field. 
We shall assume that $k$ is a perfect field in 
the main theorem and in Sections  \ref{sec:non-empty}, 
\ref{sec:empty-strata}. In that case 
the algebraic closure $\overline k$ coincides with 
the separable closure $k^{\sep}$. 
If $X,Y$ are schemes, algebraic groups, etc., over $k$ then 
$X,Y$ are said to be {\it $k$-forms} of each other if 
$X\times_k k^{\sep}\cong Y\times_k k^{\sep}$. 

Let $Q_0(v)\in \sym^2\aff^4$ be the quadratic form 
$v_1v_4-v_2v_3$.
\begin{defn}
\label{defn:Ex-defn}
\begin{itemize}
\item[(1)]
$\Ex_n(k)$ is the set of conjugacy classes of 
homomorphisms from $\gal(\overline k/k)$ to $\gS_n$.
\item[(2)]
$\Prg_2(k)$ is the set of $k$-isomorphism classes of 
$k$-forms of $\mathrm{PGL}_2$. 
\item[(3)]
$\mathrm{QF}_4(k)$ 
is the set of $k$-isomorphism classes of 
algebraic groups of the form $\text{GO}(Q)^{\circ}$ 
where $Q\in \sym^2 \aff^4$. Let 
$\mathrm{IQF}_4(k)\sub \mathrm{QF}_4(k)$ be the subset  
consisting of inner forms of $\text{GO}(Q_0)^{\circ}$. 
\end{itemize}
\end{defn}

Note that if $n=2,3$ then $\Ex_n(k)$ coincides with 
the set of $k$-isomorphism classes of 
separable extensions of $k$ of degree up to $n$.

The following theorem is our main theorem in this part.

\begin{thm} {\bf (Main Theorem)}
\label{thm:main}
Suppose that $k$ is a perfect field. 
\begin{itemize}
\item[(1)]
For the \pv{} (\ref{eq:PV}), there are $61$ non-empty strata $S_{\be}$.  
\item[(2)]
Suppose that $\ch(k)\not=2$. 
If $S_{\be}\not=\emptyset$ then $G_k\backslash S_{\be\,k}$ 
is either (i) a single point (abbreviated as SP from now on) 
(ii) $\Ex_2(k)$ (iii) $\Ex_3(k)$ 
(iv) $\Prg_2(k)$ or (v) $\mathrm{IQF}_4(k)$. Moreover the number of $S_{\be}$'s 
for (i)--(v) are as follows.
\begin{center}
\begin{tabular}{|c|c|c|c|c|c|}
\hline
\rule[-8pt]{-0.2cm}{22pt}
\upshape Type & $\;$ \hskip 5pt \upshape SP \hskip 5pt $\;$ 
& $\Ex_2(k)$ & $\Ex_3(k)$ & $\Prg_2(k)$ & $\mathrm{IQF}_4(k)$ \\
\hline
\rule[-8pt]{-0.2cm}{22pt}
\upshape Number of $S_{\be}$'s & \upshape 43 & \upshape 12 
& \upshape 3 & \upshape 2 & \upshape 1 \\
\hline
\end{tabular}
\end{center}
\end{itemize}
\end{thm}

As we pointed out in Part I, the orbit decomposition of 
these \pv s is known over $\C$ (see \cite{ozekic}). 
Our approach answers to rationality questions and provide 
the inductive structure of orbits rationally over $k$. 
The \pv{} (\ref{eq:PV}) is the ``quintic case'' and it 
was proved in \cite{wryu} that the set of generic rational orbits 
is in bijective correspondence with $\Ex_5(k)$. Integral orbits 
of this case was considered by Bhargava in \cite{bhargava4}. 

In the process of proving Theorem \ref{thm:main}, 
we end up with determining the set of generic rational orbits 
of the following \pv s $(G,V)$. 

\begin{itemize}
\item[(1)]
$G=\gl_5\times \gl_3$, $V=\wedge^2 \aff^5\otimes \aff^3$
\item[(2)]
$G=\gl_3^2\times \gl_2$, 
$V=\wedge^2 \aff^3\oplus \aff^3\otimes \aff^3\otimes \aff^2$.
\item[(3)]
$G=\gl_4\times \gl_2^2$, 
$V=\aff^4\otimes \aff^2\oplus \wedge^2\aff^4\otimes \aff^2$.
\item[(4)]
$G=\gl_4\times \gl_3$, 
$V=\wedge^2 \aff^4\otimes \aff^3$ and 
$\wedge^2 \aff^4\otimes \aff^3\oplus \aff^4$. 
\item[(5)]
$G=\gl_3\times \gl_2^2\times \gl_1$, 
$V=\wedge^2 \aff^3\oplus \aff^3\otimes \aff^2\otimes \aff^2$.
\item[(6)]
$G=\gl_3\times \gl_2^2\times \gl_1$, 
$V=\wedge^2 \aff^3\otimes \aff^2 \oplus \aff^3\otimes \aff^2\otimes \aff^2$.
\item[(7)]
$G=\gl_2^3$, 
$V=\aff^2\otimes \aff^2\otimes \aff^2\oplus \aff^2$.
\item[(8)]
$G=\gl_2^3$, 
$V=\aff^2\otimes \aff^2\otimes \aff^2\oplus \aff^2\oplus \aff^2$.
\end{itemize}
For details on the precise definitions of actions and 
interpretations of rational orbits, 
see Sections \ref{sec:rational-orbits-53}--\ref{sec:rational-orbits-222-22}. 

These cases are either irreducible, 2-simple or 3-simple \pv s.
The case (1) was considered in \cite{ochiai}, \cite{yukiem} over $k$ such that 
$\ch(k)=0$. 
Prehomogeneous vector spaces which are 2-simple or 3-simple 
were considered in \cite{kimura-2simple-typeI}, \cite{kimura-2simple-typeII}, 
\cite{kimura-2simple-typeI-inv}, \cite{kimura-2simple-some-inv}, 
\cite{kimura-3simple}.
The second case of (4) is the case (3) of \cite[p.396]{kimura-2simple-typeI}.
Relative invariants of this case are constructed in the case (4a) of 
Theorem 5.8 (the case (3) of Theorem 8.1) \cite[pp.465, 478]{kimura-2simple-some-inv}.
In this case, there are two relative invariant polynomials
($P_1(x),P_2(x)$ in Proposition \ref{prop:section8-b-P1-value} 
and (\ref{eq:43-4-P2-equivariant})). They should be more or less 
$P_1(v),P_2(v)$ in \cite[p.465]{kimura-2simple-some-inv}.
Note that the case (4a) in \cite[p.465]{kimura-2simple-some-inv} 
is bigger than the second case of the above (4), but 
$P_1(v),P_2(v)$ in \cite[p.465]{kimura-2simple-some-inv} 
give rise to relative invariant polynomials on 
the second case of the above (4) by restriction. 
We define $P_1(x),P_2(x)$ in a superficially different 
manner in Section \ref{sec:rational-orbits-wedge43}.   
The cases (2), (7), (8) are the cases (5), (2), (3) of 
\cite[p.187]{kimura-3simple}.
The case (5) is in the form of \pv s in THEOREM 1.1 
\cite[p.161]{kimura-3simple}.
The case (3) contains a 2-simple \pv{} of trivial type and is not contained 
in the list of \cite{kimura-3simple}.
The case (6) is not reduced in the sense that the 
component $\aff^3\otimes \aff^2\otimes \aff^2$ is not reduced
and is not contained 
in the list of \cite{kimura-3simple} either.

Since the ground field is $\C$ for the above papers and  
$k$ is an arbitrary perfect field in this paper, 
we carry out some Lie algebra computations to make sure 
that some cases are regular \pv{} regardless of $\ch(k)$.

The organization of this paper is as follows. 
In Section \ref{sec:notation}, we discuss notations 
used in this part. In Section \ref{sec:notation-related-git}, 
we briefly review the notion of GIT stratification and 
discuss related notations. 
We have to deal with many reducible 
\pv s in this part and we shall discuss the notion of 
``regularity'' formulated in Section 2 of \cite{kato-yukie-jordan}. 
In \cite{kato-yukie-jordan}, the regularity was used mainly for
irreducible \pv s. Even though the definition of regularity 
is the same as the one in \cite{kato-yukie-jordan}, 
it is necessary to consider the relation between 
the number of independent characters and the number 
of independent relative invariant polynomials. 
We shall discuss this issue in Section \ref{sec:regularity}. 
In Section \ref{sec:rational-orbits-general}, 
we review known results (\cite{wryu}, \cite{tajima-yukie-GIT2})
on rational orbits for some \pv s.  
In Sections \ref{sec:rational-orbits-53}--\ref{sec:rational-orbits-222-22}, 
we determine the set of generic rational orbits of 
some \pv s which appear as 
$(M_{\be},Z_{\be})$ such that $Z^{\sst}_{\be}\not=\emptyset$. 
In Section \ref{sec:non-empty}, we show that there are 
$61$ non-empty strata $S_{\be}$. Assuming $\ch(k)\not=2$, 
we determine $G_k\backslash S_{\be\,k}$ for such $\be$. 
In Section \ref{sec:empty-strata}, we show that the 
remaining strata $S_{\be}$ are the empty set. 
The method is similar to the one in 
\cite{tajima-yukie-GIT2}.

The \pv{} (\ref{eq:PV}) is a rather big \pv. 
If a \pv{} $(G',V')$ appears as $(M_{\be},Z_{\be})$ 
of a \pv{} $(G,V)$, then $(G',V')$ is ``smaller'' than $(G,V)$ 
in some sense. In that case, 
the GIT stratification of $(G',V')$ should follow from 
that of $(G,V)$. 
In Section \ref{sec:smaller-pv-s}, we prove a proposition 
by which the consideration of most strata of 
$(G',V')$ is reduced to that of $(G,V)$ with possible 
relatively easy computer computations. 
We shall consider GIT stratifications which follow from 
this series of papers in the future.

The authors would like to thank the referee for reading the manuscript
carefully, giving useful suggestions and pointing out many mistakes.

\section{Notation}
\label{sec:notation}

We discuss notations used in this part.   

When we quote statements from Part I \cite{tajima-yukie-GIT1} 
and Part II \cite{tajima-yukie-GIT2}, we write like 
Theorem I--2.1, Lemma II--4.4. 
 
Let $\ch (k)$ be the characteristic of $k$.
The separable closure and the algebraic closure of $k$ 
are denoted by $k^{\sep},\overline k$ respectively. 
We often have to refer to a set consisting of a single point. 
We use the notation SP for such a set.

If $X$ is a variety over $k$ and $x\in X$ then 
$X_k$ is the set of $k$-rational points of $X$ and 
${\mathrm T}_x(X)$ is the tangent space at $x$ of $X$. 
Let $k[\vep]/(\vep^2)$ be the ring of dual numbers. 
We identify ${\mathrm T}_x(X)$ with the space of 
$k[\vep]/(\vep^2)$-rational points of $X$ which reduce to $x$. 

If $G$ is an algebraic group over $k$ then 
$G^{\circ}$ is the identity component of $G$, i.e.,
the connected component containing the unit element 
of $G$. Unless otherwise stated, $e_G$ is the unit element of $G$
(since we use the notation $e_{ijk}$ later, it is necessary to use a 
slightly different notation for unit elements of groups). 

Let $X^*(G)$ (resp. $X_*(G)$) 
be the group of characters (resp. the group of one parameter subgroups)
of $G$.  We abbreviate ``one parameter subgroup'' of algebraic groups 
as ``1PS'' from now on. 
Suppose that $\chi$ is a non-trivial character of $G$.
It is said to be {\it primitive} if $\psi$ is a character of $G$ and 
$\chi=\psi^a$ for an integer $a$ then $a=\pm 1$. 
If $\rho:G\to \gl(V)$ is a \rep{} of $G$ then 
$\wedge^2 \rho,\sym^3 \rho$, etc., are the induced 
\rep s on $\wedge^2 V,\sym^3 V$, etc.

If $V$ is a vector space then 
we denote the dual space of $V$ by $V^*$. 
If $V$ is a \rep{} of $G$ then $V^*$ becomes a 
\rep{} of $G$ where the action of $g\in G$ is defined by 
$V^*\ni f\mapsto (V\ni v \mapsto f(g^{-1}v))$.

If $\sig,\tau\in\gal(k^{\sep}/k)$ then we define 
$\sig\tau(x)=\tau(\sig(x))$ for $x\in k^{\sep}$. So the 
action of $\gal(k^{\sep}/k)$ is a right action. 
We use the notation such as $x^{\sig}$.  
If $G$ is an algebraic group over $k$ then 
we denote by $\h^1(k,G)$ the Galois cohomology set with coefficients in 
$G$. We use the notation such as $h=(h_{\sig})_{\sig}$
for 1-cochains. The cocycle condition is $h_{\sig\tau}=h_{\tau}h_{\sig}^{\tau}$ 
and $h_1=(h_{1,\sig})_{\sig},h_2=(h_{2,\sig})_{\sig}$ are equivalent if there exists 
$g\in G_{k^{\sep}}$ such that $h_{1,\sig}=g^{-1}h_{2,\sig}g^{\sig}$ 
for all $\sig$. We denote the trivial 1-cocycle (i.e., $h_{\sig}=1$ for all
$\sig$) and its cohomology class  by $1$. 
Note that $\h^1(k,G)$ may not have a group structure 
if $G$ is not commutative.

Let $\gl_n$ (resp. $\spl_n$) be the general linear group 
(resp. special linear group), 
$\m_{n,m}$ the space of $n\times m$ matrices and $\m_n=\m_{n,n}$. 
If $G$ is an algebraic group then 
the tangent space ${\mathrm T}_{e_G}(G)$ at $e_G$ has a structure 
of a Lie algebra. If $G=\gl_n$ then we can identify 
the Lie algebra of $G$ with $\m_n$ 
by considering $I_n+\vep A$ for $A\in\m_n$. 

Let $E_{ij}$ be the matrix whose $(i,j)$-entry is $1$ and other 
entries are $0$. We use this notation only in the 
situation where the size of $E_{ij}$ will be clear from the
context. Let $\mathrm{PGL}_n=\gl_n/\{tI_n\mid t\in\gl_1\}$. 
If $A$ is an $n\times n$ matrix then we say $A$ is an 
alternating matrix if ${}^tA=-A$ and the diagonal entries are $0$. 
The latter condition is necessary if $\ch(k)=2$.
We may write alternating matrices like
\begin{equation*}
\begin{pmatrix}
0 & a & b \\
* & 0 & c \\
* & * & 0
\end{pmatrix}
\end{equation*}
because entries in $*$ are determined by other entries. 
We denote the set of $k$-rational points of $\gl_n$, etc., by 
$\gl_n(k)$, etc. We sometimes use 
the notation $[x_1\ccd x_n]$ to express column vectors 
to save space.  We denote the unit matrix of dimension $n$ 
by $I_n$.  We use the notation $\diag(g_1\ccd g_m)$ 
for the block diagonal matrix whose diagonal blocks are 
$g_1\ccd g_m$.  We often use the following matrices
\begin{equation}
\label{eq:J-defn}
\tau_0 = 
\begin{pmatrix}
0 & 1 \\
1 & 0
\end{pmatrix}, \quad
J = 
\begin{pmatrix}
0 & 1 \\
-1 & 0 
\end{pmatrix}. 
\end{equation}

If $A=(a_{ij})$ is a $2n\times 2n$ alternating matrix 
then $\pfaff(A)$ is the Pfaffian of $A$. It has the property 
that $\pfaff(A)^2=\det A$, $\pfaff(gA{}^tg)= (\det g)\,\pfaff(A)$ 
for $g\in \gl_{2n}$. We choose the sign so that $\pfaff(A)=1$ 
for the matrix 
\begin{math}
A = \sum_{i=1}^n(E_{2i-1,2i}-E_{2i,2i-1}).
\end{math}

For $u=(u_{ij})\in \aff^{n(n-1)/2}$ ($1\leq j<i\leq n$), 
let $n_n(u)$ be the lower triangular matrix whose diagonal entries are $1$ 
and the $(i,j)$-entry is $u_{ij}$ for $i>j$. 
If $v_1\ccd v_m$ are elements of a vector space $V$ then 
$\lan v_1\ccd v_m\ran$ is the subspace spanned by 
$v_1\ccd v_m$. We sometimes use the notation 
$\text{Span}\{v_1\ccd v_m\}$ also. 
For the rest of this paper, 
tensor products are always over $k$.

Let $n_1=5,n_2=4$.  
We use parabolic subgroups which consist 
of lower triangular  blocks. 
Let $i=1,2$ and $j_{i\,0}=0<j_{i\,1}<\cdots<j_{i\,N_i}=n_i$. 
We use the notation $P_{i,[j_{i1}\ccd j_{i\,N_i-1}]}$ 
(resp. $M_{i,[j_{i1}\ccd j_{i\,N_i-1}]}$)
for the parabolic subgroup (resp. reductive subgroup) of $\gl_{n_i}$ 
in the form 
\begin{equation*}
\begin{pmatrix} 
P_{11} & 0 \cdots 0 & 0 \\
\vdots & \ddots & 
\begin{matrix}
0 \\ \vdots \\ 0
\end{matrix}  \\ 
P_{N_i\,1} & \cdots & P_{N_i\,N_i} 
\end{pmatrix}, \quad
\begin{pmatrix} 
M_{11} & 0 \cdots 0 & 0 \\
\begin{matrix}
0 \\ \vdots \\ 0
\end{matrix}
& \ddots & 
\begin{matrix}
0 \\ \vdots \\ 0
\end{matrix}  \\ 
0 & 0 \cdots 0 & M_{N_i\,N_i} 
\end{pmatrix}
\end{equation*}
where the size of $P_{kl},M_{kl}$ is 
$(j_{i\,k}-j_{i\,k-1})\times (j_{i\,l}-j_{i\,l-1})$. 
If $N_i=1$ then we use the notation $P_{i,\emptyset},M_{i,\emptyset}$.

We put 
\begin{equation}
\label{eq:parabolic-defn}
\begin{aligned} 
P_{[j_{11}\ccd j_{1\,N_1-1}],[j_{21}\ccd j_{2\,N_2-1}]}
& = P_{1,[j_{11}\ccd j_{1\,N_1-1}]}
\times P_{2,[j_{21}\ccd j_{2\,N_2-1}]}, \\
M_{[j_{11}\ccd j_{1\,N_1-1}], [j_{21}\ccd j_{2\,N_2-1}]}
& = M_{1,[j_{11}\ccd j_{1N_1-1}]}
\times M_{2,[j_{a1}\ccd j_{2\,N_2-1}]}.
\end{aligned}
\end{equation}
If $N_i=1$ then we replace $[j_{i\,1}\ccd j_{i\,N_i-1}]$ by 
$\emptyset$. 
Let 
\begin{equation}
\label{eq:M1-defn}
M^{\text{st}}_{[j_{1\,1}\ccd j_{1\,N_1-1}],[j_{2\,1}\ccd j_{2\,N_2-1}]}
= (\spl_5 \times \spl_4)\cap
(M_{1,[j_{1\,1}\ccd j_{1\,N_1-1}]}
\times M_{2,[j_{2\,1}\ccd j_{2\,N_2-1}]})
\end{equation}
and $M^s_{[j_{1\,1}\ccd j_{1\,N_1-1}],[j_{2\,1}\ccd j_{2\,N_2-1}]}$ 
be the semi-simple part of 
$M_{[j_{1\,1}\ccd j_{1\,N_1-1}],[j_{2\,1}\ccd j_{2\,N_2-1}]}$. 

We consider many \rep s of groups of the 
form $M_{[j_{1\,1}\ccd j_{1\,N_1-1}],[j_{2\,1}\ccd j_{2\,N_2-1}]}$
in later sections. 
We use notations such as 
\begin{equation}
\label{eq:Lam-defn}
\Lam^{m,i}_{j,[c,d]} \quad \text{where $m=d-c+1$}. 
\end{equation}
The meaning of this notation is that this is 
$\wedge^i \aff^m$ as a vector space where 
$\aff^m$ is the standard \rep{} of $\gl_m$ 
and the indices $j,[c,d]$ mean that 
the block from the $(c,c)$-entry to 
the $(d,d)$-entry of the $j$-th factor $\gl_{n_j}$ 
of $M_{[j_{11}\ccd j_{1\,N_1-1}],[j_{21}\ccd j_{2\,N_2-1}]}$ 
($n_1=5,n_2=4$) is acting on this vector space. 
For example, 
$M_{[1],[2]}$ consists of elements of the form 
$(\diag(t_1,g_1),\diag(g_2,g_3))$ where 
$t_1,\in\gl_1,g_1\in \gl_4,g_2,g_3\in\gl_2$. 
Then $\Lam^{2,1}_{2,[1,2]}$ is the standard \rep{} of $g_2\in \gl_2$
identified with the element $(I_5,\diag(g_2,I_2))$. 
We denote by $1$ the   
trivial \rep{} of $M^s_{[j_{1\,1}\ccd j_{1\,N_1-1}],[j_{2\,1}\ccd j_{2\,N_2-1}]}$.

\section{Notation related to the GIT stratification}
\label{sec:notation-related-git}

We basically follow \cite{tajima-yukie-GIT1} for the notation 
related to the GIT stratification.  
However, there will be many indices in this part, 
and the subscript ``1'' such as in $G_1$ in 
\cite{tajima-yukie-GIT1}, \cite{tajima-yukie}  
may cause confusion. So we slightly change the notation 
related to the GIT stratification in this part. 
We include a short review of the formulation of
the GIT stratification for the sake of the reader. 
We consider a general situation where the group is split 
because that is the case in Section \ref{sec:smaller-pv-s}.
The formulation of \cite{tajima-yukie-GIT1} is that of  
Corollary 1.4 of \cite{tajima-yukie} where 
some scalar directions are removed from 
$G_{\beta}$ to consider stability.

Let $G$ be a connected split reductive group, $V$ a finite dimensional 
representation of $G$ both defined over $k$. We do not assume that 
$(G,V)$ is a \pv{} in this section. The notion of \pv s will be
reviewed in the next section. 
We assume that there is a connected split reductive subgroup 
$G_{\text{st}}$ of $G$, 
a split torus $T_0\sub Z(G)$ (the center of $G$), 
such that $T_0\cap G_{\text{st}}$ is finite and $G=T_0G_{\text{st}}$ 
as algebraic groups. We assume that there is a 
rational character $\chi$ of $T_0$ such that the action of 
$t\in T_0$ is given by the scalar multiplication by $\chi(t)$. 
We mainly consider the case where $\chi$ is non-trivial.  
Let $(T_0\cap G_{\text{st}})\sub T_{\text{st}}\sub G_{\text{st}}$ 
be a maximal split torus, 
$X_*(T_{\text{st}}),X^*(T_{\text{st}})$ the group of 1PS's
and the group of rational characters respectively. 
The subscript ``st'' stands for ``stability'' since $G_{\text{st}}$
is the subgroup by which we measure the stability. 
These $T_0$, $G_{\text{st}}$ and $T_{\text{st}}$ correspond to 
$T_0$, $G_1$ and $T$ in \cite[p.3]{tajima-yukie-GIT1}.

If $(G,V)$ is the \pv{} (\ref{eq:PV}) then we choose 
$T=\{(t_1,t_2)\in G\mid t_1,t_2\;\text{diagonal}\}$ (the center of $G$)  
and 
\begin{equation}
\label{eq:T0-defn}
G_{\text{st}}=\spl_5\times \spl_4,\; 
T_{\text{st}}=T\cap G_{\text{st}},\;
T_0=\{(t_1I_5,t_2I_4)\mid t_1,t_2\in\gl_1\}.
\end{equation}
Note that $T_0\cap G_{\text{st}}$ is finite and 
$G=T_0G_{\text{st}}$ as algebraic groups. 
The action of $(t_1I_5,t_2I_4)$ on $V$ is by scalar multiplication 
by $t_1^2t_2$.

We put 
\begin{equation*}
\mathfrak {t}=X_*(T_{\text{st}})\otimes \R, \;
\mathfrak {t}_{\Q}=X_*(T_{\text{st}})\otimes \Q, \;
\mathfrak t^*=X^*(T_{\text{st}})\otimes \R, \;
\mathfrak t^*_{\Q}=X^*(T_{\text{st}})\otimes \Q.
\end{equation*}
Let $\weyl =N_{G}(T)/T$ be the Weyl group
of $G$. The Weyl group of $G_{\text{st}}$ 
coincides with $\weyl$.  
For the \pv{} (\ref{eq:PV}), 
$\weyl\cong \gS_5\times \gS_4$.  
There is a natural action of $\weyl$ on $\gt^*$. 
There is a natural pairing 
$\langle \;  ,\;  \rangle_T :X^*(T_{\text{st}})\times 
X_*(T_{\text{st}})\rightarrow \Z$  defined by 
$t^{\langle \chi ,\lambda \rangle_{T_{\text{st}}}}=\chi (\lambda (t))$ for 
$\chi \in X^*(T_{\text{st}}),\lambda \in X_*(T_{\text{st}})$. 
This is a perfect paring (\cite[pp.113--115]{borelb}). 
There exists an inner product 
$(\;  , \; )$
on $\mathfrak{t}$ which is invariant under the actions of $\weyl$.
We may assume that this inner product is rational, i.e., 
$(\lambda,\nu)\in \Q$ for all 
$\lambda,\nu\in \mathfrak t_{\Q}$. 
Let $\|\; \|$ be the norm on $\mathfrak t$ defined 
by $(\;  , \; )$. We choose a Weyl chamber 
$\mathfrak{t}_{+}\subset \mathfrak{t}$ 
for the action of $\weyl$.

For $\lambda\in\mathfrak t$, 
let $\beta =\beta(\lambda)$ be the element of
$\mathfrak t^*$ such that 
$\langle \beta, \nu\rangle = (\lambda, \nu)$
for all $\nu\in \mathfrak t$. 
The map $\lambda\mapsto \beta(\lambda)$ is a bijection 
and we denote the inverse map by 
$\lambda=\lambda(\beta)$. 
There is a unique positive rational
number $a$ such that $a\lambda(\beta)\in X_*(T_{\text{st}})$
and is indivisible. We use the notation 
$\lambda_{\beta}$ for $a\lambda(\beta)$. 
Identifying $\mathfrak t$ with $\mathfrak t^*$ by the map $\lam\mapsto \be(\lam)$, 
we have a $\weyl$-invariant inner product 
$(\;  , \;  )_{*}$ on $\mathfrak{t^*}$, 
the norm $\|\; \|_{*}$ determined by $(\;  , \;  )_{*}$ 
and a Weyl chamber $\mathfrak t^*_{+}$. 
Note that choosing a $\weyl$-invariant inner product  
$(\;  , \;  )$ is equivalent to choosing a $\weyl$-invariant inner product  
$(\;  , \;  )_{*}$. 

For the \pv{} (\ref{eq:PV}), we can identify $\gt^*$ with 
\begin{align*}
& \left\{(a_{11}\ccd a_{15},a_{21}\ccd a_{24})\in\R^9
\,\vrule\, \sum_{i=1}^5 a_{1i}=
\sum_{i=1}^4 a_{2i}=0\right\}.
\end{align*}
We choose $(\;  , \;  )_{*}$ and $\gt^*_+$ so that 
\begin{align*}
(a,b)_* & = \sum_{i=1}^5 a_{1i}b_{1i}+\sum_{i=1}^4 a_{2i}b_{2i}, \\\
\gt^*_+ & = \{(a_{11}\ccd a_{15},a_{21}\ccd a_{24})\in \gt^*\mid 
a_{11}\leq \cdots \leq a_{15},a_{21}\leq \cdots\leq a_{24} \}.
\end{align*}

Let $N=\dim V$. We choose a coordinate system 
$v=(v_1,\dots ,v_N)$ on $V$ by which $T$ (and so $T_{\text{st}}$ also) 
acts diagonally. 
Let $\gamma_i \in \mathfrak t^{\ast }$ and $\coorda_i$ be 
the weight and the coordinate vector 
which correspond to $i$-th coordinate. 
The reason why we use this notation unlike 
the notation $e_i$ in \cite{tajima-yukie} is that
we would like to use the notation $\coorde_i$
for the coordinate vectors of $\aff^5$. 

Let $\Gam=\{\gam_1\ccd \gam_N\}$. 
For a subset $\gI\subset \Gam$, 
we denote the convex hull of $\mathfrak I$ by $\Conv \mathfrak I$. 
Let $\p(V)$ be the projective space associated with $V$ 
and $\pi_V:V\setminus\{0\}\to \p(V)$ the natural map. 
For $\gI\subset \Gam$ 
such that $0\notin \Conv \mathfrak I$, let $\beta $ be the 
closest point of $\Conv\mathfrak  I$ to the origin. 
Then $\beta $ lies in $\mathfrak t^*_{\mathbb Q}$. 

\begin{defn}
\label{defn:gB-defn}
\begin{itemize}
\item[(1)]
$\gC$ is the set of all $\beta$ obtained in the above manner.
\item[(2)]
$\gB=\gC\cap \mathfrak t_+^*$. 
\end{itemize}
\end{defn}
The GIT stratification is parametrized by $\gB$. 
However, we have to use the set $\gC$ in Section \ref{sec:smaller-pv-s}.

For the \pv{} (\ref{eq:PV}), $N=40$. 
Let $\coorde_i$ be the coordinate vector of 
$\aff^5$ with respect to the $i$-th coordinate
and $\coordf_i$ the coordinate vector of 
$\aff^4$ with respect to the $i$-th coordinate. 
We put $e_{i_1i_2i_3}=(\coorde_{i_1}\wedge \coorde_{i_2})\otimes \coordf_{i_3}$
for $i_1,i_2=1\ccd 5,i_3=1\ccd 4$. The numbering used in \cite{tajima-yukie-GIT1} 
for (\ref{eq:PV}) is as follows. 

\vskip 10pt

\tiny

\begin{center}
 
\begin{tabular}{|c|c|c|c|c|c|c|c|c|c|}
\hline
1 & 2 & 3 & 4 & 5 & 6 & 7 & 8 & 9 & 10 \\
\hline
$e_{121}$ & $e_{131}$ & $e_{141}$ & $e_{151}$ & $e_{231}$ 
& $e_{241}$ & $e_{251}$ & $e_{341}$ & $e_{351}$ & $e_{451}$ \\
\hline
11 & 12 & 13 & 14 & 15 & 16 & 17 & 18 & 19 & 20 \\
\hline
$e_{122}$ & $e_{132}$ & $e_{142}$ & $e_{152}$ & $e_{232}$ 
& $e_{242}$ & $e_{252}$ & $e_{342}$ & $e_{352}$ & $e_{452}$ \\
\hline
21 & 22 & 23 & 24 & 25 & 26 & 27 & 28 & 29 & 30 \\
\hline
$e_{123}$ & $e_{133}$ & $e_{143}$ & $e_{153}$ & $e_{233}$ 
& $e_{243}$ & $e_{253}$ & $e_{343}$ & $e_{353}$ & $e_{453}$ \\
\hline
31 & 32 & 33 & 34 & 35 & 36 & 37 & 38 & 39 & 40 \\
\hline
$e_{124}$ & $e_{134}$ & $e_{144}$ & $e_{154}$ & $e_{234}$ 
& $e_{244}$ & $e_{254}$ & $e_{344}$ & $e_{354}$ & $e_{454}$ \\
\hline
\end{tabular}

\end{center}

\normalsize

\vskip 10pt

With this numbering, let $\mathbbm a_1= e_{121}$, etc., 
and $\gam_i$ be the character corresponding to $\mathbbm a_i$ 
for $i=1\ccd 40$.

For $\be\in\gC$, we define
\begin{align*} 
& Y_{\beta }= \Span \{\coorda_i\,|\, (\gamma_i,\beta )_{*}
\geq (\beta ,\beta )_{*}\},\quad 
Z_{\beta }= \Span \{\coorda_i\,|\, (\gamma_i,\beta )_{*}
=(\beta ,\beta )_{*}\}, \\
& W_{\beta }=\Span \{ \coorda_i\,|\, (\gamma_i,\beta )_{*}
>(\beta ,\beta )_{*}\}
\end{align*}
where $\Span$ is the spanned subspace. 
Clearly $Y_{\beta}=Z_{\beta}\oplus W_{\beta}$.

If $\lambda $ is a 1PS of $G$, we define 
\begin{align*}
P(\lambda ) & = \left  \{p\in G\; \Big |\; 
\lim_{t\rightarrow 0}\lambda (t)p\lambda(t)^{-1} 
\; \textrm{exists} \right \},\; M(\lambda) = Z_G(\lambda) \;\text{(the centralizer)}, \\
U(\lambda ) & = \left  \{p\in G\; \Big |\; 
\lim_{t\rightarrow 0}\lambda (t)p\lambda(t)^{-1}= 1 \right \}. 
\end{align*} 
The group $P(\lambda )$ is a parabolic subgroup of 
$G$ (\cite[p.148]{Springer-LAG}) with Levi part $M(\lambda)$ 
and unipotent radical $U(\lambda)$. 
We put $P_{\beta}=P(\lambda_{\beta})$, 
$M_{\beta }=Z_G(\lambda_{\beta })$  and 
$U_{\beta }=U(\lambda_{\beta })$.

Let $\chi_{\beta}$ be the indivisible rational character of $M_{\beta}$
such that it  is trivial on $T_0$ and that 
the restriction of $\chi_{\beta}^a$ 
to $T_{\text{st}}$ coincides with $b \be$ 
for some positive integers $a,b$. 
There is a slight ambiguity on the domain of definition of $\chi_{\be}$
in \cite{tajima-yukie} and so we specified the domain of 
definition of $\chi_{\be}$ to be $M_{\be}$. 
We put 
\begin{align*}
M^{\text{st}}_{\be} & = (M_{\be}\cap G_{\text{st}})^{\circ} \;
\text{(the identity component)}, \\
G_{\be} & =\{g\in M_{\beta }\,|\,\chi_{\beta}(g)=1\}^{\circ }, \\ 
G_{\text{st},\beta} & =\{g\in M^{\text{st}}_{\beta}\mid \chi_{\be}(g)=1\}^{\circ}
= (G_{\be}\cap G_{\text{st}})^{\circ}. 
\end{align*}
In the situation of Corollary 1.4 \cite[p.264]{tajima-yukie} and 
\cite[p.4]{tajima-yukie-GIT1},  
the above $G_{\be}$ was considered (not explicitly in \cite[p.264]{tajima-yukie})
with $G$ replaced by
$G_1$ in \cite{tajima-yukie} ($G_1$ in \cite{tajima-yukie} is $G_{\text{st}}$ here. 
So the above $G_{\text{st},\beta}$ is $(G^1)_{\be}^{\circ}$ of 
\cite[p.264]{tajima-yukie} and \cite[p.4]{tajima-yukie-GIT1}. 
Let $M^s_{\be}$ be the semi-simple part of $M_{\be}$.

If $(G,V)$ is the \pv{} (\ref{eq:PV}) and 
$M_{\be}$ is in the form (\ref{eq:parabolic-defn}), 
then $M^{\text{st}}_{\be}$ is the group (\ref{eq:M1-defn}). 
Note that $M^{\text{st}}_{\be}$ in (\ref{eq:M1-defn}) is connected.

The group $G_{\beta }$ acts on $Z_{\beta }$. 
Note that $M_{\beta }$ and $G_{\beta }$ are defined over $k$, and 
since $\langle \chi_{\beta},\lambda_{\beta}\rangle$ is a positive
multiple of $\|\beta\|_{*}$,
$M_{\beta}=G_{\beta}\im(\lambda_{\beta})$.
Moreover, if $\nu$ is any rational 1PS in $G_{\beta}$, 
$(\nu,\lambda_{\beta})=0$.

Let $\p(V)^{\semi}$ (resp. $\mathbb P(Z_{\beta })^{\semi}$)
be the set of semi-stable points of $\p(V)$ (resp. 
$\mathbb P(Z_{\beta })$) with respect to 
the action of $G_{\text{st}}$ (resp. $G_{\text{st},\be}$). 
Since there is a difference between $V$ and $\p(V)$
(resp. $Z_{\beta}$ and $\mathbb P(Z_{\beta })$), 
we removed some scalar directions from 
$G$ (resp. $G_{\beta}$) and considered stability with respect to 
$G_{\text{st}}$ (resp. $G_{\text{st},\be}$).  
For the notion of semi-stable points, see \cite{mufoki}. 
We regard $\mathbb P(Z_{\beta })^{\semi}$ as a subset of $\p(V)$. 
Put 
\begin{align*}
& V^{\semi} = \pi_V^{-1}(\p(V)^{\semi}),\; 
Z_{\beta }^{\semi}=\pi_V^{-1}(\mathbb P(Z_{\beta })^{\semi }), \;
Y_{\beta }^{\semi}=\{(z,w)\,|\, z\in Z_{\beta}^{\semi},w\in W_{\beta}\}. 
\end{align*}
We define $S_{\beta }=GY_{\beta }^{\semi }$. 
Note that $S_{\beta }$ can be the empty set. 
We denote the set of $k$-rational points of 
$S_{\beta}$, etc., by $S_{\beta \,k}$, etc.


The following theorem is COROLLARY 1.4 \cite[p.264]{tajima-yukie}. 
\begin{thm} 
\label{KKN} 
Suppose that $k$ is a perfect field. Then  we have 
\begin{align*} 
V_k\setminus \{0\} = V^{\semi } _k
\coprod \coprod_{\beta \in \mathfrak{B}} S_{\beta \, k}. 
\end{align*}
Moreover, $S_{\beta \,k}\cong G_{k}\times_{P_{\beta \,k}} Y_{\beta \,k}^{\semi }$. 
\end{thm}  

Note that the set $\gB$, not $\gC$, is used in the above theorem. 
The point of Corollary 1.4 \cite[p.264]{tajima-yukie} is 
that even though the stability was considered for 
$G_{\text{st}}$ and $G_{\be,\text{st}}$, the inductive structure 
of $S_{\be}$ is with respect to the action of $G$.

%

\section{Regularity}
\label{sec:regularity}

In \cite{kato-yukie-jordan}, 
we discussed the notion of regularity 
mainly for irreducible \pv s. 
The notion of regularity was defined for 
not necessarily irreducible \pv s in 
\cite{kato-yukie-jordan}. In this paper, we have to deal with 
many reducible \pv s and it will be necessary
to clarify the relation between the number of 
independent characters and the number of 
independent relative invariant polynomials. 

Let $G$ be a connected reductive group, 
$Z$ the identity component of the center of $G$,  
$V_1\ccd V_N\not=\{0\}$  finite dimensional 
\rep s of $G$ over $k$  
and $\chi_1\ccd \chi_N$ characters of $Z$ over $k$. 
We assume that $t\in Z$ acts on $V_i$ 
by multiplication by $\chi_i(t)$.  
Let $\Gam \sub X^*(Z)$ be the subgroup generated by 
$\chi_1\ccd \chi_N$. We assume that 
$\Gam\cong \Z^N$. This implies that 
$\{\chi_1\ccd \chi_N\}$ is a basis of $\Gam$. 
Let $V=V_1\oplus \cdots \oplus V_N$. 

\begin{defn}\
\label{defn:PV-defn}
In the above situation,  $(G,V)$ is called a \pv{}   
if it satisfies the following properties.
\begin{itemize}
\item[(1)]
There exists a Zariski open orbit. 
\item[(2)]
There exists a non-constant polynomial $\Del(x)\in k[V]$ 
and a character $\chi$ of $G$  
such that $\Del(gx)=\chi(g)\Del(x)$. 
\end{itemize}
\end{defn}

The polynomial $\Del(x)$ in (2) is called a {\it relative invariant polynomial}.

\begin{lem}
\label{lem:Pi-homogeneos}
If $\Del(x)$ is a relative invariant polynomial then 
it is homogeneous with respect to each of $V_1\ccd V_N$.
\end{lem}
\begin{proof}
We only consider $V_1$. 
Suppose that $\Del(gx)=\chi(g)\Del(x)$. 
Since $\Gam\cong \Z^N$, 
there exists a 1PS $\lam(t)\sub Z$
such that $\chi_1(\lam(t))=t^a$ ($a>0$) 
and $\chi_2(\lam(t))=\cdots =\chi_N(\lam(t))=1$. 
Then for $v_1\in V_1\ccd v_N\in V_N$ and $t\in\gl_1$, 
\begin{equation*}
\Del(\lam(t)(v_1\ccd v_N))
= \chi(\lam(t))\Del(v_1\ccd v_N)
= \Del(t^av_1,v_2\ccd v_N). 
\end{equation*}
There exists an integer $b$ such that 
$\chi(\lam(t))=t^b$. Then 
\begin{equation*}
\Del(t^av_1,v_2\ccd v_N)=t^b\Del(v_1\ccd v_N).
\end{equation*}
Therefore, $\Del(x)$ is homogeneous with respect to $V_1$. 
\end{proof}

The following proposition will be useful 
to show the existence of an open orbit. 
One can prove it 
by the same argument as in \cite[p.321]{kato-yukie-jordan}.  
\begin{prop}
\label{prop:open-orbit}
Suppose that $G$ is an algebraic group, 
$V$ a finite dimensional \rep{} of $G$ both defined over $k$, 
$x\in V$ and that $\dim {\mathrm T}_{e_G}(G_x)=\dim G-\dim V$. 
Then $Gx\sub V$ is Zariski open and the group scheme 
$G_x$ is smooth over $k$. 
\end{prop}
\begin{prop} 
\label{prop:regularity}
Suppose that $G$ is an algebraic group,  
$V$ a finite dimensional \rep{} of $G$ both defined over $k$ 
and that there exists a point 
$w\in \rv$ such that $U=Gw$ is Zariski open and that  
$G_w$ is reductive (and so smooth as a group scheme).
Then 
\begin{itemize}
\item[(1)] $U$ is affine,  
\item[(2)] $U_{k^{\sep}}$ is a single $G_{k^{\sep}}$-orbit.
\item[(3)] There exists a relative invariant polynomial and so 
$(G,V)$ is a \pv. 
\end{itemize}
\end{prop}
\begin{proof}

(1), (2) are proved in Proposition 2.2 \cite[p.310]{kato-yukie-jordan}. 

(3) Let $F_1\ccd F_m$ be all irreducible components of codimension $1$ 
of $V\setminus U$. Since $\bigcup_i F_i$ is $G$-invariant and 
$G$ is irreducible, $F_1\ccd F_m$ are $G$-invariant. 
Let $W=V\setminus \left(\bigcup_i F_i\right)$. Then 
$U\sub W$ are both affine and the codimension of $W\setminus U$
is greater than or equal to $2$. Since $W$ is regular, it is normal. 
Therefore, any regular function on $U$ extends to $W$. 
So $W=U$. 

Since the origin $0$ does not belong to $U$, 
$V\setminus U\not=\emptyset$.  Suppose that $\Del_i(x)\in k[V]$ 
and $F_i$ is the zero set of $\Del_i(x)$ for $i=1\ccd m$.
Since $F_i$ is $G$-invariant, it is also the zero set of $\Del_i(gx)$ 
for all $g\in G$. Therefore, there exists a character 
$\psi_i$ of $G$ such that $\Del_i(gx)=\psi_i(g)\Del_i(x)$. 
Therefore, there exists a relative invariant polynomial. 
This proves (3). 
\end{proof}

\begin{defn}
\label{defn:regularity}
If $(G,V)$ is a \pv{} which satisfies the condition of 
the above proposition then $(G,V)$ is said to be {\it regular}. 
\end{defn}
\begin{prop}
\label{prop:number-of-components}
Suppose that $(G,V)$ is a regular \pv{} 
as in Proposition \ref{prop:regularity} 
and $V\setminus U=F_1\cup \cdots \cup F_m$ is the irreducible decomposition 
($F_1\ccd F_m$ are distinct). Then $m\leq N$.  
\end{prop}
\begin{proof}
Suppose that $m>N$ and that $F_i$ is the zero set of $\Del_i$ for all $i$. 
Since $\Del_i(x)$ is homogeneous with respect to $V_1\ccd V_N$, 
there exist integers $c_1\ccd c_m$ such that $(c_1\ccd c_m)\not=(0\ccd 0)$ 
and $\prod_{i=1}^m\Del_i(tx)^{c_i}=\prod_{i=1}^m\Del_i(x)^{c_i}$ 
for all $t\in Z$.  Since the derived subgroup $[G,G]$ is connected 
semi-simple, $\psi_1\ccd \psi_m$ are trivial on $[G,G]$. 
Since $G=Z\cdot [G,G]$, $\prod_{i=1}^m\Del_i(x)^{c_i}\in k(V)$  
is invariant under the action of $G$. However, since $U=Gw$ 
is Zariski open,  $\prod_{i=1}^m\Del_i(x)^{c_i}\in k^{\times}$.    
This is a contradiction since $k[V]$ is UFD and $F_1\ccd F_m$ 
are distinct.  
\end{proof}

\begin{cor}
\label{cor:reducible-sep-orbit}
In the situation of  Proposition \ref{prop:number-of-components}, 
if $m=N$ then 
$\{x\in V_{k^{\sep}}\mid \Del_1(x)\ccd \Del_N(x)\not=0\}=G_{k^{\sep}}w$. 
\end{cor}

There is an alternative definition of \pv s, specifying a character. 
Let $G,V$ be as in the beginning of this section. 
We fix a non-trivial primitive rational character $\chi$ of $G$. 
The following is an alternative definition of 
\pv s. 

\begin{defn}
\label{defn:PV-alternative}
In the above situation, $(G,V,\chi)$ is
called a \pv{} if it satisfies the properties 
(1) of Definition \ref{defn:PV-defn} and 
there exists a non-constant polynomial $\Del(x)\in k[V]$ 
and a positive integer $a$ such that $\Del(gx)=\chi(g)^a\Del(x)$. 
\end{defn}

$\Del(x)$ in the above definition is called a 
relative invariant polynomial. 
If $(G,V,\chi)$ is a \pv{} then 
$\Del(x)$ in the above definition is 
essentially unique, i.e., 
if $\Del_1(x),\Del_2(x)$ are relative invariant 
polynomials then there exist positive integers
$a,b$ and $c\in \mk$ 
such that $\Del_1(x)^a=c\Del_2(x)^b$.  
We define 
$V^{\sst}=\{x\in V\mid \Del(x)\not=0\}$.
This definition does not depend on the 
choice of $\Del(x)$.

Suppose that $(G,V,\chi)$ is a \pv{} in the above sense. 
Let $Z$ be the center of $G$. 
We assume that there is a split torus $T_0\sub Z$ 
which acts on $V$ by scalar multiplication by 
a non-trivial character $\psi$. 
Let $G_{\text{st}}=(\kernel(\chi))^{\circ}$. 
Since $\chi,\psi$ are non-trivial, 
$\dim G_{\text{st}}=\dim G-1$ and $\psi$ is surjective.
Therefore, $G=T_0G_{\text{st}}$. We can choose $T_0$ so that 
$G_{\text{st}}\cap T_0$ is finite. 
Then this is the situation of Section \ref{sec:notation-related-git}.
Moreover, $\Del(x)$ of 
Definition \ref{defn:PV-alternative} is 
indeed an invariant polynomial with respect to the action of 
$G_{\text{st}}$. Therefore, points of $V^{\sst}$ are pull backs from 
semi-stable points of $\p(V)$ in the sense 
of GIT.

Suppose that the condition of Corollary \ref{cor:reducible-sep-orbit} 
is satisfied. If moreover, a positive power of 
$\chi$ is $\psi_1^{a_1}\cdots \psi_m^{a_m}$ 
with $a_1\ccd a_m>0$ then $\Del(x)$ is a constant multiple of 
a positive power of $\Del_1(x)^{a_1}\cdots \Del_m(x)^{a_m}$. 
The converse is true, i.e., if 
$\Del(x)$ is a constant multiple of a positive power of  
$\Del_1(x)^{a_1}\cdots \Del_m(x)^{a_m}$ then 
a positive power of $\chi$ is $\psi_1^{a_1}\cdots \psi_m^{a_m}$. 
Therefore, $V^{\sst}=\{x\in V\mid \Del_1(x)\ccd \Del_m(x)\not=0\}$. 
So $V^{\sst}_{k^{\sep}}=G_{k^{\sep}}w$ is a single $G_{k^{\sep}}$-orbit.
It turns out later that in our case the condition of 
Corollary \ref{cor:reducible-sep-orbit} and this additional condition
are satisfied for all 
$(M_{\be},Z_{\be})$ where we have to verify 
$Z^{\sst}_{\be\,k^{\sep}}$ is a single $M_{\be\,k^{\sep}}$-orbit.

For the rest of this section, we review the notion of 
universally generic element in \cite{kable-yukie-quinary}. 
We consider a slightly more general situation than in 
\cite{kable-yukie-quinary}.

Let $S$ be a finite set of primes in $\Z$. 
Let $R$ be the ring generated over $\Z$ by 
$\{p^{-1}\mid p\in S\}$. 
Let $G$ be a smooth group scheme over $R$ 
with connected geometric fibers, 
$V$ a free $R$-module of finite rank, 
$\gl_R(V)$ the group scheme over $R$ 
of $R$-module isomorphims of $V$ 
and $G\to \gl_R(V)$ a homomorphism. We regard $V$ 
as  a scheme over $R$. 
If $k$ is an algebraically closed field and 
$R\to k$ is a homomorphism then 
$G\times_R k$ is an algebraic group over $k$ 
and $V\times_R k$ is a \rep{} of $G\times_R k$ 
over $k$.  Suppose that $w\in V$. 
We denote the image of $w$ in  
$V\times_R k$ by $w(k)$. 

\begin{defn}
\label{defn:universally-generic}
If for any $k$ as above $(G\times_R k) w(k)\sub V\times_R k$ 
is Zariski open then we say that 
$w$ is {\it universally generic} outside $S$. 
If $S=\emptyset$ then we say $w$ is universally generic.  
\end{defn}

The following proposition can be proved 
in the same manner as in Proposition 1 
\cite[p. 279]{kable-yukie-quinary} 
and so we do not provide the proof. 

\begin{prop}
\label{prop:universally-generic}
Suppose that $S,R,G,V,w$ are as above 
and that $w$ is universally generic outside $S$. 
Suppose that $(\pi,W)$ is a finite dimensional 
\rep{} of $G\times_R \Q$ defined over $R$. If 
$\Phi:V\times_R \Q\to W\times_R \Q$ is a 
$G\times_R \Q$-equivariant morphism and that 
$\Phi(w)\in W_R$. Then $\Phi$ is defined over $R$. 
\end{prop}

We make a remark on the notion of Castling transform. 
Suppose that $0<m<n$ are integers (then $n>1$). 
Let $G$ be an algebraic group and $V$ an $n$-dimension \rep{} 
of $G$ over $k$.  Let $\aff^m,\aff^{n-m}$ be the standard \rep s 
of $\gl_m,\gl_{n-m}$ respectively. 
Let $G_1=G\times \spl_m$, $G_2=G\times \spl_{n-m}$.
Then $(G_1,V\otimes \aff^m)$ and 
$(G_2,V^*\otimes \aff^{n-m})$ 
are called {\it Castling transforms} of each other. 
Let 
\begin{align*}
U_1 & =\{x=(x_1\ccd x_m)\in V\otimes \aff^m\mid 
x_1\ccd x_m\;\text{are linearly independent}\}, \\
U_2 & =\{y=(y_1\ccd y_{n-m})\in V^*\otimes \aff^{n-m}\mid 
y_1\ccd y_{n-m}\;\text{are linearly independent}\}. 
\end{align*}
Then the actions of $\spl_m,\spl_{n-m}$ on $U_1,U_2$ 
are free respectively, $\spl_m\backslash U_1$, 
$\spl_{n-m}\backslash U_2$ are affine and the 
coordinate rings are generated by Pl\"ucker coordinates
(see Theorem 2.1 \cite[p.20]{dolgachev-invariant} for the proof 
over any field).

Since the codimensions of $U_1,U_2$ are greater than $1$, 
regular functions on $U_1,U_2$ can be regarded as elements 
of $k[V\otimes \aff^m],k[V^*\otimes \aff^{n-m}]$ respectively. 
Since Pl\"ucker coordinates on 
$V\otimes \aff^m,V^*\otimes \aff^{n-m}$
coincide except possibly with the choice of 
signs, $\spl_m\backslash U_1$, 
$\spl_{n-m}\backslash U_2$ are isomorphic. Since 
$\spl_m,\spl_{n-m}$ have no non-trivial characters, 
relative invariant polynomials on 
$V\otimes \aff^m,V^*\otimes \aff^{n-m}$  
correspond  bijectively and $G_1,G_2$-orbits in 
$U_1,U_2$ correspond bijectively also.  
Since $\h^1(k,\spl_m),\h^1(k,\spl_{n-m})$ are
trivial, the correspondence of orbits 
is bijective rationally over $k$. 
By the above consideration, 
$(G_1,V\otimes \aff^m)$ is a \pv{}  
if and only if $(G_2,V^*\otimes \aff^{n-m})$ 
is a \pv. If so, since the stabilizers of corresponding 
points are isomorphic, the regularity 
of these \pv s coincide. 
The reader should see Proposition 18 \cite[p.68]{saki} 
for details (the argument there works over any field). 

If both $(G_1,V\otimes \aff^m)$, 
$(G_2,V^*\otimes \aff^{n-m})$ are \pv s
then the number, say $l$, of distinct invariant 
hypersurfaces for these \pv s coincide. 
If $\Del_1\ccd \Del_l$ are distinct relative 
invariant polynomials on $V\otimes \aff^m$ 
and $\Del_1^*\ccd \Del_l^*$ are the corresponding 
relative invariant polynomials on 
$V^*\otimes \aff^{n-m}$  then 
$G_1$-orbits in $\{x\in V\otimes \aff^m\mid \Del_1(x)\ccd \Del_l(x)\not=0\}$
and  $G_2$-orbits in 
$\{y\in V^*\otimes \aff^{n-m}\mid \Del^*_1(y)\ccd \Del^*_l(y)\not=0\}$
correspond  bijectively.  In particular  
$\{x\in V\otimes \aff^m\mid \Del_1(x)\ccd \Del_l(x)\not=0\}$ is a single 
$G_1$-orbit if and only if 
$\{y\in V^*\otimes \aff^{n-m}\mid \Del^*_1(y)\ccd \Del^*_l(y)\not=0\}$
is a single $G_2$-orbit. If $l=1$ then 
$G_{1\,k}\backslash (V\otimes \aff^m)^{\sst}_k$ 
and $G_{2\,k}\backslash (V^*\otimes \aff^{n-m})^{\sst}_k$ 
are in bijective correspondence. 
 
Note that any relative invariant polynomial 
on $V\otimes \aff^m$ is invariant for the action of 
$\spl_m$. The situation is similar for $V^*\otimes \aff^{n-m}$.
The only issue we have to be careful is 
the correspondence of rational orbits. 
Let $N$ be the binomial number 
\begin{math}
\left(
\begin{smallmatrix}
n \\m 
\end{smallmatrix}
\right).
\end{math}
If $R$ is the quotient of $k[z_1\ccd z_N]$ by Pl\"ucker relations, 
both $\gl_m\backslash U_1,\gl_{n-m}\backslash U_2$ 
are isomorphic to $\mathrm{Proj}(R)$ this time rather than 
$\mathrm{Spec}(R)$.  Since $\h^1(k,\gl_m),\h^1(k,\gl_{n-m})$
are trivial, this bijection is rational over $k$ also. 
Note that if $x=(x_1\ccd x_m)\in U_1$ then 
$H=\lan x_1\ccd x_m\ran\sub V$ is a subspace of dimension $m$.
Then $\{f\in V^*\mid {}^{\forall}x\in H, f(x)=0\}\sub V^*$ 
is a subspace of dimension $n-m$. By taking a basis 
$\{y_1\ccd y_{n-m}\}$, the orbit of $y=(y_1\ccd y_{n-m})$ 
is the orbit corresponding to the orbit of $x$. 
By this correspondence, 
($G_k\times \gl_m(k)$)-orbits in 
$\{x\in V_k\otimes k^m\mid \Del_1(x)\ccd \Del_m(x)\not=0\}$
and ($G_k\times \gl_{n-m}(k)$)-orbits in  
$\{y\in V^*_k\otimes k^{n-m}\mid \Del^*_1(y)\ccd \Del^*_m(y)\not=0\}$
correspond  bijectively also.

\section{Rational orbits (1)}
\label{sec:rational-orbits-general}

In this section and the next six sections, we consider
rational orbits of some \pv s which appear as 
$(M_{\be_i},Z_{\be_i})$. 
We describe known cases of rational orbits of 
some \pv s in this section. 
Even tough the results are known for all cases 
(see \cite{wryu}) except for the last case, 
we will have to consider \pv s which in some sense 
contain the \pv s in this section. 
In order to show results in later sections, we need Lie algebra 
computations for the \pv s in this section also. 
Therefore,  we reprove some known results in 
\cite{wryu} in a slightly different manner using the
notion of regularity.  
In the next six sections, we consider cases which require
more labor. 

For the rest of this paper, 
$\{\bbmp_{n,1}\ccd \bbmp_{n,n}\}$ is the standard basis of $\aff^n$. 
We put $p_{n,i_1\cdots i_m}=\bbmp_{n, i_1}\wedge \cdots\wedge \bbmp_{n, i_m}$
for $1\leq i_1\ccd i_m\leq n$. 
When we have to consider more than one standard \rep s of 
$\gl_n$'s, we may use different letters such as 
$\mathbbm q_{n,i},\mathbbm r_{n,i}$ to avoid confusion. 
In that case, we use notation such as  
$q_{n,i_1\cdots i_m}$, $r_{n,i_1\cdots i_m}$. 
We identify $\wedge^n \aff^n$ with $\aff^1$, 
so that $p_{n,12\cdots n},q_{n,12\cdots n}$, etc., correspond to $1$.

Let $\Ex_n(k)$, $\mathrm{Prg}_n(k),\text{QF}_n(k)$ be as in 
Definition \ref{defn:Ex-defn}.

\subsection{Case I}

We consider the natural action of $G=\gl_2^3$ on
$V=\aff^2\otimes \aff^2\otimes \aff^2$.  
Let 
\begin{equation}
\label{eq:w-efn-222}
w = \bbmp_{2,1}\otimes \bbmp_{2,1}\otimes \bbmp_{2,1}
+ \bbmp_{2,2}\otimes \bbmp_{2,2}\otimes \bbmp_{2,2}\in V.
\end{equation}
We identify $V$ with the space of pairs of $2\times 2$ matrices
so that $g=(g_1,g_2,g_3)$ acts on $(A_1,A_2)$ by 
\begin{math}
g_3 \left(
\begin{smallmatrix}
g_1 A_1 {}^tg_2 \\
g_1 A_2 {}^tg_2 \\
\end{smallmatrix}
\right)
\end{math}
treating $[g_1 A_1 {}^tg_2,g_1 A_2 {}^tg_2]$ as a column vector. 
Then we can identify $w$ with the element
\begin{equation}
\label{eq:w-efn-222-matrix}
(w_1,w_2) = \left(
\begin{pmatrix}
1 & 0 \\
0 & 0 
\end{pmatrix},
\begin{pmatrix}
0 & 0 \\
0 & 1 
\end{pmatrix}
\right).
\end{equation}

Note that $k[\vep]/(\vep^2)$ is the ring of dual numbers.
Let $X=(x_{ij}),Y=(y_{ij}),Z=(z_{ij})\in\m_2$. 
Then it is easy to show the following proposition.
We do not provide the details.
\begin{prop}
\label{prop:ie-alg-222}
$(e_G+\vep(X,Y,Z))w=w$ 
if and only if $X,Y,Z$ are diagonal matrices and 
$x_{11}+y_{11}+z_{11}=x_{22}+y_{22}+z_{22}=0$. 
So $\dim {\mathrm T}_{e_G}(G_w)=4$. 
\end{prop}

Since $\dim {\mathrm T}_{e_G}(G_w)=4=\dim G-\dim V$, 
Proposition \ref{prop:open-orbit} implies that 
$Gw\sub V$ is Zariski open and that $G_w$ is smooth over $k$. 

We regard elements of $V$ as $2\times 2$ matrices with entries 
in linear forms of $v=(v_1,v_2)$ where the last $\aff^2$ factor 
corresponds to $v$. 
For $x=x(v)\in V$, let 
$F_x(v)$ its determinant and $P(x)$ the discriminant of 
$F_x(v)$. It is easy to see that $F_w(v)=v_1v_2$. 
The following proposition is easy and we do not provide the details. 
\begin{prop}
\label{prop:222-invariant}
$P(x)$ is a homogeneous degree $4$ polynomial on $V$,   
$P((g_1,g_2,g_3)x)=(\det g_1\det g_2\det g_3)^2 P(x)$ 
and $P(w)=1$.  
\end{prop} 

If $x\in V$, $P(x)\not=0$ then the field generated by roots 
of $F_x(v)$ belongs to $\Ex_2(k)$.  
Suppose that $F/k$ is a separable quadratic extension. 
We choose $a=(a_1,a_2)\in k^2$ so that $f_a(t)= t^2+a_1t+a_2\in k[t]$ 
is irreducible with distinct roots $\al=(\al_1,\al_2)$ 
and that $F=k(\al_1)$. 

Let $\tau_0$ be the element in (\ref{eq:J-defn}). 
We put
\begin{equation}
\label{eq:S20-wa-defn}
\begin{aligned}
& \tau  = (\tau_0,\tau_0,\tau_0), \;
\tau'  = (\tau_0,-\tau_0,-\tau_0), \\
& x_F = \left(
\begin{pmatrix}
0 & 1 \\
1 & a_1
\end{pmatrix},
\begin{pmatrix}
1 & a_1 \\
a_1 & a_1^2-a_2
\end{pmatrix}
\right), \;
h_{\al} = 
\begin{pmatrix}
1 & -1 \\
-\al_1 & \al_2
\end{pmatrix} 
\end{aligned}
\end{equation}
and $g_{\al}=(-(\al_1-\al_2)^{-1}h_{\al},h_{\al},h_{\al})$. 
Then $\tau,\tau'$ fix $w$. 
Also, $w$ (resp. $x_F$) corresponds to the trivial extension 
(resp. quadratic extension $F$) of $k$ and  $x_F=g_{\al}w$. 

\begin{prop}
\label{prop:222-regular-sect}
\begin{itemize}
\item[(1)]
$G^{\circ}_w$ consists of 
elements of the form 
\begin{equation}
\label{eq:fix-222}
(\diag(t_1,t_2),\diag(t_3,t_4),
\diag((t_1t_3)^{-1},(t_2t_4)^{-1})).
\end{equation}
\item[(2)]
$G_w=G_w \rtimes \{1,\tau\}=G_w \rtimes \{1,\tau'\}\cong \gl_1^4 \rtimes \Z/2\Z$.
\end{itemize}
\end{prop}
\begin{proof}
(1) It is easy to see that elements of the form (\ref{eq:fix-222}) fix $w$. 
Since $\dim {\mathrm T}_{e_G}(G_w)=4$, 
$G^{\circ}_{w}$ consists of elements of the form (\ref{eq:fix-222}).

(2) Suppose that $g=(g_1,g_2,g_3)\in G_w$. 
Then $(\det g_1)(\det g_2)g_3$ fixes $F_w(v)=v_1v_2$. 
If we replace $g$ by $g\tau$, $g_3$ is replaced by $g_3\tau_0$. 
It is well-known that multiplying $\tau_0$ to
$g_3$ if necessary, we may assume that $g_3$ is a diagonal matrix. 
Then $g_1 w_i {}^t g_2$ is a scalar multiple of $w_i$ for $i=1,2$. 

We can construct two other equivariant maps to $\sym^2 \aff^2$ 
using the first and the second factors of $\aff^2$. So 
each of $g_1,g_2$ is a diagonal matrix or a diagonal matrix times $\tau_0$. 
The only possibility is both $g_1,g_2$ are diagonal matrices.  
Then it is easy to see that $g\in G^{\circ}_w$. 
Since $\tau'=\tau (I_2,-I_2,-I_2)$ and $(I_2,-I_2,-I_2)\in G^{\circ}_w$, 
the last statement follows. 
\end{proof}

Let $\sig\in\gal(F/k)$ be the non-trivial element.
As in Proposition 3.6 \cite[p.305]{wryu}, 
one can show that 
\begin{align*}
G^{\circ}_{x_F} & = g_{\al}\{(\diag(t_1,t_1^{\sig}),\diag(t_3,t_3^{\sig}),
\diag((t_1t_3)^{-1},(t_1^{\sig}t_3^{\sig})^{-1}))
\mid t_1,t_2\in {\mathrm R}_{F/k}\gl_1\}
g_{\al}^{-1} \\
& \cong ({\mathrm R}_{F/k}\gl_1)^2.
\end{align*}
Note that $g_{\al}^{\sig}=g_{\al}\tau'$. 

Since $G_w$ is reductive and 
$\h^1(k,G^{\circ}_w),\h^1(k,G^{\circ}_{x_F})$
are trivial, the following proposition follows without any assumption 
on $\ch(k)$
(see Theorem 5.3 \cite[p.310]{wryu}, LEMMA (1.8) \cite[p.120]{yukiel}). 

\begin{prop}
\label{prop:222-rat-orbits}
$(G,V)$ a regular \pv{} and $G_k\backslash V^{\sst}_k$ 
is in bijective correspondence  with $\Ex_2(k)$ 
by associating the field generated by roots of $F_x(v)$ 
to $x\in V^{\sst}_k$. 
\end{prop}

\subsection{Case II}

We consider the natural action of $G=\gl_3^2\times \gl_2$ on 
$V=\aff^3\otimes \aff^3\otimes \aff^2$. We can identify 
$V$ with the space of pairs of $3\times 3$ matrices. 

Let 
\begin{equation}
\label{eq:sec6-w-defn-sec4}
\begin{aligned}
w_1 & = 
\begin{pmatrix}
1 & 0 & 0 \\
0 & -1 & 0 \\
0 & 0 & 0 
\end{pmatrix}, \; 
w_2 = 
\begin{pmatrix}
0 & 0 & 0 \\
0 & 1 & 0 \\
0 & 0 & -1 
\end{pmatrix}
\end{aligned}
\end{equation}
and $w=(w_1,w_2)\in V$. 

Let $\{\bbmp_{3,1},\bbmp_{3,2},\bbmp_{3,3}\}$, 
$\{\bbmq_{3,1},\bbmq_{3,2},\bbmq_{3,3}\}$ be the standard bases
for the first and the 
second $\aff^3$ respectively. Let $\{\bbmr_{2,1},\bbmr_{2,2}\}$ 
be the standard basis of $\aff^2$. 
We regard 
\begin{align*}
w_1 & =\bbmp_{3,1}\otimes \bbmq_{3,1} \otimes \bbmr_{2,1} 
-\bbmp_{3,2}\otimes \bbmq_{3,2}\otimes \bbmr_{2,1}, \\
w_2 & =\bbmp_{3,2}\otimes \bbmq_{3,2}\otimes \bbmr_{2,2}
-\bbmp_{3,3}\otimes \bbmq_{3,3} \otimes \bbmr_{2,2}. 
\end{align*}

Note that $k[\vep]/(\vep^2)$ is the ring of dual numbers as before.
We determine the Lie algebra of ${\mathrm T}_{e_G}(G_w)$. 
Let $X=(x_{ij}),Y=(y_{ij})\in\m_3,Z=(z_{ij})\in\m_2$. 
We express elements of ${\mathrm T}_{e_G}(G)$ in the form
$e_G+\vep(X,Y,Z)$.  

\begin{lem}
\label{lem:S3-lie-alg}
$\dim {\mathrm T}_{e_G}(G_w)=4$. 
\end{lem}
\begin{proof}
Let $X,Y,Z$ be as above. 
Then $(e_G+\vep(X,Y,Z))w=w$ if and only if
\begin{align*}
& Xw_{21}+w_{21}{}^tY + z_{11}w_{21}+z_{12}w_{22}=0, \\
& Xw_{22}+w_{22}{}^tY + z_{21}w_{21}+z_{22}w_{22}=0. 
\end{align*}
This implies that $X,Y,Z$ are diagonal matrices and 
\begin{equation}
\label{eq:S3-lie-alg}
\begin{aligned}
& x_{11}+y_{11}+z_{11}=0,\;
x_{22}+y_{22}+z_{11}=0, \\
& x_{22}+y_{22}+z_{22}=0,\;
x_{33}+y_{33}+z_{22}=0
\end{aligned}
\end{equation}
(we do not provide the details).  
\end{proof}

It is easy to see that the following subgroup of $G$
\begin{equation}
\label{eq:stab-332}
\left\{
(\diag(t_{11},t_{12},t_{13}),\diag(t_{21},t_{22},t_{23}),t_3I_2)\in G\,\vrule\,
t_{11}t_{21}=
t_{12}t_{22}=
t_{13}t_{23}=t_3^{-1}
\right\}
\end{equation}
fix $w$. 
By Lemma \ref{lem:S3-lie-alg}, 
$\dim {\mathrm T}_{e_G}(G_w)=4=\dim G-\dim V$. So Proposition \ref{prop:open-orbit} 
implies that $Gw\sub V$ is Zariski open and that $G_w$ is smooth over $k$. 
Therefore, we obtain the following corollary. 
\begin{cor}
\label{cor:S3-regular}
$G^{\circ}_w$ coincides with the group in (\ref{eq:stab-332}) 
and $G^{\circ}_w\cong \gl_1^4$. Therefore, $(G,V)$ is a regular \pv. 
\end{cor} 
Since $k$ is arbitrary and  $G,V,w$ are defined over $\Z$, 
$w$ is universally generic 
(see Definition \ref{defn:universally-generic}).

Let 
\begin{equation}
\label{eq:tau12-defn}
\begin{aligned}
\tau_1 & = 
\left(
\begin{pmatrix}
0 & 1 & 0 \\
1 & 0 & 0 \\
0 & 0 & 1 
\end{pmatrix},
\begin{pmatrix}
0 & 1 & 0 \\
1 & 0 & 0 \\
0 & 0 & 1 
\end{pmatrix},
\begin{pmatrix}
-1 & 0 \\   
1 & 1 
\end{pmatrix}\right), \\
\tau_2 & = 
\left(
\begin{pmatrix}
1 & 0 & 0 \\
0 & 0 & 1 \\
0 & 1 & 0 
\end{pmatrix},
\begin{pmatrix}
1 & 0 & 0 \\
0 & 0 & 1 \\
0 & 1 & 0 
\end{pmatrix},
\begin{pmatrix}
1 & 1 \\
0 & -1 
\end{pmatrix}\right).
\end{aligned}
\end{equation}
Then $\tau_1,\tau_2$ fix $w$ and generate 
a subgroup of $G_k$ isomorphic to $\gS_3$. 
We denote the subgroup of $G_k$ generated by 
$\tau_1,\tau_2$ by $H$.

Let $x=(x_1,x_2)\in V$. Regarding $x_2$ as a $3\times 3$ matrix with 
entries in linear forms in two variables $v=(v_1,v_2)$, let $F_x(v)$ 
be its determinant. Let $\zero(x)\sub \p^1$ be the zero set of
$F_x(v)$. Then $F_w(v)=v_1v_2(v_1-v)$ and $\zero(w)=\{(1,0),(0,1),(1,1)\}$.

\begin{prop}
\label{prop:S3-stabilizer}
$G_w=G^{\circ}_w\rtimes H\cong G^{\circ} \rtimes \gS_3$. 
\end{prop}
\begin{proof}
Suppose that $g=(g_1,g_2,g_3)\in G_w$ where 
$g_1=(a_{ij}),g_2=(b_{ij}),g_3=(c_{ij})$.  
Since $F_{gx}(v) = (\det g_1)(\det g_2)F_x(vg_3)$, $g_3$ induces a 
permutation of $(1,0),(0,1),(1,1)\in \p^1$. 
Multiplying an element of $H$, we may assume that $g_3$ fixes 
$(1,0),(0,1),(1,1)\in \p^1$. Then $g_3$ is a scalar matrix. 
This implies that $g_1 w_i {}^t g_2$ is a scalar multiple of 
$w_i$ for $i=1,2$. 

By considering the $(1,3),(2,3),(3,1),(3,2),(3,3)$-entries 
(resp. $(1,1),(1,2),(1,3)$, $(2,1),(3,1)$-entries) of 
$g_1 w_1 {}^t g_2$ (resp. $g_1 w_2 {}^t g_2$), 
$a_{12}=a_{13}=a_{31}=a_{32}=b_{12}=b_{13}=b_{31}=b_{32}=0$. 
Then by the $(1,2),(2,1)$-entries (resp. $(2,3),(3,2)$-entries) 
of $g_1 w_1 {}^t g_2$ (resp. $g_1 w_2 {}^t g_2$), 
$a_{21}=a_{23}=b_{21}=b_{23}=0$. Therefore, $g_1,g_2$ are diagonal matrices
and $g\in G^{\circ}_w$. Since $G^{\circ}_w$ is a normal subgroup pf 
$G_w$, the proposition follows. 
\end{proof}

Let $P(x)$ be the discriminant of $F_x(v)$. 
The following proposition is obvious. 
\begin{prop}
\label{prop:332-invariant}
$P(x)$ is a homogeneous degree $12$ polynomial on $V$,   
$P((g_1,g_2,g_3)x)$ $=(\det g_1\det g_2)^4(\det g_3)^6 P(x)$ 
and $P(w)=1$.
\end{prop}

We describe rational orbits together with the next case
in Proposition \ref{prop:Rational-orbits-3-cases}.

\subsection{Case III}

We next consider the natural action of $G=\gl_4\times \gl_2$ on 
$V=(\wedge^2 \aff^4)\otimes \aff^2$. 
We identify $V$ with the space of 
pairs of $4\times 4$ alternating matrices. 
Let $\{\bbmp_{4,1}\ccd \bbmp_{4,4}\}$ (resp. $\{\bbmq_{2,1},\bbmq_{2,2}\}$)
be the standard basis of $\aff^4$ (resp. $\aff^2$).  
We use the notation such as $p_{4,12}=\bbmp_{4,1}\wedge \bbmp_{4,2}$, etc. 

Let 
\begin{equation}
\label{eq:sec7-w-defn-sec4}
\begin{aligned}
w & = p_{4,12}\otimes \bbmq_{2,1}+p_{4,34}\otimes \bbmq_{2,2} \\
& = \left(
\begin{pmatrix}
0 & 1 & 0 & 0 \\
-1 & 0 & 0 & 0 \\
0 & 0 & 0 & 0 \\
0 & 0 & 0 & 0
\end{pmatrix},
\begin{pmatrix}
0 & 0 & 0 & 0 \\
0 & 0 & 0 & 0 \\
0 & 0 & 0 & 1 \\
0 & 0 & -1 & 0
\end{pmatrix}
\right)\in V. 
\end{aligned}
\end{equation}

We express elements of ${\mathrm T}_{e_G}(G)$ 
as $e_G+\vep(X,Y)$ where $X=(x_{ij})\in\m_4$, $Y=(y_{ij})\in\m_2$. 
\begin{lem}
\label{lem:section7-Lie-alg}
$\dim {\mathrm T}_{e_G}(G_w)=8$. 
\end{lem}
\begin{proof}
It is easy to see that 
$(e_G+\vep(X,Y))w=w+\vep(X',Y')$ where 
\begin{align*}
X' & = X
\begin{pmatrix}
0 & 1 & 0 & 0 \\
-1 & 0 & 0 & 0 \\
0 & 0 & 0 & 0 \\
0 & 0 & 0 & 0
\end{pmatrix}
+ \begin{pmatrix}
0 & 1 & 0 & 0 \\
-1 & 0 & 0 & 0 \\
0 & 0 & 0 & 0 \\
0 & 0 & 0 & 0
\end{pmatrix}
{}^t X 
+ \begin{pmatrix}
0 & y_{11} & 0 & 0 \\
-y_{11} & 0 & 0 & 0 \\
0 & 0 & 0 & y_{12} \\
0 & 0 & -y_{12} & 0
\end{pmatrix}  \\
& = \begin{pmatrix}
0  & x_{11}+x_{22}+y_{11} & x_{32} & x_{42} \\ 
-x_{11}-x_{22}-y_{11}  & 0 & -x_{31} & -x_{41} \\ 
-x_{32}  & x_{31} & 0 & y_{12} \\
-x_{42}  & x_{41} & -y_{12} & 0 \\
\end{pmatrix}, \\
Y' & = X
\begin{pmatrix}
0 & 0 & 0 & 0 \\
0 & 0 & 0 & 0 \\
0 & 0 & 0 & 1 \\
0 & 0 & -1 & 0
\end{pmatrix}
+ \begin{pmatrix}
0 & 0 & 0 & 0 \\
0 & 0 & 0 & 0 \\
0 & 0 & 0 & 1 \\
0 & 0 & -1 & 0
\end{pmatrix}
{}^t X 
+ \begin{pmatrix}
0 & y_{21} & 0 & 0 \\
-y_{21} & 0 & 0 & 0 \\
0 & 0 & 0 & y_{22} \\
0 & 0 & -y_{22} & 0
\end{pmatrix}
\end{align*}

\begin{align*}
& = \begin{pmatrix}
0  & y_{21} & -x_{14} & x_{13} \\ 
-y_{21}  & 0 & -x_{24} & x_{23} \\ 
x_{14}  & x_{24} & 0 & x_{33}+x_{44}+y_{22} \\
-x_{13}  & -x_{23} & -x_{33}-x_{44}-y_{22} & 0 \\
\end{pmatrix}.
\end{align*}

So $e_G+\vep(X,Y)\in {\mathrm T}_{e_G}(G_w)$
if and only if $X,Y$ are in the following form 
\begin{equation}
\label{eq:42-lie-alg-sec4}
X = 
\begin{pmatrix}
x_{11} & x_{12} & 0 & 0 \\
x_{21} & x_{22} & 0 & 0 \\
0 & 0 & x_{33} & x_{34} \\
0 & 0 & x_{43} & x_{44} 
\end{pmatrix},\; 
Y = 
\begin{pmatrix}
-x_{11}-x_{22} & 0 \\
0 & -x_{33}-x_{44} 
\end{pmatrix}
\end{equation}
\end{proof}

Let $\tau_0$ be as in (\ref{eq:J-defn}) and 
\begin{equation*}
\tau_1 = 
\begin{pmatrix}
0 & I_2 \\
I_2 & 0 
\end{pmatrix},\;
\tau=(\tau_1,\tau_0). 
\end{equation*}
Then $\tau$ fixes $w$. 

\begin{prop}
\label{prop:S5-regular}
\begin{itemize}
\item[(1)]
$Gw\sub V$ is Zariski open and $G_w$ is smooth
over $k$.
\item[(2)]
$G^{\circ}_w\cong \gl_2\times \gl_2$. 
\item[(3)]
$G_w = G^{\circ}_w \rtimes \{1,\tau\} \cong (\gl_2\times \gl_2)\rtimes \Z/2\Z$.
\end{itemize}
\end{prop}
\begin{proof}
(1) Since $\dim {\mathrm T}_{e_G}(G_w)=4=24-20=\dim G-\dim V$, 
$Gw\sub V$ is Zariski open and $G_w$ is smooth
over $k$ by Proposition \ref{prop:open-orbit}. 

(2) If $h_1,h_2\in \gl_2$ then elements of the form
\begin{equation*}
(\diag(h_1,h_2),\diag(\det(h_1)^{-1},\det(h_2)^{-1}))
\end{equation*}
fix $w$. Since the dimension of the set of 
such elements is $8$, $G^{\circ}_w$ consists of
such elements and $G^{\circ}_w\cong \gl_2\times \gl_2$.

(3) We may assume that $k=\overline k$. 
Suppose that $g=(g_1,g_2)\in G_w$. 
Note that $\spl_2\times \spl_2\sub G^{\circ}_w$. 
Conjugation by $g$ induces an automorphism of $G^{\circ}_w$ 
and on $\spl_2\times \spl_2\sub G^{\circ}_w$. 
By multiplying $\tau$ if necessary, we may assume that 
$g$ induces an inner automorphism of $\spl_2\times \spl_2$.  
By multiplying an element of $\spl_2\times \spl_2$ if necessary,
we may assume that $g$ commutes with elements of 
$\spl_2\times \spl_2$.  So $g$ is in block form 
and Schur's lemma implies that 
$g_1$ is in the form $\diag(\al I_2,\be I_2)$. 
Since the two components of $w$ are linearly independent,
$g_2$ is a diagonal matrix. Then it is easy to see that 
$g\in G^{\circ}_w$. 
\end{proof}

\begin{cor}
\label{cor:S5-regular}
$(G,V)$ is a regular \pv.  
\end{cor}

We regard $x$ as a $4\times 4$ alternating matrix with 
entries in linear forms in two variables $v=(v_1,v_2)$.  
Let $F_x(v)$ be its Pfaffian. It is a quadratic form of $v$. 
Let $P(x)$ be the discriminant of $F_x(v)$. 
The following proposition is easy to verify
(see \cite[p.306]{wryu} also). 

\begin{prop}
\label{prop:42-invariant}
$P(x)$ is a homogeneous degree $4$ polynomial on $V$
such that $P((g_1,g_2)x)=(\det g_1)^2(\det g_2)^2 P(x)$ 
and $P(w)=1$.  
\end{prop}

Let $(G,V)$ be the \rep{} of 
Corollaries \ref{cor:S3-regular}, \ref{cor:S5-regular}. 
Since these are regular \pv, 
$V^{\sst}_{k^{\sep}}$ is a single $G_{k^{\sep}}$-orbit. 
Then one can use the usual argument of Galois cohomology. 
We do not provide the details, but one can determine 
the stabilizers as in \S 3,4 \cite{wryu}. 
Then LEMMA (1.8) \cite[p.120]{yukiel}) implies the  
following proposition without any assumption on 
$\ch(k)$. 

\begin{prop}
\label{prop:Rational-orbits-3-cases}
Let $(G,V)$ be the \rep{} of 
Corollaries \ref{cor:S3-regular}, \ref{cor:S5-regular}. 
Then $G_k\backslash V^{\sst}_k$ is in bijective correspondence with 
$\Ex_3(k)$ or $\Ex_2(k)$ respectively 
by associating the field generated by roots of $F_x(v)$ 
to $x\in V^{\sst}_k$.  
\end{prop}

\subsection{Case IV}

We consider the natural action of 
$G=\gl_3\times \gl_2^2$ on $V=\aff^3\otimes \m_2$.
We identify $V$ with the space of elements of the form 
$(A_1,A_2,A_3)$ with $A_1,A_2,A_3\in\m_2$. 
Let 
\begin{equation}
\label{eq:322-generator2}
w=\mathbbm p_{3,1}\otimes (-q_{11}+q_{22})
+\mathbbm p_{3,2}\otimes q_{12}
+\mathbbm p_{3,3}\otimes q_{21}
\end{equation}
where $q_{11}=\bbmp_{2,1}\otimes \bbmp_{2,1}$, etc.  Let
\begin{equation}
\label{eq:B123-defn}
w_1 = 
\begin{pmatrix}
-1 & 0 \\
0 & 1 
\end{pmatrix},\;
w_2 = 
\begin{pmatrix}
0 & 1 \\
0 & 0 
\end{pmatrix},\;
w_3 = 
\begin{pmatrix}
0 & 0 \\
1 & 0
\end{pmatrix}.
\end{equation}
Then we regard that 
\begin{equation}
\label{eq:322-generator2-alternative}
w = (w_1,w_2,w_3). 
\end{equation}

We consider the Lie algebra of $G_w$. 
Let $X=(x_{ij})\in \m_3$, $Y=(y_{ij}),Z=(z_{ij})\in \m_2$. 
Then $(e_G+\vep(X,Y,Z))w=w$ if and only if 
\begin{align*}
& Yw_1 + w_1{}^tZ + x_{11}w_1+x_{12}w_2+x_{13}w_3 = 0, \\
& Yw_2 + w_2{}^tZ + x_{21}w_1+x_{22}w_2+x_{23}w_3 = 0, \\
& Yw_3 + w_3{}^tZ + x_{31}w_1+x_{32}w_2+x_{33}w_3 = 0.
\end{align*}

Straightforward considerations show the following lemma 
and so we do not provide the details. 

\begin{lem}
\label{lem:S14-XYZ-form}
$X,Y,Z$ are in the form 
\begin{align*}
& X = \begin{pmatrix}
-y_{11}-z_{11} & -2y_{12} & 2y_{21} \\
-y_{21} & -y_{11}-z_{22} & 0 \\
y_{12} & 0 & -y_{22}-z_{11}
\end{pmatrix}, \\
& Y = \begin{pmatrix}
y_{11} & y_{12} \\
y_{21} & y_{22} 
\end{pmatrix},\;
Z = \begin{pmatrix}
z_{11} & -y_{21} \\
-y_{12} & z_{22}
\end{pmatrix}
\end{align*}
where $y_{11}+z_{11}=y_{22}+z_{22}$. 
\end{lem}

Since 
$\dim {\mathrm T}_{e_G}(G_w)=5=\dim G-\dim V$, 
Proposition \ref{prop:open-orbit} implies the following corollary. 

\begin{cor}
\label{cor:S14-H-dense}
$Gw\sub V$ is Zariski dense and $G_w$ is smooth over $k$.
\end{cor}

Let $\rho:\gl_2\times \gl_1\to \gl(\m_2)$ be the
homomorphism defined by 
$\rho(h,t)A=thAh^{-1}$ for $h\in\gl_2,t\in\gl_1,A\in \m_2$.
Let $L\sub \m_2$ be the subspace spanned by $w_1,w_2,w_3$. 
Obviously, $\dim L=3$. 
\begin{lem}
\label{lem:S14-L-invariant}
If $h\in\gl_2,t\in\gl_1$ then $\rho(h,t)L\sub L$. 
\end{lem}
\begin{proof}
This lemma follows from the observation that
$L$ is the subspace of matrices of trace $0$. 
%
\end{proof}

For $(h,t)\in \gl_2\times \gl_1$, 
let $\rho_1(h,t)=(\rho_1(h,t)_{ij})\in\gl_3$ 
be the matrix such that 
\begin{equation*}
\begin{pmatrix}
\rho(h,t)w_1 & \rho(h,t)w_2 & \rho(h,t)w_3
\end{pmatrix}
= \begin{pmatrix}
w_1 & w_2 & w_3
\end{pmatrix}
\rho_1(h,t). 
\end{equation*}

We define a homomorphism $\phi:\gl_2\times \gl_1\to G$ by 
\begin{equation}
\label{eq:S14-phi-defn}
\phi(h,t)=({}^t\rho_1(h,t)^{-1},th,{}^th^{-1}).
\end{equation}
\begin{lem}
\label{lem:S14-w-fix}
The image of $\phi$ fixes $w$. 
\end{lem}
\begin{proof}
The action of $g=(g_1,g_2,g_3)$ 
on $x=(x_1,x_2,x_3)$ is 
\begin{math}
g_1 \left(
\begin{smallmatrix}
g_2 x_1 {}^tg_3 \\
g_2 x_2 {}^tg_3 \\
g_2 x_3 {}^tg_3 
\end{smallmatrix}
\right)
\end{math}
treating $[g_2 x_1 {}^tg_3,g_2 x_2 {}^tg_3,g_2 x_3 {}^tg_3 ]$ 
as a column vector. 
So 
\begin{align*}
({}^t\rho_1(h,t)^{-1},th,{}^th^{-1}) w
& = {}^t\rho_1(h,t)^{-1}
\begin{pmatrix}
th w_1 h^{-1} \\
th w_2 h^{-1} \\
th w_3 h^{-1} 
\end{pmatrix} \\
& = {}^t\rho_1(h,t)^{-1}
\; {}^t\begin{pmatrix}
\rho(h,t) w_1 & 
\rho(h,t) w_2 & 
\rho(h,t) w_3
\end{pmatrix}  \\
& = {}^t\rho_1(h,t)^{-1}
\; {}^t 
(\begin{pmatrix}
w_1 & w_2 & w_3 
\end{pmatrix}
\rho_1(h,t)) \\
& = {}^t\rho_1(h,t)^{-1}\; {}^t\rho_1(h,t)
\begin{pmatrix}
w_1 \\ w_2 \\ w_3
\end{pmatrix} = \begin{pmatrix}
w_1 \\ w_2 \\ w_3
\end{pmatrix} \\
& = w. 
\end{align*}
\end{proof}
\begin{prop}
\label{prop:HR-circ}
$G_w=G^{\circ}_w = \im(\phi)$. 
\end{prop}
\begin{proof}
It is easy to see that 
$\phi$ is an imbedding and so $\dim \phi(\gl_2\times \gl_1)=5=\dim G_w$. 
So $G^{\circ}_w = \im(\phi)$. 
Suppose that $k=\overline k$ and $g\in G_{w\,k}$. 
Since $\spl_2\sub \gl_2\times \gl_1$ has no outer automorphisms, 
by multiplying an element of $\spl_2$ if necessary, 
we may assume that $g$ commutes with elements of $\phi(\spl_2)$. 

Note that if $t\in\gl_1$ then with respect to the basis 
$\{E_{11},E_{21},E_{12},E_{22}\}$, 
\begin{equation*}
\rho(\diag(t,t^{-1}),1)=\diag(1,t^{-2},t^2,1). 
\end{equation*}
So 
\begin{equation*}
\phi(\diag(t,t^{-1}),1)=\left(
\diag(1,t^2,t^{-2}),
\diag(t,t^{-1}),
\diag(t^{-1},t)
\right).
\end{equation*}
Since $g$ commutes with this element, 
all components of $g$ are diagonal matrices. 
Then it is easy to show that $g\in G^{\circ}_w$.
\end{proof}

If $h=(h_{ij})\in\gl_2$ then by computation, 
\begin{equation}
\label{eq:rho1(g)}
\rho_1(h,1) = (\det h)^{-1} 
\begin{pmatrix}
h_{11}h_{22}+h_{12}h_{21} & h_{11}h_{21} & -h_{12}h_{22} \\
2h_{11}h_{12} & h_{11}^2 & -h_{12}^2 \\
-2h_{21}h_{22} & -h_{21}^2 & h_{22}^2
\end{pmatrix}.
\end{equation}

For the following proposition, 
see the consideration of $S_{\be_4}$
in Section 3 \cite{tajima-yukie-GIT2} 
(take the determinant of $\Phi(x)$ in (3.3) \cite{tajima-yukie-GIT2} 
identifying the dual space of $\m_2$ with $\m_2$).  
Note that $w$ in (\ref{eq:322-generator2-alternative}) corresponds to 
$R_{322}$ in Section 3 \cite{tajima-yukie-GIT2}. 
\begin{prop}
\label{prop:322-invariant}
We consider the natural action of 
$G=\gl_3\times \gl_2^2$ on 
$V=\aff^3\otimes \aff^2\otimes \aff^2$.
Then there is a homogeneous degree $6$ polynomial $P(x)$ on 
$V$ such that 
\begin{equation*}
P((g_1,g_2,g_3)x)=(\det g_1)^2(\det g_2)^3(\det g_3)^3 P(x)
\end{equation*}
for $(g_1,g_2,g_3)\in G$ and that 
\begin{math}
P(w)=1.  
\end{math}
Moreover, 
\begin{math}
\{x\in V_k\mid P(x)\not=0\}
= G_k w.
\end{math}
\end{prop}

\begin{rem}
\label{rem:322-remark}
In the situation of Proposition \ref{prop:322-invariant}, 
we obtain the same result even if we change the 
action of $g\in \gl_3$ to $\aff^3\ni x\mapsto {}^tg^{-1}x\in\aff^3$.
This is because $\{{}^tg^{-1}\mid g\in \gl_3(k)\}=\gl_3(k)$. 
Similarly we can change the action of $g\in \gl_2$ to 
$\aff^2\ni x\mapsto {}^tg^{-1}x\in\aff^2$. 
\end{rem}

\subsection{Case V}

We consider $G=\gl_2^2\times \gl_1$ and
$V=\aff^2\otimes \aff^2\oplus \aff^2\oplus \aff^2$. 
We define an action of $G$ on $V$ so that 
$(g_1,g_2,t)(v_1\otimes v_2,v_3,v_4)=(g_1v_1\otimes g_2v_2,tg_1v_3,g_2v_4)$. 
for $(g_1,g_2,t)\in G$, $v_1\ccd v_4\in\aff^2$.

Let $V,W$ be vector spaces over $k$ and $0<m<n$ integers. 
We define a map
\begin{equation*}
\Phi:V\otimes \wedge^m \aff^n\oplus W\otimes \wedge^{n-m}\aff^n\to 
V\otimes W\otimes \wedge^n \aff^n
\cong V\otimes W, 
\end{equation*}
linear with respect to each component, so that 
\begin{math}
\Phi(v\otimes x,w\otimes y) = v\otimes w \otimes (x\wedge y)
\end{math}
for $v\in V,w\in W,x\in \wedge^m \aff^n,y\in \wedge^{n-m}\aff^n$. 
As we pointed out earlier, we identify 
$\wedge^n \aff^n$ with $\aff^1$ so that 
$p_{n,12\cdots n}$ corresponds to $1$.

The following lemma is obvious. 
\begin{lem}
\label{lem:natural-pairing}
In the above situation, 
\begin{equation*}
\Phi(g_1v\otimes g_3x,g_2w\otimes g_3y) = (\det g_3) 
(g_1,g_2)\Phi(v\otimes x,w\otimes y)
\end{equation*}
for $g_1\in\gl(V),g_2\in\gl(W),g_3\in \gl_n$ 
where the action of $(g_1,g_2)$ on $V\otimes W$ is the 
natural action. 
\end{lem}

The following proposition follows by applying 
Lemma \ref{lem:natural-pairing} twice. 
Also it is proved explicitly in 
Proposition 3.5 \cite{tajima-yukie-GIT2} 
during the consideration of $S_{\be_6}$ in Section 3 \cite{tajima-yukie-GIT2}.  
Note that the formulation in \cite{tajima-yukie-GIT2} is slightly 
different and there are extra $\gl_1$-factors, 
but the proof works.    
\begin{prop}
\label{prop:M2-2-2-invariant}
Let $G,V$ be as above.  
\begin{itemize}
\item[(1)]
There is a homogeneous degree $3$ polynomial $P_1(x)$ on $V$, 
linear with respect to each of $v_1\otimes v_2,v_3,v_4$ such that
$P_1(gx) = t(\det g_1)(\det g_2)P_1(x)$ for $g=(g_1,g_2,t)\in G,x\in V$
and that $P_1(q_{11}+q_{22},\bbmp_{2,1},\bbmp_{2,1})=1$ 
where $q_{11}=\bbmp_{2,1}\otimes \bbmp_{2,1}$, etc.  
\item[(2)]
For $x=(A,v_1,v_2)\in V$ where $A\in \aff^2\otimes \aff^2$, 
let $P_2(x)=\det A$ identifying $\aff^2\otimes \aff^2\cong \m_2$.
Then $P_2(gx)=(\det g_1)(\det g_2)P_2(x)$ and 
$\{x\in V_k\mid P_1(x),P_2(x)\not=0\}=G_k(q_{11}+q_{22},\bbmp_{2,1},\bbmp_{2,1})$.
\end{itemize}
\end{prop}

\section{Rational orbits (2)}
\label{sec:rational-orbits-53}

In this section and the subsequent five sections, 
we consider \rep s which appear as $(M_{\be},Z_{\be})$. 
We have to consider two questions. One is the existence of 
relative invariant polynomials and the other is the 
interpretation of generic rational orbits. 
For the sake of the GIT stratification, it is enough to 
determine whether or not $S_{\be}\not=\emptyset$. 
For that purpose we only have to construct non-zero 
relative invariant polynomials and $(M_{\be},Z_{\be})$ does not
have to be a \pv{}.  
However, to describe rational orbits in 
$S_{\be\,k}$, we need more information and it is convenient 
to show that $(M_{\be},Z_{\be})$ is a regular \pv. 
For that purpose, we sometimes have to assume that 
$\ch(k)\not=2$.

In this section let $G=\gl_5\times \gl_3$ and 
$V=\wedge^2 \aff^5\otimes \aff^3$. 
Rational orbits of this case was considered 
in \cite{yukiem} under the assumption 
$\ch(k)=0$.  The proof in \cite{yukiem} 
works as long as $\ch(k)\not=2,3$. 
However, we would like to minimize 
the number of primes excluded for $\ch(k)$.
So we make some modification of the 
argument so that the only prime excluded 
is $p=2$. Some of the considerations 
in this section are due to \cite{ochiai} if $k=\C$.

We sometimes consider over $\Z,\Q$ and then 
deduce certain assertions over any $k$ 
by the natural homomorphism $\Z\to k$.

We first assume $k=\Q$. 
Let $W=\aff^2$ be the standard \rep{} of the ring $\m_2$, 
$W_4=\sym^4 W$ and $W_2=\sym^2 W$. 
Let $\{\bbmp_{2,1},\bbmp_{2,2}\}$ be the standard basis 
of $W$ as before and 
$l_0=\bbmp_{2,1}^4,l_1=\bbmp_{2,1}^3\bbmp_{2,2}\ccd l_4=\bbmp_{2,2}^4\in W_4$.
Then $S_4=\{l_0,l_1,3l_2,l_3,l_4\}$ is a basis of $W_4$ over $\Q$. 
For $g\in\gl_2$, let $\rho(g)$ be the matrix of the action of $g\in \m_2$ 
on $W_4$ with respect to the basis $S_4$. We put
\begin{align*}
w_1 & = 
\begin{pmatrix}
0 & 0 & 0 & 1 & 0 \\
0 & 0 & -1 & 0 & 0 \\
0 & 1 & 0 & 0 & 0 \\
-1 & 0 & 0 & 0 & 0 \\
0 & 0 & 0 & 0 & 0 
\end{pmatrix}, \;
w_2 = 
\begin{pmatrix}
0 & 0 & 0 & 0 & -1 \\
0 & 0 & 0 & 2 & 0 \\
0 & 0 & 0 & 0 & 0 \\
0 & -2 & 0 & 0 & 0 \\
1 & 0 & 0 & 0 & 0 
\end{pmatrix}, \\
w_3 & = 
\begin{pmatrix}
0 & 0 & 0 & 0 & 0 \\
0 & 0 & 0 & 0 & 1 \\
0 & 0 & 0 & -1 & 0 \\
0 & 0 & 1 & 0 & 0 \\
0 & -1 & 0 & 0 & 0 
\end{pmatrix},\; 
w = 
\begin{pmatrix}
w_1 \\ w_2 \\ w_3
\end{pmatrix}
\in \wedge^2 \aff^5\otimes \aff^3=V.
\end{align*}
Let $L=\lan w_1,w_2,w_3\ran\sub \wedge^2 W_4$ 
be the subspace spanned by $w_1,w_2,w_3$ and 
$L_{\Z}$ the $\Z$-module spanned by $w_1,w_2,w_3$.  

\begin{lem}
\label{lem:S1-rho4-defined-integrally} 
\begin{itemize}
\item[(1)]
If $g\in\m_2(\Z)$ then entries of $\rho(g)$ 
are polynomials of entries of $g$ with coefficients in $\Z$.  
\item[(2)]
If $g_1,g_2\in\m_2(\Z)$ then $\rho(g_1g_2)=\rho(g_1)\rho(g_2)$. 
\end{itemize}
\end{lem}
\begin{proof}
(1) We only have to verify that the coefficient of 
$\bbmp_{2,1}^2\bbmp_{2,2}^2$ in $(\sym^4 g) l_i$, say $C_i$, is divisible by $3$ 
for $i=0,1,3,4$. If $g=(g_{ij})$ then 
\begin{align*}
& C_0 = 6g_{11}^2g_{21}^2,\;
C_1 = 3g_{11}^2g_{21}g_{22}+3g_{11}g_{21}^2g_{12}.
\end{align*}
So $C_0,C_1$ are divisible by $3$. Similarly, $C_3,C_4$ 
are divisible by $3$. 
(2) is obvious.
\end{proof}

Note that if $u\in \Z$ then 
\begin{equation}
\label{eq:S1-rho-nu}
\rho(n_2(u)) = 
\begin{pmatrix}
1 & 0 & 0 & 0 & 0 \\
4u & 1 & 0 & 0 & 0 \\
2u^2 & u & 1 & 0 & 0 \\
4u^3 & 3u^2 & 6u & 1 & 0 \\
u^4 & u^3 & 3u^2 & u & 1
\end{pmatrix},\;
\rho({}^tn_2(u)) = 
\begin{pmatrix}
1 & u & 3u^2 & u^3 & u^4 \\
0 & 1 & 6u & 3u^2 & 4u \\
0 & 0 & 1 & u & 2u^2 \\
0 & 0 & 0 & 1 & 4u \\
0 & 0 & 0 & 0 & 1
\end{pmatrix}.
\end{equation}

By Lemma \ref{lem:S1-rho4-defined-integrally}, 
the natural homomorphism $\Z\to k$ induces a map 
$\rho:\m_2(k)\to \m_5(k)$ (we use the same notation $\rho$) 
such that $\rho(g_1g_2)=\rho(g_1)\rho(g_2)$ for $g_1,g_2\in \m_2(k)$.
Since $\rho(I_2)=I_5$ over $\Z$, the same is true over $k$. 
Therefore, $\rho(g)\in\gl_5(k)$ if $g\in\gl_2(k)$.

\begin{lem}
\label{lem:S1-invariant-subspace}
If $g\in\gl_2(k)$ and $A\in L$ then 
$\wedge^2 \rho(g)A\in L$.   
\end{lem}
\begin{proof}
It is enough to consider $n_2(u),{}^tn_2(u)$ and 
diagonal matrices. If $g=\diag(t_1,t_2)$ then 
the assertion is obvious. Explicitly, 
\begin{equation*}
\wedge^2 \rho(g) w_1 = t_1^5t_2^3 w_1,\; 
\wedge^2 \rho(g) w_2 = t_1^4t_2^4 w_2,\;
\wedge^2 \rho(g) w_3 = t_1^3t_2^5 w_3.
\end{equation*}

If $g=n_2(u)$ then 
\begin{equation*}
\wedge^2 \rho(g) w_1 = w_1 -uw_2 +u^2w_3,\; 
\wedge^2 \rho(g) w_2 = w_2 -2uw_3,\; 
\wedge^2 \rho(g) w_3 = w_3
\end{equation*}
and so $\wedge^2 \rho(g)L\sub L$.   
The consideration is similar for $g={}^tn_2(u)$. 
\end{proof}

For $g\in \gl_2$, let $\rho_1(g)=(\det g)^{-2}\rho(g)$. 
Then $\rho_1(g)=I_5$ if $g$ is a scalar matrix.  
For $g\in \gl_2$, let $\rho_2(g)=(\rho_2(g)_{ij})\in\gl_3$ 
be the matrix such that 
\begin{equation*}
\begin{pmatrix}
\wedge^2 \rho_1(g)w_1 & \wedge^2 \rho_1(g)w_2 & \wedge^2 \rho_1(g)w_3
\end{pmatrix}
= \begin{pmatrix}
w_1 & w_2 & w_3
\end{pmatrix}
\rho_2(g)
\end{equation*}
treating 
\begin{math}
\begin{pmatrix}
w_1 & w_2 & w_3
\end{pmatrix}
\end{math}
as a row vector. 

\begin{lem}
\label{lem:S1-rho-rho1-fix-w}
The image of $(\rho_1,{}^t\rho_2^{-1})$ 
fixes $w$. 
\end{lem}
\begin{proof}
The action of $(g_1,g_2)\in \gl_5\times \gl_3$ on 
$x=[x_1,x_2,x_3]\in V$ is 
\begin{math}
g_2 \left(
\begin{smallmatrix}
\wedge^2 g_1 x_1 \\
\wedge^2 g_1 x_2 \\
\wedge^2 g_1 x_3 
\end{smallmatrix}
\right)
\end{math}
treating $[\wedge^2 g_1 x_1,\wedge^2 g_1 x_2,\wedge^2 g_1 x_3]$ 
as a column vector. 
So if $g\in\gl_2$ then 
\begin{align*}
(\wedge^2 \rho_1(g),{}^t\rho_2(g)^{-1}) w 
& = {}^t\rho_2(g)^{-1}
\begin{pmatrix}
\wedge^2 \rho_1(g) w_1 \\
\wedge^2 \rho_1(g) w_2 \\
\wedge^2 \rho_1(g) w_3
\end{pmatrix} \\
& = {}^t\rho_2(g)^{-1}
\; {}^t\begin{pmatrix}
\wedge^2 \rho_1(g) w_1 & 
\wedge^2 \rho_1(g) w_2 & 
\wedge^2 \rho_1(g) w_3
\end{pmatrix}  \\
& = {}^t\rho_2(g)^{-1}
\; {}^t 
(\begin{pmatrix}
w_1 & w_2 & w_3 
\end{pmatrix}
\rho_2(g)) \\
& = {}^t\rho_2(g)^{-1}\; {}^t\rho_2(g)
\begin{pmatrix}
w_1 \\ w_2 \\ w_3
\end{pmatrix} = \begin{pmatrix}
w_1 \\ w_2 \\ w_3
\end{pmatrix} \\
& = w. 
\end{align*}
\end{proof}

We show that $G_w\cong \mathrm{PGL}_2\times \gl_1$. 
For that purpose we first determine the Lie algebra 
of $G_w$. 

Let $k[\vep]/(\vep^2)$ be the ring of dual numbers, 
$X=(x_{ij})\in\m_5$ and $Y=(y_{ij})\in \m_3$. 
If $(e_G+\vep(X,Y))w = w + \vep(A,B,C)$ 
then 
\begin{equation}
\label{eq:S1-lie-algeba-Gw}
\begin{aligned}
A & = X w_1 + w_1 {}^t X + y_{11}w_1+y_{12}w_2+y_{13}w_3, \\
B & = X w_2 + w_2 {}^t X + y_{21}w_1+y_{22}w_2+y_{23}w_3, \\
C & = X w_3 + w_3 {}^t X + y_{31}w_1+y_{32}w_2+y_{33}w_3.
\end{aligned}
\end{equation}
So $e_G+\vep(X,Y)\in {\mathrm T}_{e_G}(G_w)$ if and only if the right hand sides 
of (\ref{eq:S1-lie-algeba-Gw}) are $0$. 
By long but straightforward computations, 
we obtain the following proposition. 
We do not provide the details.

\begin{prop}
\label{prop:S1-lie-alg-Gw}
Suppose that $\ch(k)\not=2$. Then 
$e_G+\vep(X,Y)\in {\mathrm T}_{e_G}(G_w)$ if and only if 
$X,Y$ are in the following form:
\begin{align*}
X & = \begin{pmatrix}
x_{11} & x_{12} & 0 & 0 & 0 \\
4x_{32} & x_{22} & 6x_{12} & 0 & 0 \\
0 & x_{32} & -x_{11}+2x_{22} & x_{12} & 0 \\
0  & 0 & 6x_{32} & -2x_{11}+3x_{22} & 4x_{12} \\
0 & 0 & 0  & x_{32} & -3x_{11}+4x_{22}
\end{pmatrix}, \\
Y & = 
\begin{pmatrix}
x_{11}-3x_{22} & x_{32} & 0  \\
2x_{12} & 2x_{11}-4x_{22} & 2x_{32}  \\
0 & x_{12} & 3x_{11}-5x_{22}
\end{pmatrix}.  
\end{align*}
\end{prop}

Since $\dim {\mathrm T}_{e_G}(G_w)=4=34-30=\dim G-\dim V$, 
the following corollary follows from Proposition \ref{prop:open-orbit}. 

\begin{cor}
\label{cor:S1-regular-dense}
If $\ch(k)\not=2$ then $Gw\sub V$ is Zariski open and 
$G_w$ is smooth over $k$. 
\end{cor}

\begin{prop}
\label{prop:S1-Gw}
If $\ch(k)\not=2$ then 
$G_w=G^{\circ}_w\cong \mathrm{PGL}_2\times \gl_1$. 
\end{prop}
\begin{proof}
Let $H=\gl_2\times \gl_1$. 
We define a homomorphism $\phi:H\to G$ 
by $\phi(g,t)=(t \rho_1(g),t^{-2}\, {}^t\rho_2(g)^{-1})$. 
By Lemma \ref{lem:S1-rho-rho1-fix-w}, 
$\phi(g,t)\in G_w$.  
Let $\psi$ be the differential of $t \rho_1(g)$. 
We compute the image of $\psi$. Note that   
${\mathrm T}_{e_H}(H)$ consists of 
elements of the form $(I_2+\vep A,1+c\vep)$ 
where $A=(a_{ij})\in \m_2,c\in\aff^1$.

If $A=E_{21},E_{12}$ (see Section \ref{sec:notation}) 
then (\ref{eq:S1-rho-nu}) 
implies that $(\det(I_2+\vep A))^{-2}\rho(I_2+\vep A)=\rho(I_2+\vep A)=I_5+\vep X$ 
where $X$ is as in  
Proposition \ref{prop:S1-lie-alg-Gw} with 
$x_{11}=x_{22}=0$ and $x_{12}=0,x_{32}=1$ 
or $x_{12}=1,x_{32}=0$. 
If $A=\diag(a_1,a_2)$ 
then $\det(I_2+\vep A)^{-2}=1-2(a_1+a_2)\vep$. 
So if we put $a_1-a_2=b$ then 
\begin{equation*}
\psi(I_2+\vep A,1+ c\vep) = I_5+\vep \diag(2b+c,b+c,c,-b+c,-2b+c). 
\end{equation*}
If $x_{11}=2b+c$, $x_{22}=b+c$ then 
$c = -x_{11}+2x_{22}$, 
$-b+c = -2x_{11}+3x_{22}$, 
$-2b+c = -3x_{11}+4x_{22}$. 
So, the dimension of the image of 
$\psi$ is $4=\dim {\mathrm T}_{e_G}(G_w)$.  Therefore, 
$\psi:{\mathrm T}_{e_H}(H)\to {\mathrm T}_{e_G}(G_w)$ 
is surjective.

Assuming that $k=\overline k$, 
we show that $\kernel(\phi)=\{(tI_2,1)\mid t\in\gl_1\}$ 
set theoretically. Then the dimension of the 
image of $\phi$ is $5-1=4=\dim G_w$. 
This implies that $\im(\phi)=G^{\circ}_w$. 

Suppose that $(g,t)\in \kernel(\phi)$. 
Then $\rho(g)$ is a scalar matrix. 
Since $\rho(g)\bbmp_{2,1}^4,\rho(g)\bbmp_{2,2}^4$ are 
scalar multiples of $\bbmp_{2,1}^4,\bbmp_{2,2}^4$ respectively, 
$g$ is a diagonal matrix, say $\diag(t_1,t_2)$. 
Since $\rho(g)$ is a scalar matrix, 
$t_1^4=t_1^3t_2$, which implies that $t_1=t_2$. 
Since $\rho_1(t_1I_2)=I_5$, $t=1$.  Therefore, 
$\kernel(\phi)=\{(t_1I_2,1)\mid t_1\in\gl_1\}$. 

As we pointed out above, $\im(\phi)=G^{\circ}_w$. 
Since $\im(\psi)={\mathrm T}_{e_G}(G_w)$, 
$\phi:H\to G^{\circ}_w$
is smooth surjective.  Since the fibers are smooth by the Jacobian 
criterion, $\phi^{-1}(e_G)=\{(tI_2,1)\mid t\in\gl_1\}$ 
as schemes.  Therefore, $G^{\circ}_w\cong \mathrm{PGL}_2\times \gl_1$.

Suppose that $g=(g_1,g_2)\in G_{w\,\overline k}$. Since all automorphims of 
$\mathrm{PGL}_2$ are inner, by multiplying an element of 
$\mathrm{PGL}_2$ to $g$, we may assume that 
$g_1$ commutes with elements of $\mathrm{PGL}_2$. 

If $\ch(k)\not=3$, it is easy to show that 
$\rho$ is an irreducible \rep (we are assuming $\ch(k)\not=2$). 
Therefore, by Schur's lemma, $g_1$ must be a scalar 
matrix.

Suppose that $\ch(k)=3$. We show that the only non-trivial 
invariant subspace is $U = \lan [0,0,1,0,0]\ran$. 
By (\ref{eq:S1-rho-nu}), $U$ is an invariant subspace
since $\ch(k)=3$. Suppose that $H\sub \aff^5$ is an invariant subspace.
Then it is invariant by the Lie algebra of $\spl_2$, 
$\tau_0$ in (\ref{eq:J-defn}) and $\diag(t,t^{-1})$. 

Coordinate vectors of $\aff^5$ are weight vectors with respect 
to $\diag(t,t^{-1})$ with distinct weights $4,2,0,-2,-4$. 
Therefore, $H$ is spanned by coordinate vectors contained in 
$H$. Let $\mathbbm e_0=[1,0,0,0,0]\ccd \mathbbm e_4=[0,0,0,0,1]$. 
By the action of the Lie algebra element 
$I_2+ n_2(\varepsilon)$, 
\begin{equation*}
\mathbbm e_0 \mapsto 4\mathbbm e_1,\;
\mathbbm e_1 \mapsto \mathbbm e_2,\;
\mathbbm e_3 \mapsto \mathbbm e_4
\end{equation*}
and similarly for $I_2+ {}^tn_2(\varepsilon)$.  
Since $\tau_0$ exchanges $\mathbbm e_0 \leftrightarrow \mathbbm e_4$
and $\mathbbm e_1 \leftrightarrow \mathbbm e_3$, 
if $H$ contains any one of $\mathbbm e_0,\mathbbm e_1,\mathbbm e_3,\mathbbm e_4$
then $H=\aff^5$. 

The above consideration implies that $g_1$ fixes the subspace
$U$ and acts on $\aff^5/U$ by scalar multiplication 
by Schur's lemma. Therefore, $g_1$ is in the form 
\begin{equation*}
g_1=\begin{pmatrix}
a & 0 & 0 & 0 & 0 \\
0 & a & 0 & 0 & 0 \\
b_1 & b_2 & b_3 & b_4 & b_5 \\
0 & 0 & 0 & a & 0 \\
0 & 0 & 0 & 0 & a \\
\end{pmatrix}
\end{equation*}
where $a,b_3\not=0$. Since $g\in G_{w\,\overline k}$, 
$g_1$ fixes the subspace $\lan w_1,w_2,w_3\ran$. 

By computation, 
\begin{align*}
g_1 w_1{}^t g_1 & = 
\begin{pmatrix}
0 & 0 & ab_4 & a^2 & 0 \\
0 & 0 & -ab_3 & 0 & 0 \\ 
-ab_4 & ab_3 & 0 & ab_1 & 0 \\
-a^2 & 0 & -ab_1 & 0 & 0 \\
0 & 0 & 0 & 0 & 0 
\end{pmatrix}, \\
g_1 w_2{}^t g_1 & = 
\begin{pmatrix}
0 & 0 & -ab_5 & 0 & -a^2 \\
0 & 0 & 2ab_4 & 2 a^2 & 0 \\
ab_5 & -2ab_4 & 0 & 2ab_2 & -ab_1 \\
0 & -2a^2 & -2ab_2 & 0 & 0 \\
a^2 & 0 & ab_1 & 0 & 0
\end{pmatrix}.
\end{align*}
By the $(1,3),(3,4)$-entries of $g_1w_1{}^tg_1$, 
$b_4=b_1=0$. By the $(1,4),(2,3)$-entries of $g_1w_1{}^tg_1$, 
$a^2=ab_3$ and so $a=b_3$. 
By the $(1,3),(3,4)$-entries of $g_1w_2{}^tg_1$, 
$b_5=b_2=0$. Therefore, $g_1=aI_5$. 

In all cases $g_1$ is a scalar matrix. 
Since $g\in G_w$, $g_2$ is a scalar matrix also. 
If $g=(t_1I_5,t_2I_3)$ then $t_2=t_1^{-2}$ and so 
$g\in \gl_1$. Therefore, $G_w=G^{\circ}_w$. 
\end{proof}

Since $G_w$ is reductive, the following corollary follows. 
\begin{cor}
\label{cor:S1-regular}
If $\ch(k)\not=2$ then 
$(G,V)$ is a regular \pv.  
\end{cor}

If $\ch(k)\not=2$ then 
Corollary \ref{cor:reducible-sep-orbit} 
implies the existence of a relative invariant 
polynomial. However, we would like to 
construct a relative invariant polynomial explicitly
without assuming $\ch(k)\not=2$. 
Since the argument is similar as in \cite{yukiem}, 
we will be brief. We also would like to find an element 
for which the value of the relative invariant polynomial 
is no-zero even if $\ch(k)=2$.  

We put $w_1'=w_1,w_3'=w_3$, 
\begin{equation*}
w_2' = 
\begin{pmatrix}
0 & 0 & 0 & 0 & -1 \\
0 & 0 & 0 & 1 & 0 \\
0 & 0 & 0 & 0 & 0 \\
0 & -1 & 0 & 0 & 0 \\
1 & 0 & 0 & 0 & 0 
\end{pmatrix}, \quad
w' = (w_1',w_2',w_3'). 
\end{equation*}

Let $u_i=\bbmp_{3,i}$ for $i=1,2,3$. 
We regard elements of $V$ as $5\times 5$ alternating 
matrices $x=(x_{ij}(u))$ with entries which are linear combinations 
of $u=(u_1,u_2,u_3)$. Let 
$\pfaff_1(x)\ccd \pfaff_5(x)$ be the Pfaffians 
of $5$ main minors.  We choose the sign so that 
\begin{align*}
\pfaff_1(x) & = x_{23}x_{45}-x_{24}x_{35}+x_{25}x_{34}, \\
\pfaff_2(x) & = -(x_{13}x_{45}-x_{14}x_{35}+x_{15}x_{34}), \\
\pfaff_3(x) & = x_{12}x_{45}-x_{14}x_{25}+x_{15}x_{24}, \\
\pfaff_4(x) & = -(x_{12}x_{35}-x_{13}x_{25}+x_{15}x_{23}), \\
\pfaff_5(x) & = x_{12}x_{34}-x_{13}x_{24}+x_{14}x_{23}.
\end{align*}
We put $\pfaff(x)=[\pfaff_1(x)\ccd \pfaff_5(x)]$.
$\pfaff_1(x)\ccd \pfaff_5(x)$ are 
quadratic forms in $u$. 

Let $R=k[u_1,u_2,u_3]$. We regard $x\in \wedge^2_R R^5$. 
Then $x\wedge x\in \wedge^4_R R^5\cong \Hom_R(R^5,R)$. 
We can regard that $\pfaff_1(x)\ccd \pfaff_5(x)$ are the 
$5$-components of $\Hom_R(R^5,R)$ and 
\begin{equation*}
\pfaff((g_1,g_2)x(u)) = (\det g_1){}^tg_1^{-1}\pfaff(x(ug_2)). 
\end{equation*}
Let  
\begin{equation*}
\Phi(x) = \pfaff_1(x)\wedge \cdots \wedge \pfaff_5(x)
\in \wedge^5 \sym^2 \aff^3 \cong (\sym^2 \aff^3)^*.
\end{equation*}
Then with respect to a basis of $\sym^2 \aff^3$ and its dual basis,
\begin{equation}
\label{eq:S1-Phi-equivariant}
\Phi(gx) = (\det g_1)^4 (\det g_2)^4 \, {}^t(\sym^2 g_2)^{-1} \Phi(x). 
\end{equation}

Let 
\begin{equation*}
l_1=u_1^2,l_2=u_1u_2,l_3=u_1u_3,l_4=u_2^2,l_5=u_2u_3,l_6=u_3^2.
\end{equation*}
Then $\{l_1\ccd l_6\}$ is a basis of $\sym^2 \aff^3$. 
Let $\{l_1^*\ccd l_6^*\}$ be the dual basis. 
We identify $\wedge^5 \sym^2 \aff^3$ with $(\sym^2 \aff^3)^*$ 
so that $l_1^*=l_2\wedge \cdots \wedge l_6$, 
$l_2^*=-l_1\wedge l_3\wedge \cdots \wedge l_6$, etc. 
Easy computations show that 
\begin{equation}
\label{eq:S1-value-phi-w}
\begin{aligned}
& \begin{pmatrix}
\pfaff_1(w) \\
\pfaff_2(w) \\
\pfaff_3(w) \\
\pfaff_4(w) \\
\pfaff_5(w) 
\end{pmatrix}
= \begin{pmatrix}
-l_6 \\
-l_5 \\
-l_3-2l_4 \\
-l_2 \\
-l_1
\end{pmatrix},\;
\begin{pmatrix}
\pfaff_1(w') \\
\pfaff_2(w') \\
\pfaff_3(w') \\
\pfaff_4(w') \\
\pfaff_5(w') 
\end{pmatrix}
= \begin{pmatrix}
-l_6 \\
-l_5 \\
-l_3-l_4 \\
-l_2 \\
-l_1
\end{pmatrix}, \\
& \Phi(w) = l_4^*-2l_3^*,\;
\Phi(w') = l_4^*-l_3^*.
\end{aligned}
\end{equation}

For $g\in \gl_3$, with respect to the above bases, we put 
\begin{math}
(\sym^2 g)^* = (\det g)^4 \, {}^t(\sym^2 g)^{-1}. 
\end{math}
For $y=\sum_{i=1}^6 y_il_i^*$, we put 
\begin{equation*}
P_1(y) = -\det
\begin{pmatrix}
y_1 & y_2 & y_3 \\
y_2 & y_4 & y_5 \\
y_3 & y_5 & y_6
\end{pmatrix}. 
\end{equation*}
\begin{prop}
\label{prop:sym2*-invariant}
Over any field $k$, 
$P_1((\sym^2 g)^* y)= (\det g)^{10}P_1(y)$. 
\end{prop}
\begin{proof}
$P_1((\sym^2 g)^* y)$, $(\det g)^{10}P_1(y)$ 
are polynomials of $y,g,(\det g)^{-1}$ 
defined over $\Z$. So if we can prove the proposition 
over $\Z$ then by considering the natural 
homomorphism $\Z\to k$, the proposition follows for $k$ 
also.  So we consider the proposition over $\Q$.  

For $a\in \aff^3\otimes \aff^3,b\in 
(\aff^3)^*\otimes (\aff^3)^*$, let $(a,b)$ 
be the natural pairing. 
We can naturally identify $(\aff^3\otimes \aff^3)^*$ 
with $(\aff^3)^*\otimes (\aff^3)^*$ by this pairing. 
The map 
\begin{math}
\iota :\sym^2 \aff^3 \ni v_1v_2\mapsto 
\frac 12(v_1\otimes v_2+v_2\otimes v_1) \in \aff^3\otimes \aff^3
\end{math}
is equivariant with respect to the group action. 
There is a similar map $\iota^*$ for $\sym^2 (\aff^3)^*$ 
also.  

Let $\{v_1,v_2,v_3\}$ be the 
dual basis of $\{u_1,u_2,u_3\}$. 
Then $(\iota(u_{i_1}u_{i_2}),\iota^*(v_{j_1}v_{j_2}))=0$ 
unless $\{i_1,i_2\}=\{j_1,j_2\}$ and  
\begin{equation*}
(\iota (u_i^2),\iota^*(v_i^2)) = 1\; (i=1,2,3),\;
(\iota (u_iu_j),\iota^*(v_iv_j)) = \frac 12 \; 
(1\leq i<j\leq 3). 
\end{equation*}
By the above consideration, we can identify 
$l_1^*\ccd l_6^*$ with 
$v_1^2,2v_1v_2,2v_1v_3,v_2^2,2v_2v_3,v_3^2$. 
For $z=z_1v_1^2+z_2v_1v_2+z_3v_1v_3+z_4v_2^2+z_5v_2v_3+z_6v_3^2$, 
we put
\begin{equation*}
P_2(z) = \det 
\begin{pmatrix}
2z_1 & z_2 & z_3 \\
z_2 & 2z_4 & z_5 \\
z_3 & z_5 & 2z_6
\end{pmatrix}. 
\end{equation*}
Then it is well known that 
$P_2(z)$ is a relative invariant polynomial. 
So $P_2((\sym^2 g)^*z)= (\det g)^{10}P_2(z)$. 
Since $y$ can be identified with 
$z=y_1v_1^2+2y_2v_1v_2+2y_3v_1v_3+y_4v_2^2+2y_5v_2v_3+2y_6v_3^2$
and $P_2(z)=-8P_1(y)$, 
$8P_1((\sym^2 g)^*y)= 8(\det g)^{10}P_1(y)$.  
We can divide by $8$ over $\Q$ and 
$P_1((\sym^2 g)^*y)= (\det g)^{10}P_1(y)$.  
\end{proof}

Let $P(x)=P_1(\Phi(x))$. The following corollary 
follows from (\ref{eq:S1-Phi-equivariant}), 
(\ref{eq:S1-value-phi-w})  
and Proposition \ref{prop:sym2*-invariant}.
\begin{cor}
\label{cor:value-Phi-ww} 
\begin{itemize}
%
%
\item[(1)]
$P((g_1,g_2)x) = (\det g_1)^{12}(\det g_2)^{10}P(x)$. 
\item[(2)]
$P(w)=4$, $P(w')=1$. 
\end{itemize}
\end{cor}

By Proposition \ref{prop:sym2*-invariant} and 
Corollary \ref{cor:value-Phi-ww}, 
$P(x)$ is a non-zero relative invariant polynomial 
without any assumption on $\ch(k)$.  
Even though $P(w')\not=0$ regardless of 
$\ch(k)$, $\dim T_{e_G}(G_{w'})\not=4$ if $\ch(k)=2$ 
and $G_{w'}$ may not be isomorphic to 
$\mathrm{PGL}_2\times \gl_1$ even if $\ch(k)\not=2$. 
So we had to use $w$ instead.

We are in the situation of of Section \ref{sec:regularity} 
where $m=N=1$. So if $\ch(k)\not=2$, 
by Corollary \ref{cor:reducible-sep-orbit}, 
$V^{\sst}_{k^{\sep}}=G_{k^{\sep}}w$. Therefore, we can use the 
standard argument of Galois cohomology.
Since $\mathrm{PGL}_2$ is the automorphism group of 
$\mathrm{PGL}_2$, 
$\h^1(k,\mathrm{PGL}_2)$ is in bijective correspondence
with $k$-forms of $\mathrm{PGL}_2$. 
However, we would like to be more explicit. 

Let $\ti T=\{(tI_5,t^{-2}I_3)\mid t\in\gl_1\}$. 
\begin{prop}
\label{prop:S1-orbit-final}
Suppose that $\ch(k)\not=2$. 
Then $G_k\backslash V^{\sst}$ is in bijective correspondence with 
$\mathrm{Prg}_2(k)$. If $x\in V^{\sst}_k$ then the corresponding 
$k$-form of $\mathrm{PGL}_2$ is $G_x/\ti T$.
%
\end{prop}
\begin{proof}
The first statement has already been shown. 
Note that $\ti T$ is contained in the center of $G$ 
and $G_w = \phi(\gl_2)\ti T$. 
Since $\h^1(k,\gl_1)=\{1\}$, $(G_x/\ti T)_k\cong G_{x\,k}/\ti T_k$. 
Let $x\in V^{\sst}_k$. Then there exists $g_x\in G_{k^{\sep}}$ 
such that $x=g_x w$. So 
\begin{equation*}
G_{x\,k^{\sep}}=g_x \phi(\gl_2(k^{\sep})) g_x^{-1} \ti T_{k^{\sep}}.  
\end{equation*}

If $g\in \phi(\gl_2(k^{\sep})),t\in \ti T_{k^{\sep}}$ 
and $\sig\in\gal(k^{\sep}/k)$ 
then $g_x gt g_x^{-1}\in G_k$ if and only if $t\in \ti T_k$ and 
$g_x^{\sig} g^{\sig} (g_x^{\sig})^{-1}=g_x g g_x^{-1}$. 
This is equivalent to $(g^{-1}g_x^{\sig})g^{\sig}(g^{-1}g_x^{\sig})^{-1}=g$. 
Let $h_{\sig}=g^{-1}g_x^{\sig}$ and $h=\{h_{\sig}\}$. Then the class of 
the 1-cocycle $h$ is the element of $\h^1(k,G_w)$ corresponding to $x$. 
The above condition is equivalent to $h_{\sig}g^{\sig}h_{\sig}^{-1}=g$ 
for all $\sig$. The set of such $g$ is precisely the $k$-form of 
$\mathrm{PGL}_2$ corresponding to $x$. Therefore, 
$G_x/\ti T$ is the $k$-form of $\mathrm{PGL}_2$ corresponding to $x$. 
\end{proof}

\section{Rational orbits (3)}
\label{sec:rational-orbits-332.3}

In this section let 
$G_1=G_2=\gl_3$, $G_3=\gl_2$, $W_1=W_2=\aff^3$. 
We regard $W_1,W_2$ the standard \rep s of $G_1,G_2$  
respectively. Let $\aff^2$ be the standard \rep{} of $G_3=\gl_2$. 
We denote the standard basis of $W_1,W_2,\aff^2$ by 
$\{\bbmp_1,\bbmp_2,\bbmp_3\}$, 
$\{\bbmq_1,\bbmq_2,\bbmq_3\}$,
$\{\bbmr_1,\bbmr_2\}$ respectively. 
Let $G=G_1\times G_2\times G_3\times \gl_1$ and  
$V=\wedge^2 W_2 \oplus W_1\otimes W_2\otimes \aff^2$. 
We denote elements of $G,V$ by $g=(g_1,g_2,g_3,t),x=(x_1,x_2)$ 
respectively. 
We define an action of $t\in\gl_1$ on $V$ 
by $V\ni (x_1,x_2)\mapsto (tx_1,x_2)\in V$. 
With the above \rep s of $G_1,G_2,G_3$,  
$V$ is a \rep{} of $G$. 

Let 
\begin{equation}
\label{eq:set6-w1}
w_1 = \bbmq_1\wedge \bbmq_2-\bbmq_1\wedge \bbmq_3
+\bbmq_2\wedge \bbmq_3
= \begin{pmatrix}
0 & 1 & -1 \\
-1 & 0 & 1 \\
1 & -1 & 0 
\end{pmatrix}
\end{equation}
and $w_{21},w_{22}$ be the elements $w_1,w_2$ 
in (\ref{eq:sec6-w-defn-sec4}). We put 
$w_2=(w_{21},w_{22})$ and $w=(w_1,w_2)\in V$.

We determine the Lie algebra of ${\mathrm T}_{e_G}(G_w)$. 
Let $k[\vep]/(\vep^2)$ be the ring of dual numbers as before
and $X=(x_{ij}),Y=(y_{ij})\in\m_3,Z=(z_{ij})\in\m_2,a\in \aff^1$. 
We express elements of ${\mathrm T}_{e_G}(G)$ in the form
$e_G+\vep(X,Y,Z,a)$.  
\begin{lem}
\label{lem:lie-alg-sec6}
$\dim {\mathrm T}_{e_G}(G_w)=2$. 
\end{lem}
\begin{proof}
By definition, 
$e_G+\vep(X,Y,Z,a)\in {\mathrm T}_{e_G}(G_w)$ 
if and only if 
$(e_G+\vep(X,Y,Z,a))w_i=w_i$ for $i=1,2$.

It is proved in the proof of Lemma \ref{lem:S3-lie-alg} that 
$(e_G+\vep(X,Y,Z,a))w_2=w_2$ implies that $X,Y$ are diagonal 
matrices, $Z$ is a scalar matrix and (\ref{eq:S3-lie-alg}) is satisfied. 
Note that the last 
$\gl_1$ acts trivially on the second component of $V$. 
Also $(e_G+\vep(X,Y,Z,a))w_1=w_1$ if and only if 
\begin{equation*}
Xw_1 + w_1 {}^tX + aw_1=0.
\end{equation*}
This implies that $x_{11}+x_{22}=x_{11}+x_{33}=x_{22}+x_{33}$ 
and so $x_{11}=x_{22}=x_{33}$. Therefore, $X$ is a scalar matrix.
By (\ref{eq:S3-lie-alg}), $Y$ is a  scalar matrix also.  
We put $X=\al I_3,Y=\be I_3,Z=\gam I_2$. Then 
$\gam=-\al-\be,a=-2\al$. 
\end{proof}

It is easy to see that elements of the form 
\begin{math}
(t_1I_3,t_2I_3,(t_1t_2)^{-1}I_2,t_2^{-2})
\end{math}
fix $w$. 
By Lemma \ref{lem:lie-alg-sec6}
$\dim {\mathrm T}_{e_G}(G_w)=2=\dim G-\dim V$. So 
Proposition \ref{prop:open-orbit} implies 
that $Gw\sub V$ is Zariski open and that $G_w$ is smooth over $k$. 
Therefore, we obtain the following corollary. 
\begin{cor}
\label{cor:S3-regular-sec6}
\begin{itemize}
\item[(1)]
\begin{math}
G^{\circ}_w=\{(t_1I_3,t_2I_3,(t_1t_2)^{-1}I_2,t_2^{-2})\mid t_1,t_2\in\gl_1\}
\cong \gl_1^2.
\end{math}
\item[(2)]
$G_w$ is smooth reductive and so $(G,V)$ is a regular \pv. 
\end{itemize}
\end{cor} 
Since $k$ is arbitrary and  $G,V,w$ are defined over $\Z$, 
$w$ is universally generic 
(see Definition \ref{defn:universally-generic}).

Let 
\begin{equation}
\label{eq:tau12-defn-sec6}
\begin{aligned}
\tau_1 & = 
\left(
\begin{pmatrix}
0 & 1 & 0 \\
1 & 0 & 0 \\
0 & 0 & 1 
\end{pmatrix},
\begin{pmatrix}
0 & 1 & 0 \\
1 & 0 & 0 \\
0 & 0 & 1 
\end{pmatrix},
\begin{pmatrix}
-1 & 0 \\   
1 & 1 
\end{pmatrix},-1\right), \\
\tau_2 & = 
\left(
\begin{pmatrix}
1 & 0 & 0 \\
0 & 0 & 1 \\
0 & 1 & 0 
\end{pmatrix},
\begin{pmatrix}
1 & 0 & 0 \\
0 & 0 & 1 \\
0 & 1 & 0 
\end{pmatrix},
\begin{pmatrix}
1 & 1 \\
0 & -1 
\end{pmatrix},-1\right).
\end{aligned}
\end{equation}
Then $\tau_1,\tau_2$ fix $w$ and generate 
a subgroup of $G_k$ isomorphic to $\gS_3$. 
The above $\tau_1,\tau_2$ coincide with those 
in (\ref{eq:tau12-defn}) except for the last component $-1$. 
We denote the subgroup of $G_k$ generated by 
$\tau_1,\tau_2$ by $H$.

\begin{prop}
\label{prop:S3-stabilizer-sec6}
$G_w\cong G^{\circ}_w\times H\cong G^{\circ}_w\times \gS_3$. 
\end{prop}
\begin{proof}
Suppose that $g=(g_1,g_2,g_3)\in G_w$.  
Since the last $\gl_1$ acts on the second component trivially, 
by Proposition \ref{prop:S3-stabilizer}, 
we may assume $g_1,g_2$ are diagonal matrices and 
$g_3$ is a scalar matrix after multiplying an element of $H$ 
if necessary. By considering $g_1 w_1{}^t g_2$, 
$g_1,g_2$ are scalar matrices. Then it is easy to see that 
$g\in G^{\circ}_w$. 
Since $G^{\circ}_w$ is contained in the center of $G_w$, 
$G_w\cong G^{\circ}_w\times H$. 
\end{proof}

Let $P_1(x)$ be the degree $12$ polynomial 
on $W_1\otimes W_2\otimes \aff^2$ in Proposition \ref{prop:332-invariant}.
Then 
\begin{equation}
\label{eq:sec6-P1}
P_1(gx) = (\det g_1)^4(\det g_2)^4(\det g_3)^6P_1(x). 
\end{equation}
and $P_1(w)=1$. 

We construct another relative invariant polynomial on $V$. 
We first construct an equivariant map from 
$W_1\otimes W_2\otimes \aff^2$ to $W_2\otimes W_2\otimes W_2$. 
Let $\Phi_1:W_1\otimes W_2\otimes \aff^2\to (W_1\otimes W_2\otimes \aff^2)^{6\otimes}$
be the map defined by 
\begin{equation*}
\Phi_1(x) = \overbrace{x\otimes \cdots \otimes x}^6.
\end{equation*}
We identify $\wedge^3 W_1,\wedge^3 W_2\cong \aff^1$, 
$\wedge^2 \aff^2\cong \aff^1$ in the usual manner. 
We define a linear map 
$\Phi_2:(W_1\otimes W_2\otimes \aff^2)^{6\otimes}\to (W_2\otimes \aff^2)^{6\otimes}$
so that 
\begin{align*}
& \Phi_2((v_{11}\otimes v_{12}\otimes v_{13})\otimes \cdots \otimes 
(v_{61}\otimes v_{62}\otimes v_{63})) \\
& = (v_{11}\wedge v_{21}\wedge v_{31}) (v_{41}\wedge v_{51}\wedge v_{61})
(v_{12}\otimes v_{13})\otimes \cdots \otimes 
(v_{62}\otimes v_{63}). 
\end{align*}
We define a linear map 
$\Phi_3:(W_2\otimes \aff^2)^{6\otimes}\to W_2^{6\otimes}$
so that 
\begin{align*}
& \Phi_3((v_{12}\otimes v_{13})\otimes \cdots \otimes 
(v_{62}\otimes v_{63})) \\
& = (v_{13}\wedge v_{43}) (v_{23}\wedge v_{53}) (v_{33}\wedge v_{63}) 
v_{12}\otimes \cdots \otimes v_{62}. 
\end{align*}
We define a linear map 
$\Phi_4:W_2^{6\otimes}\to W_2^{3\otimes}$
so that 
\begin{equation*}
\Phi_4(v_{12}\otimes \cdots \otimes v_{62})
= - (v_{12}\wedge v_{22}\wedge v_{42}) 
v_{32} \otimes v_{52} \otimes v_{62}.
\end{equation*}
Let $\Phi=\Phi_4\circ \Phi_3\circ \Phi_2\circ \Phi_1$. 

We define a polynomial $P_2(x)$ on $V$ by 
\begin{equation*}
P_2(x_1,x_2)=\frac 16(x_1\otimes x_1\otimes x_1) \wedge \Phi(x_2)
\end{equation*}
for $x_1\in \wedge^2 W_2,x_2\in W_1\otimes W_2\otimes \aff^2$ 
($\wedge$ is applied to all three components of 
$\Phi(x_2)\in W_2^{3\otimes}$). 
We shall explain in Proposition \ref{prop:S5-second-equivariant} 
that the coefficient $\tfrac 16$ can be justified. 

\begin{prop}
\label{prop:S5-second-equivariant}
\begin{itemize}
\item[(1)]
$P_2$ is a polynomial, homogeneous of degrees $3,6$ 
with respect to $x_1,x_2$ respectively and  
$P_2(gx) = t^3(\det g_1)^2(\det g_2)^4(\det g_3)^3P_2(x)$. 
\item[(2)]
\begin{math}
\Phi(w_2) = \sum_{\sig\in\gS_3}
\bbmq_{\sig(1)}\otimes \bbmq_{\sig(2)} \otimes \bbmq_{\sig(3)}.
\end{math}
\item[(3)]
\begin{math}
P_2(w) = 1. 
\end{math}
\end{itemize}
\end{prop}
\begin{proof}
(1) The first part of the statement of (1) 
is obvious. It is easy to see that 
\begin{equation*}
\Phi(gx)=(\det g_1)^2 (\det g_2) (\det g_3)^3(g_2\otimes g_2\otimes g_2) \Phi(x).
\end{equation*}
Since $\Phi(x)$ has three components of $W_2$ and 
$(tx)\otimes (tx)\otimes (tx)=t^3 x\otimes x\otimes x$, 
we obtain the second statement of (1). 

(2) Let 
\begin{align*}
a = \; & -(\bbmq_1 \otimes \bbmr_1)\otimes \bbmq_2\otimes (-\bbmr_1 + \bbmr_2) \otimes
(\bbmq_3 \otimes \bbmr_2) \\
& + (\bbmq_1 \otimes \bbmr_1)\otimes 
(\bbmq_3 \otimes \bbmr_2) \otimes \bbmq_2\otimes (-\bbmr_1 + \bbmr_2) \\
& + \bbmq_2\otimes (-\bbmr_1 + \bbmr_2) \otimes 
(\bbmq_1 \otimes \bbmr_1)\otimes 
(\bbmq_3 \otimes \bbmr_2) \\
& -\bbmq_2\otimes (-\bbmr_1 + \bbmr_2) \otimes 
(\bbmq_3 \otimes \bbmr_2) \otimes 
(\bbmq_1 \otimes \bbmr_1) \\
& -(\bbmq_3 \otimes \bbmr_2) \otimes 
(\bbmq_1 \otimes \bbmr_1)\otimes 
\bbmq_2\otimes (-\bbmr_1 + \bbmr_2) \\
& + (\bbmq_3 \otimes \bbmr_2) \otimes 
\bbmq_2\otimes (-\bbmr_1 + \bbmr_2) \otimes 
(\bbmq_1 \otimes \bbmr_1). 
\end{align*}
Then $\Phi_2\circ \Phi_1(w_2)=a\otimes a$. 

Straightforward computations show that 
\begin{align*}
\Phi_3(a\otimes a) = \;
& \bbmq_1 \otimes \bbmq_2 \otimes \bbmq_3 
\otimes \bbmq_2 \otimes \bbmq_3  \otimes \bbmq_1 
- \bbmq_1 \otimes \bbmq_2 \otimes \bbmq_3 
\otimes \bbmq_3 \otimes  \bbmq_1 \otimes \bbmq_2 \\
& + \bbmq_1 \otimes \bbmq_3 \otimes \bbmq_2
\otimes \bbmq_2 \otimes \bbmq_1 \otimes \bbmq_3  
- \bbmq_1 \otimes \bbmq_3 \otimes \bbmq_2
\otimes \bbmq_3 \otimes \bbmq_2\otimes \bbmq_1 \\
& - \bbmq_2\otimes \bbmq_1 \otimes \bbmq_3 
\otimes \bbmq_1 \otimes \bbmq_3 \otimes \bbmq_2 
+ \bbmq_2\otimes \bbmq_1 \otimes \bbmq_3 
\otimes \bbmq_3 \otimes \bbmq_2\otimes \bbmq_1 \\
& - \bbmq_2\otimes \bbmq_3 \otimes \bbmq_1 
\otimes \bbmq_1 \otimes \bbmq_2\otimes \bbmq_3 
+ \bbmq_2\otimes \bbmq_3 \otimes \bbmq_1 
\otimes \bbmq_3 \otimes \bbmq_1 \otimes \bbmq_2 \\
& + \bbmq_3 \otimes \bbmq_1 \otimes \bbmq_2 
\otimes \bbmq_1 \otimes \bbmq_2\otimes \bbmq_3 
- \bbmq_3 \otimes \bbmq_1 \otimes \bbmq_2 
\otimes \bbmq_2\otimes \bbmq_3 \otimes \bbmq_1 \\
& + \bbmq_3 \otimes \bbmq_2\otimes \bbmq_1 
\otimes \bbmq_1 \otimes \bbmq_3 \otimes \bbmq_2 
- \bbmq_3 \otimes \bbmq_2\otimes \bbmq_1 
\otimes \bbmq_2\otimes \bbmq_1 \otimes \bbmq_3.  
\end{align*}
Then 
\begin{align*}
\Phi(w_2) & = \Phi_4\circ \Phi_3(a\otimes a) \\
& = 0 + \bbmq_3\otimes \bbmq_1 \otimes \bbmq_2 
+ \bbmq_2\otimes \bbmq_1 \otimes \bbmq_3 + 0 + 0 
+ \bbmq_3\otimes \bbmq_2 \otimes \bbmq_1 \\
& \quad + \bbmq_1\otimes \bbmq_2 \otimes \bbmq_3
+ 0 + 0 + \bbmq_2\otimes \bbmq_3 \otimes \bbmq_1
+ \bbmq_1\otimes \bbmq_3 \otimes \bbmq_2 + 0.    
\end{align*}

(3) Note that  
$w_1\wedge \bbmq_i= \bbmq_1\wedge \bbmq_2\wedge \bbmq_3$ 
for $i=1,2,3$.  Therefore, 
$6P_2(w)=6$. Since $w$ is universally generic, 
Proposition \ref{prop:universally-generic} implies
that $P_2$ is defined over $\Z$, $P_2(w)=1$ and that 
$6P_2(x_1,x_2) = (x_1\otimes x_1\otimes x_1) \wedge \Phi(x_2)$. 
\end{proof}

So we are in the situation of Section \ref{sec:regularity} 
where $m=N=2$. We put $U=\{x\in V\mid P_1(x),P_2(x)\not=0\}$. 
Then by Corollary \ref{cor:reducible-sep-orbit}, 
$U_{k^{\sep}}=G_{k^{\sep}}w$. So we can use the 
standard argument of Galois cohomology. 

\begin{prop}
\label{prop:S3-orbit-rational}
The map $U_k\ni (x_1,x_2) \mapsto x_2\in \aff^3\otimes \aff^3\otimes \aff^2$ 
induces a bijection
\begin{equation*}
G_k\backslash U_k\cong (\gl_3(k)^2\times \gl_2(k))
\backslash (\aff^3\otimes \aff^3\otimes \aff^2)^{\sst}_k
\cong \Ex_3(k). 
\end{equation*}
\end{prop}
\begin{proof}
Since $\h^1(k,G)=\{1\}$, 
$G_k\backslash U_k\cong \h^1(k,G^{\circ}_w\times H)
\cong \h^1(k,\gS_3)\cong \Ex_3(k)$. 

Suppose that $z^3+a_1z^2+a_2z+a_3\in k[z]$ has distinct roots 
$\al=(\al_1,\al_2,\al_3)$. 
Let $a=(a_1,a_2,a_3),F=k(\al)=k(\al_1,\al_2,\al_3)$ and  
$D(\al)=(\al_1-\al_2)(\al_1-\al_3)(\al_2-\al_3)$. 
We put 
\begin{align*}
& P_{\al} = 
\begin{pmatrix}
1 & 1 & 1 \\
\al_1 & \al_2 & \al_3 \\
\al_1^2 & \al_2^2 & \al_3^2
\end{pmatrix},\;
Q_{\al} = D(\al)^{-1}
\begin{pmatrix}
-(\al_3-\al_2) & -(\al_1-\al_2) \\
\al_1(\al_3-\al_2) & \al_3(\al_1-\al_2)
\end{pmatrix}, \\
& g_{\al} = \left(P_{\al},P_{\al},Q_{\al},D(\al)\right), \;
x_a=g_{\al}w.
\end{align*}
Then $g_{(\al_2,\al_1,\al_3)}=g_{\al}\tau_1$, 
$g_{(\al_1,\al_3,\al_2)}=g_{\al}\tau_2$ (see (\ref{eq:tau12-defn})). 
Therefore, for any $\sig\in\gal(F/k)$, 
$g_{\al}^{-1}g_{\al}^{\sig}\in H$. 
This implies that $x_a\in U_k$. 

For $\sig\in \gal(F/k)$ let $h_{\sig}=g_{\al}^{-1}g_{\al}^{\sig}$. 
If $\sig_1,\sig_2\in \gal(F/k)$ then 
$h_{\sig_1\sig_2}=h_{\sig_2}h_{\sig_1}^{\sig_2}$. 
The action of $\gal(F/k)$ on $\al$ enables us to identify 
$\gal(F/k)$ with a subgroup of $\gS_3\cong H$. 
Since any element of $\gS_3$ is a  finite product of 
transpositions $(1\;2),(2\;3)$, 
$h_{\sig}\in G_k$ and so 
$h_{\sig_1\sig_2}=h_{\sig_2}h_{\sig_1}$. 
If $\sig\in H\cong \gS_3$ and $\sig=\sig_1\cdots \sig_n$ 
where $\sig_1\ccd \sig_n=(1\;2)$ or $(2\;3)$ then 
$h_{\sig}=h_{\sig_n}\cdots h_{\sig_1}$ and 
$h_{\sig_i}$ is either $\tau_1$ or $\tau_2$. 
This implies that $h_{\sig}$ can be identified with 
$\sig^{-1}$. Therefore, the field corresponding to 
$x_a$ is $F$. 

If we associate $x_2$ to $x=(x_1,x_2)$ then 
the element corresponding to $x_a$ is the projection 
to the second component. If $x_a=(x_{a,1},x_{a,2})$ 
then $x_{a,2}=g_{\al}w_2$. So the cohomology class 
corresponding to $x_{a,2}$ is the same as that of $x_a$. 
Therefore, $x=(x_1,x_2)\mapsto x_2$ induces a bijection 
of rational orbits. 
\end{proof}

\section{Rational orbits (4)} 
\label{sec:rational-orbits-42-wedge42}

In this section let 
$G=\gl_4\times \gl_2^2$, $W=\aff^4$ and   
$V=W\otimes \aff^2\oplus \wedge^2 W\otimes \aff^2$. 
We consider the natural action of $\gl_4$ on $W,\wedge^2 W$. 
The first and the second $\gl_2$-factor act on 
the first and the second $\aff^2$ respectively. 
We identify $W\otimes \aff^2\cong \m_{4,2}$  
and $\wedge^2 W\otimes \aff^2$ with the space of 
pairs of $4\times 4$ alternating matrices. 
We express elements of $G,V$ as 
$g=(g_1,g_2,g_3)$, $x=(A,B_1,B_2)$ ($A\in\m_{4,2}$, $B_1,B_2$ 
$4\times 4$ alternating matrices).
Then the action of $g$ on $x$ is 
\begin{math}
g(A,B_1,B_2) = (g_1A{}^tg_2,B_1',B_2')
\end{math}
where 
\begin{equation*}
\begin{pmatrix}
B_1' \\
B_2'
\end{pmatrix}
= g_3
\begin{pmatrix}
g_1 B_1{}^t g_1 \\
g_1 B_2{}^t g_1 
\end{pmatrix}
\end{equation*}
treating $[g_1 B_1{}^t g_1,g_1 B_2{}^t g_1]$ 
as a column vector. 
If $x=(A,B_1,B_2)$ then we may denote $A,B_1,B_2$ by $A(x),B_1(x),B_2(x)$, 
$B(x)=(B_1(x),B_2(x))$.

Let 
$\{\bbmp_{4,1}\ccd \bbmp_{4,4}\}$ be the standard basis of $W$ 
and $\{\bbmp_{4,1}^*\ccd \bbmp_{4,1}^*\}$ its dual basis.
We use the notation such as $p_{4,12}=\bbmp_{4,1}\wedge \bbmp_{4,2}$, etc. 
We identify $p_{4,234},p_{4,134},p_{4,124},p_{4,123}$ with 
$\bbmp_1^*,-\bbmp_2^*,\bbmp_3^*,-\bbmp_4^*$ respectively.   
To distinguish standard bases of the two $\aff^2$'s, 
we denote the standard \rep{} and its standard basis of 
the first (resp. second) $\gl_2$ by $L_1,\{\bbmp_{2,1},\bbmp_{2,2}\}$
(resp. $L_2,\{\bbmq_{2,1},\bbmq_{2,2}\}$).

Let $w=(w_1,w_2)\in V$ where $w_2$ is the element $w$ in 
(\ref{eq:sec7-w-defn-sec4}) and 
\begin{equation}
\label{eq:sec7-w-defn-sec7}
\begin{aligned}
w_1 & = (\bbmp_{4,2}+\bbmp_{4,4})\otimes \bbmp_{2,1}
+(\bbmp_{4,1}+\bbmp_{4,3})\otimes \bbmp_{2,2}
= \begin{pmatrix}
0 & 1 \\
1 & 0 \\
0 & 1 \\
1 & 0
\end{pmatrix}.
\end{aligned}
\end{equation}

We determine the Lie algebra of ${\mathrm T}_{e_G}(G_w)$. 
Let $k[\vep]/(\vep^2)$ be the ring of dual numbers as before
and $X=(x_{ij})\in\m_4,Y=(y_{ij}),Z=(z_{ij})\in\m_2$. 
We express elements of ${\mathrm T}_{e_G}(G)$ in the form
$e_G+\vep(X,Y,Z)$.  
\begin{lem}
\label{lem:section7-Lie-alg-sec7}
$\dim {\mathrm T}_{e_G}(G_w)=4$. 
\end{lem}
\begin{proof}
We consider the condition $(e_G+\vep(X,Y,Z))w_i=w_i$ for $i=1,2$. 
It is proved in the proof of Lemma \ref{lem:section7-Lie-alg} that 
$X,Z$ are in the form of $X,Y$ in (\ref{eq:42-lie-alg-sec4}).  
It is easy to see that 
$(e_G+\vep(X,Y,Z))w_1=w_1+\vep(X',0,0)$ where 
\begin{align*}
X' & = 
X \begin{pmatrix}
0 & 1 \\
1 & 0 \\
0 & 1 \\
1 & 0
\end{pmatrix}
+ \begin{pmatrix}
0 & 1 \\
1 & 0 \\
0 & 1 \\
1 & 0
\end{pmatrix}
{}^t Y
= \begin{pmatrix}
x_{12} + y_{12} & x_{11} + y_{22} \\  
x_{22} + y_{11} & x_{21} + y_{21} \\  
x_{34} + y_{12} & x_{33} + y_{22} \\      
x_{44} + y_{11} & x_{43} + y_{21}       
\end{pmatrix}.
\end{align*}
So $(e_G+\vep(X,Y,Z))w_i=w_i$ for $i=1,2$ if and only if 
$X,Y,Z$ are in the following form 
\begin{align*}
X & = 
\begin{pmatrix}
-y_{22} & -y_{12} & 0 & 0 \\
-y_{21} & -y_{11} & 0 & 0 \\
0 & 0 & -y_{22} & -y_{12} \\
0 & 0 & -y_{21} & -y_{11} 
\end{pmatrix},\;
Y = 
\begin{pmatrix}
y_{11} & y_{12} \\
y_{21} & y_{22} 
\end{pmatrix}, \\
Z & = 
\begin{pmatrix}
y_{11}+y_{22} & 0 \\
0 & y_{11}+y_{22} 
\end{pmatrix}. 
\end{align*}
\end{proof}

Let $\tau_0$ be as in (\ref{eq:J-defn}) and 
\begin{equation*}
\tau_1 = 
\begin{pmatrix}
0 & I_2 \\
I_2 & 0 
\end{pmatrix},\;
\tau=(\tau_1,I_2,\tau_0). 
\end{equation*}
Then $\tau$ fixes $w$. 

\begin{prop}
\label{prop:S5-regular-sec7}
\begin{itemize}
\item[(1)]
$Gw\sub V$ is Zariski open and $G_w$ is smooth
over $k$.
\item[(2)]
$G^{\circ}_w\cong \gl_2$. 
\item[(3)]
$G_w\cong G^{\circ}_w\times \{1,\tau\}\cong \gl_2\times \Z/2\Z$.
\end{itemize}
\end{prop}
\begin{proof}
(1) Since $\dim {\mathrm T}_{e_G}(G_w)=4=24-20=\dim G-\dim V$, 
Proposition \ref{prop:open-orbit} implies that 
$Gw\sub V$ is Zariski open and that $G_w$ is smooth
over $k$.

(2) If $h\in \gl_2$ we define
\begin{equation*}
f(h)=(\diag(\tau_0 {}^th^{-1}\tau_0 ,\tau_0 {}^th^{-1}\tau_0),
h,(\det h)I_2) \in G.
\end{equation*}
Then $f(\gl_2)$ fixes $w$.  
So $f(\gl_2)\sub G^{\circ}_w$. 
Since $\dim G^{\circ}_w=4$ and $f$ is 
an imbedding, $G^{\circ}_w=f(\gl_2)\cong \gl_2$. 

(3) We may assume that $k=\overline k$. 
Suppose that $g=(g_1,g_2,g_3)\in G_w$. 
Conjugation by $g$ induces an automorphism of 
$\spl_2\sub \gl_2\cong G^{\circ}_w$. 
Since all automorphisms of $\spl_2$ are inner, 
by multiplying an element of $G^{\circ}_w$, 
we may assume that $g$ commutes with elements of $\spl_2$.

By Proposition \ref{prop:S5-regular} (3), 
by multiplying $\tau$ is necessary, we may 
assume that $g_1=\diag(h_1,h_2)$ ($h_1,h_2\in\gl_2$). 
Since $\tau$ commutes with elements of $G^{\circ}_w$, 
we may still assume that $g$ commutes with 
elements of $\spl_2$. Elements of $\spl_2$ 
are imbedded into elements of $G$ of the form 
$(\diag(m,m),*,*)$ where $m\in\spl_2$. So 
$g_1$ must be in the form $\diag(\al I_2,\be I_2)$. 
Since $g$ fixes the first component of $w$, 
$g_2$ is a scalar matrix, say $\gam I_2$. Then  
$\al=\be$ and $\gam=\al^{-1}$. By multiplying an element 
of $G^{\circ}_w$, we may assume that $g_1=I_4,g_2=I_2$. 
Since the last two components of $w$ are linearly independent,
$g_3=I_2$.
\end{proof}

\begin{cor}
\label{cor:S5-regular-sec7}
$(G,V)$ is a regular \pv.  
\end{cor}

We construct two relative invariant polynomials. 

We define a map 
$\Phi_1:V\to W^*\otimes L_1\otimes L_2$, 
linear with respect to each of 
$W\otimes L_1,\wedge^2 W\otimes L_2$ 
so that  
\begin{align*}
\Phi_1(b_1\otimes c_1,b_2 \otimes c_2) 
= (b_1\wedge b_2)\otimes c_1\otimes c_2
\in \wedge^3 W \otimes L_1\otimes L_2 
\cong W^* \otimes L_1\otimes L_2
\end{align*}
for 
\begin{math}
b_1\in W,b_2\in \wedge^2 W,c_1\in L_1,c_2\in L_2. 
\end{math}

We put $u_i=\bbmp_{4,i}^*$ for $i=1\ccd 4$ and 
$u=(u_1\ccd u_4)$. 
We identify $W^* \otimes L_1\otimes L_2$ 
with the space of $2\times 2$ matrices with entries  which are 
linear combinations of $u$. 
So we write $\Phi_1(x)=\Phi_1(x)(u)$. 
For $y\in W^* \otimes L_1\otimes L_2$, 
we put $\Phi_2(y)(u)=\det y(u)\in\sym^2 W^*$. 
It is easy to see that
\begin{equation*}
\Phi_1(w)(u) = 
\begin{pmatrix}
u_3 & u_1 \\
-u_4 & -u_2
\end{pmatrix},\;
\Phi_2\circ \Phi_1(w)(u) =u_1u_4-u_2u_3. 
\end{equation*}
Let $P_1(x)$ be the discriminant of $\Phi_2\circ \Phi_1(x)$ 
for $x\in V$. By the above computations, 
$P_1(w)=1$.  Note that $P_1(x)$ is homogeneous of degree $8$ 
with respect to each of $W\otimes L_1,\wedge^2 W\otimes L_2$.  

Let $P_2(x)$ be the degree $4$ polynomial 
of $B(x)\in \wedge^2 W\otimes L_2$
given by Proposition \ref{prop:222-invariant}. 
Then $P_2(w)=1$. 
It is easy to see that 
\begin{equation}
\label{eq:S5-P1P2equivariant}
\begin{aligned}
\Phi_1(gx)(u) & = (\det g_1)g_2 \Phi_1(x)(u {}^tg_1^{-1}) {}^t g_3,\; \\
\Phi_2\circ \Phi_1(gx) & = (\det g_1)^2 (\det g_2)(\det g_3)
\Phi_2\circ \Phi_1(x)(u{}^tg_1^{-1}),\; \\
P_1(gx) & = (\det g_1)^6(\det g_2)^4(\det g_3)^4P_1(gx), \\
P_2(gx) & = (\det g_1)^2(\det g_3)^2P_2(gx).
\end{aligned}
\end{equation}

So we are in the situation of Section \ref{sec:regularity} 
where $m=N=2$. We put $U=\{x\in V\mid P_1(x),P_2(x)\not=0\}$. 
Then by Corollary \ref{cor:reducible-sep-orbit}, 
$U_{k^{\sep}}=G_{k^{\sep}}w$. Therefore, we can use the 
standard argument of Galois cohomology. 

\begin{prop}
\label{prop:S5-orbit-rational}
The map $U_k\ni x \mapsto B(x)\in \wedge^2 W\otimes L_2$ 
induces a bijection
\begin{equation*}
G_k\backslash U_k\cong (\gl_4(k)\times \gl_2(k))
\backslash (\wedge^2 W\otimes L_2)^{\sst}_k
\cong \Ex_2(k). 
\end{equation*}
\end{prop}
\begin{proof}
Since $\h^1(k,G)=\h^1(k,\gl_2)=\{1\}$, 
$G_k\backslash U_k\cong \h^1(k,\Z/2\Z)\cong \Ex_2(k)$.

Suppose that $z^2+a_1z+a_2\in k[z]$ has distinct roots 
$\al=(\al_1,\al_2)$. We put $a=(a_1,a_2),F=k(\al_1)$,  
\begin{equation*}
g_{\al} = \left(
\begin{pmatrix}
I_2 & I_2 \\
-\al_1I_2 & -\al_2I_2
\end{pmatrix},
(\al_1-\al_2)^4I_2,
\begin{pmatrix}
1 & 1 \\
-\al_1 & -\al_2
\end{pmatrix}
\right),\;
x_a=g_{\al}w.
\end{equation*}
Let $\sig\in\gal(F/k)$ 
be the non-trivial element. 
Then $g_{\al}^{\sig}=g_{\al}\tau$ and so $x_a\in U_k$. 
Since $g_{\al}^{-1}g_{\al}^{\sig}=\tau$, 
the element of $\h^1(k,\Z/2\Z)$ which 
corresponds to $x_a$
is determined by a 1-cocycle which 
sends $\sig$ to $\tau$.  

Let $H=\gl_4\times \gl_2$ and 
$W_1=\wedge^2 W\otimes L_2$.
Then $B(w)\in W_1$ is the element $w$ in (\ref{eq:sec7-w-defn-sec4}) 
and $H_{B(w)}/H^{\circ}_{B(w)}$ is represented by 
$\{1,(\tau_1,\tau_0)\}$ (see Proposition \ref{prop:S5-regular}). 

Since $B((g_1,g_2,g_3)x)=(g_1,g_3)B(x)$, 
the map $V\ni x \mapsto B(x)\in W_1$ 
is equivariant with respect to the natural action 
of $H$ on $W_1$. Since the map 
$G\ni (g_1,g_2,g_3)\mapsto (g_1,g_3)\in H$ 
sends $\tau$ to $(\tau_1,\tau_0)$, 
the cohomology class associated with 
$B(x_a)$ is determined by the 1-cocycle 
which sends $\sig$ to the non-trivial element of 
$\Z/2\Z$. Therefore, the map $U_k\ni x\mapsto B(x)\in W_1$ 
induces a bijection of rational orbits.   
\end{proof}

\section{Rational orbits (5)}
\label{sec:rational-orbits-wedge43}

Let $G=\gl_4\times \gl_3$, $W=\aff^4$ and 
(a) $V=\wedge^2 W\otimes \aff^3$, 
(b) $V=\wedge^2 W\otimes \aff^3\oplus W$.
We consider the natural action of $G$ in both 
cases. 

We first consider the case (a). 

We put
\begin{equation}
\label{eq:w-defn-423}
w_1= \begin{pmatrix}
J & 0 \\
0 & 0
\end{pmatrix},\;
w_2= \begin{pmatrix}
0 & J \\
J & 0
\end{pmatrix},\;
w_3=\begin{pmatrix}
0 & 0 \\
0 & J
\end{pmatrix}
\end{equation}
and $w=(w_1,w_2,w_3)$. 

Long but straightforward computations show the following proposition.
We do not provide the proof.
\begin{prop}
\label{prop:section8-Lie-alg}
${\mathrm T}_{e_G}(G_w)$ consists of elements of the form $e_G+\vep(X,Y)$ 
where  
\begin{equation}
\label{eq:A-defn}
\begin{aligned}
X & = \begin{pmatrix}
x_{11} & x_{12} & x_{13} & 0\\
x_{21} & x_{22} & 0 & x_{13} \\
x_{31} & 0 & x_{33} & x_{12} \\
0 & x_{31} & x_{21} & -x_{11}+x_{22}+x_{33}
\end{pmatrix}, \\ 
Y & = \begin{pmatrix}
-x_{11}-x_{22} & -x_{31} & 0 \\
-2x_{13} & -x_{22}-x_{33} & -2x_{31} \\
0 & -x_{13} & x_{11}-x_{22}-2x_{33}
\end{pmatrix}.
\end{aligned}
\end{equation}
\end{prop}

By Proposition \ref{prop:section8-Lie-alg}, 
$\dim {\mathrm T}_{e_G}(G_w)=7=25-18=\dim G-\dim V$. 
So Proposition \ref{prop:open-orbit} 
implies that the orbit $Gw\sub V$ is Zariski open and 
that $G_w$ is smooth over $k$. Therefore,  
$w$ is universally generic. 

The tensor product of two standard \rep s of $\gl_2$
induces a homomorphism 
$\rho:\gl_2\times \gl_2\to\gl(\aff^2\otimes \aff^2)\cong \gl_4$. 
We show that $G_w^{\circ}$ can be identified with 
the image of $\rho$. 

We identify $\aff^2\otimes \aff^2$ with $\m_2$. 
Then $\rho(g_1,g_2)M = g_1M{}^tg_2$. 
We identify $\m_2$ with $W$ using 
the basis $\{E_{11},E_{21},E_{12},E_{22}\}$ 
(see Section \ref{sec:notation}). 
The matrix \rep s of $\rho(g_1,I_2),\rho(I_2,g_2)$ for 
\begin{math}
g_2=\left(
\begin{smallmatrix}
\al & \be \\
\gam & \del  
\end{smallmatrix}
\right)
\end{math}
with respect to this basis are 
\begin{equation}
\label{eq:rho-defn}
\begin{pmatrix}
g_1 & 0 \\
0 & g_1
\end{pmatrix}, \quad
\begin{pmatrix}
\al I_2 & \be I_2 \\
\gam I_2 & \del I_2
\end{pmatrix}
\end{equation}
respectively. 

$\rho$ induces a \rep{} 
$\wedge^2 \rho:\gl_2\times \gl_2\to \gl(\wedge^2 \aff^4)$.
Since 
\begin{equation}
\label{eq:wedge2-rho}
\begin{aligned}
& \begin{pmatrix}
\al I_2 & \be I_2 \\
\gam I_2 & \del I_2
\end{pmatrix}
\begin{pmatrix}
a J & bJ \\
bJ & c J
\end{pmatrix}
\begin{pmatrix}
\al I_2 & \gam I_2 \\
\be I_2 & \del I_2
\end{pmatrix} \\
& = \begin{pmatrix}
(\al a + \be b)J & (\al b+ \be c)J \\
(\gam a+ \del b)J & (\gam b+ \del c) 
\end{pmatrix}
\begin{pmatrix}
\al I_2 & \gam I_2 \\
\be I_2 & \del I_2
\end{pmatrix} \\
& = \begin{pmatrix}
((\al a+\be b)\al + (\al b+\be c)\be)J
& ((\al a+\be b)\gam + (\al b+\be c)\del)J \\
((\gam a+\del b)\al+ (\gam b+\del c)\be)J 
& ((\gam a+\del b)\gam+ (\gam b+\del c)\del)J 
\end{pmatrix}, 
\end{aligned} 
\end{equation}
elements of the form $(I_2,g_2)$ 
leave the subspace $\lan w_1,w_2,w_3\ran$ 
invariant. Elements of the form $(g_1,I_2)$ 
act on the subspace $\lan w_1,w_2,w_3\ran$ 
by multiplication by $\det g_1$. 

For $g=(g_1,g_2)\in\gl_2\times \gl_2$, 
let $\rho_1(g)=(\rho_1(g)_{ij})\in\gl_3$ be the matrix 
such that 
\begin{equation}
\label{eq:gw-formula}
\begin{pmatrix}
\wedge^2 \rho(g)w_1 & 
\wedge^2 \rho(g)w_2 & 
\wedge^2 \rho(g)w_3 
\end{pmatrix}
= \begin{pmatrix}
w_1 & w_2 & w_3 
\end{pmatrix}
\rho_1(g) 
\end{equation}
(we are treating $\begin{pmatrix}w_1 & w_2 & w_3 \end{pmatrix}$ as a row vector).
It is easy to see that if $g,h\in \gl_2\times \gl_2$ then 
$\rho_1(gh)=\rho_1(g)\rho_1(h)$. 

The proof of the following lemma is similar to those of 
Lemmas \ref{lem:S14-w-fix}, 
\ref{lem:S1-rho-rho1-fix-w} and we do not provide the proof.
\begin{lem}
\label{lem:rho-rho1-fix-w}
Let $\phi=(\rho,{}^t\rho_1^{-1})$. 
Then the image of $\phi$ fixes $w$. 
\end{lem}
%

%
\begin{lem}
\label{lem:423-kernel}
$\kernel(\rho)=\{(t_1I_2,t_2I_2)\mid t_1t_2=1\}$.
\end{lem}
\begin{proof}
Clearly, 
$\kernel(\rho)\supset \{(t_1I_2,t_2I_2)\mid t_1t_2=1\}$.

Suppose that $(g_1,g_2)\in \kernel(\rho)$. 
If 
\begin{math}
g_2= \left(
\begin{smallmatrix}
\al & \be \\
\gam & \del 
\end{smallmatrix}
\right)
\end{math}
then 
\begin{equation*}
\rho(g_1,g_2)
= \begin{pmatrix}
\al g_1 & \be g_1 \\
\gam g_1 & \del g_1 
\end{pmatrix} = I_4.
\end{equation*}
So $\be=\gam=0$, $\al=\del$ and 
$g_1$ is a scalar matrix. Therefore, 
there exist $t_1,t_2\in\gl_1$ such that
$(g_1,g_2)=(t_1I_2,t_2I_2)$. 
Then it is easy to see that $t_1t_2=1$. 
\end{proof}

\begin{prop}
\label{prop:kernel-same}
\begin{itemize}
\item[(1)]
$\kernel(\phi)=\kernel(\rho)$.
\item[(2)]
$G^{\circ}_w = \im(\phi)\cong \im(\rho)\cong (\gl_2\times \gl_2)/\kernel(\rho)$. 
\end{itemize}
\end{prop}
\begin{proof}
(1) Clearly, $\kernel(\phi)\sub \kernel(\rho)$.
Suppose that $g=(g_1,g_2)\in \kernel(\rho)$. 
Since $\rho(g)=I_4$, 
$\wedge^2 \rho(g) w_i=w_i$ for $i=1,2,3$. 
Therefore, $\rho_1(g)=I_3$.

(2) By Lemma \ref{lem:423-kernel}, 
$\dim \im(\phi)=7=\dim G^{\circ}_w$. 
So $G^{\circ}_w = \im(\phi)$. 
\end{proof}

We shall prove later that $G_w=G^{\circ}_w$ 
after we construct an equivariant map from $V$ to 
$\sym^2 W$.

\begin{cor}
\label{cor:S9a-regular}
$(G,V)$ is a regular \pv{}.
\end{cor}

We construct an equivariant map $V\to \sym^2 W$. 

Let $\{\bbmp_{4,1}\ccd \bbmp_{4,4}\}$ be 
the standard basis of $W$ 
(this basis corresponds to the basis $\{E_{11},E_{21},E_{12},E_{22}\}$ of $\m_2$). 
We define a linear map $\iota:\wedge^2 W\to W^{2\otimes}$ 
so that $\iota(a_1\wedge a_2)=a_1\otimes a_2-a_2\otimes a_1$. 
This map is clearly well-defined. 
We regard $\sym^2 W$ as the space of quadratic forms on $W^*$. 
We identify $\wedge^4 W\cong \aff^1$ so that 
$p_{4,1234}$ corresponds to $1$.  We define a linear map
$\Phi_1:W^{6\otimes}\to \sym^2 W$ so that 
\begin{equation*}
\Phi_1(a_1\otimes a_2\otimes \cdots \otimes a_6)
= (a_1\wedge a_2\wedge a_3\wedge a_5)a_4a_6.
\end{equation*}
We define a map  
$\Phi:V=\wedge^2 W\otimes \aff^3\to \sym^2 W$ 
so that 
\begin{equation*}
\Phi(x_1,x_2,x_3)
= \frac 1 {12}\Phi_1\left(
\sum_{\sig\in\gS_3} \sgn(\sig)\iota(x_{\sig(1)})\otimes \iota(x_{\sig(2)})
\otimes \iota(x_{\sig(3)})
\right).
\end{equation*}
Note that $12\Phi$ is defined over $\Z$. 
We shall show that $12\Phi(w)$ is divisible by $12$ over $\Z$.
Then, since $w$ is universally generic (see Definition \ref{defn:PV-alternative}), 
$\Phi$ itself is defined over $\Z$ by Proposition \ref{prop:universally-generic}
and so $\Phi$ is defined over any field.  

It is easy to see that 
\begin{equation}
\label{eq:Phi-equivariant-explicit}
\Phi((g_1,g_2)x) = (\det g_1)(\det g_2) (\sym^2 g_1)\Phi(x)
.\end{equation}
%

In the following, we need notations for elements of 
$W^{2\otimes},W^{6\otimes},\sym^2W$. 
We denote $\bbmp_{4,i_1}\otimes \cdots \otimes \bbmp_{4,i_m}$ 
by $q_{i_1\cdots i_m}$ for $m=2,6$, $i_1\ccd i_m=1\ccd 4$ 
and $\bbmp_{4,i}\bbmp_{4,j}$ by $v_iv_j$ in $\sym^2 W$ for $i,j=1\ccd 4$.   
\begin{prop}
\label{prop:wedge2-3-equivariant-sym24}
$\Phi(w) = v_1v_4-v_2v_3$. 
\end{prop}
\begin{proof}
By definition, 
\begin{align*}
& \sum_{\sig\in\gS_3} \sgn(\sig)\iota(w_{\sig(1)})\otimes \iota(w_{\sig(2)})
\otimes \iota(w_{\sig(3)}) \\
& = (q_{12}-q_{21})\otimes (q_{14}-q_{41}-q_{23}+q_{32})\otimes (q_{34}-q_{43}) \\
& \quad - (q_{12}-q_{21})\otimes (q_{34}-q_{43})\otimes (q_{14}-q_{41}-q_{23}+q_{32}) \\
& \quad - (q_{14}-q_{41}-q_{23}+q_{32})\otimes (q_{12}-q_{21})\otimes (q_{34}-q_{43}) \\
& \quad + (q_{14}-q_{41}-q_{23}+q_{32})\otimes (q_{34}-q_{43})\otimes (q_{12}-q_{21}) \\
& \quad + (q_{34}-q_{43})\otimes (q_{12}-q_{21})\otimes (q_{14}-q_{41}-q_{23}+q_{32}) \\
& \quad - (q_{34}-q_{43})\otimes (q_{14}-q_{41}-q_{23}+q_{32})\otimes (q_{12}-q_{21}).
\end{align*}
Expanding the tensor product and 
applying $\Phi_1$ to each term, 
we obtain $12\Phi(w)=12(v_1v_4-v_2v_3)$. 
We do not provide further details. 
\end{proof}

Let $P(x)$ be the discriminant of $\Phi(x)$. 
Then $P(w)=1$. The following lemma follows from 
(\ref{eq:Phi-equivariant-explicit}).
\begin{lem}
\label{lem:relative-invariant-423}
$P(x)$ is a non-zero homogeneous polynomial 
of degree $12$ and if $g=(g_1,g_2)\in G$ then 
$P(gx)=(\det g_1)^6(\det g_2)^4\Phi(x)$. 
\end{lem}

We define an action of 
$(t,g_1)\in H=\gl_1\times \gl_4$ on $\sym^2 W$ by 
$\sym^2 W\ni Q\mapsto t(\sym^2 g_1)Q\in\sym^2 W$
where 
$t(\sym^2 g_1)Q$ is the scalar multiplication by $t$
to $(\sym^2 g_1)Q$.  
We define a homomorphism $f:G\to H$ 
by $f(g_1,g_2)=((\det g_1)(\det g_2),g_1)$. 
Then $\Phi((g_1,g_2)x)=f(g_1,g_2)\Phi(x)$.   
Therefore, 
\begin{math}
f(G^{\circ}_w)\sub H^{\circ}_{\Phi(w)}. 
\end{math}

Let 
\begin{equation}
\label{eq:tau-defn}
\tau = 
\begin{pmatrix}
1 & 0 & 0 & 0  \\
0 & 0 & 1 & 0 \\
0 & 1 & 0 & 0 \\
0 & 0 & 0 & 1 
\end{pmatrix}.
\end{equation}
\begin{prop}
\label{prop:Gw-section8}
$G_w=G^{\circ}_w$. Moreover, $f$ induces an isomorphism 
$G_w=G^{\circ}_w\cong H^{\circ}_{\Phi(w)}$. 
\end{prop}
\begin{proof}
We may assume that $k=\overline k$. 

By simple Lie algebra computations, $(H,\sym^2 W)$ 
is a regular \pv{} and $H\Phi(w)\sub \sym^2 W$ is Zariski 
open. Since $\dim H_{\Phi(w)}=7$ and $G^{\circ}_w\cong \im(\rho)$ 
by Proposition \ref{prop:kernel-same} (1), 
the projection from $G^{\circ}_w$ to $\gl_4$ is an isomorphism 
into its image and
$f(G^{\circ}_w)=H^{\circ}_{\Phi(w)}$. 
Let $Z\sub G^{\circ}_w$ be the center of $G^{\circ}_w$. 
Then $G^{\circ}_w/Z\cong H^{\circ}_{\Phi(w)}/f(Z)\cong \text{PGL}_2\times \text{PGL}_2$.  

Let $t=(\diag(t_1,t_1^{-1}),\diag(t_2,t_2^{-1}))\in\gl_2\times \gl_2$. 
Then 
\begin{align*}
\rho(t) & = \diag(t_1t_2,t_1^{-1}t_2,t_1t_2^{-1},t_1^{-1}t_2^{-1}), \\
\tau \rho \tau & = 
\diag(t_1t_2,t_1t_2^{-1},t_1^{-1}t_2,t_1^{-1}t_2^{-1}) 
= \rho((\diag(t_2,t_2^{-1}),\diag(t_1,t_1^{-1}))). 
\end{align*}
This implies that the conjugation by $(1,\tau)\in H$ induces 
an outer automorphism of $H^{\circ}_{\Phi(w)}/f(Z)$. 
By (3.12) \cite[p.556]{hayasaka-yukie-tamagawaI}, 
$H_{\Phi(w)}/H^{\circ}_{\Phi(w)}$ is represented by $\{1,(1,\tau)\}$.
  
If $g=(g_1,g_2)\in G_w$ then $f(g)\in H_{\Phi(w)}$. 
So $g_1\in \rho(\gl_2\times \gl_2)$ or 
$g_1\tau\in \rho(\gl_2\times \gl_2)$. 
Since (\ref{eq:gw-formula}) holds for $g$, 
$\wedge^2 g_1$ leaves the subspace $\lan w_1,w_2,w_3\ran$ 
invariant. If $g_1\tau\in \rho(\gl_2\times \gl_2)$ 
then $\wedge^2 \tau$ 
must leave the subspace $\lan w_1,w_2,w_3\ran$ invariant. 
However, $\tau w_2 \tau \notin \lan w_1,w_2,w_3\ran$. 
Therefore, $g_1\in \rho(\gl_2\times \gl_2)$. 
If $g_1=\rho(h_1,h_2)$ with $h_1,h_2\in\gl_2$ then $g_2$ has to be 
${}^t\rho_1(h_1,h_2)^{-1}$. Therefore, $g\in G^{\circ}_w$.
\end{proof}
\begin{prop}
\label{prop:stab-iso-wedge42-3}
The map $G_k\backslash V^{\sst}_k\ni G_k x\mapsto G_x\in \mathrm{IQF}_4(k)$ 
is bijective ($G_kx$ is the orbit and $G_x$ is the stabilizer).  
\end{prop}
\begin{proof}
For the definition of $\mathrm{IQF}_4(k)$, 
see Definition \ref{defn:Ex-defn}. 
We would like to interpret $G_k\backslash V^{\sst}_k$ 
using the equivariant map $\Phi$. 

Since $\h^1(k,G)=\{1\}$, $G_k\backslash V^{\sst}_k$ 
is in bijective correspondence with 
$\h^1(k,G_w)=\h^1(k,G^{\circ}_w)\cong \h^1(k,H^{\circ}_{\Phi(w)})$. 
By LEMMA (1.2) \cite[p.118]{yukiel}, the following sequence 
\begin{equation*}
1\to (\Z/2\Z)\backslash \h^1(k,H^{\circ}_{\Phi(w)})\to \h^1(k,H_{\Phi(w)}) 
\to \h^1(k,\Z/2\Z)
\end{equation*}
is exact where the action of $\Z/2\Z$ on $\h^1(k,H^{\circ}_{\Phi(w)})$ 
is by the conjugation by $\tau$. 
It is known that $\h^1(k,H_{\Phi(w)})$ is 
in bijective correspondence with 
$\mathrm{QF}_4(k)$. For this, see 
Proposition 3.15 \cite[p.557]{hayasaka-yukie-tamagawaI}
for example. We shall show below that $\Z/2\Z$ acts on
$\h^1(k,H^{\circ}_{\Phi(w)})$ trivially.

Let $i:\gl_1\to \gl_2\times \gl_2$ be the homomorphism 
defined by $i(t)=(tI_2,t^{-1}I_2)$.  
We identify $H^{\circ}_{\Phi(w)}\cong (\gl_2\times \gl_2)/i(\gl_1)$ 
by $f\circ \phi$ (see Lemma \ref{lem:rho-rho1-fix-w}). 
Let $\psi_1:\gl_1 \to\gl_1$, 
$\psi_2:\gl_2\times \gl_2\to \gl_2\times \gl_2$
be the automorphisms such that 
$\psi_1(t)=t^{-1}$, $\psi_2(g_1,g_2)=(g_2,g_1)$. 
The conjugation by $\tau$ on $H^{\circ}_{\Phi(w)}$ 
is denoted by $\psi_3$. Then the following diagram 
is commutative. 
\begin{equation*}
\xymatrix{
1 \ar[r] & \gl_1 \ar[r]\ar[d]^{\psi_1} & \gl_2\times \gl_2 \ar[r]\ar[d]^{\psi_2}
& H^{\circ}_{\Phi(w)} \ar[r]\ar[d]^{\psi_3} & 1 \\
1 \ar[r] & \gl_1 \ar[r] & \gl_2\times \gl_2 \ar[r] 
& H^{\circ}_{\Phi(w)} \ar[r] & 1 
}
\end{equation*}

Let $T_1=\{tI_2\mid t\in\gl_1\}$, $T_2=i(\gl(1))\sub \gl_2\times \gl_2$.
For $g\in \gl_2$ or $h=(h_1,h_2)\in \gl_2\times \gl_2$, we use the notation 
$\overline g=g\mod T_1$, $\overline h= h \mod T_2$. 
We denote continuous maps from $\gal(k^{\sep}/k)$ to $\gl_2(k^{\sep})$, etc., 
as $(h_{\sig})_{\sig}$ ($h_{\sig}\in\gl_2(k^{\sep})$). 
Note that we consider the discrete topology on $\gl_2(k^{\sep})$. 
For $h=(h_{\sig})_{\sig}$ ($h_{\sig}\in \gl_2(k^{\sep})$, etc.),  
let $\del h_{\sig,\tau} = h_{\tau}h_{\sig}^{\tau}h_{\sig\tau}^{-1}$ 
and $\del h =(\del h_{\sig,\tau})_{\sig,\tau}$. This is a function on 
$\gal(k^{\sep}/k)\times \gal(k^{\sep}/k)$. 
If $h=(h_{\sig})_{\sig}$ ($h_{\sig}\in \gl_2(k^{\sep})$), 
$g\in \gl_2(k^{\sep})$ and $h'=(g^{-1}h_{\sig}g^{\sig})_{\sig}$ 
then $\del h_{\sig,\tau}=g^{-1}\del h'_{\sig,\tau}g$ for all $\sig,\tau$. 
So if $\del h_{\sig,\tau}\in T_{1\,k^{\sep}}$ for all $\sig,\tau$ 
then $\del h=\del h'$.

Suppose that $h_1=(h_{1,\sig})_{\sig}$, $h_2=(h_{2,\sig})_{\sig}$ 
($h_{1,\sig},h_{2,\sig}\in\gl_2(k^{\sep})$), 
$h=(h_1,h_2)$ and $\del h_{\sig,\tau}\in i(T(k^{\sep}))$ 
for all $\sig,\tau$. This implies that 
$(\del h_1)_{\sig,\tau}(\del h_2)_{\sig,\tau}=1$ for all $\sig,\tau$. 
Then if we put $\overline h_1=(\overline h_{1,\sig})_{\sig}$, etc., 
then $\del \overline h$ is trivial. So $\overline h$ is a 1-cocycle 
with coefficients in $H^{\circ}_{\Phi(w)}$. 
Any 1-cocycle with coefficients in $H^{\circ}_{\Phi(w)}$ 
can be expressed in this manner. 
We define $\psi_2(h),\psi_3(\overline h)$ in the obvious manner.

Let $h$ be as above such that $\del \overline h=1$.
We show that $\overline h,\psi_3(\overline h)$ 
are in the same cohomology class. 
Let $q_i=(q_{i,\sig,\tau})$ with 
$q_{i,\sig,\tau}I_2=(\del h_i)_{\sig,\tau}$ 
for each $i=1,2$ and $\sig,\tau\in\gal(k^{\sep}/k)$. 
Then 
\begin{equation*}
q_{i,\sig,\tau}^2=(\det h_{i,\tau})(\det h_{i,\sig})^{\tau}
(\det h_{i,\sig\tau})^{-1},\;
q_{1,\sig,\tau}q_{2,\sig,\tau}=1. 
\end{equation*}

We put $h^{(1)}_{1,\sig}=(\det h_{1,\sig})^{-1}h_{1,\sig}$, 
$h^{(1)}_{2,\sig}=(\det h_{1,\sig})h_{2,\sig}$, 
$h^{(1)}_i=(h^{(1)}_{i,\sig})$ for $i=1,2$ and
$h^{(1)} = (h^{(1)}_1,h^{(1)}_2)$.  Then 
$\overline h = \overline h^{(1)}$. 
If $\sig,\tau\in \gal(k^{\sep}/k)$ then 
\begin{equation*}
\del h^{(1)}_{\sig,\tau} = (\del h^{(1)}_{1,\sig,\tau},\del h^{(1)}_{2,\sig,\tau}),\;
\del h^{(1)}_{1,\sig,\tau} = q_{1,\sig,\tau}^{-2}\del h_{1,\sig,\tau}
= (\del h_{1,\sig,\tau})^{-1}.  
\end{equation*}
Therefore, $\del h^{(1)}=\del \psi_2(h)$. 

If $c_{\sig}\in\gl_2(k^{\sep})$, $c = (c_{\sig})_{\sig}$
and $\del c_{\sig,\tau}\in \{tI_2\mid t\in (k^{\sep})^{\times}\}$ 
then $\overline c$ is a 1-cocycle with coefficients in $\text{PGL}_2$. 
Since $\del c$ is a 2-cocycle with coefficients in $\gl_1$, 
by associating the class of $\del c$ to the class of $\overline c$, 
we obtain a map $\h^1(k,\text{PGL}_2)\to \h^2(k,\gl_1)$. 
Since $\h^1(k,\text{PGL}_2)$ can be identified with $k$-forms of 
$\m_2$ and $\h^2(k,\gl_1)$ is the Brauer group, 
the map $\h^1(k,\text{PGL}_2)\to \h^2(k,\gl_1)$ is injective.
Note that if a $k$-form of $\m_2$ is not split then 
it is isomorphic to $D^{\times}$ where $D$ is a division algebra.

Since $\del h^{(1)}=\del \psi_2(h)$ and 
$h^{(1)}_{1,\sig}$ is a scalar multiple of $h_{1,\sig}$, 
there exists $g\in\gl_2(k^{\sep})$ and $t_{\sig}\in (k^{\sep})^{\times}$ 
such that $g^{-1}h_{1,\sig}g^{\sig}=t_{\sig}h_{2,\sig}$ for all $\sig$.
Since we are considering the action of $\Z/2\Z$ on 
$\h^1(k,H^{\circ}_{\Phi(w)})$, we may assume that 
$h_{1,\sig}=t_{\sig}h_{2,\sig}$ for all $\sig$. 
Since 
\begin{equation*}
(t_{\sig}h_{2,\sig},h_{2,\sig})
= (t_{\sig},t_{\sig}^{-1})(h_{2,\sig},t_{\sig}h_{2,\sig}),
\end{equation*}
$\overline h=\overline \psi_2(h)$. This implies that  
$\psi_3$ induces the trivial action on 
$\h^1(k,H^{\circ}_{\Phi(w)})$. 
Therefore, $\h^1(k,H^{\circ}_{\Phi})\to \h^1(k,H_{\Phi})$
is injective and the image is the inverse image 
of the trivial class of $\h^1(k,\Z/2\Z)$. 
This implies that the map 
$G_k\backslash V^{\sst}_k\to H_k\backslash (\sym^2 W)^{\sst}_k$ 
induced by $\Phi$ is injective.  

Suppose that $x=g_x w\in V^{\sst}_k$. 
Then the class of $(g_x^{-1}g_x^{\sig})_{\sig}$ is the 
element in $\h^1(k,G_w)=\h^1(k,G^{\circ}_w)$ corresponding to 
the orbit of $x$. Let $h_{x,\sig}=g_x^{-1}g_x^{\sig}$ 
for $\sig\in\gal(k^{\sep}/k)$. 
As in the proof of Proposition \ref{prop:S1-orbit-final},
$G_{x\,k}$ consists of elements $g\in G_{w\, k^{\sep}}$ such that 
$h_{x,\sig}g^{\sig}h_{x,\sig}^{-1}=g$ for all $\sig$. 
Note that $G_w\cong H^{\circ}_{\Phi(w)}\cong \text{GO}(\Phi(w))^{\circ}$ 
(the identity component of the group of similitude of the quadratic form $\Phi(w)$). 
Since $G_{w\,k^{\sep}}\ni g\mapsto h_{x,\sig}gh_{x,\sig}^{-1}\in G_{w\,k^{\sep}}$  
is an inner automorphism, $G_x$ is an inner form of $\text{GO}(\Phi(w))^{\circ}$. 
Therefore, by associating $G_x$ to $x\in V^{\sst}_k$, 
we obtain a bijective correspondence between $G_k\backslash V^{\sst}_k$ 
and $\mathrm{IQF}_4(k)$. 
\end{proof}

We consider the case (b) next. 

Let $G=\gl_4\times \gl_3$ and $V=\wedge^2 W\otimes \aff^3\oplus W$. 
The action of $G$ on $V$ is the obvious one.
We express elements of $\wedge^2 \aff^4$ by alternating matrices as 
in the case (a). 

Let $w_1,w_2,w_3$ be as in (\ref{eq:w-defn-423}), 
$w_0=[1,0,0,1]\in\aff^4$ and 
\begin{equation}
\label{eq:w-defn-sec8b}
w=(w_1,w_2,w_3,w_0)\in V. 
\end{equation}
Elements of the form 
$e_G+\vep (X,Y)$ ($X=(x_{ij})\in \m_4,Y=(y_{ij})\in \m_3$) 
fix $(w_1,w_2,w_3)$ if and only if 
$X,Y$ are in the form (\ref{eq:A-defn}).  
Then $e_G+\vep (X,Y)$ fixes $w_0$ if and only if 
\begin{equation*}
x_{11}=x_{21}+x_{13}=x_{31}+x_{12}=-x_{11}+x_{22}+x_{33}=0. 
\end{equation*}
So $X,Y$ are in the following form:
\begin{equation}
\label{eq:Lie-alg-43-3}
\begin{aligned}
X & = \begin{pmatrix}
0 & x_{12} & x_{13} & 0\\
-x_{13} & x_{22} & 0 & x_{13} \\
-x_{12} & 0 & -x_{22} & x_{12} \\
0 & -x_{12} & -x_{13} & 0
\end{pmatrix}, \;
Y = \begin{pmatrix}
-x_{22} & x_{12} & 0 \\
-2x_{13} & 0 & 2x_{12} \\
0 & -x_{13} & x_{22}
\end{pmatrix}.  
\end{aligned}
\end{equation}
Therefore, $\dim {\mathrm T}_{e_G}(G_w)=3=25-22=\dim G-\dim V$. 
So Proposition \ref{prop:open-orbit} implies that $Gw\sub V$
is Zariski open and that $G_w$ is smooth over $k$. 
Since this computation does not depend on 
$\ch(k)$, $w$ is defined over $\Z$ and is universally generic. 

We define a map 
$\Phi_1:V\to \wedge^2 W\otimes \wedge^2 W\otimes \wedge^2 W$ 
by 
\begin{equation*}
\Phi_1(x) = \sum_{\sig\in\gS_3} \sgn(\sig) 
x_{\sig(1)}\otimes x_{\sig(2)}\otimes x_{\sig(3)}
\end{equation*}
for $x=(x_1,x_2,x_3,x_0)$ 
($x_1,x_2,x_3\in \wedge^2 W$, $x_0\in W$). 

For $x_0\in W$, we define a linear map 
$\Phi_{2,x_0}:\wedge^2 W\otimes \wedge^2 W\otimes \wedge^2 W\to 
\wedge^4 W\otimes \wedge^4 W$ so that 
\begin{align*}
\Phi_{2,x_0}(a_1\otimes a_2\otimes (b_1\wedge b_2))
& = (a_1\wedge b_1\wedge x_0)\otimes (a_2\wedge b_2\wedge x_0) \\
& \quad -(a_1\wedge b_2\wedge x_0)\otimes (a_2\wedge b_1\wedge x_0)
\end{align*}
for $a_1,a_2\in \wedge^2 W,b_1,b_2\in W$. 
 
We first consider over $\Q$. We put 
$P_1(x)=-\frac 16\Phi_{2,x_0}(\Phi_1(x))$
for $x=(x_1,x_2,x_3,x_0)$.
We identify two $\wedge^4 W$'s with $\aff^1$ so that 
$p_{4,1234}$ corresponds to $1$. Then $P_1$ can be regarded as 
a polynomial on $V$ over $\Q$. It is easy to see that 
$P_1(x)$ is homogeneous of degree $3,2$ with respect to 
$\wedge^2 W\otimes \aff^3,W$ respectively.

\begin{prop}
\label{prop:section8-b-P1-value}
\begin{itemize}
\item[(1)]
$P_1(x)$ is defined over $\Z$ and $P_1(w)=1$. 
\item[(2)]
\begin{math}
P_1((g_1,g_2)x)=(\det g_1)^2(\det g_2)P_1(x)
\end{math}
for $(g_1,g_2)\in\m_4\times \m_3$. 
\end{itemize}
\end{prop}
\begin{proof}
(1) By long but straightforward computations, 
$6P_1(w)=6$. Since $w$ is universally generic, 
$P_1(x)$ is defined over $\Z$ and $P_1(w)=1$. 
(2) is easy. 
\end{proof}

By the natural homomorphism $\Z\to k$, 
there is a polynomial $P_1(x)$ over $k$  
such that $P_1(w)=1$ and Proposition \ref{prop:section8-b-P1-value} (2)
holds for $(g_1,g_2)\in \m_4\times \m_3$.
In particular, this holds for 
$(g_1,g_2)\in\gl_4\times \gl_3$.

Let $u=(u_1,u_2,u_3)$ be variables. 
For $x$ as above, let $\Phi_3(x)$ be the Pfaffian of 
$u_1x_1+u_2x_2+u_3x_3$. Then $\Phi_3(x)=\Phi_3(x)(u)\in\sym^2 \aff^3$ 
and 
\begin{equation}
\label{eq:43-3-Phi3}
\Phi_3(gx)(u) = (\det g_1)\Phi_3(x)(u g_2) 
\end{equation}
regarding $u$ as a row vector. 
It is easy to see that $\Phi_3(w)=u_1u_3+u_2^2$.  
It is known that there is a homogeneous polynomial 
$\Phi_4$ of degree $3$ on the space of quadratic forms of $u$
such that $\Phi_4(u_1u_3+u_2^2)=1$ and 
$\Phi_4(hQ)=(\det h)^2\Phi_4(Q)$ for $h\in \gl_3$,  
$Q\in\sym^2 \aff^3$. We define 
$P_2(x) = \Phi_4(\Phi_3(x))$. Then 
\begin{equation}
\label{eq:43-4-P2-equivariant}
P_2((g_1,g_2)x)=(\det g_1)^3(\det g_2)^2P_2(x).
\end{equation}
It is easy to see that $P_2(x)$ is  a homogeneous polynomial of degree $6$. 
As we pointed out in Introduction, $P_1(x),P_2(x)$ are more or less 
$P_1(v),P_2(v)$ of \cite[p.465]{kimura-2simple-some-inv}.

We use the standard argument 
to determine the set of rational orbits.
For that purpose, we have to determine the 
stabilizer of $w$. 

Let $\rho:\gl_2\to\gl(\m_2)\cong \gl_4$ be the 
\rep{} defined by 
$\gl_2\times \m_2\ni (g,A)\mapsto gAg^{-1}\in \m_2$. 
We use the basis $\{E_{11},E_{21},E_{12},E_{22}\}$
as in the case (a) for $\m_2$. 
This $\rho$ coincides with $\rho(g,{}^tg^{-1})$ of 
the case (a). So $\wedge^2\rho$ leaves 
the subspace $\lan w_1,w_2,w_3\ran$ invariant. 
We define $\rho_1(g)\in \gl_3$ in the same manner 
as in Lemma \ref{lem:rho-rho1-fix-w}. 
Since the coordinate of $I_2\in\m_2$ is 
$[1,0,0,1]$ and $gI_2g^{-1}=I_2$, 
$\rho(g)$ fixes $w_0$. Therefore, we obtain the following 
lemma. 
\begin{lem}
\label{lem:rho-rho1-fix-w-caseb}
The image of $(\rho,{}^t\rho_1^{-1})$ 
fixes $w$. 
\end{lem}
%


%
\begin{prop}
\label{prop:43-3-Gwcirc}
\begin{itemize}
\item[(1)]
$\kernel(\rho,{}^t\rho_1^{-1})=\kernel(\rho)$ 
is the center $\mathrm{Z}(\gl_2)$. 
\item[(2)]
$G_w=G^{\circ}_w =\im(\rho,{}^t\rho_1^{-1}) \cong \mathrm{PGL}_2$. 
\end{itemize}
\end{prop}
\begin{proof}
(1) is easy 

(2) 
By (1), $\im(\rho,{}^t\rho_1^{-1})\cong\mathrm{PGL}_2$
and so $\dim \im(\rho,{}^t\rho_1^{-1})=3$. 
We have shown in (\ref{eq:Lie-alg-43-3}) 
that $G_w$ is smooth and $\dim G_w=3$. Therefore,  
$G^{\circ}_w=\im(\rho,{}^t\rho_1^{-1})$.

To show that $G_w=G^{\circ}_w$, 
we may assume that $k=\overline k$. 
Suppose that $g=(g_1,g_2)\in G_w$. 
Since $\mathrm{PGL}_2$ has no outer automorphisms, 
by multiplying an element of $G^{\circ}_w$, 
we may assume that $g$ commutes with all elements of $G^{\circ}_w$.

If $t=\diag(t_1,t_2)$ then 
\begin{equation*}
(\rho(t),{}^t\rho_1(t)^{-1})
=(\diag(1,t_1^{-1}t_2,t_1t_2^{-1},1),\diag(t_1^{-1}t_2,1,t_1t_2^{-1})).
\end{equation*}
Since $g$ commutes with this element, 
$g_2$ must be diagonal and $g_1$ is in the form 
\begin{equation*}
g_1 = 
\begin{pmatrix}
h_{11} & h_{12} \\
h_{21} & h_{22}
\end{pmatrix}
\end{equation*}
where $h_{11}\ccd h_{22}\in\m_2$ and $h_{11},h_{22}$ are 
diagonal. 

Since $g$ fixes $w$ and $g_2$ is diagonal, 
there exist scalars $c_1,c_2,c_3$ 
such that $g_1w_i{}^tg_1=c_iw_i$. 
For $i=1,3$, 
\begin{equation}
\label{eq:h12-etc}
\begin{aligned}
& \begin{pmatrix}
h_{11} & h_{12} \\
h_{21} & h_{22} 
\end{pmatrix}
\begin{pmatrix}
J & 0 \\
0 & 0 
\end{pmatrix}
= c_1
\begin{pmatrix}
J & 0 \\
0 & 0 
\end{pmatrix}
\begin{pmatrix}
h_{11} & h_{12} \\
h_{21} & h_{22} 
\end{pmatrix}, \\
& \begin{pmatrix}
h_{11} & h_{12} \\
h_{21} & h_{22} 
\end{pmatrix}
\begin{pmatrix}
0 & 0 \\
0 & J 
\end{pmatrix}
= c_3
\begin{pmatrix}
0 & 0 \\
0 & J 
\end{pmatrix}
\begin{pmatrix}
h_{11} & h_{12} \\
h_{21} & h_{22} 
\end{pmatrix}.
\end{aligned}
\end{equation}
So $h_{12}=h_{21}=0$. Therefore, $g_1$ is diagonal also. 
Now it is easy to  deduce that $g\in G^{\circ}_w$. 
\end{proof}

\begin{cor}
\label{cor:433-3-regular}
$(G,V)$ is a regular \pv. 
\end{cor}

Let $U=\{x\in V\mid P_1(x),P_2(x)\not=0\}$. 
We are in the situation of Section \ref{sec:regularity} 
where $m=N=2$. So by Corollary \ref{cor:reducible-sep-orbit}, 
$U_{k^{\sep}}=G_{k^{\sep}}w$. So we can use the 
standard argument of Galois cohomology. 
Since $\mathrm{PGL}_2$ is the automorphism group of 
$\mathrm{PGL}_2$ itself, 
$\h^1(k,\mathrm{PGL}_2)$ is in bijective correspondence
with $\mathrm{Prg}_2(k)$. 

The proof of the following proposition is similar to that 
of Proposition \ref{prop:S1-orbit-final} and we do not provide the 
proof. 
\begin{prop}
\label{prop:43-3-orbit}
$G_k\backslash U_k$ is in bijective correspondence with 
$\mathrm{Prg}_2(k)$. If $x\in U_k$ the corresponding 
$k$-form of $\mathrm{PGL}_2$ is the stabilizer $G_x$.
%
\end{prop}

\section{Rational orbits (6)}
\label{sec:rational-orbits-wedge3-32}

Let (a) $G=\gl_3\times \gl_2^2\times \gl_1$ and 
$V=\wedge^2 \aff^3\oplus \aff^3\otimes \m_2$, 
(b) $G=\gl_3\times \gl_2^3$ and 
$V=\wedge^2 \aff^3\otimes \aff^2 \oplus \aff^3\otimes \m_2$. 
The actions of $G$ for these cases will be defined 
later when they are needed.

We use the notation such as $p_{3,12}$, etc.

We first consider the case (a). 
We express elements of $V$ as $(A,B)$ where 
$A=a_1p_{3,12}+a_2p_{3,13}+a_3p_{3,23}$, 
\begin{equation*}
B_i = 
\begin{pmatrix}
b_{i,11} & b_{i,12} \\
b_{i,21} & b_{i,22} 
\end{pmatrix}
\end{equation*}
for $i=1,2,3$ and 
$B=\bbmp_{3,1}\otimes B_1 + \bbmp_{3,2}\otimes B_2 + \bbmp_{3,3}\otimes B_3$. 
We sometimes write $B=(B_1,B_2,B_3)$. 
If $x=(A,B)\in V$ then we write $a_1(x),a_2(x),a_3(x),A(x),B(x)$ 
for these $a_1,a_2,a_3,A,B$. 
We identify $\wedge^2 \aff^3$ with the space of $3\times 3$ 
alternating matrices. 

We consider the natural action of $\gl_3\times \gl_2^2$ on $V$.
The action of $t\in \gl_1$ on $V$ is defined by 
$V\ni (A,B)\mapsto (tA,B)\in V$. 
We shall show  the existence of
non-zero relative invariant polynomials first without 
any condition on $\ch(k)$ and then show that 
$(G,V)$ is a regular \pv{} assuming that $\ch(k)\not=2$.

Let 
\begin{equation*}
\Phi(x) = a_1(x)B_3(x)-a_2(x)B_2(x)+a_3(x)B_1(x) \in \m_2
\end{equation*}
be the element obtained by applying Lemma \ref{lem:natural-pairing} 
to $A(x),B(x)$. We put $P_1(x) = \det \Phi(x)$.  
Let $P_2(x)$ be the degree $6$ polynomial of $B$
obtained by Proposition \ref{prop:322-invariant}. 
Let $g=(g_1,g_2,g_3,t)\in G$. Then 
\begin{math}
A(gx) = t(\wedge^2 g_1)A(x),\; 
B(gx) = g_1\otimes (g_2,g_3)B(x)
\end{math}
and 
\begin{equation}
\label{eq:322,3-invariants}
\begin{aligned}
& \Phi(gx) = t(\det g_1) (g_2,g_3) \Phi(x), \\
& P_1(gx) = t^2(\det g_1)^2(\det g_2)(\det g_3)P_1(x), \\
& P_2(gx) = (\det g_1)^2(\det g_2)^3(\det g_3)^3P_2(x). 
\end{aligned}
\end{equation}
Note that the action of $(g_2,g_3)$ on $\m_2$ is 
$\m_2\ni M\mapsto g_2 M{}^t g_3\in \m_2$. 
So $P_1(x),P_2(x)$ are relative invariant polynomials. 

Let $R_{322}$ be the element $w$  in (\ref{eq:322-generator2}). 
We put 
\begin{equation}
\label{eq:R(322-3)-defn}
R_{322,3} = (p_{3,23},R_{322}).
\end{equation}
It is easy to see that $\Phi(R_{322,3})$ is $w_1$ 
in (\ref{eq:B123-defn}). 
So $P_1(R_{322,3})=-1$. By Proposition \ref{prop:322-invariant}, 
$P_2(R_{322,3})=1$. 
We identify $p_{3,23}$ with  
\begin{equation}
\label{eq:32-3A-defn}
A = 
\begin{pmatrix}
0 & 0 & 0 \\
0 & 0 & 1 \\
0 & -1 & 0 
\end{pmatrix}.
\end{equation}
By the above consideration, $P_1,P_2\not=0$. 

Let $T$ be the torus consisting of elements of the form 
\begin{equation}
\label{eq:T-t-322-3}
t = \left(
\begin{pmatrix}
(t_1t_3)^{-1} & 0 & 0 \\
0 & (t_1^2t_2^{-1}t_3)^{-1} & 0 \\
0 & 0 & (t_2t_3)^{-1}
\end{pmatrix},
\begin{pmatrix}
t_1 & 0 \\
0 & t_2 
\end{pmatrix},
\begin{pmatrix}
t_3 & 0 \\
0 & t_1t_2^{-1}t_3
\end{pmatrix},t_1^2t_3^2
\right). 
\end{equation}
Elements of $T$ fix $R_{322,3}$.

\begin{prop}
\label{prop:322-3-regular}
Suppose that $\ch(k)\not=2$.
\begin{itemize}
\item[(1)]
$(G,V)$ is a regular \pv. 
\item[(2)]
$\{x\in V_{k^{\sep}}\mid P_1(x),P_2(x)\not=0\}=G_{k^{\sep}}R_{322,3}$. 
\item[(3)]
$G^{\circ}_{R_{322,3}}=T$.
\end{itemize}
\end{prop}
\begin{proof}
We put $R=R_{322,3}$. 
We first determine the Lie algebra of the stabilizer $G_R$. 
Let $X=(x_{ij})\in\m_3,Y=(y_{ij}),Z=(z_{ij})\in\m_2,a\in \aff^1$.
Suppose that $(e_G+\vep(X,Y,Z,a))R=R$. 
We may assume that $X,Y,Z$ are in the form in Lemma \ref{lem:S14-XYZ-form}.  
Since 
\begin{equation*}
XA+A{}^tX + aA 
= \begin{pmatrix}
0 & 2y_{21} & -2y_{12} \\
* & 0 & -2y_{11}-2z_{11}+a \\
* & * & 0  
\end{pmatrix}=0, 
\end{equation*}
$y_{12}=y_{21}=0$, $a=2y_{11}+2z_{11}$. 
So 
\begin{equation*}
X = \diag(-y_{11}-z_{11},-y_{11}-z_{22},-y_{22}-z_{11}),\;
Y = \diag(y_{11},y_{22}),\;
Z = \diag(z_{11},z_{22})
\end{equation*}
and $y_{11}+z_{11}=y_{22}+z_{22},a=2y_{11}+2z_{11}$. 
This implies that 
$\dim {\mathrm T}_{e_G}(G_R)=3=18-15=\dim G-\dim V$. 
So, Proposition \ref{prop:open-orbit} implies that 
$GR\sub V$ is Zariski dense and that $G_R$ 
is smooth over $k$. 
Since $\dim T=3$, $G^{\circ}_R=T$. Therefore, $(G,V)$ is regular. 

The assumption of Corollary \ref{cor:reducible-sep-orbit} 
is satisfied and the second assertion of the proposition follows.  
\end{proof}

Let $\tau_0$ be as in (\ref{eq:J-defn}) and 
\begin{equation}
\label{eq:332-3-tau-defn}
\tau_1 = 
\begin{pmatrix}
-1 & 0 & 0 \\
0 & 0 & 1 \\
0 & 1 & 0 
\end{pmatrix},\;
\tau = (\tau_1,\tau_0,\tau_0,-1)\in G_k.
\end{equation}
Then $\tau R=R$. 

\begin{lem}
\label{lem:332-3-GRcirc}
$G_R/G^{\circ}_R$ is represented by 
$\{1,\tau\}$. 
\end{lem}
\begin{proof}
If $(h,h',h'',t)\in G_R$ ($h=(h_{ij}),h'=(h'_{ij}),h''=(h''_{ij})$) 
then this element sends weight spaces (resp. trivial weight spaces) 
of $T$ to weight spaces (resp. trivial weight spaces). 
The subspaces spanned by $p_{3,12},p_{3,13},p_{3,23}$ 
are weight spaces and the weights of $t\in T$ in (\ref{eq:T-t-322-3}) 
are $t_1^{-1}t_2,t_1t_2^{-1},1$ respectively. 
So $h$ leaves $\lan p_{1,23}\ran$ invariant and 
either fixes or exchanges the subspaces 
$\lan p_{3,12}\ran,\lan p_{3,13}\ran$.
Since $\tau$ exchanges $\lan p_{3,12}\ran,\lan p_{3,13}\ran$,  
by multiplying $\tau$ if necessary, we may assume that 
$h$ fixes the subspaces 
$\lan p_{3,12}\ran,\lan p_{3,13}\ran,\lan p_{3,23}\ran$.  

Since
\begin{align*}
hp_{3,12}{}^th & = 
\begin{pmatrix}
0 & h_{11}h_{22} - h_{12}h_{21} 
& h_{11}h_{32} - h_{12}h_{31} \\
-h_{11}h_{22} + h_{12}h_{21} 
& 0 &  h_{21}h_{32} - h_{22}h_{31} \\
-h_{11}h_{32} + h_{12}h_{31}  
& -h_{21}h_{32} + h_{22}h_{31} & 0 
\end{pmatrix}, \\
hp_{3,13}{}^th & =
\begin{pmatrix}
0 & h_{11}h_{23} - h_{13}h_{21} 
& h_{11}h_{33} - h_{13}h_{31} \\
-h_{11}h_{23} + h_{13}h_{21} 
& 0 &  h_{21}h_{33} - h_{23}h_{31} \\
-h_{11}h_{33} + h_{13}h_{31}  
& -h_{21}h_{33} + h_{23}h_{31} & 0 
\end{pmatrix}, \\
hp_{3,23}{}^th & = 
\begin{pmatrix}
0 & h_{12}h_{23} - h_{13}h_{22} 
& h_{12}h_{33} - h_{13}h_{32} \\
-h_{12}h_{23} + h_{13}h_{22} 
& 0 & h_{22}h_{33} - h_{23}h_{32} \\
-h_{12}h_{33} + h_{13}h_{32}  
& -h_{22}h_{33} + h_{23}h_{32} & 0 
\end{pmatrix}, 
\end{align*}
we have
\begin{align*}
& h_{11}h_{32} - h_{12}h_{31} 
= h_{21}h_{32} - h_{22}h_{31} =0, \\
& h_{11}h_{23} - h_{13}h_{21} 
= h_{21}h_{33} - h_{23}h_{31} =0, \\
& h_{12}h_{23} - h_{13}h_{22} 
= h_{12}h_{33} - h_{13}h_{32} =0. 
\end{align*}
This means that the matrix $(\det h)h^{-1}$ 
is diagonal.  Therefore, $h$ is a diagonal matrix.

%

By Proposition \ref{prop:HR-circ}, 
there exists 
$h_0\in\gl_2,q\in \gl_1$ such that 
$(h,h',h'')=({}^t\rho_1(h_0,q)^{-1},qh_0,{}^th_0{}^{-1})$ 
where $\rho_1(h_0,q)$ is as in (\ref{eq:S14-phi-defn}). 
Since $h$ is diagonal, $\rho_1(h_0,q)$ is diagonal. 
By (\ref{eq:rho1(g)}), $h_0$ must be diagonal. 
So $h'=qh_0,h''={}^th_0^{-1}$ are diagonal. 
Then it is easy to verify that $(h,h',h'',t)\in G^{\circ}_R$. 
\end{proof}

Let $V_1=\aff^3\otimes \m_2$. 
We construct an equivariant map 
$V_1\to \aff^3\otimes \sym^2 \aff^2$. 
We define a map 
$\Phi_1:V_1\to V_1^{4\otimes}$ 
so that 
\begin{equation*}
\Phi_1(x)=x\otimes x\otimes x\otimes x.
\end{equation*}
We define a linear map 
\begin{math}
\Phi_2:V_1^{4\otimes}\to 
(\aff^3)^{4\otimes} \otimes (\aff^2\otimes \aff^2)^{4\otimes}
\end{math}
so that 
\begin{equation*}
\Phi_2(v_1\otimes \cdots \otimes v_{12})
= v_1\otimes v_4\otimes v_7\otimes v_{10}
\otimes v_2\otimes v_3\otimes v_5\otimes v_6
\otimes v_8\otimes v_9\otimes v_{11}\otimes v_{12}.
\end{equation*}
We define a linear map 
\begin{math}
\Phi_3:(\aff^3)^{4\otimes} \otimes 
(\aff^2\otimes \aff^2)^{4\otimes}
\to \aff^3\otimes \sym^2 \aff^2
\end{math}
so that 
\begin{align*}
& \Phi_3(v_1\otimes v_2\otimes v_3\otimes v_4\otimes v_5\otimes v_6
\otimes v_7\otimes v_8\otimes v_9\otimes v_{10}\otimes v_{11}\otimes v_{12}) \\
& = (v_1\wedge v_2\wedge v_3)(v_5\wedge v_7)(v_9\wedge v_{11})
(v_6\wedge v_{10}) v_4\otimes v_8v_{12}. 
\end{align*}
Let $\Phi_4=\Phi_3\circ\Phi_2\circ\Phi_1$. 
Suppose that $g=(g_1,g_2,g_3,t)\in G$, 
$x=(x_1,x_2)\in V$ where $x_2\in \aff^3\otimes \aff^2\otimes \aff^2$. 
Then it is easy to see that 
\begin{equation}
\label{eq:322-equivariant}
\Phi_4((g_1,g_2,g_3)x_2) =(\det g_1)(\det g_2)^2(\det g_3)
(g_1,g_3)\Phi_4(x_2). 
\end{equation}

We write $u_i=\bbmp_{2,i}$ for $i=1,2$.  
Note that $u_iu_j\in \sym^2\aff^2$ for all $i,j$. 
Long but straightforward computations show the following proposition. 
We do not provide the details. 
\begin{lem}
\label{lem:equivariant-332-3}
\begin{math}
\Phi_4(R_{322})
= 6 \bbmp_{3,1}\otimes u_1u_2 
-3 \bbmp_{3,2} \otimes u_2^2 
+3 \bbmp_{3,3} \otimes u_1^2.  
\end{math}
\end{lem} 

Let $\Phi_5:V\to\sym^2 \aff^2$ be the map
which is obtained by applying Lemma \ref{lem:natural-pairing} 
to the $\wedge^2\aff^3$ component of $V$ and the $\aff^3$-factor 
of $\Phi_4$.  The following proposition follows from 
(\ref{eq:322-equivariant}) and Lemma \ref{lem:equivariant-332-3}.

\begin{prop}
\label{prop:322,3-equivariant}
Suppose that $g=(g_1,g_2,g_3,t)\in G$ and $x\in V$.  
\begin{itemize}
\item[(1)]
$\Phi_5(gx) = t(\det g_1)^2(\det g_2)^2(\det g_3)g_3\Phi_5(x)$.
\item[(2)]
$\Phi_5(R)=6u_1u_2$
\end{itemize}
\end{prop}

Note that $G,V,\Phi$ and $R_{322,3}$ are defined over $\Z$. 
The element $R_{322,3}$ is universally generic outside $2$. 
So Proposition \ref{prop:universally-generic} implies
the following corollary.

\begin{cor}
\label{cor:S14-equiv}
If $\ch(k)\not=2$ then there is a map 
$\Phi:V\to \sym^2 \aff^2$ over $k$ such that 
$\Phi(R)=u_1u_2$ and that
$\Phi(gx) = t(\det g_1)^2(\det g_2)^2(\det g_3)g_3\Phi(x)$. 
\end{cor}

\begin{prop}
\label{prop:322+3-orbit}
If $\ch(k)\not=2$ then 
$G_k\backslash \{x\in V_k\mid P_1(x),P_2(x)\not=0\}$ 
is in bijective correspondence with $\Ex_2(k)$ by 
associating to $x\in V_k$ such that $P_1(x),P_2(x)\not=0$,  
the field generated over $k$ by a root of $\Phi(x)$. 
\end{prop}
\begin{proof}
Since $\{x\in V_{k^{\sep}}\mid P_1(x),P_2(x)\not=0\}=G_{k^{\sep}}R$ 
and $\h^1(k,G)=\{1\}$,  
by the standard argument of Galois cohomology, 
\begin{equation*}
G_k\backslash \{x\in V_k\mid P_1(x),P_2(x)\not=0\} \cong 
\h^1(k,G_R). 
\end{equation*}

By Lemma \ref{lem:332-3-GRcirc}, $G_R/G^{\circ}_R\cong \Z/2\Z$ 
and $\gal(k^{\sep}/k)$ acts on $\Z/2\Z$ trivially. So there is a natural 
map 
\begin{equation*}
\al_V:G_k\backslash \{x\in V_k\mid P_1(x),P_2(x)\not=0\}\cong 
\h^1(k,G_R)\to \h^1(k,\Z/2\Z)\cong \Ex_2(k).
\end{equation*}
To show that this map is bijective, for each $F\in \Ex_2(k)$ 
we choose a point $x_F\in\al_V^{-1}(F)$ and prove that 
$\h^1(k,G^{\circ}_{x_F})=\{1\}$. Then LEMMA (1.8) \cite[p.120]{yukiel} 
implies that $\al_V$ is bijective. Moreover, we show that 
$F$ coincides with the field generated over $k$ by a root of 
$\Phi(x_F)$. 

Let $F/k$ be a quadratic extension. Since $\ch(k)\not=2$, 
this is a Galois extension. Let $f(z) = z^2+a_1z+a_2\in k[z]$ 
be a polynomial whose roots $\al=(\al_1,\al_2)$ generate $F$ 
over $k$. Let $\sig\in\gal(F/k)$ be the non-trivial element. 
Let 
\begin{equation*}
g_F = \left(
\begin{pmatrix}
(\al_1-\al_2)^{-1} & 0 & 0 \\
0 & 1 & 1 \\
0 & -\al_1 & -\al_2
\end{pmatrix},
\begin{pmatrix}
1 & 1 \\
-\al_1 & -\al_2
\end{pmatrix},
\begin{pmatrix}
1 & 1 \\
-\al_1 & -\al_2
\end{pmatrix},
(\al_1-\al_2)^{-3} \right).
\end{equation*}
We put 
\begin{math}
x_F = g_F R. 
\end{math}
Then by Corollary \ref{cor:S14-equiv}, 
$\Phi(x_F) = (u_1-\al_1 u_2)(u_1-\al_2 u_2)=u_1^2+a_1u_1u_2+a_2u_2^2$ 
and so its roots generate $F$ over $k$. 

It is easy to see that $g_F^{\sig}=g_F \tau$. 
Since $\tau\in G_R$, 
\begin{equation*}
(g_F)^{\sig}R = g_F\tau R  = g_F R =x_F.
\end{equation*}
Therefore, $x_F\in V_k$ and $P_1(x_F),P_2(x_F)\not=0$. 
Since $g_F^{-1}g_F^{\sig}=\tau$, the cohomology class in 
$\h^1(k,\Z/2\Z)$ determined by $x_F$ corresponds to the
field $F$. 

We express $t$ in 
(\ref{eq:T-t-322-3}) as $a(t_1,t_2,t_3)$. 
Then the group of $F$-rational points of $G_{x_F}^{\circ}$ is 
\begin{equation*}
G_{x_F\,F}^{\circ} = \{g_Fa(t_1,t_2,t_3)g_F^{-1}\mid t_1,t_2,t_3\in F^{\times}\}.
\end{equation*}
Since 
\begin{math}
(g_Fa(t_1,t_2,t_3)g_F^{-1})^{\sig}=
g_F \tau a(t_1^{\sig},t_2^{\sig},t_3^{\sig})\tau g_F^{-1}, 
\end{math}
$g_Fa(t_1,t_2,t_3)g_F^{-1}\in G_{x_F\,k}$ 
if and only if $\tau a(t_1^{\sig},t_2^{\sig},t_3^{\sig})\tau=a(t_1,t_2,t_3)$. 
This is equivalent to the following conditions:
\begin{equation*}
t_1t_3\in k^{\times},\; 
t_2=t_1^{\sig},\; 
t_1^2t_2^{-1}t_3=t_2^{\sig}t_3^{\sig},\; 
t_3^{\sig}=t_1t_2^{-1}t_3. 
\end{equation*}
The last two conditions follow from the first two conditions. 
Note that the condition on the last component is $(t_1t_3)^2\in\mk$, 
which follows from the above conditions. So the above 
conditions follow from the consideration of the 
$\gl_3\times \gl_2^2$ part of the group.

By the above consideration, 
$G_{x_F\,k}^{\circ}\cong \gl_1(k)\times (\mathrm{R}_{F/k}\gl_1)(k)$. 
We only considered $k$-rational points but considering rational 
points over all $k$-algebras as in \cite{kayu}, one can show that 
$G_{x_F}^{\circ}\cong \gl_1\times (\mathrm{R}_{F/k}\gl_1)$ 
as algebraic groups. Also $g_{\al}\tau g_{\al}^{-1}\in G_{x_F\,k}$. 
Therefore, $\h^1(k,G^{\circ}_{x_F})=\{1\}$.  
Hence $\al_V$ is bijective. 
\end{proof} 

Finally, we consider the case (b). 

Let $G=\gl_3\times \gl_2^3$ and  
$V=\wedge^2 \aff^3\otimes \aff^2 \oplus \aff^3\otimes \m_2$. 
We define an action of $g=(g_1,g_2,g_3,g_4)\in G$ on $V$ 
so that the action on $\wedge^2 \aff^3\otimes \aff^2$ 
is by $(\wedge^2 g_1,g_4)$ and the action on 
$\aff^3\otimes \m_2$ is the natural action of 
$(g_1,g_2,g_3)$.  

If $x\in V$ then we write the 
$\wedge^2 \aff^3\otimes \aff^2$ component of $x$
as pairs $A(x)=(A_1(x),A_2(x))$ where $A_1(x),A_2(x)\in \wedge^2 \aff^3$. 
Let $B(x)$ be the $\aff^3\otimes \m_2$ component of $x$. 
Let $R_1=-p_{3,13}$, $R_2=p_{3,12}$. 
We put 
\begin{equation}
\label{eq:S14-b-R}
R = (R_1,R_2,R_{322}). 
\end{equation}

We define a map $\Phi:V\to \aff^2\otimes \m_2$, 
linear with respect to each of 
$\wedge^2 \aff^3\otimes \aff^2$, 
$\aff^3\otimes \m_2$ 
so that 
\begin{equation*}
\Phi(v_1\otimes v_2,v_3\otimes v_4)
= (v_1\wedge v_3)v_2\otimes v_4
\end{equation*}
for $v_1\in\wedge^2 \aff^3,v_3\in \aff^3,v_2,v_4\in\m_2$
($p_{3,123}$ corresponds to $1$). 

Then for $x\in V,g=(g_1,g_2,g_3,g_4)$, 
\begin{equation}
\label{eq:S14Phi-R-equivariant}
\Phi(gx) = (\det g_1) (g_4,g_2,g_3)\Phi(x)
\end{equation}
where $(g_4,g_2,g_3)\Phi$ is 
the natural action.  

Since $R_1\wedge \bbmp_{3,2}=R_2\wedge \bbmp_{3,3} = p_{3,123}$, 
\begin{equation}
\label{eq:S14-b-PhiR}
\Phi(R) = \left(
\begin{pmatrix}
0 & 1 \\
0 & 0 
\end{pmatrix},
\begin{pmatrix}
0 & 0 \\
1 & 0 
\end{pmatrix}
\right).
\end{equation}

Let $\tau_0,\tau_1$ be as in (\ref{eq:J-defn}), 
(\ref{eq:332-3-tau-defn}) respectively. 
We put $\tau=(\tau_1,\tau_0,\tau_0,\tau_0)$. 
It is easy to see that $\tau$ fixes $R$. 
\begin{prop}
\label{prop:S14-R-regular}
%
\begin{itemize}
\item[(1)]
$GR\sub V$ is Zariski open and $G_R$ is smooth over $k$.
\item[(2)]
$G^{\circ}_R$ consists of elements of the form 
\begin{equation*}
(\diag(t_1^{-1}t_3^{-1},t_1^{-2}t_2t_3^{-1},t_2^{-1}t_3^{-1}),
\diag(t_1,t_2),\diag(t_3,t_1t_2^{-1}t_3),\diag(t_1t_2t_3^2,t_1^3t_2^{-1}t_3^2))
\end{equation*}
where $t_1,t_2,t_3\in\gl_1$. 
\item[(3)]
$G_R/G^{\circ}_R$ is represented by $\{1,\tau\}$. 
\item[(4)]
$(G,V)$ is a regular \pv. 
\end{itemize}
\end{prop}
\begin{proof}
(1) Let $e_G+\vep(X,Y,Z,Q)\in {\mathrm T}_{e_G}(G)$ 
where $X=(x_{ij})\in\m_3,Y=(y_{ij}),Z=(z_{ij}),Q=(q_{ij})\in\m_2$. 
Suppose that $e_G+\vep(X,Y,Z,Q)$ fixes $R$. Since it fixes 
$R_{322}$, $X,Y,Z$ are in the form in Lemma \ref{lem:S14-XYZ-form}. 
Regarding $R_1,R_2$ as $3\times 3$ alternating matrices,
\begin{align*}
& X R_1 + R_1{}^t X + q_{11} R_1 + q_{12}R_2 =0, \\
& X R_2 + R_2{}^t X + q_{21} R_1 + q_{22}R_2 =0.
\end{align*}
By computations, $x_{21}=x_{31}=0$. 
By Lemma \ref{lem:S14-XYZ-form}, 
$Y$ is diagonal.  This implies that $X,Z$ 
are diagonal also.  By the above equalities, 
$Q$ is diagonal also and  
\begin{align*} 
X & = \begin{pmatrix}
-y_{11}-z_{11} & 0 & 0 \\
0 & -y_{11}-z_{22} & 0 \\
0 & 0 &  -y_{22}-z_{11}
\end{pmatrix}, \;
Y = \begin{pmatrix}
y_{11} & 0 \\
0 & y_{22} 
\end{pmatrix}, \;
Z = \begin{pmatrix}
z_{11} & 0 \\
0 & z_{22}
\end{pmatrix}, \\
Q & = \begin{pmatrix}
y_{11}+y_{22}+2z_{11} & 0 \\
0 & 2y_{11}+z_{11}+z_{22}  
\end{pmatrix},\; 
y_{11}+z_{11}=y_{22}+z_{22}.
\end{align*}
Therefore, $\dim {\mathrm T}_{e_G}(G_R)=3=21-18=\dim G-\dim V$
and (1) follows from Proposition \ref{prop:open-orbit}. 

(2) Let $T\sub G$ be the subgroup consisting of 
elements of the form in the statement. It is easy to see that 
$T$ fixes $R$. Since $\dim T=3$, $G^{\circ}_R=T$. 
So (4) follows. 

(3) We may assume that $k=\overline k$. 
Suppose that $g=(g_1,g_2,g_3,g_4)\in G_R$. 
By (\ref{eq:S14Phi-R-equivariant}), 
$(\det g_1)(g_4,g_2,g_3)$ fixes $\Phi(R)$. 
The stabilizer of $(1,1,\tau_0)\Phi(R)$
is known and is generated by $(\tau_0,\tau_0,\tau_0)$ 
and elements whose components are diagonal matrices
(see Proposition \ref{prop:222-regular-sect}). 
This implies that the stabilizer of $\Phi(R)$ 
is also generated by $(\tau_0,\tau_0,\tau_0)$ 
and elements whose components are diagonal matrices. 
So multiplying $\tau$ to $g$ if necessary, 
we may assume that $g_2,g_3$ are diagonal matrices.

Since $(g_1,g_2,g_3)$ fixes $R_{322}$, there exist 
$t\in\gl_1,h\in\gl_2$ such that $(g_1,g_2,g_3)$
is in the form (\ref{eq:S14-phi-defn}). 
So $h$ is diagonal and so $g_1$ is diagonal. 
Then it is easy to verify that $g\in T$. 
\end{proof}

Let $P_1(x)$ be the degree $4$ polynomial 
of $\Phi(x)$ obtained by Proposition \ref{prop:222-invariant}. 
Then $P_1(R)=1$. Let $P_2(x)$ be the degree $6$ polynomial 
of $B(x)$ obtained by Proposition \ref{prop:322-invariant}.
Then $P_2(R)=1$ and
\begin{equation}
\label{eq:S14-P1P2-equivariant}
\begin{aligned}
P_1(gx) & = (\det g_1)^4(\det g_2)^2(\det g_3)^2(\det g_4)^2 P_1(x), \\
P_2(gx) & = (\det g_1)^2(\det g_2)^3(\det g_3)^3P_2(x).
\end{aligned}
\end{equation}

Let $U=\{x\in V\mid P_1(x),P_2(x)\not=0\}$. 
\begin{prop}
\label{prop:S14-rational-orbits}
The map $\Phi$ induces a bijective correspondence
\begin{equation*}
G_k\backslash U_k\cong 
\gl_2(k)^3\backslash (k^2\otimes k^2\otimes k^2)^{\sst}
\cong \Ex_2(k).
\end{equation*}
\end{prop}
\begin{proof}
We proceed as in Proposition \ref{prop:322+3-orbit}. 
We first show that $G_k\backslash U_k$ 
is in bijective correspondence with $\h^1(k,\Z/2\Z)$ 
where the action of the Galois group on $\Z/2\Z$ is trivial. 

We are in the situation of Section \ref{sec:regularity}
where $m=N=2$. So by Corollary \ref{cor:reducible-sep-orbit}, 
$U_{k^{\sep}}=G_{k^{\sep}}R$. Since $\h^1(k,G)=\{1\}$, 
$G_k\backslash U_k\cong \h^1(k,G_R)$.

Let $F/k$ be a quadratic extension. Since $\ch(k)\not=2$, 
$F/k$ is a Galois extension. 
Suppose that $f(z) = z^2+a_1z+a_2\in k[z]$ 
is a polynomial whose roots $\al=(\al_1,\al_2)$ generate $F$ 
over $k$. Let $\sig\in\gal(F/k)$ be the non-trivial element. 
Let 
\begin{align*}
& g_{\al,1} = 
\begin{pmatrix}
(\al_1-\al_2)^{-5} & 0 & 0 \\
0 & 1 & 1 \\
0 & -\al_1 & -\al_2
\end{pmatrix},\;
g_{\al,2}= 
\begin{pmatrix}
1 & 1 \\
-\al_1 & -\al_2
\end{pmatrix}, \\
& g_{\al}=(g_{\al,1},g_{\al,2},g_{\al,2},g_{\al,2}),\;
x_F = g_{\al}R. 
\end{align*}
Since $g_{\al}^{\sig}=g_{\al}\tau$, 
$x_F\in U_k$.

We determine the identity component of the 
stabilizer of $x_F$. 
Let $a(t_1,t_2,t_3)$ be the element in 
the statement (2) of Proposition \ref{prop:S14-R-regular}.
Then 
\begin{equation*}
G^{\circ}_{x_F\,F}=
g_{\al}\{a(t_1,t_2,t_3)\mid t_1,t_2,t_3\in F\}g_{\al}^{-1}.
\end{equation*}
$g_{\al}a(t_1,t_2,t_3)g_{\al}^{-1}\in G^{\circ}_{x_F\,k}$
if and only if 
\begin{equation}
\label{eq:S9-b-stab-rational}
\tau a(t_1^{\sig},t_2^{\sig},t_3^{\sig})\tau
= a(t_1,t_2,t_3).
\end{equation}
Since the first three components of $a(t_1,t_2,t_3)$
are the same as those of (\ref{eq:T-t-322-3}), 
$t_1t_3\in\mk$, $t_1^{\sig}=t_2$. 
Then it is easy to show that this 
implies (\ref{eq:S9-b-stab-rational}). 
As in Proposition (2.10) \cite[p.323]{kayu}
using Theorem \cite[p.17]{mum}, 
\begin{math}
G^{\circ}_{x_F} \cong \gl_1\times \mathrm{R}_{F/k} \gl_1.
\end{math}
Also $g_{\al}\tau g_{\al}^{-1}\in G_{x_F\,k}$.

Since $g_{\al}^{-1}g_{\al}^{\sig}=\tau$, 
the cohomology class corresponding to $x_F$ 
corresponds to $F$. 
Since $\h^1(k,G^{\circ}_{x_F})=\{1\}$ for all such $F$, 
$G_k\backslash U_k\cong \h^1(k,\Z/2\Z)$. 
By (\ref{eq:S14Phi-R-equivariant}), 
\begin{equation*}
\Phi(x_F) = (\det g_{\al,1}) (g_{\al,2},g_{\al,2},g_{\al,2})\Phi(R)
= (\al_1-\al_2)^{-4} (g_{\al,2},g_{\al,2},g_{\al,2})\Phi(R)
\end{equation*}
So the cohomology class corresponding to $\Phi(x_F)$ 
is the class of the 1-cocycle 
\begin{equation*}
(g_{\al,2},g_{\al,2},g_{\al,2})^{-1}(g_{\al,2},g_{\al,2},g_{\al,2})^{\sig}
=(\tau_0,\tau_0,\tau_0).
\end{equation*}
Therefore, $\Phi(x_F)$ corresponds to the same cohomology class 
in $\h^1(k,\Z/2\Z)$. This proves the statement of this proposition. 
\end{proof}

\section{Rational orbits (7)}
\label{sec:rational-orbits-222-22}

Let $G=\gl_2^3$ and (a) 
$V=\aff^2\otimes \aff^2\otimes \aff^2\oplus \aff^2$
or (b) $V=\aff^2\otimes \aff^2\otimes \aff^2\oplus \aff^2\oplus \aff^2$. 

We identify $\aff^2\otimes \aff^2\otimes \aff^2$ with the space of 
pairs of $2\times 2$ matrices. We describe elements of $V$ as 
(a) $x=(A(x),B(x),v(x))$ where $A(x),B(x)\in\m_2$ 
and $v(x)\in\aff^2$ is a column vector or 
(b) $x=(A(x),B(x),v_1(x),v_2(x))$ where $A(x),B(x)\in\m_2$ 
and $v_1(x),v_2(x)\in\aff^2$ are column vectors. 
If $g_1,g_2\in\gl_2$ and 
\begin{math}
g_3= 
\left(
\begin{smallmatrix}
\al & \be \\
\gam & \del 
\end{smallmatrix}
\right)
\end{math}
then we define an action of 
$g=(g_1,g_2,g_3)\in G$ on $V$ by 
\begin{align*}
gx & = (\al g_1A(x){}^tg_2+ \be g_1B(x){}^tg_2,
\gam g_1A(x){}^tg_2+\del g_1B(x){}^tg_2,g_1v(x)), \tag a \\
gx & = (\al g_1A(x){}^tg_2+ \be g_1B(x){}^tg_2,
\gam g_1A(x){}^tg_2+\del g_1B(x){}^tg_2,g_1v_1(x),g_2v_2(x)). \tag b
\end{align*}

Let $P_1(x)$ be the polynomial $P$ in Proposition \ref{prop:222-invariant} 
evaluated at $A(x)$.
In the case (a), let $\Phi:V\to \aff^2\otimes \aff^2\cong \m_2$ be the 
map obtained by applying Lemma \ref{lem:natural-pairing}
to the first $\aff^2$-factor of $\aff^2\otimes \aff^2\otimes \aff^2$ 
and the last $\aff^2$. We put $P_2(x)=\det \Phi(x)$. 
In the case (b), let $\Phi_1,\Phi_2:V\to \aff^2\otimes \aff^2\cong \m_2$ be the 
map obtained by applying Lemma \ref{lem:natural-pairing}
to the first (resp. second) $\aff^2$-factor of 
$\aff^2\otimes \aff^2\otimes \aff^2$ 
and the next to the last (resp. last) $\aff^2$. 
We put $P_2(x)=\det \Phi_1(x)$, $P_3(x)=\det \Phi_2(x)$. 
Then 
\begin{equation}
\label{eq:P1P2P3}
\begin{aligned}
P_1(gx) & = (\det g_1)^2(\det g_2)^2(\det g_3)^2P_1(x), \\
P_2(gx) & = (\det g_1)^2(\det g_2)(\det g_3)P_2(x)
\end{aligned}
\end{equation}
in both cases and 
$P_3(gx) = (\det g_1)(\det g_2)^2(\det g_3)P_3(x)$ in the case (b). 
Note that $P_1(x)$ (resp. $P_2(x)$ in the case (b)) 
is homogeneous of degree $2$ with respect to each
of $\aff^2\otimes \aff^2\otimes \aff^2$ and 
the first $\aff^2$ (the second $\aff^2$ in the case (b)).

Let 
\begin{equation}
\label{eq:w-defn-sec10}
w_1 = 
\begin{pmatrix}
1 & 0 \\
0 & 0
\end{pmatrix},\;
w_2 = 
\begin{pmatrix}
0 & 0 \\
0 & 1 
\end{pmatrix},\;
v_1 = v_2 = 
\begin{pmatrix}
1 \\ 1 
\end{pmatrix}. 
\end{equation}
We put (a) $w=(w_1,w_2,v_1)$ (b) $w=(w_1,w_2,v_1,v_2)$. 
We determine the Lie algebra ${\mathrm T}_{e_G}(G_w)$. 
Let $X=(x_{ij}),Y=(y_{ij}),Z=(z_{ij})\in\m_2$.

Suppose that $(e_G+\vep(X,Y,Z))w=w$. 
By Proposition \ref{prop:ie-alg-222}, 
$X=\diag(x_{11},x_{22}),Y=\diag(y_{11},y_{22}),Z=\diag(-x_{11}-y_{11},-x_{22}-y_{22})$. 
In the case (a), $(e_G+\vep(X,Y,Z))v_1=v_1$ 
if and only if $Xv_1=0$.  This is equivalent to 
$X=0$. In the case (b), $(e_G+\vep(X,Y,Z))v_i=v_i$ for $i=1,2$ 
if and only if $Xv_1=0$, $Yv_2=0$.  This is equivalent to 
$X=Y=0$. So $X=0,Y=\diag(y_{11},y_{22}),Z=-Y$ in the case (a) 
and $X=Y=Z=0$ in the case (b). Therefore, 
\begin{equation*}
\dim {\mathrm T}_{e_G}(G_w)= 
\begin{cases}
2 = 12-10=\dim G-\dim V & \text{case (a)}, \\
0 = 12-12=\dim G-\dim V & \text{case (b)}.
\end{cases}
\end{equation*}
By Proposition \ref{prop:open-orbit}, 
$Gw\sub V$ is Zariski dense 
and $G_w$ is smooth over $k$ in both cases. 

\begin{prop}
\label{prop:S20-regular}
\begin{itemize}
\item[(1)]
$G^{\circ}_w = \{(I_2,\diag(t_1,t_2),\diag(t_1^{-1},t_2^{-1}))\}$
in the case (a) and 
$G^{\circ}_w = \{e_G\}$ in the case (b).  
\item[(2)]
$(G,V)$ is a regular \pv{} in both cases.  
\end{itemize}
\end{prop}
\begin{proof}
(1) It is easy to see that the right hand side is contained in $G^{\circ}_w$. 
Since the dimensions are the same, (1) follows. 

(2) follows from (1). 
\end{proof}

Let $\tau=(\tau_0,\tau_0,\tau_0)$ (see (\ref{eq:J-defn})). 
It is easy to see that $\tau$ fixes $w$. 

\begin{prop}
\label{prop:S20-stabilizer} 
$G_w/G^{\circ}_w$ is represented by $\{1,\tau\}$.
\end{prop}
\begin{proof}
Suppose that $g=(g_1,g_2,g_3)\in G_w$. 
Since $g$ fixes $(w_1,w_2)$, by multiplying $\tau$ 
if necessary, we may assume that $g_1,g_2,g_3$ 
are diagonal matrices. Then it is easy to 
show that $g\in G^{\circ}_w$. 
\end{proof}

We put 
\begin{equation*}
U = 
\begin{cases}
\{x\in V_k\mid P_1(x),P_2(x)\not=0\} & \text{case (a)}, \\
\{x\in V_k\mid P_1(x),P_2(x),P_3(x)\not=0\} & \text{case (b)}. 
\end{cases}
\end{equation*}
\begin{prop}
\label{prop:222+2-orbit}
The map $V\ni x\mapsto (A(x),B(x))\in\aff^2\otimes \aff^2\otimes \aff^2$
induces a bijection 
$G_k\backslash U\to (\gl_2(k)^3)\backslash 
(k^2\otimes k^2\otimes k^2)^{\sst}\cong \Ex_2(k)$. 
\end{prop}
\begin{proof}
Note that the set $U$ is $G_k$-invariant. 
We show that if $x\in U$ then $G_kx$ is determined 
by the orbit of $(A(x),B(x))$. So we assume that
$(A(x),B(x))$ is $(w_1,w_2)$ or $x_F$. 

Since the consideration is similar, 
we only consider the case (b). The case (a) is easier. 
Suppose that $x\in U$. 
We first consider the case $(A(x),B(x))=(w_1,w_2)$. 
Let $v_1(x)=[x_1,x_2]$, $v_2(x)=[x_3,x_4]$. 
Since 
\begin{math}
\Phi_1(w_1,w_2,v_1(x),v_2(x)) = \left(
\begin{smallmatrix}
x_2 & 0 \\ 
0 & -x_1 
\end{smallmatrix}
\right), 
\end{math}
$x_1,x_2\not=0$. Similarly, $x_3,x_4\not=0$.
If 
\begin{math}
t=(\diag(x_1^{-1},x_2^{-1}),\diag(x_3^{-1},x_4^{-1}),
\diag(x_1x_3,x_2x_4)) 
\end{math}
then $tx=(w_1,w_2,[1,1],[1,1])$. So $x$ is in the orbit of 
$(w_1,w_2,[1,1],[1,1])$.

We next consider the case $(A(x),B(x))=x_F$. 
Let $a_1,a_2,\al$ be as in (\ref{eq:S20-wa-defn}) 
and $F=k(\al_1)$. 
Let $\sig\in\gal(F/k)$ be the non-trivial element.
We express $v_1(x),v_2(x)$ in the form 
$-(\al_1-\al_2)^{-1}h_{\al}[x_1,x_2],h_{\al}[x_3,x_4]$ respectively 
where $x_1,x_2,x_3,x_4\in F$. 
Then $v_1(x),v_2(x)\in k^2$ if and only 
if $x_2=x_1^{\sig},x_4=-x_3^{\sig}$. 
Since $P_2(x),P_3(x)\not=0$, $x_1,x_3\not=0$. 
If 
\begin{math}
t=g_{\al}(\diag(x_1^{-1},(x_1^{-1})^{\sig}),
\diag(x_3^{-1},(x_3^{-1})^{\sig}),
\diag(x_1x_3,x_1^{\sig}x_3^{\sig}))g_{\al}^{-1}
\end{math}
then $t\in G_k$ and 
\begin{equation*}
tx=(x_F,-(\al_1-\al_2)^{-1}h_{\al}[1,1],h_{\al}[1,-1])
= (x_F,[0,1],[2,a_1]).
\end{equation*}
So $x$ is in the orbit of $(x_F,[0,1],[2,a_1])$. 
\end{proof} 

\section{Non-empty strata}
\label{sec:non-empty}

In this section and the next two sections, 
we assume that $k$ is a fixed perfect field. 

Let $(G,V)$ be the \pv{} (\ref{eq:PV}) 
and $T,G_{\text{st}},T_{\text{st}},T_0$ as in (\ref{eq:T0-defn}).

The set $\gB$ consists 
of $292$ $\be_i$'s. We use the table in 
Section 8 \cite{tajima-yukie-GIT1}. 
In this section we shall prove that $S_{\be_i}\not=\emptyset$ for 
\begin{equation}
\label{eq:list-non-empty-54}
\begin{aligned}
i & = 1,3,5,9,14,15,16,20,21,33,35,36,37,40,42,49,50,64,70,71,75,95, \\
& \quad 101,105,106,107,108,110,113,121,131,149,150,151,152,164,178,202, \\
& \quad 216,217,223,224,226,227,232,254,256,258,259,270,271,272,273,280, \\
& \quad 281,285,286,287,289,291,292 
\end{aligned}
\end{equation}
for the \pv{} (\ref{eq:PV}).  
We shall prove that $S_{\be_i}=\emptyset$ for other $\be_i$'s 
in the next section.

The following table describes $M_{\be_i}$, $Z_{\be_i}$ as a 
\rep{} of $M^s_{\be_i}$, the coordinates of $Z_{\be_i},W_{\be_i}$ 
and $G_k\backslash S_{\be_i\, k}\cong P_{\be_i\,k}\backslash Y^{\sst}_{\be_i\,k}$.

\vskip 5pt

\begin{center} 

\; 
none \\

\hline
\end{tabular}

\end{center}

\vskip 10pt

We now verify that $S_{\be_i}\not=\emptyset$ and 
determine $G_k\backslash S_{\be_i\, k}$
for the above $i$'s.  

In the following, we write $S_i,Z_i,W_i,Y_i,\chi_i$ instead of 
$S_{\be_i},Z_{\be_i},W_{\be_i},Y_i,\chi_{\be_i}$ 
(but we keep writing $M_{\be_i}$, etc.). 

\begin{prop}
If $\be\in \gB$ and $M_{\be}=T$ then 
$Z^{\sst}_{\be}\not=\emptyset$. 
\end{prop}
\begin{proof}
Note that $M_{\be}=T$ implies that $M^{\text{st}}_{\be}=T_{\text{st}}$. 
By the definition of $\gB$, 
there exist $1\leq i_1<\cdots< i_N\leq 40$ 
such that $\be$ is the closest point to the  
origin of the convex hull of their weights 
$\gam_{i_1}\ccd \gam_{i_N}\in X^*(T_{\text{st}})$. 
We may assume that this convex hull is
contained in the hyperplane 
$\{t\in \mathfrak t^*\mid (t,\be)_*=(\be,\be)_*\}$
($(\;,\;)_*$ is the inner product on $\mathfrak t^*$ 
which is used to define $\gB$, etc.). 
So $\bbma_{i_1}\ccd \bbma_{i_N}\in Z_{\be}$ 
and there exist rational numbers $0\leq c_1\ccd c_N<1$ such that 
$\be=c_1\gam_{i_1}+\cdots+c_N\gam_{i_N},c_1+\cdots+c_N=1$.

There exists an integer $d_1>0$ such that 
$d_1\be$ is actually a character $\chi$
(i.e. an element of $X^*(T_{\text{st}})$). 
We take a common denominator $d_2$ of $c_1\ccd c_N$ 
and express $c_i$ as $b_i/d_2$ for $i=1\ccd N$. 
Then 
\begin{math}
d_1b_1\gam_{i_1}+\cdots d_1b_N\gam_{i_N} 
= d_2\chi. 
\end{math}
Writing multiplicatively, 
\begin{math}
\prod_{j=1}^N\gam_{i_j}(t)^{d_1b_j}= \chi(t)^{d_2}
\end{math}
for $t\in T_{\text{st}}$. Since $\chi$ is proportional 
to $\be$, $\chi(t)=1$ for $t\in G_{\text{st},\be}$. 
Therefore, if we put 
$P(x) = \prod_{j=1}^Nx_{i_j}^{d_1b_j}$ then  
$P(x)$ is a non-constant polynomial on $Z_{\be}$ which is 
invariant under the action of $G_{\text{st},\be}$. 
\end{proof}

We consider the situation of the above proposition. 
We express elements of $T$ (not $T_{\text{st}}$) as 
\begin{equation}
\label{eq:t-defn}
t=(\diag(t_1\ccd t_5),\diag(t_6\ccd t_9)). 
\end{equation}
Let $\del_1\ccd \del_9\in X^*(T)$ be the characters
such that $\del_i(t)=t_i$. Then 
$\{\del_1\ccd \del_9\}$ is a $\Z$-basis of $X^*(T)$. 
Suppose that $Z_{\be}=\lan \bbma_{i_1}\ccd \bbma_{i_N}\ran$ 
and that $\chi_1\ccd \chi_N\in X^*(T)$ are the weights of 
$\bbma_{i_1}\ccd \bbma_{i_N}$ with respect to $T$ (not $T_{\text{st}}$). 
There exists $m_{ij}\in \Z$ for $i=1\ccd N,j=1\ccd 9$ such that 
$\chi_i = \prod_{j=1}^9 \del_j^{m_{ij}}$. 

The following proposition is obvious and so we do not provide 
the proof. 
\begin{prop}
\label{prop:M=Tsuejective} 
In the above situation, if there exists an 
($N\times N$)-minor $C$ of the matrix $(m_{ij})$ 
such that $\det C=\pm 1$, then 
$Z^{\sst}_{\be\,k}$ is a single $T_k$-orbit. 
\end{prop}

When we consider individual cases we would like to 
show that  
\begin{math}
P_{\be_i\,k}\backslash Y^{\sst}_{i\,k}
\cong M_{\be_i\,k}\backslash Z^{\sst}_{i\,k}
\end{math}
for all $i$ such that $Z^{\sst}_i\not=\emptyset$. 
What we have to do is to show  that  
if $x\in Z^{\sst}_{i\,k}$ and 
$y\in W_{i\,k}$ then there exists 
$u\in U_{\be_{i\,k}}$ such that $ux=(x,y)\in Z_{i\,k}\oplus W_{i\,k}$. 
We show that under a certain condition, 
if this is possible for a particular element 
$R\in Z^{\sst}_{i\,k}$ then it is possible for all 
$x\in Z^{\sst}_{i\,k}$. 
Since we mention the above property very often, 
we give it a name. 

\begin{property}
\label{property:W-eliminate}
Let $x\in Z^{\sst}_{i\,k}$. If 
$y\in W_{i\,k}$ then there exists 
$u\in U_{\be_{i\,k}}$ such that $ux=(x,y)\in Z_{i\,k}\oplus W_{i\,k}$. 
\end{property}

\begin{cond}
\label{cond:unipotent-eliminate-condition} 
\begin{itemize}
\item[(1)]
$Z^{\sst}_{i\,k^{\sep}}$ is a single $M_{\be_i\,k^{\sep}}$-orbit. 
\item[(2)]
There exists $R\in Z^{\sst}_{i\,k}$ with the property that 
if $y\in W_{i\,k^{\sep}}$ 
then there exists $u\in U_{\be_i\,k^{\sep}}$ such that 
$uR=(R,y)$.  
\item[(3)]
$G_R\cap U_{\be_i}$ is connected.  
\end{itemize}
\end{cond}

\begin{prop}
\label{prop:W-eliminate}
If Condition \ref{cond:unipotent-eliminate-condition} 
is satisfied then Property \ref{property:W-eliminate} 
holds for any $x\in Z^{\sst}_{i\,k}$.
\end{prop}
\begin{proof}
Since $G_R\cap U_{\be_i}$ is connected and $x\in M_{\be_i\,k^{\sep}}R$, 
$G_x\cap U_{\be_i}$ is connected also. Since $k$ is a perfect field and
$G_x\cap U_{\be_i}$ is unipotent, it splits over $k$ 
(see 15.5 Corollary (ii) \cite[p.205]{borelb}). This means that 
there is a composition series 
$G_x\cap U_{\be_i}=G_0\supset G_1\supset \cdots \supset G_s=\{e_G\}$ 
such that $G_i/G_{i+1}\cong \mathbbm G_a$ (the additive group over the 
ground field).   
Therefore, $\h^1(k,G_x\cap U_{\be_i})=\{1\}$.  

Since $x\in M_{\be_i\,k^{\sep}}R$, 
if $y\in W_{i\,k}$ then there exists $u\in U_{\be_i\,k^{\sep}}$ such that 
$ux=(x,y)$. Then $u^{\sig}x=(x,y)$ also for all $\sig\in\gal(k^{\sep}/k)$. 
Therefore, $u^{-1}u^{\sig}\in G_x\cap U_{\be_i}$. 
Since $\h^1(k,G_x\cap U_{\be_i})=\{1\}$, 
there exists $u_1\in G_{x\,k^{\sep}}\cap U_{\be_i\,k^{\sep}}$ 
such that $u^{-1}u^{\sig}=u_1^{-1}u_1^{\sig}$ for all $\sig\in\gal(k^{\sep}/k)$. 
This implies that $uu_1^{-1}\in U_{\be_i\,k}$ and that 
$uu_1^{-1}x=ux=(x,y)$.  
\end{proof} 

Also the following proposition is obvious. 
\begin{prop}
\label{prop:W-eliminate-single-orbit}
If $R\in Z^{\sst}_{i\,k}$, $Z^{\sst}_{i\,k}=M_{\be_i\,k}R$ 
and Property \ref{property:W-eliminate} holds for $R$ 
then Property \ref{property:W-eliminate} holds for 
any $x\in Z^{\sst}_{i\,k}$. 
\end{prop}

It is convenient to point out 
certain situations where we can apply 
Proposition \ref{prop:W-eliminate}. 

Suppose that  $n\geq 0,m>0$ and that 
$\phi:\aff^{n+m}\to \aff^m$
is a map in the form 
\begin{equation*}
\phi(u)=\phi(u_1\ccd u_{n+m})
=(u_{n+1}+P_1(u)\ccd u_{n+m}+P_m(u))
\end{equation*}
where $P_i(u)$ is a polynomial of 
$u_1\ccd u_{n+i-1}$.  For convenience, 
we call $u_1\ccd u_n$ ``extra variables''.

\begin{lem}
\label{lem:connected-criterion}
In the above situation we have to following. 
\begin{itemize}
\item[(1)]
$\phi$ induces a surjective map $k^{n+m}\to k^m$. 
\item[(2)]
$\phi^{-1}(0\ccd 0)\sub \aff^{n+m}$ is connected.  
\end{itemize}
\end{lem}

We do not provide the proof of the above lemma since 
it is very easy. We shall use the above lemma very often 
when we consider individual cases. 

\begin{rem}
\label{remark:beta-invariance}
Suppose that $P(x)$ is a polynomial on $Z_i$ 
and that $\psi(g)$ is a character of $M^{\text{st}}_{\be_i}$. 
If $P(gx)=\psi(g)P(x)$ for all $x\in Z_i$, $g\in M^{\text{st}}_{\be_i}$ 
and $\psi$ is proportional to $\chi_i$. Then 
$P(gx)=P(x)$ for $g\in G_{\text{st},\be_i}$. 

For, there exists an indivisible character $\om$ on $M^{\text{st}}_{\be_i}$ 
and integers $a,b>0$ such that $\psi=\om^a,\chi_i=\om^b$ on $M^{\text{st}}_{\be_i}$. 
Since $G_{\text{st},\be_i}$ is the identity component of 
$\{g\in M^{\text{st}}_{\be_i}\mid \chi_i(g)=1\}$, it is generated by 
the center $Z$ and $[G_{\text{st},\be_i},G_{\text{st},\be_i}]$,  
which is connected semi-simple and is a normal subgroup of 
$G_{\text{st},\be_i}$. So a character of 
$G_{\text{st},\be_i}$ is trivial if and only if 
it is trivial on $Z$. Since $X^*(Z)$ is torsion 
free and $\om^b=1$ on $G_{\text{st},\be_i}$, 
$\om=1$ on $G_{\text{st},\be_i}$. Therefore, 
$\psi=1$ on $G_{\text{st},\be_i}$ also. 
\end{rem}

We now consider individual strata.
We remind the reader that $\{\bbmp_{n,i}\mid i=1\ccd n\}$ 
is the standard basis of $\aff^n$. We use the notation 
such as $p_{3,12}\in \wedge^2 \aff^3$, etc. 
  
For each $i$, 

\begin{itemize}
\item[(1)]
we show that $Z^{\sst}_i\not=\emptyset$ and interpret 
$M_{\be_i\,k}\backslash Z^{\sst}_{i\,k}$, 
\item[(2)]
choose a point $R(i)\in Z^{\sst}_i$ and 
apply Lemma \ref{lem:connected-criterion}.
\end{itemize}

We shall verify the following theorem for the rest of this section. 

\begin{thm}
\label{thm:non-empty}
\begin{itemize}
\item[(1)] $S_{\be_i}\not=\emptyset$ for $i$ in (\ref{eq:list-non-empty-54}).
\item[(2)]
The information on the above table is correct. 
\end{itemize}
\end{thm}

For the rest of this section, 
$u_1=(u_{1ij})$ where $5\geq i>j\geq 1$ and 
$u_2=(u_{2ij})$ where $4\geq i>j\geq 1$.  
We put $u=(u_1,u_2)$, $n(u)=(n_5(u_1),n_4(u_2))$. 
For each $i$, we may consider additional 
conditions for $u_1,u_2$.

\vskip 5pt

(1) $S_1$,  
$\be_1=\tfrac {1} {12} (0,0,0,0,0,-3,1,1,1)$ 

In this case $G_{\text{st},\be_1}\cong \spl_5\times \spl_3$ 
is semi-simple. So any relative invariant polynomial is 
invariant by the action of $G_{\text{st},\be_1}$. 
By Corollary \ref{cor:S1-regular}, 
$Z^{\sst}_1\not=\emptyset$.  
Proposition \ref{prop:S1-orbit-final} implies that 
if $\ch(k)\not=2$ then $M_{\be_1\,k}\backslash Z^{\sst}_{1\,k}$ 
is in bijective correspondence with $\mathrm{Prg}_2(k)$. 

In this case $W_1=\{0\}$ and so 
$P_{\be_1\,k}\backslash Y^{\sst}_{1\,k} \cong M_{\be_1\,k}\backslash Z^{\sst}_{1\,k}$ 
if $\ch(k)\not=2$.

\vskip 5pt

(2) $S_3$, $\be_3 = \tfrac {1} {780} (-12,-12,8,8,8,-15,5,5,5)$ 

We identify the element 
$(\diag(g_{11},g_{12}),\diag(t_2,g_2))\in M_{[2],[1]}=M_{\be_3}$
with the element 
$g=(g_2,g_{12},g_{11},t_2)\in \gl_3^2\times \gl_2\times \gl_1$. 
On $M^{\text{st}}_{\be_3}$, 
\begin{equation*}
\chi_3(g)= (\det g_{11})^{-12}(\det g_{12})^8t_2^{-15}(\det g_2)^5
= (\det g_{12})^{20}(\det g_2)^{20}. 
\end{equation*}

For $x\in Z_3$, let 
\begin{align*}
& A(x) = 
\begin{pmatrix}
0 & x_{341} & x_{351} \\
-x_{341} & 0 & x_{451} \\
-x_{351} & -x_{451} & 0 
\end{pmatrix}, \\
& B_1(x) = \begin{pmatrix}
x_{132} & x_{142} & x_{152} \\
x_{133} & x_{143} & x_{153} \\
x_{134} & x_{144} & x_{154}
\end{pmatrix},\; 
B_2(x) = \begin{pmatrix}
x_{232} & x_{242} & x_{252} \\
x_{233} & x_{243} & x_{253} \\
x_{234} & x_{244} & x_{254}
\end{pmatrix}
\end{align*}
and $B(x)=(B_1(x),B_2(x))$. 
We identify $Z_3\cong \wedge^2 \aff^3 \oplus \m_3\otimes \aff^2$
by the map $Z_3\ni x\mapsto (A(x),B(x))$. 
This is the same vector space as the one considered in 
Section \ref{sec:rational-orbits-332.3} and 
$M_{\be_3}$ is the same as the group considered in 
Section \ref{sec:rational-orbits-332.3}. 
Since 
\begin{equation*}
A(gx) = t_2 (\wedge^2 g_{12})A(x),\;
\begin{pmatrix}
B_1(gx) \\
B_2(gx)
\end{pmatrix}
= g_{11}\begin{pmatrix}
g_2B_1(x){}^tg_{12} \\
g_2B_2(x){}^tg_{12}
\end{pmatrix}, 
\end{equation*}
the action of $(g_2,g_{12},g_{11},t_2)$
on $V$ is the same as the one in 
Section \ref{sec:rational-orbits-332.3}. 

Let $P_1(x),P_2(x)$ be the relative invariant polynomials 
defined in Section \ref{sec:rational-orbits-332.3}. 
Then by (\ref{eq:sec6-P1}) and Proposition \ref{prop:S5-second-equivariant} (1), 
\begin{align*}
P_1(gx) & = (\det g_{11})^6(\det g_{12})^4(\det g_2)^4P_1(x), \\
P_2(gx) & = (\det g_{11})^3(\det g_{12})^4t_2^3(\det g_2)^2P_2(x). 
\end{align*}
Let $P(x)=P_1(x)P_2(x)^3$. Then on $M^{\text{st}}_{\be_3}$, 
\begin{align*}
P(gx) & = ((\det g_{11})^6(\det g_{12})^4(\det g_2)^4)
((\det g_{11})^3(\det g_{12})^4t_2^3(\det g_2)^2)^3P(x) \\
& = (\det g_{11})^{15}(\det g_{12})^{16}t_2^9(\det g_2)^{10}P(x)
= (\det g_{12})(\det g_2)P(x). 
\end{align*}

Therefore, $P(x)$ is 
invariant under the action of $G_{\text{st},\be_3}$.

By Proposition \ref{prop:S3-orbit-rational}, 
$M_{\be_3\,k}\backslash Z^{\sst}_{3\,k}$ is in bijective 
correspondence with $\Ex_3(k)$. 

Let $R(3)\in Z_3$ be the element 
such that $A(x)=w_1,B_1(x)=w_{21},B_2(x)=w_{22}$ 
where $w_1$ is the element (\ref{eq:set6-w1}) and 
$w_{21},w_{22}$ are the elements 
$w_1,w_2$ in (\ref{eq:sec6-w-defn-sec4}).
Explicitly, 
$R(3)=e_{341}-e_{351}+e_{451}+e_{132}-e_{143}+e_{243}-e_{254}$. 
Then $R(3)\in Z^{\sst}_3$. 

We assume that $u_{1ij}=0$ unless 
$i=3,4,5,j=1,2$ and $u_{2ij}=0$ unless $j=1$.  
Then the four components of 
$n(u)R(3)$ are as follows: 
\begin{align*}
& \begin{pmatrix}
0 & 0 & 0 & 0 & 0 \\
0 & 0 & 0 & 0 & 0 \\
0 & 0 & 0 & 1 & -1 \\
0 & 0 & -1 & 0 & 1\\
0 & 0 & 1 & -1 & 0 
\end{pmatrix},\; 
\begin{pmatrix}
0 & 0 & 1 & 0 & 0 \\
0 & 0 & 0 & 0 & 0 \\
-1 & 0 & 0 & -u_{141} + u_{221} & -u_{151} - u_{221} \\
0 & 0 & * & 0 & u_{221} \\
0 & 0 & * & * & 0
\end{pmatrix}, \\
& \begin{pmatrix}
0 & 0 & 0 & -1 & 0 \\
0 & 0 & 0 & 1 & 0 \\
0 & 0 & 0 & -u_{131} + u_{132} + u_{231} & -u_{231} \\
1 & -1 & * & 0 & u_{151} - u_{152} + u_{231} \\
0 & 0 & * & * & 0 
\end{pmatrix}, \\
& \begin{pmatrix}
0 & 0 & 0 & 0 & 0 \\
0 & 0 & 0 & 0 & -1 \\
0 & 0 & 0 & u_{241} & -u_{132} - u_{241} \\
0 & 0 & * & 0 & -u_{142} + u_{241} \\
0 & 1 & * & * & 0 \\
\end{pmatrix}.
\end{align*}
Note that we obtain $R(3)$ by substituting $u_{1ij}=u_{2ij}=0$ 
for all $i,j$. 

We can apply Lemma \ref{lem:connected-criterion} 
to the map $\aff^9\to \aff^9$ defined by the sequence
\begin{align*}
& u_{221},u_{141}-u_{221},u_{151}+u_{221},
u_{241},u_{132}+u_{241}, \\
& u_{142}-u_{241},u_{231},
u_{152}-u_{151}-u_{231},u_{131}-u_{132}-u_{231}
\end{align*}
(there are no extra variables). 
Note that the above $9$ entries exhaust coordinates 
of $W_3$ (we shall not point this out in the remaining cases). 
So by Proposition \ref{prop:W-eliminate}, 
Property \ref{property:W-eliminate} holds for 
any $x\in Z^{\sst}_{3\,k}$. Therefore, 
$P_{\be_3\,k}\backslash Y^{\sst}_{3\,k}$ is in bijective 
correspondence with $\Ex_3(k)$ also. 

\vskip 5pt 

(3) $S_5$, 
$\be_5 = \tfrac 1 {180} (-2,-2,-2,-2,8,-5,-5,5,5)$

We identify the element 
$(\diag(g_1,t_1),\diag(g_{21},g_{22}))\in M_{[4],[2]}=M_{\be_5}$
with the element 
$g=(g_1,g_{21},g_{22},t_1)\in \gl_4\times \gl_2^2\times \gl_1$. 
On $M^{\text{st}}_{\be_5}$, 
\begin{equation*}
\chi_5(g)= (\det g_1)^{-2}t_1^8(\det g_{21})^{-5}(\det g_{22})^5
= t_1^{10}(\det g_{22})^{10}. 
\end{equation*}

For $x\in Z_5$, let
\begin{align*}
& A(x) = 
\begin{pmatrix}
x_{151} & x_{152} \\
x_{251} & x_{252} \\
x_{351} & x_{352} \\
x_{451} & x_{452} 
\end{pmatrix},\;
B_1(x) =
\begin{pmatrix}
0 & x_{123} & x_{133} & x_{143} \\
-x_{123} & 0 & x_{233} & x_{243} \\
-x_{133} & -x_{233} & 0 & x_{343} \\
-x_{143} & -x_{243} & -x_{343} & 0 
\end{pmatrix}, \\
& B_2(x) =
\begin{pmatrix}
0 & x_{124} & x_{134} & x_{144} \\
-x_{124} & 0 & x_{234} & x_{244} \\
-x_{134} & -x_{234} & 0 & x_{344} \\
-x_{144} & -x_{244} & -x_{344} & 0 
\end{pmatrix},\; B(x) = (B_1(x),B_2(x)).
\end{align*}
We identify $Z_5\cong \m_{4,2}\oplus \wedge^2 \aff^4\otimes \aff^2$ 
by the map $Z_5\ni x\mapsto (A(x),B(x))$. 
Since 
\begin{equation*}
A(gx) = t_1g_1 A(x) {}^t g_{21},\; 
\begin{pmatrix}
B_1(gx) \\
B_2(gx)  
\end{pmatrix}
= g_{22}\begin{pmatrix}
g_1B_1(x){}^tg_1 \\
g_1B_2(x){}^tg_1 
\end{pmatrix},
\end{equation*}
$Z_5$ is the same vector space as the one considered in 
Section \ref{sec:rational-orbits-42-wedge42} and 
$M_{\be_5}$ is almost the same as the group considered in 
Section \ref{sec:rational-orbits-42-wedge42} except for the extra 
$\gl_1$-factor. 
The action of $(g_1,g_{21},g_{22},1)$ 
is the same as the one in Section \ref{sec:rational-orbits-42-wedge42}. 
If $t=(I_4,I_2,I_2,t_1)$ then 
$A(tx)=t_1A(x)$,  $B(tx)=B(x)$.

Let $P_1(x),P_2(x)$ be the polynomials in (\ref{eq:S5-P1P2equivariant}).
Since $P_1(x)$ is homogeneous of degree $8$ with respect to 
$A(x)$, $P_1(tx)=t^8P_1(x)$. Since $P_2(x)$ is a polynomial of $B(x)$, 
$P_2(tx)=P_2(x)$. We put $P(x)=P_1(x)^2P_2(x)$. Then on $M^{\text{st}}_{\be_5}$,  
by (\ref{eq:S5-P1P2equivariant}), 
\begin{align*}
P(gx) & = ((\det g_1)^6 t_1^8(\det g_{21})^4(\det g_{22})^4)^2
((\det g_1)^2(\det g_{22})^2)P(x) \\
& = (\det g_1)^{14}t_1^{16} (\det g_{21})^8(\det g_{22})^{10}P(x)
= t_1^2 (\det g_{22})^2 P(x). 
\end{align*}
Therefore, $P(x)$ is 
invariant under the action of $G_{\text{st},\be_5}$.

Let $R(5)\in Z_5$ be the element which corresponds to $w$ 
in Section \ref{sec:rational-orbits-42-wedge42} (see 
(\ref{eq:sec7-w-defn-sec4}) and (\ref{eq:sec7-w-defn-sec7})).
Explicitly, $R(5)=e_{251}+e_{451}+e_{152}+e_{352}+e_{123}+e_{344}$. 
Then $R(5)\in Z^{\sst}_5$. 

Let $g'=(g_1,t_1g_{21},g_{22},1)$. 
Then $A(g'x)=A(gx)$ and $B(g'x)=B(gx)$. 
So the action of $\gl_1$ can be absorbed into the action 
of $\gl_4\times \gl_2^2$. Therefore, 
by Proposition \ref{prop:S5-orbit-rational}, 
$M_{\be_5\,k}\backslash Z^{\sst}_{5\,k}$ is in bijective 
correspondence with $\Ex_2(k)$. 

We assume that $u_{1ij}=0$ unless
$i=5$ and $u_{2ij}=0$ unless $i=3,4,j=1,2$. 
Then the first two components of 
$n(u)R(5)$ are the same as those of $R(5)$ 
and the remaining components are as follows:
\begin{align*}
& \begin{pmatrix}
0 & 1 & 0 & 0 & u_{152}+u_{232} \\
-1 & 0 & 0 & 0 & -u_{151}+u_{231}\\
0 & 0 & 0 & 0 & u_{232} \\
0 & 0 & 0 & 0 & u_{231} \\
* & * & * & * & 0 
\end{pmatrix}, \;
\begin{pmatrix}
0 & 0 & 0 & 0 & u_{242} \\
0 & 0 & 0 & 0 & u_{241} \\
0 & 0 & 0 & 1 & u_{154}+u_{242} \\
0 & 0 & -1 & 0 & -u_{153}+u_{241} \\
* & * & * & * & 0 
\end{pmatrix}.
\end{align*}

We can apply Lemma \ref{lem:connected-criterion} 
to the map $\aff^8\to \aff^8$ defined by the sequence
\begin{align*}
& u_{231},u_{232},u_{151}-u_{231},u_{152}+u_{232}, 
u_{241},u_{242},u_{153}-u_{241},u_{154}+u_{242}
\end{align*}
(there are no extra variables). 
So by Proposition \ref{prop:W-eliminate}, 
Property \ref{property:W-eliminate} holds for 
any $x\in Z^{\sst}_{5\,k}$. Therefore, 
$P_{\be_5\,k}\backslash Y^{\sst}_{5\,k}$ is in bijective 
correspondence with $\Ex_2(k)$ also.

\vskip 5pt

(4) $S_9$, $\be_9=\tfrac {3} {620} (-16,4,4,4,4,-5,-5,-5,15)$ 

We identify the element 
$(\diag(t_1,g_1),\diag(g_2,t_2))\in M_{[1],[3]}=M_{\be_9}$
with the element 
$g=(g_1,g_2,t_1,t_2)\in \gl_4\times \gl_3\times \gl_1^2$. 
On $M^{\text{st}}_{\be_9}$, 
\begin{equation*}
\chi_9(g)= t_1^{-16}(\det g_1)^4(\det g_2)^{-5}t_2^{15}
= (\det g_1)^{20}t_2^{20}. 
\end{equation*}

For $x\in Z_9$, let 
\begin{equation*}
A_i(x) = 
\begin{pmatrix}
0 & x_{23i} & x_{24i} & x_{25i} \\
-x_{23i} & 0 & x_{34i} & x_{35i} \\
-x_{24i} & -x_{34i} & 0 & x_{45i} \\
-x_{25i} & -x_{35i} & -x_{45i} & 0
\end{pmatrix}\; (i=1,2,3),\;
v(x) = 
\begin{pmatrix}
x_{124} \\
x_{134} \\
x_{144} \\
x_{154}
\end{pmatrix}
\end{equation*}
and $A(x)=(A_1(x),A_2(x),A_3(x))$. 
We identify $Z_9\cong \wedge^2 \aff^4\otimes \aff^3\oplus \aff^4$ 
by the map $Z_9\ni x \mapsto (A(x),v(x))$.
This is the same vector space considered in the case (b) of 
Section \ref{sec:rational-orbits-wedge43} and 
$M_{\be_9}$ is almost the same as the group considered in 
Section \ref{sec:rational-orbits-wedge43} except for the extra 
$\gl_1$-factors. Since 
\begin{equation*}
A(gx) = 
g_2 \begin{pmatrix}
g_1 A_1(x) {}^tg_1 \\
g_1 A_1(x) {}^tg_1 \\
g_1 A_1(x) {}^tg_1 
\end{pmatrix},\;
v(gx) = t_1t_2 g_1 v(x),  
\end{equation*}
the action of $(g_1,g_2,1,1)$ is the same as the 
one in Section \ref{sec:rational-orbits-wedge43}, i.e., 
the natural action of $\gl_4\times \gl_3$.

Let $P_1(x),P_2(x)$ be the polynomials on $Z_9$ 
obtained in the consideration of the case (b) 
of Section \ref{sec:rational-orbits-wedge43} 
(see Proposition \ref{prop:section8-b-P1-value} and (\ref{eq:43-4-P2-equivariant})).
$P_1(x)$ is homogeneous of degrees $3,2$ with 
respect to $A(x),v(x)$ respectively and 
$P_2(x)$ is a homogeneous degree $6$ 
polynomial of $A(x)$.   So by 
Proposition \ref{prop:section8-b-P1-value} (2)
and (\ref{eq:43-4-P2-equivariant}), 
\begin{equation*}
P_1(gx) = t_1^2(\det g_1)^2(\det g_2)t_2^2P_1(x),\;
P_2(gx) = (\det g_1)^3(\det g_2)^2P_2(x).
\end{equation*}
Let $P(x)=P_1(x)^5P_2(x)$. Then on $M^{\text{st}}_{\be_9}$, 
\begin{align*}
P(gx)
& = (t_1^2(\det g_1)^2(\det g_2)t_2^2)^5
((\det g_1)^3(\det g_2)^2) \\
& =t_1^{10}(\det g_1)^{13}(\det g_2)^7t_2^{10}P(x)
= (\det g_1)^3t_2^3P(x). 
\end{align*}
Therefore, $P(x)$ is invariant under the action 
of $G_{\text{st},\be_9}$.

Let $g'=(t_1t_2g_1,(t_1t_2)^{-2}g_2,1,1)$. 
Then $A(g'x)=A(gx)$ and $v(g'x)=v(gx)$. 
So the action of $\gl_1^2$ can be absorbed into the action 
of $\gl_4\times \gl_3$. Therefore, 
By Proposition \ref{prop:43-3-orbit}, 
$M_{\be_9\,k}\backslash Z^{\sst}_{9\,k}$ 
is in bijective correspondence with $\mathrm{Prg}_2(k)$. 

Let $R(9)\in Z_9$ be the element which corresponds to 
$(w_1,w_2,w_3,w_0)$ in (\ref{eq:w-defn-sec8b}).
Explicitly, 
$R(9)=e_{231}+e_{252}-e_{342}+e_{453}+e_{124}+e_{154}$. 
We assume that $u_{1ij}=0$ unless $j=1$ and $u_{2ij}=0$ unless $i=4$. 
Then the first three components of 
$n(u)R(9)$ are the same as those of $R(9)$
and the last component is as follows: 
\begin{align*}
& \begin{pmatrix}
0 & 1 & 0 & 0 & 1 \\
-1 & 0 & -u_{131} + u_{241} & -u_{141} &  -u_{151} + u_{121} + u_{242} \\
0 & * & 0 & -u_{242} & u_{131} \\
0 & * & * & 0 & u_{141} + u_{243} \\
-1 & * & * & * & 0
\end{pmatrix}.
\end{align*}

We can apply Lemma \ref{lem:connected-criterion} 
to the map $\aff^7\to \aff^6$ defined by the sequence
\begin{align*}
u_{131},u_{141},u_{242},u_{241}-u_{131},u_{121}-u_{151}+u_{242},u_{243}+u_{141}
\end{align*}
where $u_{151}$ is an extra variable.
So by Proposition \ref{prop:W-eliminate}, 
Property \ref{property:W-eliminate} holds for 
any $x\in Z^{\sst}_{9\,k}$. Therefore, 
$P_{\be_9\,k}\backslash Y^{\sst}_{9\,k}$ 
is in bijective correspondence with $\mathrm{Prg}_2(k)$ also.

\vskip 5pt

(5) $S_{14}$, $\be_{14}=\tfrac {3} {220} (-6,-6,4,4,4,-5,-5,5,5)$

We identify the element 
$(\diag(g_{11},g_{12}),\diag(g_{21},g_{22}))\in M_{[2],[2]}=M_{\be_{14}}$
with the element 
$g=(g_{12},g_{11},g_{22},g_{21})\in \gl_3\times \gl_2^3$.  
On $M^{\text{st}}_{\be_{14}}$, 
\begin{equation*}
\chi_{{14}}(g) = 
(\det g_{11})^{-6}(\det g_{12})^4 (\det g_{21})^{-5} (\det g_{22})^5
= (\det g_{12})^{10}(\det g_{22})^{10}. 
\end{equation*}

For $x\in Z_{14}$, let 
\begin{align*}
& A_1(x) = 
\begin{pmatrix}
0 & x_{341} & x_{351} \\
-x_{341} & 0 & x_{451} \\
-x_{351} & -x_{451} & 0 
\end{pmatrix},\;
A_2(x) = 
\begin{pmatrix}
0 & x_{342} & x_{352} \\
-x_{342} & 0 & x_{452} \\
-x_{352} & -x_{452} & 0 
\end{pmatrix}, \\
& B_1(x) = 
\begin{pmatrix}
x_{133} & x_{134} \\
x_{233} & x_{234}
\end{pmatrix},\;
B_2(x) = 
\begin{pmatrix}
x_{143} & x_{144} \\
x_{243} & x_{244}
\end{pmatrix},\;
B_3(x) = 
\begin{pmatrix}
x_{153} & x_{144} \\
x_{253} & x_{254}
\end{pmatrix}, \\
& A(x) =(A_1(x),A_2(x)),\; B(x) = (B_1(x),B_2(x),B_3(x)).
\end{align*}
We identify $Z_{14}\cong \wedge^2 \aff^3\otimes \aff^2\oplus \aff^3\otimes \m_2$
by the map $Z_{14}\ni x\mapsto (A(x),B(x))$. 
Since 
\begin{equation*}
\begin{pmatrix}
A_1(gx) \\
A_2(gx) 
\end{pmatrix}
= g_{21}
\begin{pmatrix}
g_{12} A_1(x) {}^tg_{12} \\
g_{12} A_2(x) {}^tg_{12}
\end{pmatrix},\;
\begin{pmatrix}
B_1(gx) \\
B_2(gx) \\
B_3(gx)
\end{pmatrix}
= g_{12}
\begin{pmatrix}
g_{11} B_1(x) {}^tg_{22} \\
g_{11} B_2(x) {}^tg_{22} \\
g_{11} B_3(x) {}^tg_{22}
\end{pmatrix},\;
\end{equation*}
$(M_{\be_{14}},Z_{14})$ can be identified with 
the case (b) of Section \ref{sec:rational-orbits-wedge3-32}.

Let $P_1(x),P_2(x)$ be polynomials which
correspond to $P_1(x),P_2(x)$ in the case (b) of 
Section \ref{sec:rational-orbits-wedge3-32}. 
Then by (\ref{eq:S14-P1P2-equivariant}), 
\begin{align*}
P_1(gx) & =  (\det g_{11})^2 (\det g_{12})^4
(\det g_{21})^2(\det g_{22})^2P_1(x), \\
P_2(gx) & = (\det g_{11})^3 (\det g_{12})^2 (\det g_{22})^3P_2(x). 
\end{align*}
Let $P(x)=P_1(x)^2P_2(x)$. Then on $M^{\text{st}}_{\be_{14}}$, 
\begin{align*}
P(gx) 
& = ((\det g_{11})^2(\det g_{12})^4(\det g_{21})^2(\det g_{22})^2)^2
((\det g_{11})^3(\det g_{12})^2(\det g_{22})^3) P(x) \\
& = (\det g_{11})^7(\det g_{12})^{10}(\det g_{21})^4(\det g_{22})^7P(x)
= (\det g_{12})^3(\det g_{22})^3P(x).
\end{align*}
Therefore, $P(x)$ is 
invariant under the action of $G_{\text{st},\be_{14}}$. 

Proposition \ref{prop:S14-rational-orbits} implies that 
\begin{math}
M_{\be_{14}\,k}\backslash Z^{\sst}_{14\,k}\cong \Ex_2(k).
\end{math}

Let $R(14)\in Z_{14}$ be the element such that 
\begin{equation*}
A_1(R(14))=-p_{3,13},\; A_2(R(14))=p_{3,12},\;
B(R(14))=R_{322}
\end{equation*}
($R_{322}$ is the element $w$ in (\ref{eq:322-generator2-alternative})). 
Explicitly, 
$R(14)=-e_{351}+e_{342}-e_{133}+e_{253}+e_{144}+e_{234}$. 
This element corresponds to $R$ in (\ref{eq:S14-b-R}).

We assume that $u_{1ij}=0$ unless 
$i=3,4,5,j=1,2$ and $u_{2ij}=0$ unless 
$i=3,4,j=1,2$. Then the first two components
of $n(u)R(14)$ are the same as those of $R(14)$
and the remaining components are as follows: 
\begin{align*}
& \begin{pmatrix}
0 & 0 & -1 & 0 & 0 \\
0 & 0 & 0 & 0 & 1 \\
1 & 0 & 0 & u_{141}+u_{232} & u_{151}+u_{132}-u_{231} \\
0 & 0 & * & 0 & u_{142} \\
0 & -1 & * & * & 0 
\end{pmatrix}, \\
& \begin{pmatrix}
0 & 0 & 0 & 1 & 0 \\
0 & 0 & 1 & 0 & 0 \\ 
0 & -1 & 0 & -u_{142} + u_{131} + u_{242} & -u_{152}-u_{241} \\
-1 & 0 & * & 0 & -u_{151}  \\
0 & 0 & * & * & 0
\end{pmatrix}.
\end{align*}

We can apply Lemma \ref{lem:connected-criterion} 
to the map $\aff^{10}\to \aff^6$ defined by the sequence
\begin{align*}
u_{141}+u_{232},u_{142},u_{151},u_{152}+u_{241},
u_{132}+u_{151}-u_{231},u_{131}-u_{142}+u_{242}
\end{align*}
where $u_{231},u_{232},u_{241},u_{242}$ are extra variables.
So by Proposition \ref{prop:W-eliminate}, 
Property \ref{property:W-eliminate} holds for 
any $x\in Z^{\sst}_{14\,k}$. Therefore, 
$P_{\be_{14}\,k}\backslash Y^{\sst}_{14\,k}$ 
is in bijective correspondence with $\Ex_2(k)$ also.

\vskip 5pt

(6) $S_{15}$, $\be_{15}=\tfrac {1} {80} (-2,-2,-2,3,3,-5,0,0,5)$

We identify the element 
$(\diag(g_{11},g_{12}),\diag(t_{21},g_2,t_{22}))
\in M_{[3],[1,3]}=M_{\be_{15}}$ with the element 
$g=(g_{11},g_{12},g_2,t_{21},t_{22})\in\gl_3\times \gl_2^2\times \gl_1^2$. 
On $M^{\text{st}}_{\be_{15}}$, 
\begin{equation*}
\chi_{{15}}(g) = 
(\det g_{11})^{-2}(\det g_{12})^3 t_{21}^{-5}t_{22}^5
= (\det g_{12})^5(\det g_2)^5t_{22}^{10}. 
\end{equation*}

For $x\in Z_{15}$, let 
$A(x) = x_{124}p_{3,12}+x_{134}p_{3,13}+x_{234}p_{3,23}$, 
\begin{equation*}
B_1(x) = 
\begin{pmatrix}
x_{142} & x_{143} \\
x_{152} & x_{153} 
\end{pmatrix},\;
B_2(x) = 
\begin{pmatrix}
x_{242} & x_{243} \\
x_{252} & x_{253} 
\end{pmatrix},\;
B_3(x) = 
\begin{pmatrix}
x_{342} & x_{343} \\
x_{352} & x_{353} 
\end{pmatrix}
\end{equation*}
and 
\begin{math}
B(x)=\bbmp_{3,1} \otimes B_1(x)+\bbmp_{3,2}\otimes B_2(x)
+\bbmp_{3,3}\otimes B_3(x).
\end{math}
We identify $Z_{15}$ with 
$1\oplus \wedge^2\aff^3\oplus \aff^3\otimes \m_2$ 
by the map $Z_{15}\ni x \mapsto (x_{451},A(x),B(x))$. 
Since 
\begin{equation*}
A(gx) = t_{22} (\wedge^2 g_{11}) A(x),\;
\begin{pmatrix}
B_1(gx) \\
B_2(gx) \\
B_3(gx)
\end{pmatrix}
= g_{11}
\begin{pmatrix}
g_{12} B_1(x) {}^tg_2 \\
g_{12} B_2(x) {}^tg_2 \\
g_{12} B_3(x) {}^tg_2
\end{pmatrix},
\end{equation*}
$Z_{15}$ is the same vector space considered in   
the case (a) of Section \ref{sec:rational-orbits-wedge3-32}
except for an extra $\gl_1$-factor and 
a component of the trivial \rep{} of $M^s_{\be_{15}}$ 
and the action of $(g_{11},g_{12},g_2,1,t_{22})$ is the same 
as that of the case (a) of Section \ref{sec:rational-orbits-wedge3-32}.
If $t=(I_3,I_2,I_2,t_{21},t_{22})$ then the 
$x_{451}$-coordinate of $tx$ is $t_{21}x_{451}$ and 
\begin{math}
A(tx) = t_{22}A(x),\; B(tx)=B(x). 
\end{math}
Let $P_1(x),P_2(x)$ be the polynomials considered in 
(\ref{eq:322,3-invariants}). Note that 
$P_2(x)$ is the homogeneous degree $6$ polynomial of $B(x)$ obtained 
by Proposition \ref{prop:322-invariant} and that 
$P_1(x)$ is homogeneous of degree $2$ 
for each of $A(x),B(x)$.  So by (\ref{eq:322,3-invariants})
\begin{align*}
P_1(gx) & = (\det g_{11})^2 (\det g_{12})(\det g_2) t_{22}^2 P_1(x), \\
P_2(gx) & = (\det g_{11})^2(\det g_{12})^3(\det g_2)^3P_2(x).
\end{align*}
We put $P(x)=P_1(x)^5P_2(x)x_{451}^6$. 
Then on $M^{\text{st}}_{\be_{15}}$, 
\begin{align*}
P(gx) & = ((\det g_{11})^2 (\det g_{12})(\det g_2)t_{22}^2)^5
((\det g_{11})^2(\det g_{12})^3(\det g_2)^3) \\
& \quad \times ((\det g_{12})t_{21})^6 P(x) \\
& = (\det g_{11})^{12} (\det g_{12})^{14} 
t_{21}^6(\det g_2)^8t_{22}^{10} P(x) 
= (\det g_{12})^2(\det g_2)^2t_{22}^4 P(x). 
\end{align*}
Therefore, $P(x)$ is 
invariant under the action of $G_{\text{st},\be_{15}}$.
Let $R(15)\in Z_{15}$ be the element such that 
$(A(R(15)),B(R(15)))=R_{322,3}$ (see (\ref{eq:R(322-3)-defn}))
and that the $x_{451}$-coordinate is $1$. 
Explicitly, 
$R(15)=e_{451}-e_{142}+e_{352}+e_{153}+e_{243}+e_{234}$. 
Then $P(R(15))=1$ and so $R(15)\in Z^{\sst}_{15}$ 
(there is no restriction on $\ch(k)$). 

Suppose that $x\in Z^{\sst}_{15\,k}$. 
Then $x_{451}\not=0$. 
If $t=(I_3,I_2,I_2,x_{451}^{-1},1)$ then the 
$x_{451}$-coordinate of $tx$ is $1$.  
If $x_{451}=1$ and 
$g=(g_{11},g_{12},g_2,t_{21},t_{22})$
then $g$ does not change the $x_{451}$-coordinate
if and only if $t_{21}=(\deg g_{12})^{-1}$. 
If so, the action of $g$ on $(A(x),B(x))$ is the 
same as that of the case (a) of Section \ref{sec:rational-orbits-wedge3-32}. 
Therefore, by Proposition \ref{prop:322+3-orbit}, 
$M_{\be_{15}\,k}\backslash Z^{\sst}_{15\,k}$ 
is in bijective correspondence with $\Ex_2(k)$ 
if $\ch(k)\not=2$. 

We assume that $u_{1ij}=0$ unless $i=4,5,j=1,2,3$ and 
$u_{232}=0$. Then the four components of 
$n(u)R(15)$ are as follows: 
\begin{align*}
& \begin{pmatrix}
0 & 0 & 0 & 0 & 0 \\
0 & 0 & 0 & 0 & 0 \\
0 & 0 & 0 & 0 & 0 \\
0 & 0 & 0 & 0 & 1 \\
0 & 0 & 0 & -1 & 0 
\end{pmatrix}, \;
\begin{pmatrix}
0 & 0 & 0 & -1 & 0 \\
0 & 0 & 0 & 0 & 0 \\
0 & 0 & 0 & 0 & 1 \\
1 & 0 & 0 & 0 & u_{143}+u_{151}+u_{221} \\
0 & 0 & -1 & * & 0 
\end{pmatrix}, \\
& \begin{pmatrix}
0 & 0 & 0 & 0 & 1 \\
0 & 0 & 0 & 1 & 0 \\
0 & 0 & 0 & 0 & 0 \\
0 & -1 & 0 & 0 & u_{141}-u_{152}+u_{231} \\
-1 & 0 & 0 & * & 0 
\end{pmatrix},\;
\begin{pmatrix}
0 & 0 & 0 & -u_{242} & u_{243} \\
0 & 0 & 1 & u_{143}+u_{243} & u_{153} \\
0 & -1 & 0 & -u_{142} & -u_{152}+u_{242}  \\
* & * & * & 0 & Q(u) + u_{241}\\
* & * & * & * & 0 
\end{pmatrix}
\end{align*}
where $Q(u)$ is a polynomial which does not depend on $u_{241}$. 

We can apply Lemma \ref{lem:connected-criterion}  
to the map $\aff^{11}\to \aff^9$ defined by the sequence
\begin{align*}
& u_{142},u_{242},u_{243},u_{143}+u_{243},u_{153},u_{152}-u_{242}, \\
& u_{231}+u_{141}-u_{152},u_{221}+u_{143}+u_{151},u_{241}+Q(u)
\end{align*}
where $u_{141},u_{151}$ are extra variables.
So by Proposition \ref{prop:W-eliminate}, 
Property \ref{property:W-eliminate} holds for 
any $x\in Z^{\sst}_{15\,k}$. Therefore, 
$P_{\be_15\,k}\backslash Y^{\sst}_{15\,k}$ 
is in bijective correspondence with $\Ex_2(k)$ also 
if $\ch(k)\not=2$.

\vskip 5pt

(7) $S_{16}$, $\be_{16}=\tfrac {1} {30} (-2,-2,-2,3,3,0,0,0,0)$ 

Since $G_{\text{st},\be_{16}}$ is semi-simple, 
any relative invariant polynomial is invariant 
by $G_{\text{st},\be_{16}}$. 
We identify the element 
$(\diag(g_{11},g_{12}),g_2)
\in M_{[3],\emptyset}=M_{\be_{16}}$ with the element 
$g=(g_2,g_{11},g_{22})\in\gl_4\times \gl_3\times \gl_2$.

For $x\in Z_{16}$, let 
\begin{equation*}
A_i(x) = \begin{pmatrix}
x_{14i} & x_{15i} \\
x_{24i} & x_{25i} \\
x_{34i} & x_{35i}
\end{pmatrix}
\end{equation*}
for $i=1,2,3,4$ and $A(x)=(A_1(x)\ccd A_4(x))$. 
We identify $Z_{16}\cong \aff^4\otimes \aff^3\otimes \aff^2$ 
by the map $Z_{16}\ni x\mapsto A(x)$. 
Since 
\begin{equation*}
\begin{pmatrix}
A_1(gx) \\
\vdots \\
A_4(gx)
\end{pmatrix}
= g_2
\begin{pmatrix}
g_{11}A_1(x){}^tg_{12} \\
\vdots \\
g_{11}A_4(x){}^tg_{12} 
\end{pmatrix}, 
\end{equation*}
$Z_{16}$ is the tensor product
of standard \rep s. 

Let $W=\aff^3\otimes \aff^2$. Then $\dim W=6$.  
We denote $q_{ij}=\bbmp_{3,i}\otimes \bbmp_{2,j}$ for $i=1,2,3,j=1,2$.
Let $\{q_{ij}^*\}\sub W^*$ be the dual basis of 
$\{q_{ij}\}$. 
If $H\sub W$ is a subspace of dimension $4$ then  
$H'=\{f\in W^*\mid {}^{\forall} h\in H, f(h)=0\}$
is a subspace of $W^*$ of dimension $2$. 
By taking bases of $H,H'$, we obtain a bijective 
correspondence between 
$(\gl_3(k)\times \gl_2(k)\times \gl_4(k))\bk 
(W\otimes \aff^4)^{\sst}_k$ and 
\begin{math}
(\gl_3(k)\times \gl_2(k)\times \gl_2(k))
\bk (W^*\otimes \aff^2)^{\sst}_k  
\end{math}
(see the end of Section \ref{sec:regularity}). 

Let $R(16)\in Z_{16}$ be the element such that 
\begin{equation*}
A_1(R(16)) = q_{11}+q_{32},\;
A_2(R(16)) = q_{12}-q_{21},\;
A_3(R(16)) = q_{31},\;
A_4(R(16)) = q_{22}. 
\end{equation*}
Explicitly, 
$R(16)=e_{141}+e_{351}+e_{152}-e_{242}+e_{343}+e_{254}$. 
By the Castling transform, $R(16)$ 
corresponds to the element 
\begin{equation*}
(-q_{11}^*+q_{32}^*,q_{12}^*+q_{21}^*)
\in W^*\otimes \aff^2.
\end{equation*}
Let $\{\bbmp^*_{3,1},\bbmp^*_{3,2},\bbmp^*_{3,3}\}$ 
be the dual basis of 
$\{\bbmp_{3,1},\bbmp_{3,2},\bbmp_{3,3}\}$.
Since $(\aff^3\otimes \aff^2)^*\cong (\aff^3)^*\otimes (\aff^2)^*$, 
the above element is 
\begin{equation*}
-\bbmp^*_{3,1}\otimes \bbmp^*_{2,1}\otimes \bbmp_{2,1}
+\bbmp^*_{3,1}\otimes \bbmp^*_{2,2}\otimes \bbmp_{2,2}
+\bbmp^*_{3,2}\otimes \bbmp^*_{2,1}\otimes \bbmp_{2,2}
+\bbmp^*_{3,3}\otimes \bbmp^*_{2,2}\otimes \bbmp_{2,1}.
\end{equation*}
This element of  
$(\aff^3)^*\otimes (\aff^2)^*\otimes \aff^2$ 
corresponds to 
$R_{322}$ in Proposition \ref{prop:322-invariant} 
(see Remark \ref{rem:322-remark}). 
Therefore, $Z^{\sst}_{{16}\,k} = M_{\be_{16}\,k}R(16)\not=\emptyset$.

We assume that $u_{1ij}=0$ unless 
$i=4,5,j=1,2,3$ and $u_2=0$. 
Then the four components of 
$n(u)R(16)$ are as follows: 
\begin{align*}
& \begin{pmatrix}
0 & 0 & 0 & 1 & 0 \\
0 & 0 & 0 & 0 & 0 \\
0 & 0 & 0 & 0 & 1 \\
-1 & 0 & 0 & 0 & -u_{151}+u_{143} \\
0 & 0 & -1 & * & 0 
\end{pmatrix}, \;
\begin{pmatrix}
0 & 0 & 0 & 0 & 1 \\
0 & 0 & 0 & -1 & 0 \\
0 & 0 & 0 & 0 & 0 \\
0 & 1 & 0 & 0 & u_{152}+u_{141} \\
-1 & 0 & 0 & * & 0 
\end{pmatrix}, \\
& \begin{pmatrix}
0 & 0 & 0 & 0 & 0 \\
0 & 0 & 0 & 0 & 0 \\
0 & 0 & 0 & 1 & 0 \\
0 & 0 & -1 & 0 & -u_{153} \\
0 & 0 & 0 & * & 0 
\end{pmatrix}, \;
\begin{pmatrix}
0 & 0 & 0 & 0 & 0 \\
0 & 0 & 0 & 0 & 1 \\
0 & 0 & 0 & 0 & 0 \\
0 & 0 & 0 & 0 & u_{142} \\
0 & -1 & 0 & * & 0 
\end{pmatrix}.
\end{align*}

We can apply Lemma \ref{lem:connected-criterion} 
to the map $\aff^6\to \aff^4$ defined by the sequence
\begin{align*}
u_{142},u_{153},u_{152}+u_{141},u_{151}-u_{143}
\end{align*}
where $u_{141},u_{143}$ are extra variables.
So by Proposition \ref{prop:W-eliminate}, 
Property \ref{property:W-eliminate} holds for 
any $x\in Z^{\sst}_{16\,k}$. Therefore, 
$Y^{\sst}_{16\,k}=P_{\be_{16}\,k}R(16)$.

\vskip 5pt

(8)  $S_{20}$, $\be_{20}=\tfrac {1} {740} (-16,-16,4,4,24,-25,-5,15,15)$ 

We identify the element 
$(\diag(g_{11},g_{12},t_1),\diag(t_{21},t_{22},g_2))
\in M_{[2,4],[1,2]}=M_{\be_{20}}$ 
with the element 
$g=(g_{12},g_{11},g_2,t_1,t_{21},t_{22})\in \gl_2^3\times \gl_1^3$. 
On $M^{\text{st}}_{\be_{20}}$, 
\begin{equation*}
\chi_{{20}}(g) = 
(\det g_{11})^{-16}(\det g_{12})^4t_1^{24}
t_{21}^{-25}t_{22}^{-5}(\det g_2)^{15}
= (\det g_{12})^{20}t_1^{40}t_{22}^{20}(\det g_2)^{40}.
\end{equation*}

For $x\in Z_{20}$, let 
\begin{align*}
& A_1(x) = 
\begin{pmatrix}
x_{133} & x_{233} \\
x_{143} & x_{243} 
\end{pmatrix}, \;
A_2(x) = 
\begin{pmatrix}
x_{134} & x_{234} \\
x_{144} & x_{244} 
\end{pmatrix}, \\
& A(x) = (A_1(x),A_2(x))\in 
\Lam^{2,1}_{1,[3,4]}\otimes \Lam^{2,1}_{1,[1,2]}
\otimes \Lam^{2,1}_{2,[3,4]}, \\
& v_1(x) = 
\begin{pmatrix}
x_{351} \\ x_{451}
\end{pmatrix} 
\in \Lam^{2,1}_{1,[3,4]},\;
v_2(x) = 
\begin{pmatrix}
x_{152} \\ x_{252}
\end{pmatrix}
\in \Lam^{2,1}_{1,[1,2]}.  
\end{align*}
We identify $Z_{20}$ with 
$\aff^2\otimes \aff^2\otimes \aff^2\oplus \aff^2\oplus \aff^2\oplus 1$ 
by the map $Z_{20}\ni \mapsto (A(x),v_1(x),v_2(x),x_{342})$. 
Since 
\begin{equation*}
v_1(gx) = t_1t_{21}g_{12}v_1(x),\;
v_2(gx) = t_1t_{22}g_{11}v_2(x),\;
\begin{pmatrix}
A_1(gx) \\
A_2(gx)
\end{pmatrix}
= g_2\begin{pmatrix}
g_{12}A_1(x){}^tg_{11} \\
g_{12}A_2(x){}^tg_{11}
\end{pmatrix}, 
\end{equation*}
the action of $(g_{12},g_{11},g_2,1,1,1)$
on $(A(x),v_1(x),v_2(x))$ is the same 
as that of the case (b) of Section \ref{sec:rational-orbits-222-22}. 

Let $P_1(x)$ be the degree $4$ polynomial of $A(x)$ obtained by 
Proposition \ref{prop:222-invariant} and 
$P_2(x),P_3(x)$ the polynomials of $(A(x),v_1(x))$, $(A(x),v_2(x))$ 
respectively for the case (b) of (\ref{eq:P1P2P3}). 
$P_1(x)$ is homogeneous of degree $2$ 
for each of $A_1(x),A_2(x)$ and 
$P_2(x)$ (resp. $P_3(x)$) is homogeneous of degree $2$ 
for each of $A(x),v_1(x)$ (resp. $A(x),v_2(x)$). 
So by (\ref{eq:P1P2P3}), 
\begin{align*}
P_1(gx) & = (\det g_{11})^2(\det g_{12})^2(\det g_2)^2P_1(x), \\
P_2(gx) & = (t_1t_{21})^2(\det g_{11})(\det g_{12})^2(\det g_2)P_2(x), \\
P_3(gx) & = (t_1t_{22})^2(\det g_{11})^2(\det g_{12})(\det g_2)P_3(x).
\end{align*}
Also the $x_{342}$-coordinate of $gx$ is 
$t_{22}(\det g_{12}) x_{342}$. 

Let 
\begin{math}
P(x) = P_1(x)P_2(x)^4P_3(x)^4x_{342}. 
\end{math}
Then on $M^{\text{st}}_{\be_{20}}$, 
\begin{align*}
P(gx) 
& = ((\det g_{11})^2(\det g_{12})^2(\det g_2)^2)
((t_1t_{21})^2(\det g_{11})(\det g_{12})^2(\det g_2))^4 \\
& \quad \times ((t_1t_{22})^2(\det g_{11})^2(\det g_{12})(\det g_2))^4
(t_{22}(\det g_{12}))P(x) \\
& = (\det g_{11})^{14}(\det g_{12})^{15}t_1^{16}
t_{21}^8t_{22}^9 (\det g_2)^{10} P(x)
= (\det g_{12})t_1^2 t_{22}(\det g_2)^2P(x). 
\end{align*}
Therefore, $P(x)$ is 
invariant under the action of $G_{\text{st},\be_{20}}$. 

Suppose that $x\in Z^{\sst}_{20\,k}$.
Then $x_{342}\not=0$.  
If $t=(I_2,I_2,I_3,x_{342},x_{342}^{-1},x_{342}^{-1})$ then 
$A(tx)=A(x),v_1(tx)=v_1(x),v_2(tx)=v_2(x)$ and 
the $x_{342}$-coordinate of $tx$ is $1$.
If $g=(g_{12},g_{11},g_2,t_1,t_{21},t_{22})$ then 
$g$ does not change the $x_{342}$-coordinate if and only if
$t_{22}=(\det g_{12})^{-1}$. 

Let 
\begin{align*}
h_1 & = (g_{12},g_{11},g_2,t_1,t_{21},(\det g_{12})^{-1}), \\
h_2 & = (t_1t_{21}g_{12},t_1(\det g_{12})^{-1}g_{11},
t_1^{-2}t_{21}^{-1}(\det g_{12})g_2). 
\end{align*}
Then the action of $h_1$ on 
$(A(x),v_1(x),v_2(x))$ 
coincides with the action of $h_2$ in 
the case (b) of Section \ref{sec:rational-orbits-222-22} 
on $(A(x),v_1(x),v_2(x))$.  
Therefore, the action of $\gl_1$'s can be absorbed 
into the action of $\gl_2^3$. 
Since $Z^{\sst}_{20\,k}$ corresponds to 
$U$ in Proposition \ref{prop:222+2-orbit}, 
$M_{\be_{20}\,k}\backslash Z^{\sst}_{20\,k}$ 
is in bijective correspondence with $\Ex_2(k)$. 

Let $R(20)\in Z_{20}$ be the element such that 
\begin{equation*}
A_1(R(20)) = 
\begin{pmatrix}
1 & 0 \\
0 & 0 
\end{pmatrix},\;
A_2(R(20)) = 
\begin{pmatrix}
0 & 0 \\
0 & 1 
\end{pmatrix},\;
v_1(R(20)) = v_2(R(20)) = 
\begin{pmatrix}
1 \\ 1
\end{pmatrix}
\end{equation*}
and that the $x_{342}$-coordinate is $1$. 
Explicitly, 
$R(20)=e_{351}+e_{451}+e_{152}+e_{252}+e_{133}+e_{244}$. 
Then $P_1(R(20))=1,P_i(R_{20})=-1$ for $i=2,3$. 
So $R(20)\in Z^{\sst}_{20}$. 

We assume that $u_{1ij}=0$ for $(i,j)=(2,1),(4,3)$ 
and $u_{243}=0$.  Then the four components of 
$n(u)R(20)$ are as follows: 
\begin{align*}
& \begin{pmatrix}
0 & 0 & 0 & 0 & 0 \\
0 & 0 & 0 & 0 & 0 \\
0 & 0 & 0 & 0 & 1 \\
0 & 0 & 0 & 0 & 1 \\
0 & 0 & -1 & -1 & 0 
\end{pmatrix},\;
\begin{pmatrix}
0 & 0 & 0 & 0 & 1 \\
0 & 0 & 0 & 0 & 1 \\
0 & 0 & 0 & 1 & u_{131} + u_{132} + u_{154} + u_{221}  \\
0 & 0 & -1 & 0 & u_{141} + u_{142} -u_{153} + u_{221} \\
-1 & -1 & * & * & 0 
\end{pmatrix}, \\
& \begin{pmatrix}
0 & 0 & 1 & 0 & u_{153}+u_{232} \\
0 & 0 & 0 & 0 & u_{232} \\
-1 & 0 & 0 & -u_{141}+u_{232} & Q_1(u) -u_{151}+u_{231} \\
0 & 0 & * & 0 & Q_2(u) + u_{231} \\
* & * & * & * & 0 
\end{pmatrix}, \\
& \begin{pmatrix}
0 & 0 & 0 & 0 & u_{242} \\
0 & 0 & 0 & 1 & u_{154} + u_{242} \\
0 & 0 & 0 & u_{132}+u_{242} & Q_3(u)+u_{241} \\
0 & -1 & * & 0 & Q_4(u) -u_{152} +u_{241} \\
* & * & * & * & 0 
\end{pmatrix}
\end{align*}
where $Q_1(u),Q_2(u),Q_3(u),Q_4(u)$ are polynomials 
which do not depend on $u_{151},u_{152}$, 
$u_{231},u_{241}$.

We can apply Lemma \ref{lem:connected-criterion} 
to the map $\aff^{13}\to \aff^{12}$ defined by the sequence
\begin{align*}
& u_{232},u_{242},u_{153}+u_{232},u_{154}+u_{242}, 
u_{132}+u_{242},u_{141}-u_{232}, \\
& u_{131}+u_{132}+u_{154}+u_{221}, 
u_{142}+u_{141}-u_{153}+u_{221}, \\
& u_{231}+Q_2(u),u_{151}-u_{231}-Q_1(u),  
u_{241}+Q_3(u),u_{152}-u_{241}-Q_4(u)
\end{align*}
where $u_{221}$ is an extra variable.
So by Proposition \ref{prop:W-eliminate}, 
Property \ref{property:W-eliminate} holds for 
any $x\in Z^{\sst}_{20\,k}$. Therefore, 
$P_{\be_{20}\,k}\backslash Y^{\sst}_{20\,k}$ 
is in bijective correspondence with $\Ex_2(k)$ also.

\vskip 5pt

(9) $S_{21}$, $\be_{21}=\tfrac {1} {60} (-4,-4,1,1,6,-5,0,0,5)$

We identify the element 
$(\diag(g_{11},g_{12},t_1),\diag(t_{21},g_2,t_{22}))
\in M_{[2,4],[1,3]}=M_{\be_{21}}$ 
with the element 
$g=(g_{11},g_{12},g_2,t_1,t_{21},t_{22})\in \gl_2^3\times \gl_1^3$. 
On $M^{\text{st}}_{\be_{21}}$, 
\begin{equation*}
\chi_{{21}}(g) = 
(\det g_{11})^{-4}(\det g_{12})
t_1^6t_{21}^{-5}t_{22}^5
= (\det g_{12})^5t_1^{10}(\det g_2)^5t_{22}^{10}. 
\end{equation*}

For $x\in Z_{21}$, let 
\begin{align*}
& A(x) = \begin{pmatrix}
x_{152} & x_{153} \\
x_{252} & x_{253}
\end{pmatrix}
\in \Lam^{2,1}_{1,[1,2]}\otimes \Lam^{2,1}_{2,[2,3]},\;
v_1(x) = 
\begin{pmatrix}
x_{342} \\ x_{343}
\end{pmatrix}
\in \Lam^{2,1}_{2,[2,3]}, \\
& B(x) = \begin{pmatrix}
x_{134} & x_{144} \\
x_{234} & x_{244}
\end{pmatrix}
\in \Lam^{2,1}_{1,[1,2]}\otimes \Lam^{2,1}_{1,[3,4]},\;
v_2(x) = 
\begin{pmatrix}
x_{351} \\ x_{451} 
\end{pmatrix}
\in \Lam^{2,1}_{1,[3,4]}.
\end{align*}
We identify $Z_{21}$ with 
$\m_2\oplus \aff^2\oplus \m_2\oplus \aff^2$ 
by the map $Z_{21}\ni x\mapsto (A(x),v_1(x),B(x),v_2(x))$.

It is easy to see that 
\begin{align*}
A(gx) & =t_1g_{11}A(x){}^tg_2,\;
B(gx)=t_{22}g_{11}B(x){}^tg_{12}, \\
v_1(gx) & = (\det g_{12})g_2v_1(x),\;
v_2(gx) = t_1t_{21}g_{12}v_2(x).
\end{align*}
By applying Lemma \ref{lem:natural-pairing} 
to $(A(x),v_1(x))$, $(B(x),v_2(x))$, 
we have maps 
\begin{align*}
& \Phi_1: Z_{21}\to \Lam^{2,1}_{1,[1,2]}\otimes \Lam^{2,1}_{2,[2,3]}
\oplus \Lam^{2,1}_{2,[2,3]}
\to \Lam^{2,1}_{1,[1,2]}\cong \aff^2, \\
& \Phi_2:
Z_{21}\to \Lam^{2,1}_{1,[1,2]}\otimes \Lam^{2,1}_{1,[3,4]}
\oplus \Lam^{2,1}_{1,[3,4]}
\to \Lam^{2,1}_{1,[1,2]}\cong \aff^2.
\end{align*}

Let $P_1(x)$ be the determinant of 
$(\Phi_1(x)\; \Phi_2(x))\in \m_2$. Then 
\begin{align*}
& \Phi_1(gx) = (\det g_{12}) t_1 (\det g_2) g_{11}\Phi_1(x),\;
\Phi_2(gx) = (\det g_{12})t_1t_{21}t_{22} g_{11}\Phi_2(x), \\
& P_1(gx) = (\det g_{11})(\det g_{12})^2t_1^2
t_{21}(\det g_2)t_{22}P_1(x). 
\end{align*}
Let $P_2(x) = \det A(x),P_3(x)=\det B(x)$.
Then 
\begin{equation*}
P_2(gx) = (\det g_{11})t_1^2(\det g_2)P_2(x), \;
P_3(gx) = (\det g_{11})(\det g_{12})t_{22}^2P_3(x).
\end{equation*}

We put 
$P(x) = P_1(x)^2P_2(x)P_3(x)$. 
Then on $M^{\text{st}}_{\be_{21}}$, 
\begin{align*}
P(gx) 
& = ((\det g_{11})(\det g_{12})^2t_1^2t_{21}(\det g_2)t_{22})^2 \\
& \quad \times ((\det g_{11})t_1^2(\det g_2))
((\det g_{11})(\det g_{12})t_{22}^2)P(x) \\
& = (\det g_{11})^4(\det g_{12})^5t_1^6t_{21}^2(\det g_2)^3t_{22}^4 P(x)
= (\det g_{12})t_1^2(\det g_2)t_{22}^2 P(x). 
\end{align*}
Therefore, $P(x)$ is 
invariant under the action of $G_{\text{st},\be_{21}}$. 

Let $R(21)\in Z_{21}$ be the element such that 
\begin{equation*}
A(R(21)) = B(R(21)) = I_2,\;
v_1(R(21))=[1,0],\; v_2(R(21))=[0,1]. 
\end{equation*}
Explicitly, $R(21)=e_{451}+e_{152}+e_{342}+e_{253}+e_{134}+e_{244}$. 
Then $\Phi_1(R(21))=[0,-1]$, $\Phi_2(R(21))=[1,0]$. 
So $P_i(R(21))=1$ for $i=1,2,3$ 
and so $R(21)\in Z^{\sst}_{21\,k}$. 

We show that  $Z^{\sst}_{21\,k}=M_{\be_{21}\,k}R(21)$. 
Suppose that $x\in Z^{\sst}_{21\,k}$. 
Since $\det A(x)$, $\det B(x)\not=0$, 
there exists $g\in M_{\be_{21}\,k}$ 
such that $A(gx)=B(gx)=I_2$. 
So we may assume that $A(x)=B(x)=I_2$. 
Then $\Phi_1(x)=[x_{343},-x_{342}]$, 
$\Phi_2(x)=[x_{451},-x_{351}]$. 
Since $P_1(x)\not=0$, either $x_{342}\not=0$ or $x_{343}\not=0$.  
Let $\tau_0$ be the element in (\ref{eq:J-defn}). 
By applying $(\tau_0,\tau_0,\tau_0,1,1,1)$ if necessary, 
we may assume that $x_{342}\not=0$. 

Let $g=(I_2,I_2,x_{342}^{-1}I_2,x_{342},1,1)$. 
Then $A(gx)=B(gx)=I_2$ and the $x_{342}$-coordinate 
of $gx$ is $1$. So we may assume that $x_{342}=1$. 
Let $h=n_2(-x_{343})$ (see Section \ref{sec:notation})
and $g=({}^th^{-1},h,h,1,1,1)$. 
Then $A(gx)=B(gx)=I_2,v_1(x)=[1,0]$. Since $P_1(x)\not=0$, 
$x_{451}\not=0$. 
By applying the element $(I_2,I_2,I_2,1,x_{451}^{-1},1)$ to $x$, 
we may assume that $x_{451}=1$. Let 
$h={}^tn_2(-x_{351})$ and $g=({}^th^{-1},h,h,1,1,1)$.
Then $gx =R(21)$. 
Therefore, $Z^{\sst}_{21\,k}=M_{\be_{21}\,k}R(21)$. 

We assume that $u_{1ij}=0$ for $(i,j)=(2,1),(4,3)$ 
and $u_{232}=0$.  
Then the fist component of 
$n(u)R(21)$ is the same as that of $R(21)$ 
and the remaining components are as follows: 
\begin{align*}
\begin{pmatrix}
0 & 0 & 0 & 0 & 1 \\
0 & 0 & 0 & 0 & 0 \\
0 & 0 & 0 & 1 & u_{131}+u_{154}  \\
0 & 0 & -1 & 0 & u_{141} -u_{153} +u_{221} \\
-1 & 0 & * & * & 0 
\end{pmatrix}, \;
\begin{pmatrix}
0 & 0 & 0 & 0 & 0 \\
0 & 0 & 0 & 0 & 1 \\
0 & 0 & 0 & 0 & u_{132} \\
0 & 0 & 0 & 0 & u_{142}+u_{231} \\
0 & -1 & * & * & 0 
\end{pmatrix}, 
\end{align*}

\begin{align*}
& \begin{pmatrix}
0 & 0 & 1 & 0 & u_{153}+u_{242} \\
0 & 0 & 0 & 1 & u_{154}+u_{243} \\
-1 & 0 & 0 & u_{132}-u_{141}+u_{242} & -u_{151} + Q_1(u) \\
0 & -1 & * & 0 & -u_{152} + u_{241} + Q_2(u) \\
* & * & * & * & 0 
\end{pmatrix}
\end{align*}
where $Q_1(u),Q_2(u)$ are polynomials 
which do not depend on $u_{151},u_{152},u_{241}$. 

We can apply Lemma \ref{lem:connected-criterion} 
to the map $\aff^{13}\to \aff^9$ defined by the sequence
\begin{align*}
& u_{132},u_{153}+u_{242},u_{154}+u_{243}, 
u_{131}+u_{154},u_{141}-u_{132}-u_{242}, \\
& u_{221}+u_{141}-u_{153},u_{142}+u_{231},
u_{151}-Q_1(u),u_{241}-u_{152}+Q_2(u) 
\end{align*}
where $u_{152},u_{231},u_{242},u_{243}$ are extra variables.
So by Proposition \ref{prop:W-eliminate}, 
Property \ref{property:W-eliminate} holds for 
any $x\in Z^{\sst}_{21\,k}$. Therefore, 
$Y^{\sst}_{21\,k} = P_{\be_{21}\,k} R(21)$ also.  

\vskip 5pt

(10) $S_{33}$, $\be_{33}=\tfrac {1} {340} (-16,-6,4,4,14,-15,-5,5,15)$. 

We identify the element 
$(\diag(t_{11},t_{12},g_1,t_{13}),\diag(t_{21},t_{22},t_{23},t_{24}))
\in M_{[1,2,4],[1,2,3]}=M_{\be_{33}}$ 
with the element 
$g=(g_1,t_{11}\ccd t_{24})\in \gl_2\times \gl_1^7$. 
On $M^{\text{st}}_{\be_{35}}$, 
\begin{equation*}
\chi_{{33}}(g) = 
t_{11}^{-16}t_{12}^{-6}(\det g_1)^4t_{13}^{14}
t_{21}^{-15}t_{22}^{-5}t_{23}^5t_{24}^{15}
= t_{12}^{10}(\det g_1)^{20}t_{13}^{30}t_{22}^{10}t_{23}^{20}t_{24}^{30}. 
\end{equation*}

For $x\in Z_{33}$, let 
\begin{equation*}
v_1(x)=[x_{351},x_{451}],
v_2(x)=[x_{233},x_{243}],
v_3(x)=[x_{134},x_{144}].
\end{equation*}
We identify $Z_{33}\cong (\aff^2)^{3\oplus}\oplus 1^{3\oplus}$
by the map $Z_{33}\ni x\mapsto (v_1(x),v_2(x),v_3(x),x_{252},x_{342}$, $x_{153})$.

We put 
\begin{equation*}
P_1(x)=\det (v_1(x) \; v_2(x)),\;
P_2(x)=\det(v_1(x)\; v_3(x)),\;
P_3(x)=\det(v_2(x)\; v_3(x))
\end{equation*}
and 
\begin{math}
P(x)=P_1(x)^2P_2(x)^5P_3(x)^5x_{252}^6x_{342}^2x_{153}^2.
\end{math}
Then on $M^{\text{st}}_{\be_{33}}$, 
\begin{align*}
v_1(gx) & = t_{13}t_{21}g_1v_1(x),\;
v_2(gx) = t_{12}t_{23}g_1v_2(x),\;
v_3(gx) = t_{11}t_{24}g_1v_3(x), \\
P_1(gx) & = t_{12}(\det g_1)t_{13}t_{21}t_{23} P_1(x), \;
P_2(gx)  = t_{11}(\det g_1)t_{13}t_{21}t_{24} P_2(x), \\
P_3(gx) & = t_{11}t_{12}(\det g_1)t_{23}t_{24} P_3(x), \\
P(gx) 
& = (t_{12}(\det g_1)t_{13}t_{21}t_{23})^2
(t_{11}(\det g_1)t_{13}t_{21}t_{24})^5
(t_{11}t_{12}(\det g_1)t_{23}t_{24})^5 \\
& \quad \times (t_{12}t_{13}t_{22})^6
((\det g_1)t_{22})^2
(t_{11}t_{13}t_{23})^2 P(x) \\
& = t_{11}^{12}t_{12}^{13}(\det g_1)^{14}t_{13}^{15}
t_{21}^7t_{22}^8t_{23}^9t_{24}^{10} P(x) 
= t_{12}(\det g_1)^2t_{13}^3t_{22}t_{23}^2t_{24}^3 P(x). 
\end{align*}
Therefore, $P(x)$ is 
invariant under the action of $G_{\text{st},\be_{33}}$. 

Let $R(33)\in Z_{33}$ be the element such that 
\begin{equation*}
v_1(R(33))=[1,0],v_2(R(33))=[0,1],v_3(R(33))=[1,1]
\end{equation*}
and that the $x_{252},x_{342},x_{153}$-coordinates are $1$. 
Explicitly, $R(33)=e_{351}+e_{252}+e_{342}+e_{153}+e_{243}+e_{134}+e_{144}$. 
Then $P_1(R(33))=P_2(R(33))=1$,  $P_3(R(33))=-1$   
and so $R(33)\in Z^{\sst}_{33\,k}$.

We show that  $Z^{\sst}_{33\,k}=M_{\be_{33}\,k}R(33)$. 
Suppose that $x\in Z^{\sst}_{33\,k}$. 
It is easy to see that 
there exists $g\in M_{\be_{33}\,k}$ 
such that $v_1(gx)=[1,0],v_2(gx)=[0,1]$. 
So we may assume that $v_1(x)=[1,0],v_2(x)=[0,1]$. 
By assumption, $x_{134},x_{144},x_{252},x_{342},x_{153}\not=0$. 
If 
\begin{equation*}
t=(t_{24}^{-1},x_{144},
\diag(x_{134}^{-1},x_{144}^{-1}),t_{21}^{-1}x_{134},
t_{21},t_{22},1,t_{24})
\end{equation*}
then 
\begin{math}
v_1(tx) = [1,0],\;
v_2(tx) = [0,1],\;
v_3(tx) = [1,1]  
\end{math}
and the $x_{252},x_{342},x_{153}$-coordinates of $tx$ 
are 
\begin{equation*}
t_{21}^{-1}t_{22}x_{134}x_{144}x_{252},\; 
t_{22}x_{134}^{-1}x_{144}^{-1}x_{342},\;
t_{21}^{-1}t_{24}^{-1}x_{134}x_{153}.
\end{equation*}
We can choose $t_{22},t_{21},t_{24}$ 
in this order so that $x_{342},x_{252},x_{153}$-coordinates 
of $tx$ are $1$. So, there exists such $t$ 
such that $tx=R(33)$. 
Therefore, $Z^{\sst}_{33\,k}=M_{\be_{33}\,k}R(33)$. 

We assume that $u_{143}=0$. 
Then the four components of 
$n(u)R(33)$ are as follows: 
\begin{align*}
& \begin{pmatrix}
0 & 0 & 0 & 0 & 0 \\
0 & 0 & 0 & 0 & 0 \\
0 & 0 & 0 & 0 & 1 \\
0 & 0 & 0 & 0 & 0 \\
0 & 0 & -1 & 0 & 0 
\end{pmatrix},\;
\begin{pmatrix}
0 & 0 & 0 & 0 & 0 \\
0 & 0 & 0 & 0 & 1 \\
0 & 0 & 0 & 1 & u_{132} + u_{154} + u_{221}  \\
0 & 0 & -1 & 0 & u_{142} -u_{153} \\
0 & -1 & * & * & 0 
\end{pmatrix}, \\
& \begin{pmatrix}
0 & 0 & 0 & 0 & 1 \\
0 & 0 & 0 & 1 & u_{121}+u_{154}+u_{232} \\
0 & 0 & 0 & u_{132}+u_{232} & Q_1(u) + u_{131}+u_{231} \\
0 & -1 & * & 0 & Q_2(u) + u_{141}-u_{152} \\
-1 & * & * & * & 0 
\end{pmatrix}, \\
& \begin{pmatrix}
0 & 0 & 1 & 1 & u_{153}+u_{154}+u_{243} \\
0 & 0 & u_{121} & u_{121}+u_{243} & Q_3(u)+u_{242} \\
-1 & * & 0 & Q_4(u)+u_{131}-u_{141}+u_{242} & Q_5(u)-u_{151}+u_{241} \\
-1 & * & * & 0 & Q_6(u)-u_{151} \\
* & * & * & * & 0 
\end{pmatrix}
\end{align*}
where $Q_1(u)\ccd Q_4(u)$ are polynomials  
which do not depend on 
$u_{131},u_{141},u_{151},u_{152}$, $u_{231},u_{241},u_{242}$
and $Q_5(u),Q_6(u)$ are polynomials which do not 
depend on $u_{151},u_{241}$. 

We can apply Lemma \ref{lem:connected-criterion} 
to the map $\aff^{15}\to \aff^{13}$ defined by the sequence
\begin{align*}
& u_{121},u_{243}+u_{121},
u_{154}+u_{153}+u_{243},u_{142}-u_{153}, 
u_{232}+u_{121}+u_{154}, \\
& u_{132}+u_{232},u_{221}+u_{132}+u_{154},
u_{242}+Q_3(u),u_{131}-u_{141}+u_{242}+Q_4(u), \\
& u_{231}+u_{131}+Q_1(u),u_{152}-u_{141}-Q_2(u),
u_{151}-Q_6(u),u_{241}-u_{151}+Q_5(u)
\end{align*}
where $u_{141},u_{153}$ are extra variables.
So by Proposition \ref{prop:W-eliminate}, 
Property \ref{property:W-eliminate} holds for 
any $x\in Z^{\sst}_{33\,k}$. Therefore, 
$Y^{\sst}_{33\,k} = P_{\be_{33}\,k} R(33)$ also.  

\vskip 5pt

(11)  $S_{35}$, $\be_{35}=\tfrac {1} {10} (-4,1,1,1,1,0,0,0,0)$ 

Since $G_{\text{st},\be_{35}}$ is semi-simple, 
any relative invariant polynomial is invariant 
with respect to $G_{\text{st},\be_{35}}$.

Let $W=\aff^4$. As in the case $S_{16}$, by the Castling transform, 
there is a bijective correspondence between 
$(\gl_4(k)\times \gl_4(k))\bk (\wedge^2 W\otimes \aff^4)^{\sst}_k$
and $(\gl_4(k)\times \gl_2(k))\bk (\wedge^2 W^*\otimes \aff^2)^{\sst}_k$.
Since  the latter is in bijective
correspondence with $\Ex_2(k)$ by Proposition \ref{prop:Rational-orbits-3-cases}, 
so is $M_{\be_{35}\,k}\bk Z^{\sst}_{{35}\, k}$. 
Since $W_{35}=\{0\}$, 
$P_{\be_{35}\,k}\bk Y^{\sst}_{{35}\, k}$ 
is in bijective correspondence with $\Ex_2(k)$ also.  

\vskip 5pt

(12) $S_{36}$,  $\be_{36}=\tfrac {1} {40} (-6,-1,-1,4,4,-5,0,0,5)$ 

We identify the element 
$(\diag(t_1,g_{11},g_{12}),\diag(t_{21},g_2,t_{22}))
\in M_{[1,3],[1,3]}=M_{\be_{36}}$ 
with the element 
$g=(g_{12},g_{11},g_2,t_1,t_{21},t_{22})\in \gl_2^3\times \gl_1^3$. 
On $M^{\text{st}}_{\be_{36}}$, 
\begin{equation*}
\chi_{{36}}(g) = 
t_1^{-6}(\det g_{11})^{-1}(\det g_{12})^4t_{21}^{-5}t_{22}^5
= (\det g_{11})^5(\det g_{12})^{10} (\det g_2)^5t_{22}^{10}. 
\end{equation*}

For $x\in Z_{36}$, let 
\begin{equation*}
A_1(x) = 
\begin{pmatrix}
x_{242} & x_{252} \\ 
x_{342} & x_{352}
\end{pmatrix},\;
A_2(x) = 
\begin{pmatrix}
x_{243} & x_{253} \\ 
x_{343} & x_{353}
\end{pmatrix},\;
v(x)=[x_{144},x_{154}]\in \Lam^{2,1}_{1,[4,5]},\; 
\end{equation*}
and 
\begin{math}
A(x)= (A_1(x),A_2(x)) \in \Lam^{2,1}_{1,[2,3]}
\otimes \Lam^{2,1}_{1,[4,5]}\otimes \Lam^{2,1}_{2,[2,3]}.
\end{math}
We identify $Z_{36}$ with 
$\m_2\otimes \aff^2\oplus \aff^2\oplus 1^{2\oplus}$
by the map $Z_{36}\ni x\mapsto (A_1(x),A_2(x),v(x),x_{451},x_{234})$. 

It is easy to see that 
\begin{equation*}
\begin{pmatrix}
A_1(gx) \\
A_2(gx)
\end{pmatrix}
= g_2
\begin{pmatrix}
g_{11}A_1(x){}^tg_{12} \\
g_{11}A_2(x){}^tg_{12}  
\end{pmatrix},\; 
v(gx) = t_1t_{22}g_{12}v(x).
\end{equation*}

Let $P_1(x),P_2(x)$ be the polynomials of $(A(x),v(x))$ 
which correspond to those in (\ref{eq:P1P2P3}). 
Then by (\ref{eq:P1P2P3}), 
\begin{align*}
P_1(gx) & = (\det g_{11})^2(\det g_{12})^2(\det g_2)^2 P_2(x),  \\
P_2(gx) & = (t_1 \det(g_{12})t_{22})^2 (\det g_{11})(\det g_2)P_1(x).
\end{align*}
We put $P(x) = P_1(x) P_2(x)^2 x_{451}^2 x_{234}^2$. 
Then on $M^{\text{st}}_{\be_{36}}$, 
\begin{align*}
P(gx) & = 
((\det g_{11})^2(\det g_{12})^2(\det g_2)^2)
((t_1 \det(g_{12})t_{22})^2 (\det g_{11})(\det g_2))^2  \\
& \quad \times ((\det g_{12})t_{21})^2 ((\det g_{11})t_{22})^2 P(x) \\
& = t_1^4(\det g_{11})^6 (\det g_{12})^8 t_{21}^2
(\det g_2)^4 t_{22}^6 P(x) \\
& = (\det g_{11})^2 (\det g_{12})^4 (\det g_2)^2t_{22}^4 P(x). 
\end{align*}
Therefore, $P(x)$ is 
invariant under the action of $G_{\text{st},\be_{36}}$.

Suppose that $x\in Z^{\sst}_{36\,k}$. 
Then $x_{451},x_{234}\not=0$. 
If $t=(I_2,I_2,I_2,1,x_{451}^{-1},x_{234}^{-1})$ 
then the $x_{451},x_{234}$-coordinates of 
$tx$ are $1$.   
Then 
\begin{math}
g=(g_{12},g_{11},g_2,t_1,t_{21},t_{22})
\end{math}
does not change the $x_{451},x_{234}$-coordinates of $x$
if and only if 
\begin{math}
t_1=(\det g_{11}),
t_{21}=(\det g_{12})^{-1},t_{22}=(\det g_{11})^{-1}
\end{math}
and its action $g$ on $(A(x),v(x))$ is the same as 
the case (a) of Section \ref{sec:rational-orbits-222-22}. 
Proposition \ref{prop:222+2-orbit} 
implies that $M_{\be_{36}\,k}\backslash Z^{\sst}_{36\,k}$ 
is in bijective correspondence with $\Ex_2(k)$. 

Let $R(36)\in Z_{36}$ be the element such that 
\begin{equation*}
A_1(R(36)) = 
\begin{pmatrix}
1 & 0 \\
0 & 0 
\end{pmatrix},\;
A_2(R(36)) = 
\begin{pmatrix}
0 & 0 \\
0 & 1 
\end{pmatrix},\;
v(R(36)) = 
\begin{pmatrix}
-1 \\ 1 
\end{pmatrix}
\end{equation*}
and that the $x_{451},x_{234}$-coordinates are $1$. 
Explicitly, 
$R(36)=e_{451}+e_{242}+e_{353}-e_{144}+e_{154}+e_{234}$. 
Then $P_1(R(36))=P_2(R(36))=1$ and so $R(36)\in Z^{\sst}_{36}$. 

We assume that $u_{132}=u_{154}=0$ 
and $u_{232}=0$. 
Then the first component of 
$n(u)R(36)$ is the same as that of $R(36)$ 
and the remaining components are as follows: 
\begin{align*}
& 
\begin{pmatrix}
0 & 0 & 0 & 0 & 0 \\
0 & 0 & 0 & 1 & 0 \\
0 & 0 & 0 & 0 & 0  \\
0 & -1 & 0 & 0 & -u_{152} + u_{221}  \\
0 & 0 & 0 & * & 0 
\end{pmatrix}, \;
\begin{pmatrix}
0 & 0 & 0 & 0 & 0 \\
0 & 0 & 0 & 0 & 0 \\
0 & 0 & 0 & 0 & 1 \\
0 & 0 & 0 & 0 & u_{143} + u_{231} \\
0 & 0 & -1 & * & 0 
\end{pmatrix}, \\
& \begin{pmatrix}
0 & 0 & 0 & -1 & 1 \\
0 & 0 & 1 & -u_{121} + u_{143} +  u_{242} & u_{121}+u_{153} \\
0 & -1 & 0 & -u_{131} -u_{142} & u_{131} -u_{152}  + u_{243} \\
1 & * & * & 0 & Q(u) + u_{141} + u_{151} + u_{241} \\
-1 & * & * & * & 0 
\end{pmatrix}
\end{align*}
where $Q(u)$ is a polynomial  
which does not depend on 
$u_{141},u_{151},u_{241}$.

We can apply Lemma \ref{lem:connected-criterion} 
to the map $\aff^{13}\to \aff^7$ defined by the sequence
\begin{align*}
& u_{221}-u_{152},u_{231}+u_{143},u_{242}-u_{121}+u_{143},
u_{153}+u_{121}, \\
& u_{142}+u_{131},u_{243}+u_{131}-u_{152},
u_{151}+u_{141}+u_{241}+Q(u)
\end{align*}
where $u_{121},u_{131},u_{141},u_{143},u_{152},u_{241}$ 
are extra variables. 
So by Proposition \ref{prop:W-eliminate}, 
Property \ref{property:W-eliminate} holds for 
any $x\in Z^{\sst}_{36\,k}$. Therefore, 
$P_{\be_{36}\,k}\backslash Y^{\sst}_{36\,k}$ 
is in bijective correspondence with $\Ex_2(k)$ 
also.  

\vskip 5pt

(13)  $S_{37}$, $\be_{37}=\tfrac {1} {380} (-12,-2,-2,8,8,-5,-5,5,5)$ 

We identify the element 
$(\diag(t_1,g_{11},g_{12}),\diag(g_{21},g_{22}))
\in M_{[1,3],[2]}=M_{\be_{37}}$ 
with the element 
$g=(g_{11}\ccd g_{22},t_1)\in \gl_2^4\times \gl_1$. 
On $M^{\text{st}}_{\be_{37}}$, 
\begin{align*}
\chi_{{37}}(g) 
& = t_1^{-12}(\det g_{11})^{-2}(\det g_{12})^8
(\det g_{21})^{-5}(\det g_{22})^5 \\
& = (\det g_{11})^{10}(\det g_{12})^{20}(\det g_{22})^{10}. 
\end{align*}

For $x\in Z_{37}$, let 
\begin{align*}
A(x) & = (x_{241}\ccd x_{352})\in 
\Lam^{2,1}_{1,[2,3]}\otimes \Lam^{2,1}_{1,[4,5]}
\otimes \Lam^{2,1}_{2,[1,2]}, \\
B(x) & = 
\begin{pmatrix}
x_{143} & x_{144} \\
x_{153} & x_{154}
\end{pmatrix}
\in \Lam^{2,1}_{1,[4,5]}\otimes \Lam^{2,1}_{2,[3,4]},\;
v(x) = 
\begin{pmatrix}
x_{233} \\ x_{234}
\end{pmatrix}
\in \Lam^{2,1}_{2,[3,4]}.
\end{align*}
We identify $Z_{37}$ with 
\begin{math}
\aff^2\otimes \aff^2\otimes \aff^2\oplus 
\m_2 \oplus\aff^2
\end{math}
by the map $Z_{37}\ni x\mapsto (A(x),B(x),v(x))$. 

It is easy to see that 
\begin{equation*}
A(gx) = (g_{11},g_{12},g_{21})A(x),\;
B(gx) = t_1g_{12}B(x){}^tg_{22},\;
v(gx) = (\det g_{11})g_{22}v(x)
\end{equation*}
where the right hand sides are the natural actions. 
We apply Lemma \ref{lem:natural-pairing} twice to 
$(A(x),B(x),v(x))$ and obtain a map 
\begin{equation*}
\Phi:Z_{37}\to 
\Lam^{2,1}_{1,[2,3]}\otimes \Lam^{2,1}_{1,[4,5]}
\otimes \Lam^{2,1}_{2,[1,2]}
\oplus \Lam^{2,1}_{1,[4,5]}
\to \Lam^{2,1}_{1,[2,3]}\otimes \Lam^{2,1}_{2,[1,2]}.
\end{equation*}
Then 
\begin{math}
\Phi(gx)
= t_1(\det g_{11})(\det g_{12})
(\det g_{22})(g_{11},g_{21})\Phi(x).
\end{math}
Identifying $\Lam^{2,1}_{1,[2,3]}\otimes \Lam^{2,1}_{2,[1,2]}$ 
with $\m_2$, we put $P_1(x) = \det \Phi(x)$. 
Then 
\begin{equation*}
P_1(gx) = t_1^2(\det g_{11})^3(\det g_{12})^2
(\det g_{21})(\det g_{22})^2P_1(x). 
\end{equation*}
Let $P_2(x)$ be the degree $4$ polynomial of 
\begin{math}
A(x)
\end{math}
obtained by Proposition \ref{prop:222-invariant}. 
Then 
\begin{equation*}
P_2(gx) = (\det g_{11})^2(\det g_{12})^2(\det g_{21})^2P_2(x). 
\end{equation*}
Let $P_3(x)=\det B(x)$. Then 
\begin{math}
P_3(gx) = t_1^2(\det g_{12})(\det g_{22})P_3(x). 
\end{math}

We put $P(x) = P_1(x)^3P_2(x)^3P_3(x)^4$. Then on $M^{\text{st}}_{\be_{37}}$,  
\begin{align*}
P(gx) & = (t_1^2(\det g_{11})^3(\det g_{12})^2
(\det g_{21})(\det g_{22})^2)^3
((\det g_{11})^2(\det g_{12})^2(\det g_{21})^2)^3 \\
& \quad \times (t_1^2(\det g_{12})(\det g_{22}))^4P(x) \\
& = t_1^{14}(\det g_{11})^{15}(\det g_{12})^{16}
(\det g_{21})^9(\det g_{22})^{10} P(x) \\
& = (\det g_{11})(\det g_{12})^2(\det g_{22})P(x). 
\end{align*}
Therefore, $P(x)$ is 
invariant under the action of $G_{\text{st},\be_{37}}$. 

Let $R(37)\in Z_{\be_{37}}$ be the element such that 
\begin{equation*}
A(R(37)) = (1,0\ccd 0,1),\;
B(R(37)) = I_2,\; 
v(R(37)) = [-1,1]. 
\end{equation*}
Explicitly, 
$R(37)=e_{241}+e_{352}+e_{143}-e_{233}+e_{154}+e_{234}$. 
Then 
\begin{equation*}
\Phi_1(R(37))= \diag(1,-1),P_1(R(37))=-1,P_2(R(37))=P_3(R(37))=1.
\end{equation*}
So $P(R(37))=1$ and $R(37)\in Z^{\sst}_{\be_{37}}$. 

Suppose that $x\in Z^{\sst}_{\be_{37}\,k}$. 
Then there exists $g_{22}\in\gl_2(k)$ such that 
$B(x){}^tg_{22}=I_2$. So we may assume that $B(x)=I_2$. 
Then $g= (g_{11},g_{12},g_{21},g_{22},t_1)$ 
does not change this condition if and only if 
$g_{22}=t_1^{-1}{}^tg_{12}^{-1}$. 
If $g=(g_{11},g_{12},g_{21},t_1^{-1}{}^tg_{12}^{-1},t_1)$ 
then $A(gx)=(g_{11},g_{12},g_{21})A(x)$ and 
$v(gx)=(\det g_{11})^{-1}t_1^{-1}{}^tg_{12}^{-1}v(x)$. 
Substituting $g_{12},g_{21}$ by 
$(\det g_{11})t_1g_{12},(\det g_{11})^{-1}t_1^{-1}g_{21}$ 
respectively, 
the action of $g$ on $A(x),v(x)$ are the natural action and 
the action by ${}^tg_{12}^{-1}$. Since this action by 
${}^tg_{12}^{-1}$ is still the standard \rep, if we exchange the 
order of $(g_{11},g_{12},g_{21})$ then 
Proposition \ref{prop:222+2-orbit} implies that 
$M_{\be_{37}\,k}\backslash Z^{\sst}_{37\,k}$ 
is in bijective correspondence with $\Ex_2(k)$.

We assume that  $u_{132}=u_{154}=0$ 
and $u_{221}=u_{243}=0$. 
Then the four components of 
$n(u)R(37)$ are as follows: 
\begin{align*}
& \begin{pmatrix}
0 & 0 & 0 & 0 & 0 \\
0 & 0 & 0 & 1 & 0 \\
0 & 0 & 0 & 0 & 0 \\
0 & -1 & 0 & 0 & -u_{152} \\
0 & 0 & 0 & * & 0 
\end{pmatrix},\;
\begin{pmatrix}
0 & 0 & 0 & 0 & 0 \\
0 & 0 & 0 & 0 & 0 \\
0 & 0 & 0 & 0 & 1 \\
0 & 0 & 0 & 0 & u_{143} \\
0 & 0 & -1 & * & 0 
\end{pmatrix}, \\
& \begin{pmatrix}
0 & 0 & 0 & 1 & 0 \\
0 & 0 & -1 & u_{121} -u_{143} + u_{231} & -u_{153} \\
0 & 1 & 0 & u_{131}+u_{142} & u_{152} +u_{232}  \\
-1 & * & * & 0 & Q_1(u)-u_{151}  \\
0 & * & * & * & 0 
\end{pmatrix}, \\
& \begin{pmatrix}
0 & 0 & 0 & 0 & 1 \\
0 & 0 & 1 & u_{143} + u_{241} & u_{121} + u_{153} \\
0 & -1 & 0 & -u_{142} & u_{131}-u_{152}+ u_{242} \\
0 & * & * & 0 & Q_2(u) + u_{141}  \\
-1 & * & * & * & 0 
\end{pmatrix}
\end{align*}
where $Q_1(u),Q_2(u)$ are polynomials  
which do not depend on $u_{141},u_{151}$.

We can apply Lemma \ref{lem:connected-criterion} 
to the map $\aff^{12}\to \aff^{12}$ defined by the sequence
\begin{align*}
& u_{142},u_{143},u_{152},u_{153},u_{232}+u_{152},u_{121}+u_{153}, 
u_{231}+u_{121}-u_{143},u_{131}+u_{142}, \\
& u_{241}+u_{143}, 
u_{242}+u_{131}-u_{152},u_{151}-Q_1(u),u_{141}+Q_2(u)
\end{align*}
with no extra variables. 
So by Proposition \ref{prop:W-eliminate}, 
Property \ref{property:W-eliminate} holds for 
any $x\in Z^{\sst}_{37\,k}$. Therefore, 
$P_{\be_{37}\,k}\backslash Y^{\sst}_{37\,k}$ 
is in  bijective correspondence with $\Ex_2(k)$ also.

\vskip 5pt

(14) $S_{40}$, $\be_{40}=\tfrac {1} {180} (-7,-2,-2,3,8,-5,0,0,5)$ 

We identify the element 
$(\diag(t_{11},g_1,t_{12},t_{13}),\diag(t_{21},g_2,t_{22}))
\in M_{[1,3,4],[1,3]}=M_{\be_{40}}$ 
with the element 
$g=(g_1,g_2,t_{11}\ccd t_{22})\in \gl_2^2\times \gl_1^5$. 
On $M^{\text{st}}_{\be_{40}}$, 
\begin{equation*}
\chi_{40}(g) = 
(t_{11})^{-7}(\det g_1)^{-2}t_{12}^3t_{13}^8
t_{21}^{-5}t_{22}^5
= (\det g_1)^5t_{12}^{10}t_{13}^{15}
(\det g_2)^5t_{22}^{10}. 
\end{equation*}

For $x\in Z_{40}$, let 
\begin{align*}
& A(x) =
\begin{pmatrix}
x_{242} & x_{243} \\
x_{342} & x_{343}
\end{pmatrix}
\in \Lam^{2,1}_{1,[2,3]}\otimes \Lam^{2,1}_{2,[2,3]}, \\
& v_1(x) =[x_{251},x_{351}]\in \Lam^{2,1}_{1,[2,3]}, \; 
v_2(x) =[x_{152},x_{153}]\in \Lam^{2,1}_{2,[2,3]}. 
\end{align*}
We identify $Z_{40}$ with 
\begin{math}
\aff^2\otimes \aff^2\oplus \aff^2
\oplus \aff^2\oplus 1^{2\oplus}
\end{math}
by the map $Z_{40}\ni x\mapsto (A(x),v_1(x),v_2(x),x_{144},x_{234})$. 

It is easy to see that 
\begin{equation*}
A(gx) = t_{12}g_1A(x){}^tg_2,\;
v_1(gx) = t_{13}t_{21}g_1v_1(x),\;
v_2(gx) = t_{11}t_{13}g_2v_2(x). 
\end{equation*}
The \rep{} of $M_{\be_{40}}$ on 
$(A(x),v_1(x),v_2(x))$ can be identified 
with the \rep{} considered in Proposition \ref{prop:M2-2-2-invariant} 
except for the extra $\gl_1$-factors. 

Let $P_1(x)$ be the degree $3$
polynomial of $(A(x),v_1(x),v_2(x))$ 
obtained by Proposition \ref{prop:M2-2-2-invariant}. 
Since $P_1(x)$ is linear with respect to each of 
$A(x),v_1(x),v_2(x)$, 
\begin{equation*}
P_1(gx) = t_{11}(\det g_1)t_{12}t_{13}^2t_{21}(\det g_2)P_1(x).  
\end{equation*}
Let $P_2(x)=\det A(x)$. Then 
\begin{math}
P_2(gx) = (\det g_1)t_{12}^2(\det g_2) P_2(x). 
\end{math}

We put $P(x) = P_1(x)^8 P_2(x) x_{144}^5x_{234}^5$.
Then on $M^{\text{st}}_{\be_{40}}$, 
\begin{align*}
P(gx) & = 
(t_{11}(\det g_1)t_{12}t_{13}^2t_{21}(\det g_2))^8
((\det g_1)t_{12}^2(\det g_2))
(t_{11}t_{12}t_{22})^5 ((\det g_1)t_{22})^5 P(x) \\
& = t_{11}^{13}(\det g_1)^{14}t_{12}^{15}t_{13}^{16}
t_{21}^8(\det g_2)^9t_{22}^{10}P(x) 
= (\det g_1)t_{12}^2t_{13}^3(\det g_2)t_{22}^2P(x). 
\end{align*}
Therefore, $P(x)$ is 
invariant under the action of $G_{\text{st},\be_{40}}$. 

Let $R(40)\in Z_{{40}}$ be the element 
such that 
\begin{equation*}
A(R(40)) = I_2 ,\;
v_1(R(40)) = v_2(R(40)) = [1,0]
\end{equation*}
and that the $x_{144},x_{234}$-coordinates are $1$. 
Explicitly, 
$R(40)=e_{251}+e_{152}+e_{242}+e_{343}+e_{144}+e_{234}$. 
Then $P(R(40))=1$ and so $R(40)\in Z^{\sst}_{40}$. 

We show that $Z^{\sst}_{{40}\,k} = M_{\be_{40}\,} R(40)$.  
Suppose that $x\in Z^{\sst}_{40\,k}$. 
Then $x_{144},x_{234}\not=0$. 
By Proposition \ref{prop:M2-2-2-invariant}, 
there exists $g\in M_{\be_{40}\,k}$ such that 
$A(gx)=I_2,v_1(gx)=v_2(gx) = [1,0]$. 
So we may assume that $A(x)=I_2,v_1(x)=v_2(x) = [1,0]$. 
If $t=(I_2,I_2,t_{11},1,t_{11}^{-1},t_{11},t_{22})$ then 
$A(tx)=I_2,v_1(tx)=v_2(tx)=[1,0]$ and 
the $x_{144},x_{234}$-coordinates of $tx$ are 
$t_{11}t_{22}x_{144},t_{22}x_{234}$. 
Therefore, there exists such $t$ such that $tx=R(40)$. 
Hence, $Z^{\sst}_{{40}\,k} = M_{\be_{40}\,} R(40)$.  

We assume that $u_{132}=0$ and $u_{232}=0$. 
Then the four components of 
$n(u)R(40)$ are as follows: 
\begin{align*}
& \begin{pmatrix}
0 & 0 & 0 & 0 & 0 \\
0 & 0 & 0 & 0 & 1 \\
0 & 0 & 0 & 0 & 0 \\
0 & 0 & 0 & 0 & u_{142} \\
0 & -1 & 0 & * & 0 
\end{pmatrix},\;
\begin{pmatrix}
0 & 0 & 0 & 0 & 1 \\
0 & 0 & 0 & 1 & u_{121} + u_{154} + u_{221} \\
0 & 0 & 0 & 0 & u_{131} \\
0 & -1 & 0 & 0 & Q_1(u) + u_{141}-u_{152}  \\
-1 & * & * & * & 0 
\end{pmatrix}, \\
& \begin{pmatrix}
0 & 0 & 0 & 0 & 0 \\
0 & 0 & 0 & 0 & u_{231} \\
0 & 0 & 0 & 1 & u_{154} \\
0 & 0 & -1 & 0 & Q_2(u) -u_{153} \\
* & * & * & * & 0 
\end{pmatrix}, \\
& \begin{pmatrix}
0 & 0 & 0 & 1 & u_{154}+u_{242} \\
0 & 0 & 1 & u_{121} +u_{143} + u_{242} 
& Q_3(u)+u_{153}+u_{241} \\
0 & -1 & 0 & u_{131}-u_{142} + u_{243} 
& Q_4(u) -u_{152}  \\
-1 & * & * & 0 & Q_5(u)-u_{151}  \\
* & * & * & * & 0 
\end{pmatrix}
\end{align*}
where $Q_1(u)\ccd Q_4(u)$ are polynomials  
which do not depend on $u_{141},u_{151},u_{152},u_{153},u_{241}$ 
and $Q_5(u)$ is a polynomial which does not depend on $u_{151}$. 

We can apply Lemma \ref{lem:connected-criterion} 
to the map $\aff^{14}\to \aff^{13}$ defined by the sequence
\begin{align*}
& u_{131},u_{142},u_{154},u_{231},u_{242}+u_{154},
u_{243}+u_{131}-u_{142},u_{221}+u_{121}+u_{154}, \\
& u_{143}+u_{121}+u_{242}, 
u_{152}-Q_4(u),u_{141}-u_{152}+Q_1(u),
u_{153}-Q_2(u), \\
& u_{241}+u_{153}+Q_3(u),u_{151}-Q_5(u) 
\end{align*}
where $u_{121}$ is an extra variable.
So by Proposition \ref{prop:W-eliminate}, 
Property \ref{property:W-eliminate} holds for 
any $x\in Z^{\sst}_{40\,k}$. Therefore, 
$Y^{\sst}_{40\,k}=P_{\be_{40}\,k}R(40)$ also.  

\vskip 5pt

(15) $S_{42}$, $\be_{42}=\tfrac {3} {260} (-8,-8,-8,12,12,-15,5,5,5)$ 

We identify the element 
$(\diag(g_{11},g_{12}),\diag(t_2,g_2))
\in M_{[3],[1]}=M_{\be_{42}}$ 
with the element 
$g=(g_{11},g_2,g_{12},t_2)\in \gl_3^2\times \gl_2\times \gl_1$. 
On $M^{\text{st}}_{\be_{42}}$, 
\begin{equation*}
\chi_{{42}}(g) = 
(\det g_{11})^{-8}(\det g_{12})^{12}
t_2^{-15}(\det g_2)^5
= (\det g_{12})^{20}(\det g_2)^{20}.
\end{equation*}

For $x\in Z_{42}$, let 
\begin{equation*}
A_1(x) = 
\begin{pmatrix}
x_{142} & x_{143} & x_{144} \\
x_{242} & x_{243} & x_{244} \\
x_{342} & x_{343} & x_{344} 
\end{pmatrix},\;
A_2(x) = 
\begin{pmatrix}
x_{152} & x_{153} & x_{154} \\
x_{252} & x_{253} & x_{254} \\
x_{352} & x_{353} & x_{354} 
\end{pmatrix}
\end{equation*}
and $A(x)=(A_1(x),A_2(x))$. 
We identify $Z_{42}\cong \m_3\otimes \aff^2\oplus 1$
by the map $Z_{42}\ni x\mapsto (A(x),x_{451})$. 

It is easy to see that 
\begin{equation*}
\begin{pmatrix}
A_1(gx) \\
A_2(gx)
\end{pmatrix}
= g_{12}
\begin{pmatrix}
g_{11}A_1(x){}^tg_2 \\
g_{11}A_2(x){}^tg_2 
\end{pmatrix}. 
\end{equation*}
Let $P_1(x)$ be the degree $12$ polynomial of 
$A(x)$ obtained by Proposition \ref{prop:332-invariant}. 
Then 
\begin{equation*}
P_1(gx) = (\det g_{11})^4(\det g_{12})^6(\det g_2)^4P_1(x)
\end{equation*}

We put $P(x) = P_1(x)x_{451}$. Then on $M^{\text{st}}_{\be_{42}}$, 
\begin{align*}
P(gx) & = ((\det g_{11})^4(\det g_{12})^6(\det g_2)^4) 
(t_2(\det g_{12}))P(x) \\
& = (\det g_{11})^4(\det g_{12})^7t_2(\det g_2)^4P(x)
= (\det g_{12})^3(\det g_2)^3P(x). 
\end{align*}
Therefore, $P(x)$ is 
invariant under the action of $G_{\text{st},\be_{42}}$. 

Suppose that $x\in Z^{\sst}_{42\,k}$. 
Then $x_{451}\not=0$. 
If $g = (I_3,I_3,I_2,x_{451}^{-1})$ then 
the $x_{451}$-coordinate of $gx$ is $1$. 
If $x_{451}=1$ then 
$g = (g_{11},g_2,g_{12},(\det g_{12})^{-1})$ 
does not change $x_{451}$ and acts on 
$A(x)$ by the tensor product of standard \rep s. 
Therefore, we can apply Proposition 
\ref{prop:Rational-orbits-3-cases} to 
$M_{\be_{42}\,k}\backslash Z^{\sst}_{42\,k}$. 
Hence, $M_{\be_{42}\,k}\backslash Z^{\sst}_{42\,k}$ is 
in bijective correspondence with $\Ex_3(k)$. 

Let $R(42)$ be the element such that 
$A(R(42))=(\diag(1,-1,0),\diag(0,1,-1))$
and the $x_{451}$-coordinate is $1$. 
Explicitly, 
\begin{math}
R(42)=e_{451}+e_{142}-e_{243} + e_{253}-e_{354}.
\end{math}
Then $P(R(42))=1$ by Proposition 
\ref{prop:332-invariant}. 

We assume that 
$u_{1ij}=0$ unless $i=4,5,j=1,2,3$ 
and $u_{2ij}=0$ unless $j=1$. 
Then the four components of 
$n(u)R(42)$ are as follows: 
\begin{align*}
& \begin{pmatrix}
0 & 0 & 0 & 0 & 0 \\
0 & 0 & 0 & 0 & 0 \\
0 & 0 & 0 & 0 & 0 \\
0 & 0 & 0 & 0 & 1 \\
0 & 0 & 0 & -1 & 0 
\end{pmatrix},\;
\begin{pmatrix}
0 & 0 & 0 & 1 & 0 \\
0 & 0 & 0 & 0 & 0 \\
0 & 0 & 0 & 0 & 0 \\
-1 & 0 & 0 & 0 & -u_{151}+u_{221}  \\
0 & 0 & 0 & * & 0 
\end{pmatrix}, 
\end{align*}

\begin{align*}
& \begin{pmatrix}
0 & 0 & 0 & 0 & 0 \\
0 & 0 & 0 & -1 & 1 \\
0 & 0 & 0 & 0 & 0 \\
0 & 1 & 0 & 0 & u_{142}+u_{152}+u_{231} \\
0 & -1 & 0 & * & 0 
\end{pmatrix}, \;
\begin{pmatrix}
0 & 0 & 0 & 0 & 0 \\
0 & 0 & 0 & 0 & 0 \\
0 & 0 & 0 & 0 & -1 \\
0 & 0 & 0 & 0 & -u_{143}+u_{241}  \\
0 & 0 & 1 & * & 0 
\end{pmatrix}.
\end{align*}

We can apply Lemma \ref{lem:connected-criterion}
to the map $\aff^9\to \aff^3$ defined by the sequence
\begin{align*}
& u_{221}-u_{151},u_{231}+u_{142}+u_{152},
u_{241}-u_{143}
\end{align*}
where $u_{1ij}$ ($i=4,5,j=1,2,3$) are extra variables.
So by Proposition \ref{prop:W-eliminate}, 
Property \ref{property:W-eliminate} holds for 
any $x\in Z^{\sst}_{42\,k}$. Therefore, 
$P_{\be_{42}\,k}\backslash Y^{\sst}_{42\,k}$ 
is in bijective correspondence with $\Ex_3(k)$ also.

\vskip 5pt

(16) $S_{49}$,  $\be_{49}=\tfrac {1} {100} (-8,-4,0,4,8,-9,-1,3,7)$ 

(17) $S_{50}$, $\be_{50}=\tfrac {1} {60} (-4,-2,0,2,4,-3,-1,1,3)$ 

For $l=49,50$, $M_{\be_l}=T$ and so $Z^{\sst}_l\not=\emptyset$ 
by Proposition \ref{prop:M=Tsuejective}.  
Let $R(l)\in Z_l$ be element whose coordinates in $Z_l$ are all $1$. 
 
We express elements of $T$ as (\ref{eq:t-defn}). 
Then the matrix $(m_{ij})$ of Proposition \ref{prop:M=Tsuejective} is 
\begin{equation*}
\begin{pmatrix}
0 & 0 & 0 & 1 & 1 & 1 & 0 & 0 & 0 \\
0 & 1 & 0 & 0 & 1 & 0 & 1 & 0 & 0 \\
0 & 0 & 1 & 1 & 0 & 0 & 1 & 0 & 0 \\
1 & 0 & 0 & 0 & 1 & 0 & 0 & 1 & 0 \\
0 & 1 & 0 & 1 & 0 & 0 & 0 & 1 & 0 \\
1 & 0 & 0 & 1 & 0 & 0 & 0 & 0 & 1 \\
0 & 1 & 1 & 0 & 0 & 0 & 0 & 0 & 1 
\end{pmatrix}, \;
\begin{pmatrix}
0 & 0 & 1 & 0 & 1 & 1 & 0 & 0 & 0 \\
0 & 1 & 0 & 0 & 1 & 0 & 1 & 0 & 0 \\
0 & 0 & 1 & 1 & 0 & 0 & 1 & 0 & 0 \\
1 & 0 & 0 & 0 & 1 & 0 & 0 & 1 & 0 \\
0 & 1 & 0 & 1 & 0 & 0 & 0 & 1 & 0 \\
1 & 0 & 0 & 1 & 0 & 0 & 0 & 0 & 1 \\
0 & 1 & 1 & 0 & 0 & 0 & 0 & 0 & 1 
\end{pmatrix}
\end{equation*}
for $S_{49},S_{50}$ respectively.

The determinant of the $7\times 7$ minor 
of columns $1,3,4,5,6,7,8$ (resp. columns $1,2,3,5,6,8,9$) 
is $-1$ (resp. $1$).   
Therefore, $Z^{\sst}_{l\,k} = T_k R(l)$ for $l=49,50$.   

We consider the case $S_{49}$. 
Explicitly, $R(49)=e_{451}+e_{252}+e_{342}+e_{153}+e_{243}+e_{144}+e_{234}$.
The four components of $n(u)R(49)$ are as follows: 
\begin{align*}
& \begin{pmatrix}
0 & 0 & 0 & 0 & 0 \\
0 & 0 & 0 & 0 & 0 \\
0 & 0 & 0 & 0 & 0 \\
0 & 0 & 0 & 0 & 1 \\
0 & 0 & 0 & -1 & 0 
\end{pmatrix},\;
\begin{pmatrix}
0 & 0 & 0 & 0 & 0 \\
0 & 0 & 0 & 0 & 1 \\
0 & 0 & 0 & 1 & u_{154}+u_{132} \\
0 & 0 & -1 & 0 & Q_1(u)+u_{142}-u_{153}+u_{221}  \\
0 & -1 & * & * & 0 
\end{pmatrix}, \\
& \begin{pmatrix}
0 & 0 & 0 & 0 & 1 \\
0 & 0 & 0 & 1 & u_{121}+u_{154}+u_{232} \\
0 & 0 & 0 & u_{132}+u_{232} & Q_2(u)+u_{131} \\
0 & -1 & * & 0 & Q_3(u)+u_{141}-u_{152}+u_{231} \\
-1 & * & * & * & 0 
\end{pmatrix}, \\
& \begin{pmatrix}
0 & 0 & 0 & 1 & u_{154}+u_{243} \\
0 & 0 & 1 & u_{121}+u_{143}+u_{243} & Q_4(u)+u_{153}+u_{242} \\
0 & -1 & 0 & Q_5(u)+u_{131}-u_{142}+u_{242} & Q_6(u)-u_{152} \\
-1 & * & * & 0 & Q_7(u) -u_{151}+u_{241} \\
* & * & * & * & 0 
\end{pmatrix}
\end{align*}
where $Q_1(u),Q_2(u),Q_4(u),Q_5(u)$
do not depend on 
$u_{131},u_{141},u_{142},u_{151},u_{152},u_{153},u_{221}$, 
$u_{241},u_{242}$, 
$Q_3(u),Q_6(u)$ do not depend on 
$u_{141},u_{151},u_{152},u_{241}$ 
and $Q_7(u)$ does not depend on 
$u_{151},u_{241}$.

We can apply Lemma \ref{lem:connected-criterion} 
to the map $\aff^{16}\to \aff^{12}$ defined by the sequence
\begin{align*}
& u_{154}+u_{132},u_{232}+u_{132},u_{121}+u_{154}+u_{232}, 
u_{243}+u_{154},u_{143}+u_{121}+u_{243}, \\ 
& u_{131}+Q_2(u),u_{153}+u_{242}+Q_4(u),
u_{142}-u_{131}-u_{242}-Q_5(u), \\
& u_{221}+u_{142}-u_{153}+Q_1(u), 
u_{152}-Q_6(u),u_{141}-u_{152}+Q_3(u),
u_{241}-u_{151}+Q_7(u)   
\end{align*}
where $u_{132},u_{151},u_{231},u_{242}$  are extra variables.
So by Proposition \ref{prop:W-eliminate}, 
Property \ref{property:W-eliminate} holds for 
any $x\in Z^{\sst}_{49\,k}$. Therefore, 
$Y^{\sst}_{49\,k} = P_{\be_{49}\,k}R(49)$ also. 

\vskip 5pt
We consider the case $S_{50}$. 
$R(50)=e_{351}+e_{252}+e_{342}+e_{153}+e_{243}+e_{144}+e_{234}$. 
The four components of $n(u)R(50)$ are as follows: 
\begin{align*}
& \begin{pmatrix}
0 & 0 & 0 & 0 & 0 \\
0 & 0 & 0 & 0 & 0 \\
0 & 0 & 0 & 0 & 1 \\
0 & 0 & 0 & 0 & u_{143} \\
0 & 0 & -1 & * & 0 
\end{pmatrix},\;
\begin{pmatrix}
0 & 0 & 0 & 0 & 0 \\
0 & 0 & 0 & 0 & 1 \\
0 & 0 & 0 & 1 & u_{132} +u_{154}+u_{221} \\
0 & 0 & -1 & 0 & Q_1(u)+u_{142}-u_{153} \\
0 & -1 & * & * & 0 
\end{pmatrix}, \\
& \begin{pmatrix}
0 & 0 & 0 & 0 & 1 \\
0 & 0 & 0 & 1 & u_{121}+u_{154}+u_{232} \\
0 & 0 & 0 & u_{132}+u_{232} & Q_2(u)+u_{131}+u_{231} \\
0 & -1 & * & 0 & Q_3(u)+u_{141}-u_{152} \\
-1 & * & * & * & 0 
\end{pmatrix}, 
\end{align*}

\begin{align*}
& \begin{pmatrix}
0 & 0 & 0 & 1 & u_{154}+u_{243} \\
0 & 0 & 1 & u_{121}+u_{143}+u_{243} & Q_4(u)+u_{153}+u_{242} \\
0 & -1 & 0 & Q_5(u)+u_{131}-u_{142}+u_{242} & Q_6(u)-u_{152}+u_{241} \\
-1 & * & * & 0 & Q_7(u) -u_{151} \\
* & * & * & * & 0 
\end{pmatrix}
\end{align*}
where $Q_1(u),Q_2(u),Q_4(u),Q_5(u)$ do not depend on 
$u_{131},u_{141},u_{151},u_{153},u_{231},u_{241},u_{242}$, 
$Q_3(u)$ does not depend on $u_{141},u_{151},u_{152},u_{241}$,
$Q_6(u)$ does not depend on $u_{151},u_{152},u_{241}$ 
and $Q_7(u)$ does not depend on $u_{151}$.

We can apply Lemma \ref{lem:connected-criterion} 
to the map $\aff^{16}\to \aff^{13}$ defined by the sequence
\begin{align*}
& u_{143},u_{154}+u_{243},u_{121}+u_{143}+u_{243}, 
u_{232}+u_{121}+u_{154}, u_{132}+u_{232}, 
u_{221}+u_{132}+u_{154}, \\
& u_{153}-u_{142}-Q_1(u),
u_{242}+u_{153}+Q_4(u),u_{131}-u_{142}+u_{242}+Q_5(u), \\
& u_{231}+u_{131}+Q_2(u),u_{141}-u_{152}+Q_3(u),u_{241}-u_{152}+Q_6(u),
u_{151}-Q_7(u)   
\end{align*}
where $u_{142},u_{152},u_{243}$  are extra variables.
So by Proposition \ref{prop:W-eliminate}, 
Property \ref{property:W-eliminate} holds for 
any $x\in Z^{\sst}_{50\,k}$. Therefore, 
$Y^{\sst}_{50\,k} = P_{\be_{50}\,k}R(50)$ also.

\vskip 5pt

(18) $S_{64}$, 
$\be_{64}=\tfrac {1} {220} (-8,-8,-8,-8,32,-55,5,5,45)$

We identify the element 
$(\diag(g_1,t_1),\diag(t_{21},g_2,t_{22}))
\in M_{[4],[1,3]}=M_{\be_{64}}$ 
with the element 
$g=(g_1,g_2,t_1,t_{21},t_{22})\in \gl_4\times \gl_2\times \gl_1^3$. 
On $M^{\text{st}}_{\be_{64}}$, 
\begin{equation*}
\chi_{\be_{64}}(g) = 
(\det g_1)^{-8}t_1^{32}t_{21}^{-55}(\det g_{12})^5t_{22}^{45}
= t_1^{40}(\det g_{12})^{60}t_{22}^{100}. 
\end{equation*}

Let $\bbmp_{4,i},p_{4,ij}$, etc.,  be as before. 
For $x\in Z_{64}$, let 
\begin{align*}
& v_1(x) = 
[x_{152},x_{252},x_{352},x_{452}],\; 
v_2(x) = 
[x_{153},x_{253},x_{353},x_{453}], \\
& \;A(x) = (v_1(x)\; v_2(x))\in \m_{4,2},\; 
B(x) = 
\begin{pmatrix}
0 & x_{124} & x_{134} & x_{144} \\
-x_{124} & 0 & x_{234} & x_{244} \\
-x_{134} & -x_{234} & 0 & x_{344} \\
-x_{144} & -x_{244} & -x_{344} & 0  
\end{pmatrix}.
\end{align*}
We identify $Z_{64}\cong \m_{4,2}\oplus \wedge^2 \aff^4$ 
by the map $Z_{64}\ni x\mapsto (A(x),B(x))$. 
We also identify $B(x)$ with 
$\sum_{1\leq i<j\leq 4} x_{ij4} p_{4,ij}\in\wedge^2 \aff^4$. 

It is easy to see that 
\begin{equation}
\label{eq:equivariant-64}
A(gx) = t_1 g_1 A(x) {}^tg_2,\;
B(gx) = t_{22} (\wedge^2 g_1) B(x).
\end{equation}
Let $P_1(x)$ be the Pfaffian of $B(x)$
and $P_2(x)$ the polynomial such that  
\begin{equation*}
v_1(x)\wedge v_2(x)\wedge B(x) = P_2(x) p_{4,1234}.
\end{equation*}
We choose the sign of $P_1(x)$ so that 
$P(x)=1$ if $B(x)=p_{2,12}+p_{2,34}$. 
By (\ref{eq:equivariant-64}), 
\begin{equation*}
P_1(gx) =  (\det g_1)t_{22}^2P_1(x),\;
P_2(gx) = (\det g_1)t_1^2(\det g_2)t_{22}P_2(x).
\end{equation*}

We put $P(x)=P_1(x)P_2(x)^3$. Then on $M^{\text{st}}_{\be_{64}}$, 
\begin{align*}
P(gx) 
& = ((\det g_1)t_{22}^2)((\det g_1)t_1^2(\det g_2)t_{22})^3P(x) \\
& = (\det g_1)^4t_1^6(\det g_2)^3t_{22}^5P(x)
= t_1^2(\det g_2)^3t_{22}^5P(x). 
\end{align*}
Therefore, $P(x)$ is 
invariant under the action of $G_{\text{st},\be_{64}}$. 

Let $R(64)\in Z_{64}$ be the element such that 
$v_1(R(64))=[1,0,0,0],v_1(R(64))=[0,0,1,0]$
and $B(R(64))=p_{4,13}+p_{4,24}$, 
Explicitly, 
$R(64)=e_{152}+e_{353}+e_{134}+e_{244}$. 
Then $P_1(R(64))=P_2(R(64))=-1$. 
So $P(R(64))=1$ and $R(64)\in Z^{\sst}_{64}$.

Suppose that $x\in Z^{\sst}_{64\,k}$. 
We show that $x\in M_{\be_{64}\,k}R(64)$. 
By Lemma II--4.6 and applying  
the permutation matrix corresponding to the transposition 
$(2\,3)$, there exists $g\in M_{\be_{64}\,k}$ 
such that $B(gx)=B(R(64))$. So we may assume that 
$B(x)=B(R(64))$. By assumption, 
\begin{equation}
\label{eq:64-condition}
\det 
\begin{pmatrix}
x_{152} & x_{153} \\
x_{352} & x_{353}
\end{pmatrix}
+ \det 
\begin{pmatrix}
x_{252} & x_{253} \\
x_{452} & x_{453}
\end{pmatrix} \not=0.
\end{equation}
Let $\tau_0$ be as in (\ref{eq:J-defn}), $u,a,b,c,d\in k$ and 
\begin{align*}
& \tau_1 = \diag(\tau_0,\tau_0),\;
m_1(u) = \diag(n_2(u),{}^tn_2(-u)), \\
& m_2(a,b,c) = 
\begin{pmatrix}
1 & 0 & 0 & 0 \\
0 & 1 & 0 & 0 \\
a & b & 1 & 0 \\
b & c & 0 & 1 
\end{pmatrix},\;
s(a,b,c,d) = 
\begin{pmatrix}
a & 0 & b & 0 \\
0 & 1 & 0 & 0 \\
c & 0 & d & 0 \\
0 & 0 & 0 & 1 
\end{pmatrix}\; (ad-bc=1).
\end{align*}
Regarding that 
$\tau_1,m_1(u),{}^tm_1(u),m_2(a,b,c),{}^tm_2(a,b,c),
s(a,b,c,d)\in\gl_4\sub M_{\be_{64}}$, 
these elements fix $B(R(64))=p_{4,13}+p_{4,24}$.

By applying $\tau_1$ if necessary, 
we may assume that 
\begin{equation}
\label{eq:S64-R64-cond}
\det 
\begin{pmatrix}
x_{152} & x_{153} \\
x_{352} & x_{353}
\end{pmatrix} \not=0.
\end{equation}
By applying an element of the form 
$s(a,b,c,d)$ if necessary, we may assume that 
$x_{152}=1$. By applying an element of the form 
$m_1(u)$, we may assume that 
$x_{252}=0$. By applying an 
element of the form $m_2(a,b,c)$, 
we may assume that $x_{352}=x_{452}=0$, 
i.e., $v_1(x)=[1,0,0,0]$.

Then (\ref{eq:S64-R64-cond}) is satisfied and so
$x_{353}\not=0$. Let $h=\diag(1,x_{353}^{-1})$. 
By replacing $x$ by $(I_4,h,1,1,1)x$, 
we may assume that $x_{353}=1$. 
By applying an element of the form 
${}^tm_1(u)$, we may assume that $x_{453}=0$. 
Note that $v_1(x)$ does not change. 
Then by applying an element of the form 
${}^t m_2(a,b,c)$, $x$ becomes $R(64)$. 
Therefore, $Z^{\sst}_{64\,k}=M_{\be_{64}\,k}R(64)$. 

We assume that $u_{1ij}=0$ unless $i=5$ 
and $u_{232}=0$. 
Then the first three components of 
$n(u)R(64)$ are the same as those of 
$R(64)$ and the last component 
is as follows: 
\begin{align*}
\begin{pmatrix}
0 & 0 & 1 & 0 & u_{153}+u_{242} \\
0 & 0 & 0 & 1 & u_{154} \\
-1 & 0 & 0 & 0 &  -u_{151}+u_{243} \\
0 & -1 & 0 & 0 & -u_{152} \\
* & * & * & * & 0 
\end{pmatrix}.
\end{align*}

We can apply Lemma \ref{lem:connected-criterion} 
to the map $\aff^9\to \aff^4$ defined by the sequence
\begin{align*}
& u_{153}+u_{242},u_{154}, 
u_{151}-u_{243},u_{152}
\end{align*}
where $u_{221},u_{231},u_{241},u_{242},u_{243}$  are extra variables.
So by Proposition \ref{prop:W-eliminate}, 
Property \ref{property:W-eliminate} holds for 
any $x\in Z^{\sst}_{64\,k}$. Therefore, 
$Y^{\sst}_{64\,k} = P_{\be_{64}\,k}R(64)$ also. 

\vskip 5pt

(19) $S_{70}$, $\be_{70}=\tfrac {1} {140} (-6,-6,4,4,4,-35,5,15,15)$ 

We identify the element 
$(\diag(g_{11},g_{12}),\diag(t_{21},t_{22},g_2))
\in M_{[2],[1,2]}=M_{\be_{70}}$ 
with the element 
$g=(g_{12},g_{11},g_2,t_{21},t_{22})
\in \gl_3\times \gl_2^2\times \gl_1^2$. 
On $M^{\text{st}}_{\be_{40}}$, 
\begin{equation*}
\chi_{{70}}(g) = 
(\det g_{11})^{-6}(\det g_{12})^4 
t_{21}^{-35}t_{22}^5(\det g_2)^{15}
= (\det g_{12})^{10} t_{22}^{40}(\det g_2)^{50}.
\end{equation*}

For $x\in Z_{70}$, 
let $A(x) = x_{342}p_{3,12}+x_{352}p_{3,13}+x_{452}p_{3,23}$,  
\begin{equation*}
B_1(x) = 
\begin{pmatrix}
x_{133} & x_{134} \\
x_{233} & x_{234} 
\end{pmatrix},\;
B_2(x) = 
\begin{pmatrix}
x_{143} & x_{144} \\
x_{243} & x_{244} 
\end{pmatrix},\;
B_3(x) = 
\begin{pmatrix}
x_{153} & x_{154} \\
x_{253} & x_{254} 
\end{pmatrix}
\end{equation*}
and 
\begin{math}
B(x)=\bbmp_{3,1}\otimes B_1(x)+ \bbmp_{3,2}\otimes B_2(x)
+ \bbmp_{3,3}\otimes B_3(x).
\end{math}
We identify 
\begin{math}
Z_{70}$ with $\wedge^2 \aff^3 \oplus 
\aff^3\otimes \m_2
\end{math}
by the map $Z_{70}\ni x\mapsto (A(x),B(x))$. 
It is easy to see that 
\begin{equation*}
A(gx) = t_{22}(\wedge^2 g_{12}) A(x),\;
\begin{pmatrix}
B_1(gx) \\
B_2(gx) \\
B_3(gx) 
\end{pmatrix}
= g_{12} \begin{pmatrix}
g_{11}B_1(x){}^tg_2 \\
g_{11}B_2(x){}^tg_2 \\
g_{11}B_3(x){}^tg_2 
\end{pmatrix}. 
\end{equation*}
Since $A(gx),B(gx)$ do not depend on $t_{21}$, if
we ignore $t_{21}\in\gl_1$, then  
$(M_{\be_{70}},Z_{70})$ can be identified with 
the case (a) of Section \ref{sec:rational-orbits-wedge3-32}. 

Let $P_1(x),P_2(x)$ be the relative invariant polynomials 
in (\ref{eq:322,3-invariants}). Then 
\begin{align*}
& P_1(gx) = (\det g_{11})(\det g_{12})^2t_{22}^2(\det g_2)P_1(x), \\
& P_2(gx) = (\det g_{11})^3(\det g_{12})^2(\det g_2)^3P_2(x).
\end{align*}

We put $P(x)=P_1(x)^2P_2(x)$. Then on $M^{\text{st}}_{\be_{70}}$, 
\begin{align*}
P(gx) 
& = ((\det g_{11})(\det g_{12})^2t_{22}^2(\det g_2))^2
((\det g_{11})^3(\det g_{12})^2(\det g_2)^3)P(x) \\
& = (\det g_{11})^5(\det g_{12})^6
t_{22}^4(\det g_2)^5 P(x) 
= (\det g_{12})t_{22}^4(\det g_2)^5 P(x). 
\end{align*}
Therefore, $P(x)$ is 
invariant under the action of $G_{\text{st},\be_{70}}$. 

Let $R(70)\in Z_{70}$ be the element 
such that $(A(R(70)),B(R(70)))=R_{322,3}$
(see (\ref{eq:R(322-3)-defn})). 
Explicitly, 
$R(70)=e_{452}-e_{133}+e_{253}+e_{144}+e_{234}$. 
Then $R(70)\in Z_{70}^{\sst}$, 
which implies that $S_{70}\not=\emptyset$. 
By Proposition \ref{prop:322+3-orbit}, if $\ch(k)\not=2$ 
then $M_{\be_{70}\,k}\backslash Z^{\sst}_{70\,k}$ 
is in bijective correspondence with $\Ex_2(k)$. 

We assume that and $u_{1ij}=0$ unless $i=3,4,5,j=1,2$ 
and $u_{243}=0$. 
Then the first two components of 
$n(u)R(70)$ are the same as those of $R(70)$ 
and the remaining components are as follows: 
\begin{align*}
\begin{pmatrix}
0 & 0 & -1 & 0 & 0 \\
0 & 0 & 0 & 0 & 1 \\
1 & 0 & 0 & u_{141} & u_{151}+u_{132} \\
0 & 0 & * & 0 & u_{142}+u_{232} \\
0 & -1 & * & * & 0 \\
\end{pmatrix}, \;
\begin{pmatrix}
0 & 0 & 0 & 1 & 0 \\
0 & 0 & 1 & 0 & 0 \\
0 & -1 & 0 & u_{131}-u_{142} &  -u_{152} \\
-1 & 0 & * & 0 & -u_{151}+u_{242} \\
0 & 0 & * & * & 0 
\end{pmatrix}.
\end{align*}

We can apply Lemma \ref{lem:connected-criterion} 
to the map $\aff^{11}\to \aff^6$ defined by the sequence
\begin{math}
u_{141}, u_{151}+u_{132}, u_{142}+u_{232}, 
u_{131}-u_{142}, u_{152}, u_{242}-u_{151}
\end{math}
where $u_{132},u_{221},u_{231},u_{232}$, $u_{241}$  are extra variables.
So by Proposition \ref{prop:W-eliminate}, 
Property \ref{property:W-eliminate} holds for 
any $x\in Z^{\sst}_{70\,k}$. Therefore, if $\ch(k)\not=2$ then 
$P_{\be_{70}\,k}\backslash Y^{\sst}_{70\,k}$ 
is in bijective correspondence with $\Ex_2(k)$ also.

\vskip 5pt

(20)  $S_{71}$, $\be_{71}=\tfrac {1} {140} (-16,-16,-16,24,24,-35,-15,25,25)$ 

We identify the element 
$(\diag(g_{11},g_{12}),\diag(t_{21},t_{22},g_2))
\in M_{[3],[1,2]}=M_{\be_{71}}$ 
with the element 
$g=(g_{11},g_2,g_{12},t_{21},t_{22})\in \gl_3\times \gl_2^2\times \gl_1^2$. 
On $M^{\text{st}}_{\be_{71}}$, 
\begin{equation*}
\chi_{{71}}(g) 
= (\det g_{11})^{-16}(\det g_{12})^{24}
t_{21}^{-35}t_{22}^{-15}(\det g_2)^{25}
= (\det g_{12})^{40}t_{22}^{20}(\det g_2)^{60}. 
\end{equation*}

For $x\in Z_{71}$, let 
\begin{equation*}
A_1(x) = \begin{pmatrix}
x_{143} & x_{153} \\
x_{144} & x_{154} 
\end{pmatrix},\; 
A_2(x) = \begin{pmatrix}
x_{243} & x_{253} \\
x_{244} & x_{254} 
\end{pmatrix},\;
A_3(x) = \begin{pmatrix}
x_{343} & x_{353} \\
x_{344} & x_{354} 
\end{pmatrix}
\end{equation*}
and $A(x)=(A_1(x),A_2(x),A_3(x))$. 
We identify $Z_{71}$ with $\aff^3\otimes \m_2\oplus 1$ 
by the map $Z_{71}\ni x\mapsto (A(x),x_{452})$. 

It is easy to see that the action of $(g_{11},g_2,g_{12},1,1)$ 
on $A(x)$ is the tensor product of standard \rep s. 
So we are in the situation of Proposition \ref{prop:322-invariant}. 
Let $P_1(x)$ be the degree $6$ polynomial of $A(x)$ 
obtained by Proposition \ref{prop:322-invariant}. 
Since $A(gx)$ does not depend on $t_{21},t_{22}$, 
\begin{math}
P_1(gx) = (\det g_{11})^2 (\det g_{12})^3 (\det g_2)^3 P_1(x). 
\end{math}

We put $P(x) = P_1(x)x_{452}$. Then on $M^{\text{st}}_{\be_{71}}$, 
\begin{align*}
P(gx) & = ((\det g_{11})^2 (\det g_{12})^3 (\det g_2)^3)
((\det g_{12})t_{22})P(x) \\
& = (\det g_{11})^2 (\det g_{12})^4 t_{22}(\det g_2)^3 P(x)
= (\det g_{12})^2 t_{22} (\det g_2)^3 P(x).
\end{align*}
Therefore, $P(x)$ is 
invariant under the action of $G_{\text{st},\be_{71}}$. 

Let $R(71)\in Z_{71}$ be the element 
such that $A(R(71))=R_{322}$ ($R_{322}$ is the element $w$ in (\ref{eq:322-generator2-alternative}))
and the $x_{452}$-coordinate is $1$.
Explicitly, $R(71)=e_{452}-e_{143}+e_{253}+e_{154}+e_{344}$.
Then $P(R(71))=1$ and so $R(71)\in Z^{\sst}_{71}$. 

Suppose that $x\in Z^{\sst}_{71\,k}$.  
By Proposition \ref{prop:322-invariant}, 
there exists $g\in M_{\be_{71}\,k}$
such that $A(x)=A(R(71))$. 
Let $t=(I_3,I_2,I_2,1,x_{452}^{-1})$. 
Then $A(tx)=A(R(71))$ and the 
$x_{452}$-coordinate of $tx$ is $1$. 
Therefore, $Z^{\sst}_{71\,k}=M_{\be_{71}\,k}R(71)$. 

We assume that $u_{1ij}=0$ unless $i=4,5,j=1,2,3$ 
and $u_{243}=0$. 
Then the first two components of 
$n(u)R(71)$ are the same as those of $R(71)$ 
and the remaining components are as follows: 
\begin{align*}
\begin{pmatrix}
0 & 0 & 0 & -1 & 0 \\
0 & 0 & 0 & 0 & 1 \\
0 & 0 & 0 & 0 & 0 \\
1 & 0 & 0 & 0 & u_{142}+u_{151}+u_{232} \\
0 & -1 & 0 & * & 0 \\
\end{pmatrix}, \;
\begin{pmatrix}
0 & 0 & 0 & 0 & 1 \\
0 & 0 & 0 & 0 & 0 \\
0 & 0 & 0 & 1 & 0 \\
0 & 0 & -1 & 0 & u_{141}-u_{153}+u_{242} \\
-1 & 0 &  0 & * & 0 
\end{pmatrix}.
\end{align*}

We can apply Lemma \ref{lem:connected-criterion} 
to the map $\aff^{11}\to \aff^2$ defined by the sequence
\begin{math}
u_{142}+u_{151}+u_{232},u_{141}-u_{153}+u_{242} 
\end{math}
where $u_{131},u_{132},u_{151},u_{152},u_{221}\ccd u_{242}$  
are extra variables.
So by Proposition \ref{prop:W-eliminate}, 
Property \ref{property:W-eliminate} holds for 
any $x\in Z^{\sst}_{71\,k}$. Therefore, 
$Y^{\sst}_{71\,k} = P_{\be_{71}\,k}R(71)$ also. 

\vskip 5pt

(21) $S_{75}$, $\beta_{75}=\tfrac {1} {60} (-4,-4,1,1,6,-15,0,5,10)$ 

We identify the element 
$(\diag(g_{11},g_{12},t_1),\diag(t_{21}\ccd t_{24}))
\in M_{[2,4],[1,2,3]}=M_{\be_{75}}$ 
with the element 
$g=(g_{12},g_{11},t_1,t_{21}\ccd t_{24})\in \gl_2^2\times \gl_1^5$. 
On $M^{\text{st}}_{\be_{75}}$, 
\begin{equation*}
\chi_{{75}}(g) = 
(\det g_{11})^{-4}(\det g_{12})t_1^6
t_{21}^{-15}t_{23}^5t_{24}^{10}
= (\det g_{12})^5t_1^{10}t_{22}^{15}t_{23}^{20}t_{24}^{25}.
\end{equation*}

For $x\in Z_{75}$, let 
\begin{align*}
& A(x) = 
\begin{pmatrix}
x_{134} & x_{234} \\
x_{144} & x_{244}
\end{pmatrix},\;
v_1(x) = 
\begin{pmatrix}
x_{352} \\ x_{452}
\end{pmatrix},\;
v_2(x) = 
\begin{pmatrix}
x_{153} \\ x_{253}
\end{pmatrix}.
\end{align*}
We regard that $A(x)\in \Lam^{2,1}_{1,[3,4]}\otimes \Lam^{2,1}_{1,[1,2]}$, 
$v_1(x)\in \Lam^{2,1}_{1,[3,4]}$, $v_2(x)\in \Lam^{2,1}_{1,[1,2]}$. 
We identify $Z_{75}$ with $\m_2\oplus \aff^2\oplus \aff^2\oplus 1$ 
by the map $Z_{75}\ni x\mapsto (A(x),v_1(x),v_2(x),x_{343})$. 
It is easy to see that 
\begin{equation*}
A(gx) = t_{24} g_{12}A(x){}^tg_{11},\;
v_1(gx) = t_1t_{22}g_{12} v_1(x),\; 
v_2(gx) = t_1t_{23}g_{11} v_2(x).
\end{equation*}
The \rep{} of $M_{\be_{75}}$ on $(A(x),v_1(x),v_2(x))$ 
can be identified with the \rep{} considered in 
Proposition \ref{prop:M2-2-2-invariant} except for the 
extra $\gl_1$-factors. 

Let $P_1(x)$ be the degree $3$ polynomial
obtained by Proposition \ref{prop:M2-2-2-invariant}.
We put $P_2(x)=\det A(x)$. 
Then 
\begin{equation*}
P_1(gx) = (\det g_{11})(\det g_{12})t_1^2t_{22}t_{23}t_{24}P_1(x),\;
P_2(gx) = (\det g_{11})(\det g_{12})t_{24}^2P_2(x). 
\end{equation*}

We put $P(x) = P_1(x)^3P_2(x)x_{343}$. 
Then on $M^{\text{st}}_{\be_{75}}$. 
\begin{align*}
P(gx) 
& = ((\det g_{11})(\det g_{12})t_1^2t_{22}t_{23}t_{24})^3
((\det g_{11})(\det g_{12})t_{24}^2)((\det g_{12})t_{23})P(x) \\
& = (\det g_{11})^4(\det g_{12})^5t_1^6 t_{22}^3t_{23}^4t_{24}^5P(x) \
= (\det g_{12})t_1^2t_{22}^3t_{23}^4t_{24}^5 P(x). 
\end{align*}
Therefore, $P(x)$ is 
invariant under the action of $G_{\text{st},\be_{75}}$. 

Let $R(75)\in Z_{75}$ be the element such that
$v_1(R(75))=v_2(R(75))=[1,0],A(R(75))$ $=I_2$ 
and the $x_{343}$-coordinate is $1$. 
Explicitly, $R(75)=e_{352}+e_{153}+e_{343}+e_{134}+e_{244}$. 
Then $P_1(R(75))=P_2(R(75))=1$.  
So $P(R(75))=1$ and $R(75)\in Z^{\sst}_{75}$. 

We show that  $Z^{\sst}_{75\,k}=M_{\be_{75}\,k}R(75)$. 
Suppose that $x\in Z^{\sst}_{75\,k}$. 
Then $x_{343}\not=0$. 
By Proposition \ref{prop:M2-2-2-invariant}, 
there exists $g\in M_{\be_{75}\,k}$ such that 
$A(gx)=I_2,v_1(gx)=v_2(gx)=[1,0]$. 
So we may assume that $A(x)=I_2,v_1(x)=v_2(x)=[1,0]$. 
If $t=(I_2,I_2,1,1,1,x_{343}^{-1},1)$
then $tx=R(75)$. Therefore, 
$Z^{\sst}_{75\,k}=M_{\be_{75}\,k}R(75)$. 

We assume that $u_{121}=u_{143}=0$. 
Then the first component of 
$n(u)R(75)$ is $0$ and the remaining components are 
are as follows: 
\begin{align*}
& \begin{pmatrix}
0 & 0 & 0 & 0 & 0 \\
0 & 0 & 0 & 0 & 0 \\
0 & 0 & 0 & 0 & 1 \\
0 & 0 & 0 & 0 & 0 \\
0 & 0 & -1 & 0 & 0 
\end{pmatrix},\;
\begin{pmatrix}
0 & 0 & 0 & 0 & 1 \\
0 & 0 & 0 & 0 & 0 \\
0 & 0 & 0 & 1 & u_{131}+u_{154}+u_{232} \\
0 & 0 & -1 & 0 & u_{141}-u_{153} \\
-1 & 0 & 0 & * & 0 \\
\end{pmatrix}, 
\end{align*}

\begin{align*}
& \begin{pmatrix}
0 & 0 & 1 & 0 & u_{153}+u_{243} \\
0 & 0 & 0 & 1 & u_{154} \\
-1 & 0 & 0 & u_{132}-u_{141}+u_{243} & Q_1(u)-u_{151}+u_{242} \\
0 & -1 & * & 0 & Q_2(u)-u_{152} \\
* & * &  * & * & 0 
\end{pmatrix}
\end{align*}
where $Q_1(u),Q_2(u)$ do not depend on $u_{151},u_{152},u_{242}$. 

We can apply Lemma \ref{lem:connected-criterion} 
to the map $\aff^{14}\to \aff^7$ defined by the sequence
\begin{align*}
& u_{154}, u_{153}+u_{243}, u_{141}-u_{153}, 
u_{131}+u_{154}+u_{232}, \\
& u_{132}-u_{141}+u_{243},
u_{151}-u_{242}-Q_1(u),u_{152}-Q_2(u)
\end{align*}
where $u_{142},u_{221}\ccd u_{243}$  are extra variables.
So by Proposition \ref{prop:W-eliminate}, 
Property \ref{property:W-eliminate} holds for 
any $x\in Z^{\sst}_{75\,k}$. Therefore, 
$Y^{\sst}_{75\,k}=P_{\be_{75}\,k}R(75)$ also.

\vskip 5pt

(22)  $S_{95}$, $\be_{95}=\tfrac {1} {620} (-28,-8,-8,12,32,-35,-15,5,45)$ 

We identify the element 
$(\diag(t_{11},g_1,t_{12},t_{13}),\diag(t_{21}\ccd t_{24}))
\in M_{[1,3,4],[1,2,3]}=M_{\be_{95}}$ 
with the element 
$g=(g_1,t_{11}\ccd t_{24})\in \gl_2\times \gl_1^7$. 
On $M^{\text{st}}_{\be_{95}}$, 
\begin{equation*}
\chi_{{95}}(g) = 
t_{11}^{-28}(\det g_1)^{-8}t_{12}^{12}t_{13}^{32}
t_{21}^{-35}t_{22}^{-15}t_{23}^5t_{24}^{45}
= (\det g_1)^{20}t_{12}^{40}t_{13}^{60}
t_{22}^{20}t_{23}^{40}t_{24}^{80}. 
\end{equation*}

For $x\in Z_{95}$, let 
\begin{equation*}
v_1(x) = [x_{252},x_{352}],\;
v_2(x) = [x_{243},x_{343}],\;
v_3(x) = [x_{124},x_{134}].
\end{equation*}
We identify $Z_{95}\cong (\aff^2)^{3\oplus}\oplus 1^{2\oplus}$ 
by the map $Z_{95}\ni x\mapsto (v_1(x),v_2(x),v_3(x),x_{451},x_{153})$. 

It is easy to see that 
\begin{equation*}
v_1(gx) = t_{13}t_{22}g_1v_1(x),\;
v_2(gx) = t_{12}t_{23}g_1v_2(x),\;
v_3(gx) = t_{11}t_{24}g_1v_3(x).
\end{equation*}
Let 
\begin{equation*}
A(x) = (v_1(x)\; v_2(x)),\;
B(x) = (v_1(x)\; v_3(x)),\;
C(x) = (v_2(x)\; v_3(x)).
\end{equation*}
We put $P_1(x)=\det A(x),\; P_2(x)=\det B(x),\; P_3(x)=\det C(x)$.
Then 
\begin{align*}
P_1(gx) & = t_{12}t_{13}t_{22}t_{23}(\det g_1)P_1(x), \\
P_2(gx) & = t_{11}t_{13}t_{22}t_{24}(\det g_1)P_2(x), \\
P_3(gx) & = t_{11}t_{12}t_{23}t_{24}(\det g_1)P_3(x).
\end{align*}

We put $P(x) = P_1(x)^2P_2(x)^5P_3(x)^5x_{451}^6x_{153}$. 
Then on $M^{\text{st}}_{\be_{95}}$, 
\begin{align*}
P(gx) 
& = (t_{12}t_{13}t_{22}t_{23}(\det g_1))^2
(t_{11}t_{13}t_{22}t_{24}(\det g_1))^5
(t_{11}t_{12}t_{23}t_{24}(\det g_1))^5 \\
& \quad \times (t_{12}t_{13}t_{21})^6(t_{11}t_{13}t_{23})P(x) \\
& = t_{11}^{11}(\det g_1)^{12}t_{12}^{13}t_{13}^{14}
t_{21}^6t_{22}^7t_{23}^8t_{24}^{10} P(x) 
= (\det g_1)t_{12}^2t_{13}^3 t_{22}t_{23}^2t_{24}^4P(x).  
\end{align*}
Therefore, $P(x)$ is 
invariant under the action of $G_{\text{st},\be_{95}}$. 

Let $R(95)$ be the element such that 
$v_1(R(95))=[1,0],v_2(R(95))=[0,1],v_3(R(95))$ 
$=[1,1]$
and that the $x_{451},x_{153},$-coordinates are $1$. 
Then $P_1(R(95))=P_2(R(95))=1$, $P_3(R(95))=-1$.
So $P(R(95))=-1$ and $R(95)\in Z^{\sst}_{95\,k}$. 
Explicitly, $R(95)=e_{451}+e_{252}+e_{153}+e_{343}+e_{124}+e_{134}$.

We show that  $Z^{\sst}_{95\,k}=M_{\be_{95}\,k}R(95)$. 
Suppose that $x\in Z^{\sst}_{95\,k}$. 
Since $P_1(x)\not=0$, 
there exists $g\in M_{\be_{95}\,k}$ 
such that $v_1(gx)=[1,0],v_2(gx)=[0,1]$. 
So we may assume that $v_1(x)=[1,0],v_2(x)=[0,1]$. 
Then elements of the form 
\begin{equation*}
t=(\diag(q_1,q_2),1,q_2^{-1},q_1^{-1},t_{21},1,1,t_{24})
\end{equation*}
do not change this condition. 
Since $P_2(x),P_3(x)\not=0$, $x_{124},x_{134}\not=0$. 
Since $v_3(tx)=t_{24}[q_1x_{124},q_2x_{134}]$ 
and the $x_{451},x_{153}$-coordinates of $tx$ 
are $q_1^{-1}q_2^{-1}t_{21}x_{451},q_1^{-1}x_{153}$, 
there exists such $t$ such that $tx=R(95)$. 
Therefore, $Z^{\sst}_{95\,k}=M_{\be_{95}\,k}R(95)$. 

We assume that $u_{132}=0$. 
Then the first component of $n(u)R(95)$ 
is the same as that of $R(95)$ and the remaining components 
are as follows: 
\begin{align*}
& \begin{pmatrix}
0 & 0 & 0 & 0 & 0 \\
0 & 0 & 0 & 0 & 1 \\
0 & 0 & 0 & 0 & 0 \\
0 & 0 & 0 & 0 & u_{142}+u_{221} \\
0 & -1 & 0 & * & 0 \\
\end{pmatrix}, \;
\begin{pmatrix}
0 & 0 & 0 & 0 & 1 \\
0 & 0 & 0 & 0 & u_{121}+u_{232} \\
0 & 0 & 0 & 1 & u_{131}+u_{154} \\
0 & 0 & -1 & 0 & Q_1(u)+u_{141}-u_{153}+u_{231} \\
-1 & * &  * & * & 0 
\end{pmatrix}, \\
& \begin{pmatrix}
0 & 1 & 1 & u_{142}+u_{143} & u_{152}+u_{153}+u_{243} \\
-1 & 0 & u_{121}-u_{131} & Q_2(u)-u_{141} & Q_3(u)-u_{151}+u_{242} \\
-1 & * & 0 & Q_4(u)-u_{141}+u_{243} & Q_5(u)-u_{151} \\
* & * & * & 0 & Q_6(u)+u_{241} \\
* & * & * & * & 0 
\end{pmatrix}
\end{align*}
where 
$Q_1(u)=Q_1(u_{142},u_{143},u_{154},u_{232})$, 
$Q_2(u)=Q_2(u_{121},u_{142},u_{143})$, 
$Q_4(u)=Q_4(u_{131}$, $u_{142},u_{143})$, 
$Q_3(u),Q_5(u)$ do not depend on $u_{151},u_{241},u_{242}$
and $Q_6(u)$ does not depend on $u_{241}$. 

We can apply Lemma \ref{lem:connected-criterion} 
to the map $\aff^{15}\to \aff^{12}$ defined by the sequence
\begin{align*}
& u_{221}+u_{142}, u_{232}+u_{121}, u_{131}-u_{121}, 
u_{154}+u_{131}, u_{143}+u_{142},u_{141}-Q_2(u), \\
&  u_{231}+u_{141}-u_{153}+Q_1(u), u_{243}-u_{141}+Q_4(u),
u_{152}+u_{153}+u_{243}, \\ 
& u_{151}-Q_5(u),u_{242}-u_{151}+Q_3(u),u_{241}+Q_6(u)
\end{align*}
where $u_{121},u_{142},u_{153}$  are extra variables.
So by Proposition \ref{prop:W-eliminate}, 
Property \ref{property:W-eliminate} holds for 
any $x\in Z^{\sst}_{95\,k}$. Therefore, 
$Y^{\sst}_{95\,k}=P_{\be_{95}\,k}R(95)$ also.

\vskip 5pt

(23) $S_{101}$, 
$\be_{101}=\tfrac {1} {420} (-48,-28,12,32,32,-25,-25,15,35)$ 

We identify the element 
$(\diag(t_{11},t_{12},t_{13},g_1),\diag(g_2,t_{21},t_{22}))
\in M_{[1,2,3],[2,3]}=M_{\be_{101}}$ 
with the element 
$g=(g_1,g_2,t_{11}\ccd t_{22})\in \gl_2^2\times \gl_1^5$. 
On $M^{\text{st}}_{\be_{101}}$, 
\begin{equation*}
\chi_{{101}}(g) = 
t_{11}^{-48}t_{12}^{-28}t_{13}^{12}(\det g_1)^{32}
(\det g_1)^{-25}t_{21}^{15}t_{22}^{35}
= t_{12}^{20}t_{13}^{60}(\det g_1)^{80}
t_{21}^{40}t_{22}^{60}.
\end{equation*}

For $x\in Z_{101}$, let 
\begin{align*}
A(x) = \begin{pmatrix}
x_{341} & x_{342} \\
x_{351} & x_{352}
\end{pmatrix},\; 
B(x)= \begin{pmatrix}
x_{243} & x_{144} \\ 
x_{253} & x_{154}
\end{pmatrix}.
\end{align*}
We identify $Z_{101}\cong \m_2\oplus \m_2\oplus 1$
by the map $Z_{101}\ni x\mapsto (A(x),B(x),x_{234})$. 
Note that the second $\m_2$ is not an irreducible \rep{} 
of $\m_{\be_{101}}$. We sometimes consider $\m_2$ 
which is a reducible \rep{} of $\spl_2$, 
but shall not point this out in the following.

Let $P_1(x) = \det A(x),\; P_2(x) = \det B(x)$. 
It is easy to see that 
\begin{align*}
& A(gx) = t_{13} g_1 A(x) {}^t g_2,\;
B(gx) = g_1 B(x) 
\begin{pmatrix}
t_{12}t_{21} & 0 \\
0 & t_{11}t_{22}
\end{pmatrix}, \\
& P_1(gx) = t_{13}^2 (\det g_1)(\det g_2) P_1(x),\;
P_2(gx) = t_{11}t_{12}(\det g_1)t_{21}t_{22} P_2(x). 
\end{align*}

We put $P(x) = P_1(x)^4P_2(x)^6 x_{234}$. 
Then on $M^{\text{st}}_{\be_{101}}$, 
\begin{align*}
P(gx) & = (t_{13}^2 (\det g_1)(\det g_2))^4
(t_{11}t_{12}(\det g_1)t_{21}t_{22})^6
(t_{12}t_{13}t_{22}) P(x) \\
& = t_{11}^6t_{12}^7t_{13}^9(\det g_1)^{10}
(\det g_2)^4 t_{21}^6t_{22}^7 P(x) 
= t_{12}t_{13}^3(\det g_1)^4
t_{21}^2t_{22}^3 P(x).  
\end{align*}
Therefore, $P(x)$ is 
invariant under the action of $G_{\text{st},\be_{101}}$. 

Let $R(101)$ be the element such that 
$A(R(101))=B(R(101))=I_2$
and that the $x_{234}$-coordinate is $1$. 
Explicitly, $R(101)=e_{341}+e_{352}+e_{243}+e_{154}+e_{234}$. 
Then $P_1(R(101))=P_2(R(101))=1$. So $P(R(101))=1$ and 
$R(101)\in Z^{\sst}_{101}$. 

We show that  $Z^{\sst}_{101\,k}=M_{\be_{101}\,k}R(101)$. 
Suppose that $x\in Z^{\sst}_{101\,k}$. 
Since $P_1(x),P_2(x)\not=0$, 
there exists $g\in M_{\be_{101}\,k}$ 
such that $A(gx)=B(gx)=I_2$. 
So we may assume that $A(x)=B(x)=I_2$. 
Then elements of the form 
$t=(I_2,I_2,1,t_{12},1,t_{12}^{-1},1)$
do not change this condition. 
Since the $x_{234}$-coordinate of $tx$ 
is $t_{12}x_{234}$, there exists $t$ such that 
$tx=R(101)$. 
Therefore, $Z^{\sst}_{101\,k}=M_{\be_{101}\,k}R(101)$. 

We assume that $u_{154}=0$ and $u_{221}=0$. 
Then the four components of 
$n(u)R(101)$ are as follows: 
\begin{align*}
& \begin{pmatrix}
0 & 0 & 0 & 0 & 0 \\
0 & 0 & 0 & 0 & 0 \\
0 & 0 & 0 & 1 & 0 \\
0 & 0 & -1 & 0 & -u_{153} \\
0 & 0 & 0 & * & 0 
\end{pmatrix},\;
\begin{pmatrix}
0 & 0 & 0 & 0 & 0 \\
0 & 0 & 0 & 0 & 0 \\
0 & 0 & 0 & 0 & 1 \\
0 & 0 & 0 & 0 & u_{143} \\
0 & 0 & -1 & * & 0 \\
\end{pmatrix}, \\
& \begin{pmatrix}
0 & 0 & 0 & 0 & 0 \\
0 & 0 & 0 & 1 & 0 \\
0 & 0 & 0 & u_{132}+u_{231} & u_{232} \\
0 & -1 & * & 0 & Q_1(u)-u_{152} \\
0 & 0 &  * & * & 0 
\end{pmatrix}, \\
& \begin{pmatrix}
0 & 0 & 0 & 0 & 1 \\
0 & 0 & 1 & u_{143}+u_{243} & u_{121}+u_{153} \\
0 & -1 & 0 & Q_2(u)-u_{142}+u_{241} & Q_3(u)+u_{131}-u_{152}+u_{242} \\
0 & * & * & 0 & Q_4(u)+u_{141} \\
-1 & * & * & * & 0 
\end{pmatrix} 
\end{align*}
where $Q_1(u),Q_2(u),Q_3(u)$ do not depend on 
$u_{131},u_{141},u_{142},u_{152}$ 
and $Q_4(u)$ does not depend on $u_{141}$. 

We can apply Lemma \ref{lem:connected-criterion} 
to the map $\aff^{14}\to \aff^{10}$ defined by the sequence
\begin{align*}
& u_{143},u_{153},u_{232},u_{132}+u_{231},u_{243}+u_{143},
u_{121}+u_{153}, u_{152}-Q_1(u), \\
& u_{142}-u_{241}-Q_2(u),
u_{131}-u_{152}+u_{242}+Q_3(u),u_{141}+Q_4(u) 
\end{align*}
where $u_{231},u_{241},u_{242},u_{243}$  are extra variables.
So by Proposition \ref{prop:W-eliminate}, 
Property \ref{property:W-eliminate} holds for 
any $x\in Z^{\sst}_{101\,k}$. Therefore, 
$Y^{\sst}_{101\,k}=P_{\be_{101}\,k}R(101)$ also. 

\vskip 5pt

(24) $S_{105}$, $\be_{105}=\tfrac {1} {140} (-56,4,4,24,24,-15,5,5,5)$

We identify the element 
$(\diag(t_1,g_{11},g_{12}),\diag(t_2,g_2))
\in M_{[1,3],[1]}=M_{\be_{105}}$ 
with the element 
$g=(g_2,g_{11},g_{12},t_1,t_2)\in \gl_3\times \gl_2^2\times \gl_1^2$. 
On $M^{\text{st}}_{\be_{105}}$, 
\begin{equation*}
\chi_{{105}}(g) = 
t_1^{-56}(\det g_{11})^4(\det g_{12})^{24}
t_2^{-15}(\det g_2)^5
= (\det g_{11})^{60}(\det g_{12})^{80}(\det g_2)^{20}.
\end{equation*}

For $x\in Z_{105}$, let 
\begin{equation*}
A_1(x) = 
\begin{pmatrix}
x_{242} & x_{252} \\
x_{342} & x_{352} 
\end{pmatrix},\;
A_2(x) = 
\begin{pmatrix}
x_{243} & x_{253} \\
x_{343} & x_{353} 
\end{pmatrix},\;
A_3(x) = 
\begin{pmatrix}
x_{244} & x_{254} \\
x_{344} & x_{354} 
\end{pmatrix}
\end{equation*}
and $A(x)=(A_1(x),A_2(x),A_3(x))$. 
We identify $Z_{105}\cong \aff^3\otimes \m_2\oplus 1$ 
by the map $Z_{105}\ni x \mapsto (A(x),x_{451})$. 
The action of $(g_2,g_{11},g_{12},1,1)$ on $A(x)$ 
is the same as that of Proposition \ref{prop:322-invariant}. 
Also if $t=(I_3,I_2,I_2,t_1,t_2)$ then 
$A(tx)=A(x)$. 

Let $P_1(x)$ be the degree $6$ polynomial of $A(x)$ 
obtained by Proposition \ref{prop:322-invariant}. 
Then  
\begin{equation*}
P_1(gx) = (\det g_{11})^3(\det g_{12})^3(\det g_2)^2 P_1(x). 
\end{equation*}

We put $P(x)=P_1(x)x_{451}$. Then on $M^{\text{st}}_{\be_{105}}$, 
\begin{align*}
P(gx) & = ((\det g_{11})^3(\det g_{12})^3(\det g_2)^2)
((\det g_{12})t_2)P(x) \\
& = (\det g_{11})^3(\det g_{12})^4t_2(\det g_2)^2 P(x)
= (\det g_{11})^3(\det g_{12})^4(\det g_2) P(x). 
\end{align*}
Therefore, $P(x)$ is 
invariant under the action of $G_{\text{st},\be_{105}}$. 

Let $R(105)$ be the element such that 
$A(R(105))=R_{322}$ and the $x_{451}$-coordinate is $1$. 
Explicitly, 
$R(105)=e_{451}-e_{242}+e_{352}+e_{253}+e_{344}$. 
Then $P(R(105))=1$ and so $R(105)\in Z^{\sst}_{105}$.

We show that  $Z^{\sst}_{105\,k}=M_{\be_{105}\,k}R(105)$.  
Suppose that $x\in Z^{\sst}_{105\,k}$.
Then $x_{451}\not=0$. 
By Proposition \ref{prop:322-invariant}, 
there exists $g\in M_{\be_{105}\,k}$ such that 
$A(gx)=A(R(105))$. So we may assume that 
$A(x)=A(R(105))$. If 
$t=(I_3,I_2,I_2,1,x_{451}^{-1})$. Then 
$gx=R(105)$. 
Therefore, $Z^{\sst}_{105\,k}=M_{\be_{105}\,k}R(105)$.

We assume that $u_{132}=u_{154}=0$ 
and $u_{2ij}=0$ unless $j=1$. 
Then the four components of 
$n(u)R(105)$ are as follows: 
\begin{align*}
& \begin{pmatrix}
0 & 0 & 0 & 0 & 0 \\
0 & 0 & 0 & 0 & 0 \\
0 & 0 & 0 & 0 & 0 \\
0 & 0 & 0 & 0 & 1 \\
0 & 0 & 0 & -1 & 0 
\end{pmatrix},\;
\begin{pmatrix}
0 & 0 & 0 & 0 & 0 \\
0 & 0 & 0 & -1 & 0 \\
0 & 0 & 0 & 0 & 1 \\
0 & 1 & 0 & 0 & u_{143}+u_{152}+u_{221} \\
0 & 0 & -1 & * & 0 \\
\end{pmatrix}, \\
& \begin{pmatrix}
0 & 0 & 0 & 0 & 0 \\
0 & 0 & 0 & 0 & 1 \\
0 & 0 & 0 & 0 & 0 \\
0 & 0 & 0 & 0 & u_{142}+u_{231} \\
0 & -1 &  0 & * & 0 
\end{pmatrix}, \;
\begin{pmatrix}
0 & 0 & 0 & 0 & 0 \\
0 & 0 & 0 & 0 & 0 \\
0 & 0 & 0 & 1 & 0 \\
0 & 0 & -1 & 0 & -u_{153}+u_{241} \\
0 & 0 & 0 & * & 0 
\end{pmatrix}. 
\end{align*}

We can apply Lemma \ref{lem:connected-criterion} 
to the map $\aff^{11}\to \aff^3$ defined by the sequence
\begin{align*}
& u_{143}+u_{152}+u_{221},u_{142}+u_{231},u_{153}-u_{241}
\end{align*}
where $u_{121},u_{131},u_{141},u_{151},u_{152},u_{221},u_{231},u_{241}$  
are extra variables.
So by Proposition \ref{prop:W-eliminate}, 
Property \ref{property:W-eliminate} holds for 
any $x\in Z^{\sst}_{105\,k}$. Therefore, 
$Y^{\sst}_{105\,k}=P_{\be_{105}\,k}R(105)$ also. 

\vskip 5pt

(25)  $S_{106}$, $\be_{106}=\tfrac {1} {260} (-24,-4,-4,16,16,-65,15,15,35)$

We identify the element 
$(\diag(t_1,g_{11},g_{12}),\diag(t_{21},g_2,t_{22}))
\in M_{[1,3],[1,3]}=M_{\be_{106}}$ 
with the element 
$g=(g_{12},g_{11},g_2,t_1,t_{21},t_{22})\in \gl_2^3\times \gl_1^3$. 
On $M^{\text{st}}_{\be_{106}}$, 
\begin{align*}
\chi_{{106}}(g) 
& = t_1^{-24}(\det g_{11})^{-4}(\det g_{12})^{16}
t_{21}^{-65}(\det g_2)^{15}t_{22}^{35} \\
& = (\det g_{11})^{20}(\det g_{12})^{40}
(\det g_2)^{80}t_{22}^{100}. 
\end{align*}

For $x\in Z_{106}$, let 
\begin{equation*}
A(x) = (x_{242},x_{252},x_{342},x_{352},x_{243}\ccd x_{353}),\;
v(x) = [x_{144},x_{154}].  
\end{equation*}
We regard $A(x)\in \Lam^{2,1}_{1,[4,5]}\otimes  
\Lam^{2,1}_{1,[2,3]}\otimes \Lam^{2,1}_{2,[2,3]}$, 
$v(x) \in \Lam^{2,1}_{1,[4,5]}$. 
We identify $Z_{106}$ with 
$\aff^2\otimes \aff^2\otimes \aff^2\oplus \aff^2\oplus 1$
by the map $Z_{106}\ni x\mapsto (A(x),v(x),x_{234})$. 

Let $P_1(x),P_2(x)$ be the polynomials of $(A(x),v(x))$
defined in the case (a) of Section \ref{sec:rational-orbits-222-22}. 
By (\ref{eq:P1P2P3}), 
\begin{align*}
& A(gx) = (g_{12},g_{11},g_2)A(x),\;
v(gx) = t_1t_{22}g_{12}v(x), \\
& P_1(gx) = (\det g_{11})^2 (\det g_{12})^2 (\det g_2)^2 P_2(x), \\
& P_2(gx) = t_1^2 (\det g_{11}) (\det g_{12})^2 (\det g_2) t_{22}^2P_1(x).
\end{align*}

We put $P(x)=P_1(x)P_2(x)^2x_{234}$. 
Then on $M^{\text{st}}_{\be_{106}}$,  
\begin{align*}
P(gx) 
& = ((\det g_{11})^2 (\det g_{12})^2 (\det g_2)^2) 
(t_1^2 (\det g_{11}) (\det g_{12})^2 (\det g_2)t_{22}^2)^2 \\
& \quad \times ((\det g_{11})t_{22})P(x) \\
& = t_1^4(\det g_{11})^5 (\det g_{12})^6 (\det g_2)^4 t_{22}^5 P(x) 
=  (\det g_{11}) (\det g_{12})^2 (\det g_2)^4 t_{22}^5 P(x). 
\end{align*}
Note that $P_1(x)$ depends only on $A(x)$ and 
$P_2(x)$ is homogeneous of degree $2$ with respect to 
each of $A(x),v(x)$. Therefore, $P(x)$ is 
invariant under the action of $G_{\text{st},\be_{106}}$. 

If $x\in Z^{\sst}_{106\,k}$ then 
applying $t=(I_2,I_2,I_2,1,1,x_{234}^{-1})$, 
we may assume that $x_{234}=1$. 
Then elements of the form 
$g=(g_{12},g_{11},g_2,(\det g_{11}),1,(\det g_{11})^{-1})$ 
do not change this condition 
and $A(gx)=(g_{12},g_{11},g_2)A(x),v(gx)=g_{12}v(x)$. 
So by Proposition \ref{prop:222+2-orbit}, 
$M_{\be_{106}\,k}\backslash Z^{\sst}_{106\,k}$ 
is in bijective correspondence with $\Ex_2(k)$. 

Let $R(106)\in Z_{106}$ be the element such that 
\begin{equation*}
A(R(106))=(1,0\ccd 0,1),\; v(R(106))=[1,1]
\end{equation*}
and that the $x_{234}$-coordinate is $1$. 
Explicitly, 
$R(106)=e_{242}+e_{353}+e_{144}+e_{154}+e_{234}$. 
Then $R(106)\in Z^{\sst}_{106}$ 
(see (\ref{eq:w-defn-sec10})).

We assume that $u_{132}=u_{154}=0$ and $u_{232}=0$. 
Then the first component of 
$n(u)R(106)$ is $0$ and the remaining components are 
are as follows: 
\begin{align*}
& \begin{pmatrix}
0 & 0 & 0 & 0 & 0 \\
0 & 0 & 0 & 1 & 0 \\
0 & 0 & 0 & 0 & 0 \\
0 & -1 & 0 & 0 & -u_{152} \\
0 & 0 & 0 & * & 0 
\end{pmatrix}, \; 
\begin{pmatrix}
0 & 0 & 0 & 0 & 0 \\
0 & 0 & 0 & 0 & 0 \\
0 & 0 & 0 & 0 & 1 \\
0 & 0 & 0 & 0 & u_{143} \\
0 & 0 & -1 & * & 0 \\
\end{pmatrix}, \\
& \begin{pmatrix}
0 & 0 & 0 & 1 & 1 \\
0 & 0 & 1 & u_{121}+u_{143}+u_{242} & u_{121}+u_{153} \\
0 & -1 & 0 & u_{131}-u_{142} & u_{131}-u_{152}+u_{243} \\
-1 & * & * & 0 & Q(u)+u_{141}-u_{151} \\
-1 & * &  * & * & 0 
\end{pmatrix}
\end{align*}
where $Q(u)$ does not depend on $u_{141},u_{151}$. 

We can apply Lemma \ref{lem:connected-criterion} 
to the map $\aff^{13}\to \aff^7$ defined by the sequence
\begin{align*}
& u_{143},u_{152},u_{121}+u_{143}+u_{242},
u_{153}+u_{121},u_{131}-u_{142}, \\
& u_{243}+u_{131}-u_{152},u_{141}-u_{151}+Q(u)
\end{align*}
where $u_{142},u_{151},u_{221},u_{231},u_{241},u_{242}$  
are extra variables.
So by Proposition \ref{prop:W-eliminate}, 
Property \ref{property:W-eliminate} holds for 
any $x\in Z^{\sst}_{106\,k}$. Therefore, 
$P_{\be_{106}\,k}\backslash Y^{\sst}_{106\,k}$ 
is in bijective correspondence with $\Ex_2(k)$ also.  

\vskip 5pt

(26)  $S_{107}$, 
$\be_{107}= \tfrac {1} {60} (-4,0,0,2,2,-3,-1,-1,5)$ 

We identify the element 
$(\diag(t_1,g_{11},g_{12}),\diag(t_{21},g_2,t_{22}))
\in M_{[1,3],[1,3]}=M_{\be_{107}}$ 
with the element 
$g=(g_{11},g_{12},g_2,t_1,t_{21},t_{22})\in \gl_2^3\times \gl_1^3$. 
On $M^{\text{st}}_{\be_{107}}$, 
\begin{equation*}
\chi_{{107}}(g) = 
t_1^{-4}(\det g_{12})^2
t_{21}^{-3}(\det g_2)^{-1}t_{22}^5
= (\det g_{11})^4(\det g_{12})^6
(\det g_2)^2t_{22}^8.
\end{equation*}

For $x\in Z_{107}$, let 
\begin{equation*}
A(x) = (x_{242},x_{252},x_{342},x_{352},x_{243}\ccd x_{353}),\; 
v(x) = [x_{124},x_{134}].
\end{equation*}
We regard $A(x) \in \Lam^{2,1}_{1,[2,3]}\otimes 
\Lam^{2,1}_{1,[4,5]}\otimes \Lam^{2,1}_{2,[2,3]}$, 
$v(x)\in \Lam^{2,1}_{1,[2,3]}$. 
We identify $Z_{107}\cong \aff^2\otimes \aff^2\otimes \aff^2\oplus \aff^2\oplus 1$
by the map $Z_{107}\ni x\mapsto (A(x),v(x),x_{451})$. 

Let $P_1(x),P_2(x)$ be the polynomials of $(A(x),v(x))$
defined in the case (a) of Section \ref{sec:rational-orbits-222-22}. 
By (\ref{eq:P1P2P3}), 
\begin{align*}
& A(gx) = (g_{11},g_{12},g_2)A(x),\; v(gx) =  t_1t_{22}g_{11}v(x), \\
& P_1(gx) = (\det g_{11})^2 (\det g_{12})^2 (\det g_2)^2 P_2(x), \\
& P_2(gx) = t_1^2 (\det g_{11})^2 (\det g_{12}) (\det g_2)t_{22}^2P_1(x).
\end{align*}

We put $P(x)=P_1(x)P_2(x)^5x_{451}^6$. 
Then on $M^{\text{st}}_{\be_{107}}$,  
\begin{align*}
P(gx) 
& = ((\det g_{11})^2 (\det g_{12})^2 (\det g_2)^2) 
(t_1^2 (\det g_{11})^2 (\det g_{12})(\det g_2)t_{22}^2)^5 \\
& \quad \times ((\det g_{12})t_{21})^6 P(x) \\
& = t_1^{10}(\det g_{11})^{12} (\det g_{12})^{13} t_{21}^6 (\det g_2)^7 t_{22}^{10} 
=  (\det g_{11})^2 (\det g_{12})^3 (\det g_2) t_{22}^4 P(x). 
\end{align*}
Therefore, $P(x)$ is 
invariant under the action of $G_{\text{st},\be_{107}}$. 

Similarly as in the previous case, 
$M_{\be_{107}\,k}\backslash Z^{\sst}_{107\,k}$ is in bijective 
correspondence with $\Ex_2(k)$ by Proposition \ref{prop:222+2-orbit}. 
Note that the action of $\gl_1^3$ can be absorbed into the action of
$\gl_2^3$. 

Let $R(107)\in Z_{107}$ be the element such that 
$A(R(107))=(1,0\ccd 0,1)$, $v(R(107))$ $=[1,1]$
and that the $x_{451}$-coordinate is $1$ (see (\ref{eq:w-defn-sec10})). 
Explicitly, $R(107)=e_{451}+e_{242}+e_{353}+e_{124}+e_{134}$. 
Then $R(107)\in Z^{\sst}_{107}$. 

We assume that $u_{132}=u_{154}=0$ and $u_{232}=0$. 
Then the four components of 
$n(u)R(107)$ are as follows: 
\begin{align*}
& \begin{pmatrix}
0 & 0 & 0 & 0 & 0 \\
0 & 0 & 0 & 0 & 0 \\
0 & 0 & 0 & 0 & 0 \\
0 & 0 & 0 & 0 & 1 \\
0 & 0 & 0 & -1 & 0 
\end{pmatrix},\;
\begin{pmatrix}
0 & 0 & 0 & 0 & 0 \\
0 & 0 & 0 & 1 & 0 \\
0 & 0 & 0 & 0 & 0 \\
0 & -1 & 0 & 0 & -u_{152}+u_{221} \\
0 & 0 & 0 & * & 0 
\end{pmatrix}, \;
\begin{pmatrix}
0 & 0 & 0 & 0 & 0 \\
0 & 0 & 0 & 0 & 0 \\
0 & 0 & 0 & 0 & 1 \\
0 & 0 & 0 & 0 & u_{143}+u_{231} \\
0 & 0 & -1 & * & 0 \\
\end{pmatrix}, \\
& \begin{pmatrix}
0 & 1 & 1 & u_{142}+u_{143} & u_{152}+u_{153} \\
-1 & 0 & u_{121}-u_{131} & Q_1(u)-u_{141}+u_{242} & Q_2(u)-u_{151} \\
-1 & * & 0 & Q_3(u)-u_{141} & Q_4(u)-u_{151}+u_{243} \\
* & * & * & 0 & Q_5(u)+u_{241} \\
* & * &  * & * & 0 
\end{pmatrix}
\end{align*}
where $Q_1(u),Q_2(u),Q_3(u),Q_4(u)$
do not depend on $u_{141},u_{151},u_{241},u_{242},u_{243}$ 
and $Q_5(u)$ does not depend on $u_{241}$. 

We can apply Lemma \ref{lem:connected-criterion} 
to the map $\aff^{13}\to \aff^{10}$ defined by the sequence
\begin{align*}
& u_{152}-u_{221},u_{143}+u_{231},u_{121}-u_{131},
u_{142}+u_{143},u_{153}+u_{152}, u_{141}-Q_3(u), \\
& u_{242}-u_{141}+Q_1(u),u_{151}-Q_2(u),
u_{243}-u_{151}+Q_4(u),u_{241}+Q_5(u)
\end{align*}
where $u_{131},u_{221},u_{231}$  
are extra variables.
So by Proposition \ref{prop:W-eliminate}, 
Property \ref{property:W-eliminate} holds for 
any $x\in Z^{\sst}_{107\,k}$. Therefore, 
$P_{\be_{107}\,k}\backslash Y^{\sst}_{107\,k}$ 
is in bijective correspondence with $\Ex_2(k)$ also.  

\vskip 5pt

(27)  $S_{108}$, $\be_{108}=\tfrac {1} {20} (-2,-2,0,2,2,-1,-1,1,1)$ 

We identify the element 
$(\diag(g_{11},t_1,g_{12}),\diag(g_{21},g_{22}))
\in M_{[2,3],[2]}=M_{\be_{108}}$ 
with the element 
$g=(g_{11},g_{12},g_{21},g_{22},t_1)\in \gl_2^4\times \gl_1$. 
On $M^{\text{st}}_{\be_{108}}$, 
\begin{equation*}
\chi_{{108}}(g) = 
(\det g_{11})^{-2}(\det g_{12})^2
(\det g_{21})^{-1}(\det g_{22})
= t_1^2(\det g_{12})^4(\det g_{22})^2. 
\end{equation*}

For $x\in Z_{108}$, let 
\begin{equation*}
A(x) = (x_{143},x_{153},x_{243},x_{253},x_{144}\ccd x_{254}),\;
B(x) = 
\begin{pmatrix}
x_{341} & x_{342} \\
x_{351} & x_{352}
\end{pmatrix}.
\end{equation*}
We regard 
$A(x)\in \Lam^{2,1}_{1,[1,2]}\otimes 
\Lam^{2,1}_{1,[4,5]}\otimes \Lam^{2,1}_{2,[3,4]}$, 
$B(x)\in \Lam^{2,1}_{1,[4,5]}\otimes \Lam^{2,1}_{2,[1,2]}$. 
We identify $Z_{108}\cong \aff^2\otimes \aff^2\otimes \aff^2\oplus \m_2$
by the map $Z_{108}\ni x\mapsto (A(x),B(x))$.

It is easy to see that 
\begin{equation*}
A(gx) = (g_{11},g_{12},g_{22})A(x),\;
B(gx) = t_1g_{12}B(x){}^tg_{21}.
\end{equation*}
Let $P_1(x)$ be the degree $4$ polynomial of $A(x)$ 
obtained by Proposition \ref{prop:322-invariant}. 
we put $P_2(x)=\det B(x)$. Then 
\begin{equation*}
P_1(gx) = (\det g_{11})^2(\det g_{12})^2(\det g_{22})^2 P_1(x),\; 
P_2(gx) = t_1^2 (\det g_{12}) (\det g_{21}) P_2(x).  
\end{equation*}

We put $P(x)=P_1(x)^3P_2(x)^4$. Then on $M^{\text{st}}_{\be_{108}}$, 
\begin{align*}
P(gx) 
& = ((\det g_{11})^2(\det g_{12})^2(\det g_{22})^2)^3
(t_1^2 (\det g_{12}) (\det g_{21}))^4 P(x) \\
& = (\det g_{11})^6 t_1^8 
(\det g_{12})^{10}(\det g_{21})^4(\det g_{22})^6 P(x) \\
& = t_1^2 (\det g_{12})^4 (\det g_{22})^2 P(x).
\end{align*}
Therefore, $P(x)$ is 
invariant under the action of $G_{\text{st},\be_{108}}$. 

Suppose that $x\in Z^{\sst}_{108\,k}$. 
It is easy to see that 
there exists $g\in M_{\be_{108}\,k}$ 
such that $B(gx)=I_2$. 
If $B(x)=I_2$ then $B(gx)=B(x)$ if and only if 
$t_1g_{12}{}^t g_{21}=I_2$. 
If $g=(g_{11},g_{12},t_1^{-1}{}^tg_{12}^{-1},g_{22},t_1)$
then $A(gx)=(g_{11},g_{12},g_{22})A(x)$. 
Therefore, $M_{\be_{108}\,k}\backslash Z^{\sst}_{108\,k}$ 
is in bijective correspondence with the set of orbits 
for the \pv{} in Proposition \ref{prop:222-rat-orbits}.
Therefore, $M_{\be_{108}\,k}\backslash Z^{\sst}_{108\,k}$ 
is in bijective correspondence with $\Ex_2(k)$. 

Let $R(108)\in Z_{108}$ be element such that 
\begin{math}
A(R(108)) = (1,0\ccd 0,1),\;
B(R(108)) = I_2
\end{math}
(see (\ref{eq:w-efn-222-matrix})).
Explicitly, 
$R(108)=e_{341}+e_{352}+e_{143}+e_{254}$. 
Then $P(R(108))=1$ and so $R(108)\in Z^{\sst}_{108}$.

We assume that $u_{121}=u_{154}=0$ 
and $u_{2ij}=0$ unless $i=3,4,j=1,2$. 
Then the four components of 
$n(u)R(108)$ are as follows: 
\begin{align*}
& \begin{pmatrix}
0 & 0 & 0 & 0 & 0 \\
0 & 0 & 0 & 0 & 0 \\
0 & 0 & 0 & 1 & 0 \\
0 & 0 & -1 & 0 & -u_{153} \\
0 & 0 & 0 & * & 0 
\end{pmatrix},\;
\begin{pmatrix}
0 & 0 & 0 & 0 & 0 \\
0 & 0 & 0 & 0 & 0 \\
0 & 0 & 0 & 0 & 1 \\
0 & 0 & 0 & 0 & u_{143} \\
0 & 0 & -1 & * & 0 
\end{pmatrix}, \\
& \begin{pmatrix}
0 & 0 & 0 & 1 & 0 \\
0 & 0 & 0 & 0 & 0 \\
0 & 0 & 0 & u_{131}+u_{231} & u_{232} \\
-1 & 0 & * & 0 & Q_1(u)-u_{151} \\
0 & 0 & * & * & 0 \\
\end{pmatrix}, \;
\begin{pmatrix}
0 & 0 & 0 & 0 & 0 \\
0 & 0 & 0 & 0 & 1 \\
0 & 0 & 0 & u_{241} & u_{132}+u_{242} \\
0 & 0 & * & 0 & Q_2(u)+u_{142} \\
0 & -1 & * & * & 0 
\end{pmatrix}
\end{align*}
where $Q_1(u),Q_2(u)$ do not depend on $u_{142},u_{151}$.

We can apply Lemma \ref{lem:connected-criterion} 
to the map $\aff^{12}\to \aff^8$ defined by the sequence
\begin{align*}
& u_{143},u_{153},u_{232},u_{241},
u_{131}+u_{231},u_{132}+u_{242},
u_{151}-Q_1(u),u_{142}+Q_2(u)
\end{align*}
where $u_{141},u_{152},u_{231},u_{242}$  
are extra variables.
So by Proposition \ref{prop:W-eliminate}, 
Property \ref{property:W-eliminate} holds for 
any $x\in Z^{\sst}_{107\,k}$. Therefore, 
$P_{\be_{107}\,k}\backslash Y^{\sst}_{107\,k}$ 
is in bijective correspondence with $\Ex_2(k)$ also.  

\vskip 5pt

(28)  $S_{110}$, 
$\be_{110}=\tfrac {1} {60} (-4,-4,-4,6,6,-15,5,5,5)$ 

We identify the element 
$(\diag(g_{11},g_{12}),\diag(t_2,g_2))
\in M_{[3],[1]}=M_{\be_{110}}$ 
with the element 
$g=(g_{11},g_2,g_{12},t_2)\in \gl_3^2\times \gl_2\times \gl_1$. 
On $M^{\text{st}}_{\be_{110}}$, 
\begin{equation*}
\chi_{{110}}(g) = 
(\det g_{11})^{-4}(\det g_{12})^6
t_2^{-15}(\det g_2)^5
= (\det g_{12})^{10}(\det g_2)^{20}. 
\end{equation*}

For $x\in Z_{110}$, let 
\begin{equation*}
A_1(x) = 
\begin{pmatrix}
x_{142} & x_{143} & x_{144} \\
x_{242} & x_{243} & x_{244} \\
x_{342} & x_{343} & x_{344}
\end{pmatrix},\; 
A_2(x) = 
\begin{pmatrix}
x_{152} & x_{153} & x_{154} \\
x_{252} & x_{253} & x_{254} \\
x_{352} & x_{353} & x_{354}
\end{pmatrix}
\end{equation*}
and $A(x)=(A_1(x),A_2(x))$. 
We identify $Z_{110}\cong \m_3\otimes \aff^2$ 
by the map $Z_{110}\ni x \mapsto A(x)$. 

It is easy to see that 
\begin{equation*}
\begin{pmatrix}
A_1(gx) \\
A_2(gx)
\end{pmatrix}
= g_{12}\begin{pmatrix}
g_{11}A_1(x){}^tg_2 \\
g_{11}A_2(x){}^tg_2.
\end{pmatrix}. 
\end{equation*}
So $gx$ does not depend on $t_2$ and 
this is the tensor product
of standard \rep s of $g_{11},g_{12},g_2$. 

Let $P(x)$ be the degree $12$ polynomial on 
$Z_{\be_{110}}\cong\aff^3\otimes \aff^3\otimes \aff^2$ 
obtained by Proposition \ref{prop:332-invariant}. 
Then on $M^{\text{st}}_{\be_{110}}$, 
\begin{equation*}
P(gx) = (\det g_{11})^4(\det g_{12})^6(\det g_2)^4 P(x)
= (\det g_{12})^2(\det g_2)^4 P(x).
\end{equation*}
Therefore, $P(x)$ is 
invariant under the action of $G_{\text{st},\be_{110}}$. 

By Proposition \ref{prop:Rational-orbits-3-cases}, 
$M_{\be_{110}\,k}\backslash Z^{\sst}_{110\,k}$
is in bijective correspondence with $\Ex_3(k)$. 

Let $R(110)\in Z_{110}$ be the element such that 
\begin{equation*}
A_1(R(110))=\diag(1,-1,0),\; A_2(R(110))=\diag(0,1,-1)
\end{equation*}
(see (\ref{eq:sec6-w-defn-sec4})). 
Explicitly, $R(110)=e_{142}-e_{243}+e_{253}-e_{354}$. 
Then $R(110)\in Z^{\sst}_{110}$ by Proposition \ref{prop:332-invariant}. 

We assume that  $u_{1ij}=0$ unless $i=4,5,j=1,2,3$  
and $u_{2ij}=0$ unless $j=1$. 
Then the first component of 
$n(u)R(110)$ is $0$ and the remaining components are 
are as follows: 
\begin{align*}
& \begin{pmatrix}
0 & 0 & 0 & 1 & 0 \\
0 & 0 & 0 & 0 & 0 \\
0 & 0 & 0 & 0 & 0 \\
-1 & 0 & 0 & 0 & -u_{151} \\
0 & 0 & 0 & * & 0 
\end{pmatrix},\;
\begin{pmatrix}
0 & 0 & 0 & 0 & 0 \\
0 & 0 & 0 & -1 & 1 \\
0 & 0 & 0 & 0 & 0 \\
0 & 1 & 0 & 0 & u_{142}+u_{152} \\
0 & -1 & 0 & * & 0 
\end{pmatrix}, \;
\begin{pmatrix}
0 & 0 & 0 & 0 & 0 \\
0 & 0 & 0 & 0 & 0 \\
0 & 0 & 0 & 0 & -1 \\
0 & 0 & 0 & 0 & -u_{143} \\
0 & 0 & 1 & * & 0 \\
\end{pmatrix}.
\end{align*}

We can apply Lemma \ref{lem:connected-criterion} 
to the map $\aff^9\to \aff^3$ defined by the sequence
\begin{math}
u_{143},u_{151},u_{142}+u_{152}
\end{math}
where $u_{141},u_{152},u_{153},u_{221}\ccd u_{241}$  
are extra variables.
So by Proposition \ref{prop:W-eliminate}, 
Property \ref{property:W-eliminate} holds for 
any $x\in Z^{\sst}_{110\,k}$. Therefore, 
$P_{\be_{110}\,k}\backslash Y^{\sst}_{110\,k}$ 
is in bijective correspondence with $\Ex_3(k)$ also.

\vskip 5pt

(29)  $S_{113}$, $\be_{113}=\tfrac {1} {60} (-9,-4,-4,6,11,-10,0,5,5)$ 

We identify the element 
$(\diag(t_{11},g_1,t_{12},t_{13}),\diag(t_{21},t_{22},g_2))
\in M_{[1,3,4],[1,2]}=M_{\be_{113}}$ 
with the element 
$g=(g_1,g_2,t_{11}\ccd t_{22})\in \gl_2^2\times \gl_1^5$. 
On $M^{\text{st}}_{\be_{113}}$, 
\begin{equation*}
\chi_{{113}}(g) = 
t_{11}^{-9}(\det g_1)^{-4}t_{12}^6t_{13}^{11}
t_{21}^{-10}(\det g_2)^5
= (\det g_1)^5t_{12}^{15}t_{13}^{20}
t_{22}^{10}(\det g_2)^{15}. 
\end{equation*}

For $x\in Z_{113}$, let 
\begin{equation*}
A(x) = 
\begin{pmatrix}
x_{243} & x_{244} \\
x_{343} & x_{344}
\end{pmatrix},\;
v_1(x) = 
\begin{pmatrix}
x_{252} \\ x_{352}
\end{pmatrix},\;
v_2(x) = 
\begin{pmatrix}
x_{153} \\x_{154}
\end{pmatrix}.
\end{equation*}
We regard that 
$A(x)\in \Lam^{2,1}_{1,[2,3]}\otimes \Lam^{2,1}_{2,[3,4]}$, 
$v_1(x)\in \Lam^{2,1}_{1,[2,3]}$, 
$v_2(x)\in \Lam^{2,1}_{2,[3,4]}$. 
We identify 
$Z_{113}\cong \m_2\oplus \aff^2\oplus \aff^2\oplus 1$ 
by the map $Z_{113}\ni x\mapsto (A(x),v_1(x),v_2(x),x_{451})$.

It is easy to see that 
\begin{equation*}
A(gx) = t_{12}g_1A(x){}^tg_2,\;
v_1(gx) = t_{13}t_{22}g_1v_1(x),\;
v_2(gx) = t_{11}t_{13}g_2v_2(x). 
\end{equation*}
Let $P_1(x)$ be the degree $3$ polynomial of 
$(A(x),v_1(x),v_2(x))$ obtained by 
Proposition \ref{prop:M2-2-2-invariant} 
and $P_2(x)=\det A(x)$. Then 
\begin{equation*}
P_1(gx) = t_{11}(\det g_1)t_{12}t_{13}^2t_{22}(\det g_2) P_1(x),\;
P_2(gx) = (\det g_1)t_{12}^2(\det g_2) P_2(x). 
\end{equation*}

We put $P(x) = P_1(x)^3P_2(x)x_{451}$. 
Then on $M^{\text{st}}_{\be_{113}}$, 
\begin{align*}
P(gx) 
& = (t_{11}(\det g_1)t_{12}t_{13}^2t_{22}(\det g_2))^3
((\det g_1)t_{12}^2(\det g_2))(t_{12}t_{13}t_{21}) P(x) \\
& = t_{11}^3(\det g_1)^4t_{12}^6t_{13}^7t_{21}t_{22}^3(\det g_2)^4 P(x) 
= (\det g_1)t_{12}^3t_{13}^4t_{22}^2(\det g_2)^3 P(x). 
\end{align*}
Therefore, $P(x)$ is 
invariant under the action of $G_{\text{st},\be_{113}}$. 

Let $R(113)\in Z_{113}$ be element such that 
\begin{equation*}
A(R(113)) = I_2,\;
v_1(R(113)) = v_2(R(113)) = [1,0]
\end{equation*}
and that the $x_{451}$-coordinate is $1$ 
(see Proposition~\ref{prop:M2-2-2-invariant}). 
Explicitly, 
$R(113)=e_{451}+e_{252}+e_{153}+e_{243}+e_{344}$. 
Then $P_1(R(113))=P_2(R(113))=1$.  So $P(R(113))=1$ and 
$R(113)\in Z^{\sst}_{113}$. 

We show that  $Z^{\sst}_{113\,k}=M_{\be_{113}\,k}R(113)$. 
Suppose that $x\in Z^{\sst}_{113\,k}$. 
By Proposition \ref{prop:M2-2-2-invariant}, 
there exists $g\in M_{\be_{113}\,k}$ 
such that $A(gx)=I_2,v_1(gx) = v_1(gx) = [1,0]$. 
So we may assume that $A(x)=I_2,v_1(x)=v_2(x)=[1,0]$. 
Let $t=(I_2,I_2,1,1,1,t_{21},1)$. 
Then $A(tx)=I_2,v_1(tx)=v_2(tx)=[1,0]$ and the 
$x_{451}$-coordinate of $tx$ is $t_{21}x_{451}$.
Therefore, $Z^{\sst}_{113\,k}=M_{\be_{113}\,k}R(113)$. 

We assume that  $u_{132}=0$ and $u_{243}=0$. 
Then the four components of 
$n(u)R(113)$ are as follows: 
\begin{align*}
& \begin{pmatrix}
0 & 0 & 0 & 0 & 0 \\
0 & 0 & 0 & 0 & 0 \\
0 & 0 & 0 & 0 & 0 \\
0 & 0 & 0 & 0 & 1 \\
0 & 0 & 0 & -1 & 0 
\end{pmatrix},\;
\begin{pmatrix}
0 & 0 & 0 & 0 & 0 \\
0 & 0 & 0 & 0 & 1 \\
0 & 0 & 0 & 0 & 0 \\
0 & 0 & 0 & 0 & u_{142}+u_{221} \\
0 & -1 & 0 & * & 0 
\end{pmatrix}, \\
& \begin{pmatrix}
0 & 0 & 0 & 0 & 1 \\
0 & 0 & 0 & 1 & u_{121}+u_{154}+u_{232} \\
0 & 0 & 0 & 0 & u_{131} \\
0 & -1 & 0 & 0 & Q_1(u) + u_{141}-u_{152}+u_{231} \\
-1 & * & * & * & 0 
\end{pmatrix}, \;
\begin{pmatrix}
0 & 0 & 0 & 0 & 0 \\
0 & 0 & 0 & 0 & u_{242} \\
0 & 0 & 0 & 1 & u_{154} \\
0 & 0 & -1 & 0 & Q_2(u)-u_{153}+u_{241} \\
0 & * & * & * & 0 \\
\end{pmatrix}
\end{align*}
where $Q_1(u),Q_2(u)$ do not depend on 
$u_{141},u_{152},u_{153},u_{231},u_{241}$. 

We can apply Lemma \ref{lem:connected-criterion} 
to the map $\aff^{14}\to \aff^7$ defined by the sequence
\begin{align*}
& u_{131},u_{154},u_{242},
u_{142}+u_{221},u_{121}+u_{154}+u_{232}, \\
& u_{141}-u_{152}+u_{231}+Q_1(u),
u_{153}-u_{241}-Q_2(u)
\end{align*}
where $u_{143},u_{151},u_{152},u_{221},u_{231},u_{232},u_{241}$  
are extra variables.
So by Proposition \ref{prop:W-eliminate}, 
Property \ref{property:W-eliminate} holds for 
any $x\in Z^{\sst}_{113\,k}$. Therefore, 
$Y^{\sst}_{113\,k}=P_{\be_{113}\,k}R(113)$ also.

\vskip 5pt

(30) $S_{121}$, 
$\be_{121}=\tfrac {1} {140} (-4,-4,0,0,8,-3,-3,1,5)$ 

We identify the element 
$(\diag(g_{11},g_{12},t_1),\diag(g_2,t_{21},t_{22}))
\in M_{[2,4],[2,3]}=M_{\be_{121}}$ 
with the element 
$g=(g_{11},g_{12},g_2,t_1,t_{21},t_{22})\in \gl_2^3\times \gl_1^3$. 
On $M^{\text{st}}_{\be_{121}}$, 
\begin{equation*}
\chi_{{121}}(g) = 
(\det g_{11})^{-4}t_1^8
(\det g_2)^{-3}t_{21}t_{22}^5
= (\det g_{12})^4t_1^{12}
t_{21}^4t_{22}^8.
\end{equation*}

Let 
\begin{equation*}
A(x) = 
\begin{pmatrix}
x_{151} & x_{152} \\
x_{251} & x_{252}
\end{pmatrix},\; 
B(x) = 
\begin{pmatrix}
x_{134} & x_{144} \\
x_{234} & x_{244} 
\end{pmatrix}.
\end{equation*}
We identify $Z_{121}\cong \m_2\oplus \m_2\oplus 1$ 
by the map $Z_{121}\ni x\mapsto (A(x),B(x),x_{343})$.

It is easy to see that
\begin{align*}
A(gx) & = t_1g_{11}A(x){}^tg_2,\;
B(gx) = t_{22}g_{11}B(x){}^tg_{12}
\end{align*}
Let
\begin{math}
P_1(x) = \det A(x),\; 
P_2(x) = \det B(x).
\end{math}
Then 
\begin{equation*}
P_1(gx) = (\det g_{11})t_1^2(\det g_2) P_1(x),\;
P_2(gx) = (\det g_{11})(\det g_{12})t_{22}^2 P_2(x).
\end{equation*}

We put $P(x) = P_1(x)^8 P_2(x)^5 x_{343}^9$. Then 
on $M^{\text{st}}_{\be_{121}}$, 
\begin{align*}
P(gx) & = 
((\det g_{11})t_1^2(\det g_2))^8
((\det g_{11})(\det g_{12})t_{22}^2)^5
((\det g_{12})t_{21})^9 P(x) \\
& = (\det g_{11})^{13}(\det g_{12})^{14}t_1^{16}
(\det g_2)^8t_{21}^9 t_{22}^{10}P(x) 
= (\det g_{12}) t_1^3t_{21}t_{22}^2 P(x). 
\end{align*}
Therefore, $P(x)$ is 
invariant under the action of $G_{\text{st},\be_{121}}$. 

Let $R(121)\in Z_{121}$ be element such that 
\begin{math}
A(R(121)) = B(R(121)) = I_2
\end{math}
and the $x_{343}$-coordinate is $1$. 
Explicitly, 
$R(121)=e_{151}+e_{252}+e_{343}+e_{134}+e_{244}$. 
Then $P_1(R(121))=P_2(R(121))=1$. 
So $P(R(121))=1$ and $R(121)\in Z^{\sst}_{121}$. 

We show that  $Z^{\sst}_{121\,k}=M_{\be_{121}\,k}R(121)$. 
Suppose that $x\in Z^{\sst}_{121\,k}$.
Since $P_1(x),P_2(x)\not=0$,  
there exists $g\in M_{\be_{121}\,k}$ 
such that $A(gx)=B(gx)=I_2$. So we may assume that $A(x)=B(x)=I_2$. 
Let $t=(I_2,I_2,I_2,1,t_{21},1)$. 
Then $A(tx)=B(tx)=I_2$ and the 
$x_{343}$-coordinate of $tx$ is $t_{21}x_{343}$.
Therefore, $Z^{\sst}_{121\,k}=M_{\be_{121}\,k}R(121)$. 

We assume that  $u_{121}=u_{143}=0$ 
and $u_{221}=0$. 
Then the four components of 
$n(u)R(121)$ are as follows: 
\begin{align*}
& \begin{pmatrix}
0 & 0 & 0 & 0 & 1 \\
0 & 0 & 0 & 0 & 0 \\
0 & 0 & 0 & 0 & u_{131} \\
0 & 0 & 0 & 0 & u_{141} \\
-1 & 0 & * & * & 0 
\end{pmatrix},\;
\begin{pmatrix}
0 & 0 & 0 & 0 & 0 \\
0 & 0 & 0 & 0 & 1 \\
0 & 0 & 0 & 0 & u_{132} \\
0 & 0 & 0 & 0 & u_{142} \\
0 & -1 & * & * & 0 
\end{pmatrix}, \\
& \begin{pmatrix}
0 & 0 & 0 & 0 & u_{231} \\
0 & 0 & 0 & 0 & u_{232} \\
0 & 0 & 0 & 1 & Q_1(u)+u_{154} \\
0 & 0 & -1 & 0 & Q_2(u)-u_{153} \\
* & * & * & * & 0 
\end{pmatrix}, \;
\begin{pmatrix}
0 & 0 & 1 & 0 & u_{153}+u_{241} \\
0 & 0 & 0 & 1 & u_{154}+u_{242} \\
-1 & 0 & 0 & u_{132}-u_{141}+u_{243} & Q_3(u)-u_{151} \\
0 & -1 & * & 0 & Q_4(u)-u_{152} \\
* & * & * & * & 0 \\
\end{pmatrix}
\end{align*}
where $Q_1(u),Q_2(u)$ do not depend on
$u_{151}\ccd u_{154},u_{241},u_{242}$ 
and $Q_3(u),Q_4(u)$ do not depend on
$u_{151},u_{152}$.  
 

We can apply Lemma \ref{lem:connected-criterion} 
to the map $\aff^{13}\to \aff^{13}$ defined by the sequence
\begin{align*}
& u_{131},u_{132},u_{141},u_{142},u_{231},u_{232},
u_{243}+u_{132}-u_{141}, u_{154}+Q_1(u), \\
& u_{153}-Q_2(u),u_{241}+u_{153},u_{242}+u_{154},
u_{151}-Q_3(u),u_{152}-Q_4(u)
\end{align*}
with no extra variables. 
So by Proposition \ref{prop:W-eliminate}, 
Property \ref{property:W-eliminate} holds for 
any $x\in Z^{\sst}_{121\,k}$. Therefore, 
$Y^{\sst}_{121\,k}=P_{\be_{121}\,k}R(121)$ also. 

\vskip 5pt

(31) $S_{131}$, $\be_{131}=\tfrac {1} {60} (-24,6,6,6,6,-15,5,5,5)$ 

We identify the element 
$(\diag(t_1,g_1),\diag(t_2,g_2))
\in M_{[1],[1]}=M_{\be_{131}}$ 
with the element 
$g=(g_1,g_2,t_1,t_2)\in \gl_4\times \gl_3\times \gl_1^2$. 
On $M^{\text{st}}_{\be_{131}}$, 
\begin{equation*}
\chi_{{131}}(g) = 
t_1^{-24}(\det g_1)^6t_2^{-15}(\det g_2)^5
= (\det g_1)^{30}(\det g_2)^{20}. 
\end{equation*}

We identify of $Z_{{131}}\cong \wedge^2 \aff^4\otimes \aff^3$. 
Let $P(x)$ be the degree $12$ polynomial on $Z_{131}$ 
obtained by Lemma \ref{lem:relative-invariant-423}.
Since elements of the form 
$g=(I_4,I_3,t_1,t_2)$ act on $Z_{131}$ trivially, 
$P(gx)=(\det g_1)^6(\det g_2)^4 P(x)$. 
Therefore, $P(x)$ is 
invariant under the action of $G_{\text{st},\be_{131}}$. 

$M_{\be_{131}\,}\backslash Z^{\sst}_{131\,k}$ 
can be identified with the set of rational orbits 
considered in the case (a) of Section \ref{sec:rational-orbits-wedge43}.
So by Proposition \ref{prop:stab-iso-wedge42-3}, 
$M_{\be_{131}\,}\backslash Z^{\sst}_{131\,k}$ 
is in bijective correspondence with $\mathrm{IQF}_4(k)$. 
Since $W_{131}=\{0\}$, 
$P_{\be_{131}\,}\backslash Y^{\sst}_{131\,k}$ 
is in bijective correspondence with $\mathrm{IQF}_4(k)$
also. 

\vskip 5pt

(32) $S_{149}$, $\be_{149}=\tfrac {1} {220} (-28,-8,2,12,22,-15,-5,5,15)$ 

(33) $S_{150}$, 
$\be_{150}=\tfrac {1} {60} (-8,-4,0,4,8,-7,-3,1,9)$ 

(34) $S_{151}$, 
$\be_{151} = \tfrac{1} {20} (-2,-1,0,1,2,-2,-1,1,2)$ 

(35) $S_{152}$, 
$\be_{152} = \tfrac {1} {20} (-4,-2,0,2,4,-3,-1,1,3)$ 

\vskip 5pt

For $l=149$--$152$, 
$M_{\be_l}=T$ and so $Z^{\sst}_l\not=\emptyset$ 
by Proposition \ref{prop:M=Tsuejective}. 
Let $R(l)\in Z_l$ be the element whose coordinates 
in $Z_l$ are all $1$. 

We express elements of $T$ as (\ref{eq:t-defn}). 
Then the matrix $(m_{ij})$ of Proposition \ref{prop:M=Tsuejective} is 
\begin{align*}
& \begin{pmatrix}
0 & 0 & 1 & 0 & 1 & 1 & 0 & 0 & 0 \\
0 & 1 & 0 & 0 & 1 & 0 & 1 & 0 & 0 \\
0 & 0 & 1 & 1 & 0 & 0 & 1 & 0 & 0 \\
0 & 1 & 0 & 1 & 0 & 0 & 0 & 1 & 0 \\
1 & 0 & 0 & 1 & 0 & 0 & 0 & 0 & 1 \\
0 & 1 & 1 & 0 & 0 & 0 & 0 & 0 & 1 
\end{pmatrix},\; 
\begin{pmatrix}
0 & 0 & 0 & 1 & 1 & 1 & 0 & 0 & 0 \\
0 & 0 & 1 & 0 & 1 & 0 & 1 & 0 & 0 \\
0 & 1 & 0 & 0 & 1 & 0 & 0 & 1 & 0 \\
0 & 0 & 1 & 1 & 0 & 0 & 0 & 1 & 0 \\
1 & 0 & 0 & 1 & 0 & 0 & 0 & 0 & 1 \\
0 & 1 & 1 & 0 & 0 & 0 & 0 & 0 & 1 
\end{pmatrix}, \\
& \begin{pmatrix}
0 & 0 & 0 & 1 & 1 & 1 & 0 & 0 & 0 \\
0 & 0 & 1 & 0 & 1 & 0 & 1 & 0 & 0 \\
1 & 0 & 0 & 0 & 1 & 0 & 0 & 1 & 0 \\
0 & 1 & 0 & 1 & 0 & 0 & 0 & 1 & 0 \\
1 & 0 & 0 & 1 & 0 & 0 & 0 & 0 & 1 \\
0 & 1 & 1 & 0 & 0 & 0 & 0 & 0 & 1 
\end{pmatrix},\;
\begin{pmatrix}
0 & 0 & 0 & 1 & 1 & 1 & 0 & 0 & 0 \\
0 & 0 & 1 & 0 & 1 & 0 & 1 & 0 & 0 \\
0 & 1 & 0 & 0 & 1 & 0 & 0 & 1 & 0 \\
0 & 0 & 1 & 1 & 0 & 0 & 0 & 1 & 0 \\
1 & 0 & 0 & 0 & 1 & 0 & 0 & 0 & 1 \\
0 & 1 & 0 & 1 & 0 & 0 & 0 & 0 & 1 
\end{pmatrix}
\end{align*}
for $S_{149}$--$S_{152}$ respectively. 

The determinant of the $6\times 6$ minor of the columns 
$1,5,6,7,8,9$ is (a) $-1$ for $S_{149}$, 
(b) $1$ for $S_{150}$, (c) $1$ for $S_{151}$ and 
(d) $1$ for $S_{152}$. 
Therefore, $Z^{\sst}_{l\,k} = T_k R(l)$ 
for $l=149$--$152$ by Proposition \ref{prop:M=Tsuejective}.

We consider the case $S_{149}$. 
Explicitly, 
$R(149)=e_{351}+e_{252}+e_{342}+e_{243}+e_{154}+e_{234}$.
The four components of 
$n(u)R(149)$ are as follows: 
\begin{align*}
& \begin{pmatrix}
0 & 0 & 0 & 0 & 0 \\
0 & 0 & 0 & 0 & 0 \\
0 & 0 & 0 & 0 & 1 \\
0 & 0 & 0 & 0 & u_{143} \\
0 & 0 & -1 & * & 0 
\end{pmatrix},\;
\begin{pmatrix}
0 & 0 & 0 & 0 & 0 \\
0 & 0 & 0 & 0 & 1 \\
0 & 0 & 0 & 1 & u_{221}+u_{132}+u_{154} \\
0 & 0 & -1 & 0 & Q_1(u) + u_{142}-u_{153} \\
0 & -1 & * & * & 0 
\end{pmatrix}, \\
& \begin{pmatrix}
0 & 0 & 0 & 0 & 0 \\
0 & 0 & 0 & 1 & u_{154}+u_{232} \\
0 & 0 & 0 & u_{132}+u_{232} & Q_2(u)+u_{231} \\
0 & -1 & * & 0 & Q_3(u)-u_{152} \\
0 & * & * & * & 0 \\
\end{pmatrix}, \\
& \begin{pmatrix}
0 & 0 & 0 & 0 & 1 \\
0 & 0 & 1 & u_{143}+u_{243} & Q_4(u)+u_{121}+u_{153}+u_{242} \\
0 & -1 & 0 & Q_5(u)-u_{142}+u_{242} & Q_6(u)+u_{131}-u_{152}+u_{241} \\
0 & * & * & 0 & Q_7(u)+u_{141} \\
-1 & * & * & * & 0 
\end{pmatrix}
\end{align*}
where $Q_1(u)$ is a polynomial of $u_{143},u_{221},u_{154}$, 
$Q_2(u),Q_4(u),Q_5(u)$ are polynomials of 
$u_{132},u_{143},u_{154},u_{232},u_{243}$, 
$Q_3(u),Q_6(u)$ do not depend on 
$u_{131},u_{141},u_{152},u_{241}$ 
and $Q_7(u)$ does not depend on $u_{141}$.

We can apply Lemma \ref{lem:connected-criterion} 
to the map $\aff^{16}\to \aff^{12}$ defined by the sequence
\begin{align*}
 & u_{143},u_{243}+u_{143},u_{154}+u_{232},u_{132}+u_{232},
u_{221}+u_{132}+u_{154},
u_{142}-u_{242}-Q_5(u), \\
& u_{153}-u_{142}-Q_1(u),
u_{121}+u_{153}+u_{242}+Q_4(u),
u_{231}+Q_2(u),u_{152}-Q_3(u), \\
& u_{131}-u_{152}+u_{241}+Q_6(u),
u_{141}+Q_7(u)  
\end{align*}
where $u_{151},u_{232},u_{241},u_{242}$  
are extra variables.
So by Proposition \ref{prop:W-eliminate}, 
Property \ref{property:W-eliminate} holds for 
any $x\in Z^{\sst}_{149\,k}$. Therefore, 
$Y^{\sst}_{149\,k}=P_{\be_{149}\,k}R(149)$ also. 

\vskip 5pt

We consider the case $S_{150}$. 
Explicitly, $R(150)=e_{451}+e_{352}+e_{253}+e_{343}+e_{144}+e_{234}$.
The first component of 
$n(u)R(150)$ is the same as that of $R(150)$ 
and the remaining components are as follows: 
\begin{align*}
& \begin{pmatrix}
0 & 0 & 0 & 0 & 0 \\
0 & 0 & 0 & 0 & 0 \\
0 & 0 & 0 & 0 & 1 \\
0 & 0 & 0 & 0 & u_{143}+u_{221} \\
0 & 0 & -1 & * & 0 
\end{pmatrix}, \;
\begin{pmatrix}
0 & 0 & 0 & 0 & 0 \\
0 & 0 & 0 & 0 & 1 \\
0 & 0 & 0 & 1 & u_{132}+u_{154}+u_{232} \\
0 & 0 & -1 & 0 & Q_1(u)+u_{142}-u_{153}+u_{231} \\
0 & -1 & * & * & 0 \\
\end{pmatrix}, \\
& \begin{pmatrix}
0 & 0 & 0 & 1 & u_{154} \\
0 & 0 & 1 & u_{121}+u_{143} & Q_2(u)+u_{153}+u_{243} \\
0 & -1 & 0 & Q_3(u)+u_{131}-u_{142}+u_{243} 
& Q_4(u)-u_{152}+u_{242} \\
-1 & * & * & 0 & Q_5(u)-u_{151}+u_{241} \\
* & * & * & * & 0 
\end{pmatrix}
\end{align*}
where $Q_1(u),Q_2(u),Q_3(u)$ 
are polynomial of $u_{121},u_{132},u_{143},u_{154},u_{232}$,
$Q_4(u)$ is a polynomial of $u_{131},u_{132},u_{153},u_{154},u_{243}$
and $Q_5(u)$ does not depend on $u_{151},u_{241}$. 

We can apply Lemma \ref{lem:connected-criterion} 
to the map $\aff^{16}\to \aff^9$ defined by the sequence
\begin{align*}
& u_{154},u_{121}+u_{143},u_{221}+u_{143},
u_{132}+u_{154}+u_{232},
u_{142}-u_{153}+u_{231}+Q_1(u), \\
& u_{243}+u_{153}+Q_2(u),
u_{131}-u_{142}+u_{243}+Q_3(u),
u_{152}-u_{242}-Q_4(u), \\
& u_{151}-u_{241}-Q_5(u)
\end{align*}
where $u_{141},u_{143},u_{153},u_{231},u_{232},u_{241},u_{242}$  
are extra variables.
So by Proposition \ref{prop:W-eliminate}, 
Property \ref{property:W-eliminate} holds for 
any $x\in Z^{\sst}_{150\,k}$. Therefore, 
$Y^{\sst}_{150\,k}=P_{\be_{150}\,k}R(150)$ also. 

\vskip 5pt

We consider the case $S_{151}$. 
Explicitly, $R(151)=e_{451}+e_{352}+e_{153}+e_{243}+e_{144}+e_{234}$.
The first component of 
$n(u)R(151)$ is the same as that of $R(151)$ 
and the remaining components are as follows: 
\begin{align*}
& \begin{pmatrix}
0 & 0 & 0 & 0 & 0 \\
0 & 0 & 0 & 0 & 0 \\
0 & 0 & 0 & 0 & 1 \\
0 & 0 & 0 & 0 & u_{143}+u_{221} \\
0 & 0 & -1 & * & 0 
\end{pmatrix}, \;
\begin{pmatrix}
0 & 0 & 0 & 0 & 1 \\
0 & 0 & 0 & 1 & u_{121}+u_{154} \\
0 & 0 & 0 & u_{132} & Q_1(u) + u_{131}+u_{232} \\
0 & -1 & * & 0 & Q_2(u)+u_{141}-u_{152}+u_{231} \\
-1 & * & * & * & 0 \\
\end{pmatrix}, \\
& \begin{pmatrix}
0 & 0 & 0 & 1 & u_{243}+u_{154} \\
0 & 0 & 1 & u_{121}+u_{143}+u_{243} & Q_3(u)+u_{153} \\
0 & -1 & 0 & Q_4(u)+u_{131}-u_{142} 
& Q_5(u)-u_{152}+u_{242} \\
-1 & * & * & 0 & Q_6(u)-u_{151}+u_{241} \\
* & * & * & * & 0 
\end{pmatrix}
\end{align*}
where $Q_1(u),Q_3(u),Q_4(u)$ 
are polynomials of 
$u_{121},u_{132},u_{143},u_{154},u_{243}$, 
$Q_2(u)$ is a polynomial of 
$u_{142},u_{143},u_{154},u_{232}$, 
$Q_5(u)$ does not depend on 
$u_{151},u_{152},u_{241},u_{242}$
and $Q_6(u)$ does not depend on 
$u_{151},u_{241}$.

We can apply Lemma \ref{lem:connected-criterion} 
to the map $\aff^{16}\to \aff^{11}$ defined by the sequence
\begin{align*}
& u_{132},u_{121}+u_{154},u_{243}+u_{154},
u_{143}+u_{121}+u_{243},u_{221}+u_{143},
u_{131}+u_{232}+Q_1(u), \\
& u_{153}+Q_3(u),u_{142}-u_{131}-Q_4(u), 
u_{141}-u_{152}+u_{231}+Q_2(u), \\
& u_{ 242}-u_{152}+Q_5(u),
u_{241}-u_{151}+Q_6(u)
\end{align*}
where $u_{151},u_{152},u_{154},u_{231},u_{232}$  
are extra variables.
So by Proposition \ref{prop:W-eliminate}, 
Property \ref{property:W-eliminate} holds for 
any $x\in Z^{\sst}_{151\,k}$. Therefore, 
$Y^{\sst}_{151\,k}=P_{\be_{151}\,k}R(151)$ also. 

\vskip 5pt

We consider the case $S_{152}$. 
Explicitly, $R(152)=e_{451}+e_{352}+e_{253}+e_{343}+e_{154}+e_{244}$. 
The first component of 
$n(u)R(152)$ is the same as that of 
$R(152)$ and the remaining components are as follows: 
\begin{align*}
& \begin{pmatrix}
0 & 0 & 0 & 0 & 0 \\
0 & 0 & 0 & 0 & 0 \\
0 & 0 & 0 & 0 & 1 \\
0 & 0 & 0 & 0 & u_{143}+u_{221} \\
0 & 0 & -1 & * & 0 
\end{pmatrix}, \;
\begin{pmatrix}
0 & 0 & 0 & 0 & 0 \\
0 & 0 & 0 & 0 & 1 \\
0 & 0 & 0 & 1 & u_{132}+u_{154}+u_{232} \\
0 & 0 & -1 & 0 & Q_1(u)+u_{142}-u_{153}+u_{231} \\
0 & -1 & * & * & 0 \\
\end{pmatrix}, \\
& \begin{pmatrix}
0 & 0 & 0 & 0 & 1 \\
0 & 0 & 0 & 1 & u_{121}+u_{154}+u_{243} \\
0 & 0 & 0 & u_{132}+u_{243} & Q_2(u)+u_{131}+u_{242} \\
0 & -1 & * & 0 & Q_3(u)+u_{141}-u_{152}+u_{241} \\
-1 & * & * & * & 0 
\end{pmatrix}
\end{align*}
where $Q_1(u),Q_2(u)$ are polynomials of 
$u_{132},u_{143},u_{154},u_{232},u_{243}$
and $Q_3(u)$ does not depend on 
$u_{141},u_{152},u_{241}$. 

We can apply Lemma \ref{lem:connected-criterion} 
to the map $\aff^{16}\to \aff^7$ defined by the sequence
\begin{align*}
& u_{143}+u_{221},u_{132}+u_{243},
u_{154}+u_{132}+u_{232},
u_{121}+u_{154}+u_{243}, \\
& u_{142}-u_{153}+u_{231}+Q_1(u), 
u_{131}+u_{242}+Q_2(u), 
u_{141}-u_{152}+u_{241}+Q_3(u)
\end{align*}
where $u_{151},u_{152},u_{153},u_{221},u_{231},u_{232},u_{241},u_{242},u_{243}$  
are extra variables.
So by Proposition \ref{prop:W-eliminate}, 
Property \ref{property:W-eliminate} holds for 
any $x\in Z^{\sst}_{152\,k}$. Therefore, 
$Y^{\sst}_{152\,k}=P_{\be_{152}\,k}R(152)$ also. 

\vskip 5pt

(36) $S_{164}$, 
$\be_{164}=\tfrac {1} {60} (-4,-4,-4,6,6,-15,-15,15,15)$ 

We identify the element 
$(\diag(g_{11},g_{12}),\diag(g_{21},g_{22}))
\in M_{[3],[2]}=M_{\be_{164}}$ 
with the element 
$g=(g_{11},g_{12},g_{21},g_{22})\in \gl_3\times \gl_2^3$. 
On $M^{\text{st}}_{\be_{164}}$, 
\begin{equation*}
\chi_{{164}}(g) = 
(\det g_{11})^{-4}(\det g_{12})^6
(\det g_{21})^{-15}(\det g_{22})^{15}
= (\det g_{12})^{10}(\det g_{22})^{30}. 
\end{equation*}

We identify $Z_{164}\cong \aff^3\otimes \aff^2\otimes \aff^2$. 
If $g=(g_{11},g_{12},g_{21},g_{22})$ and $x\in Z_{164}$ 
then $gx$ does not depend on $g_{21}$ and 
the action of $(g_{11},g_{12},I_2,g_{22})$ on $Z_{164}$ 
is the same as that of 
Proposition \ref{prop:322-invariant}. 
Therefore, if $R(164)\in Z_{164}$ is the element 
which corresponds to $R_{322}$ in Proposition \ref{prop:322-invariant}, 
then $R(164)\in Z^{\sst}_{164}$ and $Z^{\sst}_{164\,k}=M_{\be_{164}\,k}R(164)$.
Explicitly, 
\begin{math}
R(164) = -e_{143}+e_{154}+e_{244}+e_{353}.
\end{math}

We assume that $u_{1ij}=0$ unless $i=4,5,j=1,2,3$ 
and $u_{221}=u_{243}=0$. Then the first two 
components of $n(u)R(164)$ are $0$ 
and the remaining components are as follows: 
\begin{align*}
& \begin{pmatrix}
0 & 0 & 0 & -1 & 0 \\
0 & 0 & 0 & 0 & 0 \\
0 & 0 & 0 & 0 & 1 \\
1 & 0 & 0 & 0 & u_{143}+u_{151} \\
0 & 0 & -1 & * & 0 
\end{pmatrix},\;
\begin{pmatrix}
0 & 0 & 0 & 0 & 1 \\
0 & 0 & 0 & 1 & 0 \\
0 & 0 & 0 & 0 & 0 \\
0 & -1 & 0 & 0 & u_{141}-u_{152} \\
-1 & 0 & 0 & * & 0 \\
\end{pmatrix}.
\end{align*}

We can apply Lemma \ref{lem:connected-criterion} 
to the map $\aff^{10}\to \aff^2$ defined by the sequence
\begin{math}
u_{143}+u_{151},u_{141}-u_{152},
\end{math}
where $u_{142},u_{151},u_{152},u_{153},u_{231},u_{232},u_{241},u_{242}$  
are extra variables.
So by Proposition \ref{prop:W-eliminate}, 
Property \ref{property:W-eliminate} holds for 
any $x\in Z^{\sst}_{164\,k}$. Therefore, 
$Y^{\sst}_{164\,k}=P_{\be_{164}\,k}R(164)$ also. 

\vskip 5pt

(37) $S_{178}$, 
$\be_{178} = \tfrac {1} {220} (-28,-8,-8,12,32,-55,5,25,25)$ 

We identify the element 
$(\diag(t_{11},g_1,t_{12},t_{13}),\diag(t_{21},t_{22},g_2))
\in M_{[1,3,4],[1,2]}=M_{\be_{178}}$ 
with the element 
$g=(g_1,g_2,t_{11}\ccd t_{22})\in \gl_2^2\times \gl_1^5$. 
On $M^{\text{st}}_{\be_{178}}$, 
\begin{equation*}
\chi_{{178}}(g) = 
t_{11}^{-28}(\det g_1)^{-8}t_{12}^{12}t_{13}^{32}
t_{21}^{-55}t_{22}^5(\det g_2)^{25}
= (\det g_1)^{20}t_{12}^{40}t_{13}^{60}
t_{22}^{60}(\det g_2)^{80}.
\end{equation*}

For $x\in Z_{164}$, let 
\begin{equation*}
A(x) = \begin{pmatrix}
x_{243} & x_{244} \\
x_{343} & x_{344}
\end{pmatrix},\; 
v_1(x) = \begin{pmatrix}
x_{252} \\ x_{352}
\end{pmatrix},\; 
v_2(x) = \begin{pmatrix}
x_{153} \\ x_{154}
\end{pmatrix}.  
\end{equation*}
We regard that $A(x)\in \Lam^{2,1}_{1,[2,3]}\otimes \Lam^{2,1}_{2,[3,4]}$, 
$v_1(x)\in \Lam^{2,1}_{1,[2,3]}$, 
$v_2(x)\in \Lam^{2,1}_{2,[3,4]}$. 
We identify $Z_{178}\cong \m_2\oplus \aff^2\oplus \aff^2$
by the map $Z_{178}\ni x\mapsto (A(x),v_1(x),v_2(x))$. 

It is easy to see that 
\begin{equation*}
A(gx) = t_{12} g_1A(x){}^tg_2,\;
v_1(gx) = t_{13}t_{22} g_1v_1(x),\;
v_2(gx) = t_{11}t_{13} g_2v_2(x).
\end{equation*}
Let $P_1(x)$ be the homogeneous degree $3$ polynomial of 
$A(x),v_1(x),v_2(x)$ obtained by 
Proposition \ref{prop:M2-2-2-invariant} and  
$P_2(x) = \det A(x)$. Since $P_1(x)$ is linear with respect to 
each of $A(x),v_1(x),v_2(x)$,  
\begin{equation*}
P_1(gx) = t_{11}(\det g_1)t_{12}t_{13}^2t_{22}(\det g_2)P_1(x),\;
P_2(gx) = (\det g_1)t_{12}^2(\det g_2)P_2(x).
\end{equation*}

We put $P(x) = P_1(x)^3P_2(x)$. Then on $M^{\text{st}}_{\be_{178}}$, 
\begin{align*}
P(gx) & = (t_{11}(\det g_1)t_{12}t_{13}^2t_{22}(\det g_2))^3
((\det g_1)t_{12}^2(\det g_2))P(x) \\
& = t_{11}^3(\det g_1)^4t_{12}^5t_{13}^6t_{22}^3(\det g_2)^4P(x)
= (\det g_1)t_{12}^2t_{13}^3t_{22}^3(\det g_2)^4P(x). 
\end{align*}
Therefore, $P(x)$ is 
invariant under the action of $G_{\text{st},\be_{178}}$. 

Let $R(178)\in Z_{178}$ be the element such that 
\begin{math}
A(R(178))=I_2,\; v_1(R(178))=[1,0],\; v_2(R(178))=[1,0]
\end{math}
(see Proposition \ref{prop:M2-2-2-invariant}). 
Explicitly, $R(178)=e_{252}+e_{153}+e_{243}+e_{344}$. 
By Proposition \ref{prop:M2-2-2-invariant}, 
$P(R(178))\not=0$ and $Z^{\sst}_{178\,k}$ is the orbit of $R(178)$ 
by the action of $(g_1,g_2,1,1,1,1,1,t_{22})$. 
Adding the action of $t_{11},t_{12},t_{13},t_{21}$ 
does not make the orbit space bigger. 
Therefore, 
$Z^{\sst}_{178\,k}=M_{\be_{178}\,k}R(178)$.

We assume that $u_{132}=0$ and $u_{243}=0$. Then the first  
component of $n(u)R(178)$ is $0$ 
and the remaining components are as follows: 
\begin{align*}
& \begin{pmatrix}
0 & 0 & 0 & 0 & 0 \\
0 & 0 & 0 & 0 & 1 \\
0 & 0 & 0 & 0 & 0 \\
0 & 0 & 0 & 0 & u_{142} \\
0 & -1 & 0 & * & 0 
\end{pmatrix}, \;
\begin{pmatrix}
0 & 0 & 0 & 0 & 1 \\
0 & 0 & 0 & 1 & u_{121}+u_{154}+u_{232} \\
0 & 0 & 0 & 0 & u_{131} \\
0 & -1 & 0 & 0 & Q_1(u) + u_{141}-u_{152} \\
-1 & * & 0 & * & 0 
\end{pmatrix}, \\
& \begin{pmatrix}
0 & 0 & 0 & 0 & 0 \\
0 & 0 & 0 & 0 & u_{242} \\
0 & 0 & 0 & 1 & u_{154} \\
0 & 0 & -1 & 0 & Q_2(u)-u_{153} \\
0 & * & * & * & 0 \\
\end{pmatrix}
\end{align*}
where $Q_1(u)=Q_1(u_{142},u_{154},u_{232})$ 
and $Q_2(u)$ does not depend on $u_{153}$. 

We can apply Lemma \ref{lem:connected-criterion} 
to the map $\aff^{14}\to \aff^7$ defined by the sequence
\begin{align*}
& u_{131},u_{142},u_{154},u_{242},
u_{121}+u_{154}+u_{232},
u_{141}-u_{152}+Q_1(u),
u_{153}-Q_2(u)
\end{align*}
where $u_{143},u_{151},u_{152},u_{221},u_{231},u_{232},u_{241}$  
are extra variables.
So by Proposition \ref{prop:W-eliminate}, 
Property \ref{property:W-eliminate} holds for 
any $x\in Z^{\sst}_{178\,k}$. Therefore, 
$Y^{\sst}_{178\,k}=P_{\be_{178}\,k}R(178)$ also. 

\vskip 5pt

(38) $S_{202}$, 
$\be_{202}=\tfrac {1} {180} (-32,-32,8,28,28,-25,-5,-5,35)$ 

We identify the element 
$(\diag(g_{11},t_1,g_{12}),\diag(t_{21},g_2,t_{22}))
\in M_{[2,3],[1,3]}=M_{\be_{202}}$ 
with the element 
$g=(g_{11},g_{12},g_2,t_1,t_{21},t_{22})\in \gl_2^3\times \gl_1^3$. 
On $M^{\text{st}}_{\be_{202}}$, 
\begin{equation*}
\chi_{{202}}(g) 
= (\det g_{11})^{-32}t_1^8(\det g_{12})^{28}
t_{21}^{-25}(\det g_2)^{-5}t_{22}^{35}
= t_1^{40}(\det g_{12})^{60}
(\det g_2)^{20}t_{22}^{60}. 
\end{equation*}

For $x\in Z_{202}$, let 
\begin{equation*}
A(x) = 
\begin{pmatrix}
x_{342} & x_{343} \\
x_{352} & x_{353}
\end{pmatrix},\;
B(x) = 
\begin{pmatrix}
x_{144} & x_{154} \\
x_{244} & x_{254}
\end{pmatrix}.
\end{equation*}
We identify $Z_{202}\cong \m_2\oplus \m_2\oplus 1$ 
by the map $Z_{202}\ni x\mapsto (A(x),B(x),x_{451})$. 

Let 
\begin{math}
P_1(x) = \det A(x),\; 
P_2(x) = \det B(x)
.\end{math}
We put $P(x) = P_1(x)^2P_2(x)^2 x_{451}$. 
Then on $M^{\text{st}}_{\be_{202}}$, 
\begin{align*}
A(gx) & = t_1g_{12}A(x){}^tg_2,\; 
B(gx) = t_{22}g_{11}A(x){}^tg_{12}, \\
P_1(gx) & = t_1^2(\det g_{12})(\det g_2)P_1(x),\;
P_2(gx) = (\det g_{11})(\det g_{12})P_2(x)t_{22}^2, \\
P(gx) & = (t_1^2(\det g_{12})(\det g_2))^2
((\det g_{11})(\det g_{12})t_{22}^2)^2 ((\det g_{12})t_{21})P(x) \\
& = (\det g_{11})^2t_1^4(\det g_{12})^5 t_{21}(\det g_2)^2t_{22}^4P(x) 
= t_1^2(\det g_{12})^3 (\det g_2)t_{22}^3P(x).
\end{align*}
Therefore, $P(x)$ is 
invariant under the action of $G_{\text{st},\be_{202}}$. 

Let $R(202)\in Z_{202}$ be element such that 
\begin{math}
A(R(202)) = B(R(202)) = I_2
\end{math}
and the $x_{451}$-coordinate is $1$. 
Then $P_1(R(202))=P_2(R(202))=1$. So 
$P(R(202))=1$ and $R(202)\in Z^{\sst}_{202}$. 
Explicitly, $R(202)=e_{451}+e_{342}+e_{353}+e_{144}+e_{254}$. 

We show that  $Z^{\sst}_{202\,k}=M_{\be_{202}\,k}R(202)$. 
Suppose that $x\in Z^{\sst}_{202\,k}$. 
It is easy to see that there exists $g\in M_{\be_{202}\,k}$ 
such that $A(gx)=B(gx)=I_2$. So we may assume that $A(x)=B(x)=I_2$. 
Let $t=(I_2,I_2,I_2,1,t_{21},1)$. 
Then $A(tx)=B(tx)=I_2$ and the 
$x_{451}$-coordinate of $tx$ is $t_{21}x_{451}$.
Therefore, $Z^{\sst}_{202\,k}=M_{\be_{202}\,k}R(202)$. 

We assume that $u_{121}=u_{154}=0$ and $u_{232}=0$.
Then the first component of $n(u)R(202)$ is the same as 
that of $R(202)$ and the remaining components 
are as follows: 
\begin{align*}
& \begin{pmatrix}
0 & 0 & 0 & 0 & 0 \\
0 & 0 & 0 & 0 & 0 \\
0 & 0 & 0 & 1 & 0 \\
0 & 0 & -1 & 0 & -u_{153}+u_{221} \\
0 & 0 & 0 & * & 0 
\end{pmatrix}, \;
\begin{pmatrix}
0 & 0 & 0 & 0 & 0 \\
0 & 0 & 0 & 0 & 0 \\
0 & 0 & 0 & 0 & 1 \\
0 & 0 & 0 & 0 & u_{143}+u_{231} \\
0 & 0 & 1 & * & 0 
\end{pmatrix}, \\
& \begin{pmatrix}
0 & 0 & 0 & 1 & 0 \\
0 & 0 & 0 & 0 & 1 \\
0 & 0 & 0 & u_{131}+u_{242} & u_{132}+u_{243} \\
-1 & 0 & * & 0 & Q(u)+u_{142}-u_{151}+u_{241} \\
0 & -1 & * & * & 0 \\
\end{pmatrix}
\end{align*}
where $Q(u)=Q(u_{143},u_{153},u_{242},u_{243})$.

We can apply Lemma \ref{lem:connected-criterion} 
to the map $\aff^{13}\to \aff^5$ defined by the sequence
\begin{align*}
& u_{143}+u_{231},u_{153}-u_{221},u_{131}+u_{242},
u_{132}+u_{243},u_{142}-u_{151}+u_{241}+Q(u)
\end{align*}
where $u_{141},u_{151},u_{152},u_{221},u_{231},u_{241},u_{242},u_{243}$  
are extra variables.
So by Proposition \ref{prop:W-eliminate}, 
Property \ref{property:W-eliminate} holds for 
any $x\in Z^{\sst}_{202\,k}$. Therefore, 
$Y^{\sst}_{202\,k}=P_{\be_{202}\,k}R(202)$ also.

\vskip 5pt

(39) $S_{216}$, 
$\be_{216} = \tfrac {1} {20} (-8,0,0,4,4,-5,-1,3,3)$ 

We identify the element 
$(\diag(t_1,g_{11},g_{12}),\diag(t_{21},t_{22},g_2))
\in M_{[1,3],[1,2]}=M_{\be_{216}}$ 
with the element 
$g=(g_{11},g_{12},g_2,t_1,t_{21},t_{22})\in \gl_2^3\times \gl_1^3$. 
On $M^{\text{st}}_{\be_{216}}$, 
\begin{equation*}
\chi_{{216}}(g) = t_1^{-8}(\det g_{12})^4
t_{21}^{-5}t_{22}^{-1}(\det g_2)^3
= (\det g_{11})^8(\det g_{12})^{12}t_{22}^4(\det g_2)^8. 
\end{equation*}

For $x\in Z_{216}$, let 
\begin{equation*}
A(x) = (x_{243},x_{253},x_{343},x_{353},x_{244}\ccd x_{354}). 
\end{equation*}
%
We identify $Z_{216}\cong \aff^2\otimes \aff^2\otimes \aff^2\oplus 1$ 
by the map $Z_{216}\ni x\mapsto (A(x),x_{452})$.

It is easy to see that $A(gx)=(g_{11},g_{12},g_2)A(x)$ 
(the natural action). Note that $gx$ does not depend on 
$t_1,t_{21}$. 
Let $P_1(x)$ be the degree $4$ polynomial of $A(x)$ 
obtained by Proposition \ref{prop:222-invariant}. 
We put $P(x)= P_1(x)x_{452}$. Then on $M^{\text{st}}_{\be_{216}}$, 
\begin{align*}
P_1(gx) & = (\det g_{11})^2(\det g_{12})^2(\det g_2)^2 P_1(x), \\
P(gx) & = ((\det g_{11})^2(\det g_{12})^2(\det g_2)^2) ((\det g_{12})t_{22})P(x) \\
& = (\det g_{11})^2(\det g_{12})^3t_{22}(\det g_2)^2 P(x). 
\end{align*}
Therefore, $P(x)$ is 
invariant under the action of $G_{\text{st},\be_{216}}$. 

By applying an element of the form  
$t=(I_2,I_2,I_2,1,1,t_{22})$ to $x$, 
the $x_{452}$-coordinate becomes $1$. 
Elements of the form 
$(g_{11},g_{12},g_2,1,1,(\det g_{12})^{-1})$
do not change this condition. 
Therefore, $M_{\be_{216}\,k}\backslash Z^{\sst}_{216\,k}$
is in bijective correspondence with the set of rational orbits 
of the \rep{} of Proposition \ref{prop:222-invariant}. 
Therefore, by Proposition \ref{prop:222-rat-orbits},  
$M_{\be_{216}\,k}\backslash Z^{\sst}_{216\,k}$ 
is in bijective correspondence with $\Ex_2(k)$. 

Let $R(216)\in Z_{216}$ be the element such that 
$A(R(216))=(1,0\ccd 0,1)$ and the $x_{452}$-coordinate is $1$
(see (\ref{eq:w-efn-222-matrix})).  
Explicitly, $R(216)=e_{452}+e_{243}+e_{354}$. 
Then $P(R(216))=1$ and so $R(216)\in Z^{\sst}_{216}$. 

We assume that $u_{132}=u_{154}=0$ and $u_{243}=0$.
Then the first component of $n(u)R(216)$ 
is $0$ and the remaining components are as follows: 
\begin{align*}
& \begin{pmatrix}
0 & 0 & 0 & 0 & 0 \\
0 & 0 & 0 & 0 & 0 \\
0 & 0 & 0 & 0 & 0 \\
0 & 0 & 0 & 0 & 1 \\
0 & 0 & 0 & -1 & 0 
\end{pmatrix},\;
\begin{pmatrix}
0 & 0 & 0 & 0 & 0 \\
0 & 0 & 0 & 1 & 0 \\
0 & 0 & 0 & 0 & 0 \\
0 & -1 & 0 & 0 & -u_{152}+u_{232} \\
0 & 0 & 0 & * & 0 
\end{pmatrix}, \;
\begin{pmatrix}
0 & 0 & 0 & 0 & 0 \\
0 & 0 & 0 & 0 & 0 \\
0 & 0 & 0 & 0 & 1 \\
0 & 0 & 0 & 0 & u_{143}+u_{242} \\
0 & 0 & -1 & * & 0 \\
\end{pmatrix}.
\end{align*}

We can apply Lemma \ref{lem:connected-criterion} 
to the map $\aff^{13}\to \aff^2$ defined by the sequence
\begin{math}
u_{143}+u_{242},u_{152}-u_{232}
\end{math}
where 
\begin{math}
u_{121},u_{131},u_{141},
u_{142},u_{151},u_{153},
u_{221},u_{231},u_{232},
u_{241},u_{242}
\end{math}
are extra variables.
So by Proposition \ref{prop:W-eliminate}, 
Property \ref{property:W-eliminate} holds for 
any $x\in Z^{\sst}_{216\,k}$. Therefore, 
$P_{\be_{216}\,k}\backslash Y^{\sst}_{216\,k}$ 
is in bijective correspondence with $\Ex_2(k)$ also.

\vskip 5pt

(40) $S_{217}$, 
$\be_{217} = \tfrac {1} {180} (-32,-12,8,8,28,-45,-5,15,35)$ 

We identify the element 
$(\diag(t_{11},t_{12},g_1,t_{13}),\diag(t_{21}\ccd t_{24}))
\in M_{[1,2,4],[1,2,3]}=M_{\be_{217}}$ 
with the element 
$g=(g_1,t_{11}\ccd t_{24})\in \gl_2\times \gl_1^7$. 
On $M^{\text{st}}_{\be_{217}}$, 
\begin{equation*}
\chi_{{217}}(g) = 
(t_{11})^{-32}t_{12}^{-12}(\det g_1)^8t_{13}^{28}
t_{21}^{-45}t_{22}^{-5}t_{23}^{15}t_{22}^{35}
= t_{12}^{20}(\det g_1)^{40}t_{13}^{60}
t_{22}^{40}t_{23}^{60}t_{24}^{80}. 
\end{equation*}

For $x\in Z_{217}$, let 
\begin{math}
A(x) = 
\left(
\begin{smallmatrix}
x_{352} & x_{234} \\
x_{452} & x_{244} 
\end{smallmatrix}
\right).
\end{math}
We identify $Z_{217}\cong \m_2\oplus 1^{3\oplus}$ 
by the map $Z_{217}\ni x\mapsto (A(x),x_{253},x_{343},x_{154})$. 

It is easy to see that 
\begin{math}
A(gx) = g_1 A(x) 
\diag(t_{13}t_{22},t_{12}t_{24}).
\end{math}
We put 
$P_1(x)=\det A(x)$ and 
$P(x)= P_1(x)^2x_{253}x_{343}^2x_{154}^2$. 
Then on $M^{\text{st}}_{\be_{217}}$, 
\begin{align*}
P_1(gx) & = t_{12}(\det g_1)t_{13}t_{22}t_{24}P_1(x), \\
P(gx) & = (t_{12}(\det g_1)t_{13}t_{22}t_{24})^2
(t_{12}t_{13}t_{23}) ((\det g_1)t_{23})^2 (t_{11}t_{13}t_{24})^2P(x) \\
& = t_{11}^2t_{12}^3(\det g_1)^4t_{13}^5 
t_{22}^2t_{23}^3t_{24}^4 P(x) 
= t_{12}(\det g_1)^2t_{13}^3 t_{22}^2t_{23}^3t_{24}^4 P(x).  
\end{align*}
Therefore, $P(x)$ is 
invariant under the action of $G_{\text{st},\be_{217}}$. 

Let $R(217)\in Z_{217}$ be element such that 
\begin{math}
A(R(217)) = I_2
\end{math}
and the $x_{253},x_{154},x_{344}$-coordinates are $1$. 
Explicitly, 
$R(217)=e_{352}+e_{253}+e_{343}+e_{154}+e_{244}$. 
Then $P(R(217))=1$ and so $R(217)\in Z^{\sst}_{217}$. 

We show that  $Z^{\sst}_{217\,k}=M_{\be_{217}\,k}R(217)$. 
Suppose that $x\in Z^{\sst}_{217\,k}$. 
It is easy to see that there exists $g\in M_{\be_{217}\,k}$ 
such that $A(gx)=I_2$. So we may assume that $A(x)=I_2$. 
Let $t=(\diag(t_{22}^{-1},t_{12}^{-1}),t_{11},t_{12},1,1,t_{22},$ $1,1)$.
Then $A(tx)=I_2$ and the 
$x_{253},x_{343},x_{154}$-coordinates 
of $tx$ are $t_{12}x_{253}$,
$t_{12}^{-1}t_{22}^{-1}x_{343}$, $t_{11}x_{154}$ respectively. 
Therefore, $Z^{\sst}_{217\,k}=M_{\be_{217}\,k}R(217)$. 

We assume that $u_{143}=0$.
Then the first component of $n(u)R(217)$ 
is $0$ and the remaining components are as follows: 
\begin{align*}
& \begin{pmatrix}
0 & 0 & 0 & 0 & 0 \\
0 & 0 & 0 & 0 & 0 \\
0 & 0 & 0 & 0 & 1 \\
0 & 0 & 0 & 0 & 0 \\
0 & 0 & -1 & 0 & 0 
\end{pmatrix},\;
\begin{pmatrix}
0 & 0 & 0 & 0 & 0 \\
0 & 0 & 0 & 0 & 1 \\
0 & 0 & 0 & 1 & u_{132}+u_{154}+u_{232} \\
0 & 0 & -1 & 0 & u_{142}-u_{153} \\
0 & -1 & * & * & 0 
\end{pmatrix}, 
\end{align*}

\begin{align*}
& \begin{pmatrix}
0 & 0 & 0 & 0 & 1 \\
0 & 0 & 0 & 1 & u_{121}+u_{154}+u_{243} \\
0 & 0 & 0 & u_{132}+u_{243} & Q_1(u)+u_{131}+u_{242} \\
0 & -1 & * & 0 & Q_2(u)+u_{141}-u_{152} \\
-1 & * & * & * & 0 \\
\end{pmatrix}
\end{align*}
where $Q_1(u) = Q_1(u_{132},u_{154},u_{243})$ and 
$Q_2(u)=Q_2(u_{142},u_{153},u_{154},u_{243})$. 

We can apply Lemma \ref{lem:connected-criterion} 
to the map $\aff^{15}\to \aff^6$ defined by the sequence
\begin{align*}
& u_{121}+u_{154}+u_{243},u_{132}+u_{243},u_{232}+u_{132}+u_{154},
u_{142}-u_{153},\\
&  u_{131}+u_{242}+Q_1(u),u_{141}-u_{152}+Q_2(u)
\end{align*}
where $u_{151},u_{152},u_{153},u_{154},u_{221},u_{231},u_{241},u_{242},u_{243}$
are extra variables.
So by Proposition \ref{prop:W-eliminate}, 
Property \ref{property:W-eliminate} holds for 
any $x\in Z^{\sst}_{217\,k}$. Therefore, 
$Y^{\sst}_{217\,k}=P_{\be_{217}\,k}R(217)$ also. 

\vskip 5pt

(41) $S_{223}$, 
$\be_{223} = \tfrac {1} {140} (-36,-16,4,4,44,-15,-15,5,25)$ 

We identify the element 
$(\diag(t_{11},t_{12},g_1,t_{13}),\diag(g_2,t_{21},t_{22}))
\in M_{[1,2,4],[2,3]}=M_{\be_{223}}$ 
with the element 
$g=(g_1,g_2,t_{11}\ccd t_{22})\in \gl_2^2\times \gl_1^5$. 
On $M^{\text{st}}_{\be_{223}}$, 
\begin{equation*}
\chi_{{223}}(g) = 
t_{11}^{-36}t_{12}^{-16}(\det g_1)^4t_{13}^{44}
(\det g_2)^{-15}t_{21}^5t_{22}^{25}
= t_{12}^{20}(\det g_1)^{40}t_{13}^{80}
t_{21}^{20}t_{22}^{40}. 
\end{equation*}

For $x\in Z_{223}$, let 
\begin{math}
A(x) = 
\left(
\begin{smallmatrix}
x_{351} & x_{352} \\
x_{451} & x_{452} 
\end{smallmatrix}
\right).
\end{math}
We identify $Z_{223}\cong \m_2\oplus 1^{3\oplus}$
by the map $Z_{223}\ni x\mapsto (A(x),x_{253},x_{154},x_{344})$. 

It is easy to see that 
\begin{math}
A(gx) = t_{13}g_1 A(x) {}^t g_2. 
\end{math}
We put 
$P_1(x)=\det A(x)$ and 
$P(x) = (\det A(x))x_{253}^2x_{154}x_{344}^2$. 
Then on $M^{\text{st}}_{\be_{223}}$, 
\begin{align*}
P_1(gx) & = (\det g_1)t_{13}^2(\det g_2)P_1(x), \\
P(gx) & = ((\det g_1)t_{13}^2(\det g_2))
(t_{12}t_{13}t_{21})^2 (t_{11}t_{13}t_{22}) 
((\det g_1)t_{22})^2 P(x) \\
& = t_{11}t_{12}^2(\det g_1)^3t_{13}^5 
(\det g_2)t_{21}^2t_{22}^3 P(x) 
= t_{12}(\det g_1)^2t_{13}^4 t_{21}t_{22}^2 P(x).  
\end{align*}
Therefore, $P(x)$ is 
invariant under the action of $G_{\text{st},\be_{223}}$.

Let $R(223)\in Z_{223}$ be element such that 
\begin{math}
A(R(223)) = I_2
\end{math}
and the $x_{253},x_{154},x_{344}$-coordinates are $1$. 
Explicitly, $R(223)=e_{351}+e_{452}+e_{253}+e_{154}+e_{344}$. 
Then $P(R(223))=1$ and so $R(223)\in Z^{\sst}_{223}$. 

We show that  $Z^{\sst}_{223\,k}=M_{\be_{223}\,k}R(223)$. 
Suppose that $x\in Z^{\sst}_{223\,k}$. 
It is easy to see that there exists $g\in M_{\be_{223}\,k}$ 
such that $A(gx)=I_2$. So we may assume that $A(x)=I_2$. 
Let $t=(I_2,I_2,t_{11},t_{12},1,1,t_{22})$.
Then $A(tx)=I_2$ and the $x_{253},x_{154},x_{344}$-coordinates 
of $tx$ are $t_{12}x_{253},t_{11}t_{22}x_{154},t_{22}x_{344}$. 
Therefore, $Z^{\sst}_{223\,k}=M_{\be_{223}\,k}R(223)$. 

We assume that $u_{143}=0$ and $u_{221}=0$. 
Then the first two components of $n(u)R(223)$ 
is the same as those of $R(223)$ and the remaining 
components are as follows: 
\begin{align*}
& \begin{pmatrix}
0 & 0 & 0 & 0 & 0 \\
0 & 0 & 0 & 0 & 1 \\
0 & 0 & 0 & 0 & u_{132}+u_{231} \\
0 & 0 & 0 & 0 & u_{142}+u_{232} \\
0 & -1 & * & * & 0 \\
\end{pmatrix}, \;
\begin{pmatrix}
0 & 0 & 0 & 0 & 1 \\
0 & 0 & 0 & 0 & u_{121}+u_{243} \\
0 & 0 & 0 & 1 & Q_1(u)+u_{131}+u_{154}+u_{241} \\
0 & 0 & -1 & 0 & Q_2(u)+u_{141}-u_{153}+u_{242} \\
-1 & * & * & * & 0 
\end{pmatrix}
\end{align*}
where $Q_1(u),Q_2(u)$ are polynomials of
$u_{132},u_{142},u_{243}$. 

We can apply Lemma \ref{lem:connected-criterion} 
to the map $\aff^{14}\to \aff^5$ defined by the sequence
\begin{align*}
& u_{121}+u_{243},u_{231}+u_{132},u_{232}+u_{142}, \\
& u_{131}+u_{154}+u_{241}+Q_1(u),u_{141}-u_{153}+u_{242}+Q_2(u)
\end{align*}
where $u_{132},u_{142},u_{151},u_{152},u_{153},u_{154},u_{241},u_{242},u_{243}$
are extra variables.
So by Proposition \ref{prop:W-eliminate}, 
Property \ref{property:W-eliminate} holds for 
any $x\in Z^{\sst}_{223\,k}$. Therefore, 
$Y^{\sst}_{223\,k}=P_{\be_{223}\,k}R(223)$ also. 

\vskip 5pt

(42) $S_{224}$, 
$\be_{224} = \tfrac {1} {580} (-32,-12,8,8,28,-25,-25,-5,55)$ 

We identify the element 
$(\diag(t_{11},t_{12},g_1,t_{13}),\diag(g_2,t_{21},t_{22}))
\in M_{[1,2,4],[2,3]}=M_{\be_{224}}$ 
with the element 
$g=(g_1,g_2,t_{11}\ccd t_{22})\in \gl_2^2\times \gl_1^5$. 
On $M^{\text{st}}_{\be_{224}}$, 
\begin{equation*}
\chi_{{224}}(g)  
= t_{11}^{-32}t_{12}^{-12}(\det g_1)^8t_{13}^{28}(\det g_2)^{-25}t_{21}^{-5}t_{22}^{55} 
= t_{12}^{20}(\det g_1)^{40}t_{13}^{60}t_{21}^{20}t_{22}^{80}. 
\end{equation*}

For $x\in Z_{224}$, let 
\begin{math}
A(x) = 
\left(
\begin{smallmatrix}
x_{351} & x_{352} \\
x_{451} & x_{452} 
\end{smallmatrix}
\right).
\end{math}
We identify $Z_{224}\cong \m_2\oplus 1^{3\oplus}$ 
by the map $Z_{224}\ni x\mapsto (A(x),x_{253},x_{343},x_{124})$. 

It is easy to see that 
\begin{math}
A(gx) = t_{13}g_1 A(x) {}^t g_2.
\end{math}
We put $P_1(x)=\det A(x)$ and 
$P(x) = P_1(x)^6x_{253}x_{343}^6x_{124}^{10}$. 
Then on $M^{\text{st}}_{\be_{224}}$, 
\begin{align*}
P_1(gx) & = (\det g_1)t_{13}^2(\det g_2)P_1(x), \\
P(gx) & = ((\det g_1)t_{13}^2(\det g_2))^6
(t_{12}t_{13}t_{21}) ((\det g_1)t_{21})^6  
(t_{11}t_{12}t_{22})^{10} P(x) \\
& = t_{11}^{10}t_{12}^{11}(\det g_1)^{12}t_{13}^{13} 
(\det g_2)^6t_{21}^7t_{22}^{10} P(x) 
= t_{12}(\det g_1)^2t_{13}^3 t_{21}t_{22}^4 P(x).  
\end{align*}
Therefore, $P(x)$ is 
invariant under the action of $G_{\text{st},\be_{224}}$.

Let $R(224)\in Z_{224}$ be the element such that 
\begin{math}
A(R(224)) = I_2
\end{math}
and the $x_{253},x_{343},x_{124}$-coordinates are $1$. 
Explicitly, 
$R(224)=e_{351}+e_{452}+e_{253}+e_{343}+e_{124}$. 
Then $P(R(224))=1$ and so $R(224)\in Z^{\sst}_{224}$. 

We show that  $Z^{\sst}_{224\,k}=M_{\be_{224}\,k}R(224)$. 
Suppose that $x\in Z^{\sst}_{224\,k}$. 
It is easy to see that there exists $g\in M_{\be_{224}\,k}$ 
such that $A(gx)=I_2$. So we may assume that $A(x)=I_2$. 
Let $t=(I_2,I_2,t_{11},t_{12},1,t_{21},1)$.
Then $A(tx)=I_2$ and the $x_{253},x_{343},x_{124}$-coordinates 
of $tx$ are $t_{12}t_{21}x_{253},t_{21}x_{343},t_{11}t_{12}x_{124}$. 
Therefore, $Z^{\sst}_{224\,k}=M_{\be_{224}\,k}R(224)$. 

We assume that $u_{143}=u_{221}=0$. 
Then the four components of $n(u)R(224)$ 
are as follows: 
\begin{align*}
& \begin{pmatrix}
0 & 0 & 0 & 0 & 0 \\
0 & 0 & 0 & 0 & 0 \\
0 & 0 & 0 & 0 & 1 \\
0 & 0 & 0 & 0 & 0 \\
0 & 0 & -1 & 0 & 0 
\end{pmatrix},\;
\begin{pmatrix}
0 & 0 & 0 & 0 & 0 \\
0 & 0 & 0 & 0 & 0 \\
0 & 0 & 0 & 0 & 0 \\
0 & 0 & 0 & 0 & 1 \\
0 & 0 & 0 & -1 & 0 
\end{pmatrix}, \;
\begin{pmatrix}
0 & 0 & 0 & 0 & 0 \\
0 & 0 & 0 & 0 & 1 \\
0 & 0 & 0 & 1 & u_{132}+u_{154}+u_{231} \\
0 & 0 & -1 & 0 & u_{142}-u_{153}+u_{232} \\
0 & -1 & * & * & 0 \\
\end{pmatrix}, \\
& \begin{pmatrix}
0 & 1 & u_{132} & u_{142} & u_{152} \\
-1 & 0 & Q_1(u)-u_{131} & Q_2(u)-u_{141} & Q_3(u)-u_{151}+u_{243} \\
* & * & 0 & Q_4(u)+u_{243} & Q_5(u)+u_{241} \\
* & * & * & 0 & Q_6(u)+u_{242} \\
* & * & * & * & 0 
\end{pmatrix}
\end{align*}
where $Q_1(u),Q_2(u),Q_3(u)$ are polynomials of 
$u_{121},u_{132},u_{142},u_{152}$, 
$Q_4(u)$ is a polynomial of 
$u_{131},u_{132},u_{141},u_{142}$ 
and $Q_5(u),Q_6(u)$ do not depend on $u_{241},u_{242}$. 

We can apply Lemma \ref{lem:connected-criterion} 
to the map $\aff^{14}\to \aff^{11}$ defined by the sequence
\begin{align*}
& u_{132},u_{142},u_{152},u_{154}+u_{132}+u_{231},
u_{153}-u_{142}-u_{232}, 
u_{131}-Q_1(u),\\
&  u_{141}-Q_2(u),u_{243}+Q_4(u),
u_{151}-u_{243}-Q_3(u), 
u_{241}+Q_5(u),u_{242}+Q_6(u)
\end{align*}
where $u_{121},u_{231},u_{232}$ are extra variables.
So by Proposition \ref{prop:W-eliminate}, 
Property \ref{property:W-eliminate} holds for 
any $x\in Z^{\sst}_{224\,k}$. Therefore, 
$Y^{\sst}_{224\,k}=P_{\be_{224}\,k}R(224)$ also.

\vskip 5pt

(43) $S_{226}$, 
$\be_{226} = \tfrac {1} {420} (-48,-28,-8,32,52,-25,-5,15,15)$ 

We identify the element 
$(\diag(t_{11}\ccd t_{15}),\diag(t_{21},t_{22},g_2))
\in M_{[1,2,3,4],[1,2]}=M_{\be_{226}}$ 
with the element 
$g=(g_2,t_{11}\ccd t_{22})\in \gl_2\times \gl_1^7$. 
On $M^{\text{st}}_{\be_{226}}$, 
\begin{equation*}
\chi_{{226}}(g) 
= t_{11}^{-48}t_{12}^{-28}t_{13}^{-8}t_{14}^{32}t_{15}^{52}
t_{21}^{-25}t_{22}^{-5}(\det g_2)^{15}
= t_{12}^{20}t_{13}^{40}t_{14}^{80}t_{15}^{100}
t_{22}^{20}(\det g_2)^{40}. 
\end{equation*}

For $x\in Z_{226}$, let 
\begin{math}
A(x) = 
\left(
\begin{smallmatrix}
x_{153} & x_{154} \\
x_{243} & x_{244}
\end{smallmatrix}
\right).
\end{math}
We identify $Z_{226}\cong \m_2\oplus 1^{3\oplus}$ 
by the map $Z_{226}\ni x \mapsto (A(x),x_{351},x_{252},x_{342})$. 

It is easy to see that 
\begin{math}
A(gx) = \left(
\begin{smallmatrix}
t_{11}t_{15} & 0 \\
0 & t_{12}t_{14} 
\end{smallmatrix}
\right)
A(x) {}^t g_2.
\end{math}
We put $P_1(x)=\det A(x)$ and 
\begin{math}
P(x) = P_1(x)^6x_{351}^4x_{252}x_{342}^4. 
\end{math}
Then on $M^{\text{st}}_{\be_{226}}$, 
\begin{align*}
P_1(gx) & = t_{11}t_{12}t_{14}t_{15}(\det g_2)P_1(x), \\
P(gx) & = (t_{11}t_{12}t_{14}t_{15}(\det g_2))^6
(t_{13}t_{15}t_{21})^4 (t_{12}t_{15}t_{22})  
(t_{13}t_{14}t_{22})^4 P(x) \\
& = t_{11}^6t_{12}^7t_{13}^8t_{14}^{10}t_{15}^{11}
t_{21}^4t_{22}^5(\det g_2)^6 P(x)
= t_{12}t_{13}^2t_{14}^4t_{15}^5
t_{22}(\det g_2)^2 P(x). 
\end{align*}
Therefore, $P(x)$ is 
invariant under the action of $G_{\text{st},\be_{226}}$.

Let $R(226)\in Z_{226}$ be element such that 
\begin{math}
A(R(226)) = I_2
\end{math}
and the $x_{351},x_{252},x_{342}$-coordinates are $1$. 
Explicitly, 
$R(226)=e_{351}+e_{252}+e_{342}+e_{153}+e_{244}$. 
Then $P(R(226))=1$ and so $R(226)\in Z^{\sst}_{226}$. 

We show that  $Z^{\sst}_{226\,k}=M_{\be_{224}\,k}R(226)$. 
Suppose that $x\in Z^{\sst}_{226\,k}$. 
It is easy to see that there exists $g\in M_{\be_{226}\,k}$ 
such that $A(gx)=I_2$. So we may assume that $A(x)=I_2$. 
Let $t=(t_{11}^{-1}I_2,t_{11},t_{11},t_{13},1,1,t_{21},1)$.
Then $A(tx)=I_2$ and the $x_{351},x_{252},x_{342}$-coordinates 
of $tx$ are $t_{13}t_{21}x_{351},t_{11}x_{252},t_{13}x_{342}$. 
Therefore, $Z^{\sst}_{226\,k}=M_{\be_{226}\,k}R(226)$.

We assume that $u_{243}=0$. 
Then the four components of $n(u)R(226)$ 
are as follows: 
\begin{align*}
& \begin{pmatrix}
0 & 0 & 0 & 0 & 0 \\
0 & 0 & 0 & 0 & 0 \\
0 & 0 & 0 & 0 & 1 \\
0 & 0 & 0 & 0 & u_{143} \\
0 & 0 & -1 & * & 0 
\end{pmatrix},\;
\begin{pmatrix}
0 & 0 & 0 & 0 & 0 \\
0 & 0 & 0 & 0 & 1 \\
0 & 0 & 0 & 1 & u_{132}+u_{154}+u_{221} \\
0 & 0 & -1 & 0 & Q_1(u)+u_{142}-u_{153} \\
0 & -1 & * & * & 0 
\end{pmatrix}, \\
& \begin{pmatrix}
0 & 0 & 0 & 0 & 1 \\
0 & 0 & 0 & 0 & u_{121}+u_{232} \\
0 & 0 & 0 & u_{232} & Q_2(u)+u_{131}+u_{231} \\
0 & 0 & * & 0 & Q_3(u)+u_{141} \\
-1 & * & * & * & 0 \\
\end{pmatrix}, \;
\begin{pmatrix}
0 & 0 & 0 & 0 & 0 \\
0 & 0 & 0 & 1 & u_{154}+u_{242} \\
0 & 0 & 0 & u_{132}+u_{242} & Q_4(u)+u_{241} \\
0 & -1 & * & 0 & Q_5(u)-u_{152} \\
0 & * & * & * & 0 
\end{pmatrix}
\end{align*}
where $Q_1(u),Q_2(u),Q_4(u)$ are polynomials of 
$u_{132},u_{143},u_{154},u_{221},u_{232},u_{242}$, 
$Q_3(u)$ does not depend on $u_{141},u_{152}$ 
and $Q_5(u)$ does not depend on $u_{152}$.  

We can apply Lemma \ref{lem:connected-criterion} 
to the map $\aff^{15}\to \aff^{11}$ defined by the sequence
\begin{align*}
& u_{143},u_{232},u_{121}+u_{232},u_{242}+u_{132},
u_{154}+u_{242},u_{221}+u_{132}+u_{154}, \\
& u_{142}-u_{153}+Q_1(u), 
u_{131}+u_{231}+Q_2(u), 
u_{241}+Q_4(u), 
u_{141}+Q_3(u),
u_{152}-Q_5(u)
\end{align*}
where $u_{132},u_{151},u_{153},u_{231}$ are extra variables.
So by Proposition \ref{prop:W-eliminate}, 
Property \ref{property:W-eliminate} holds for 
any $x\in Z^{\sst}_{226\,k}$. Therefore, 
$Y^{\sst}_{226\,k}=P_{\be_{226}\,k}R(226)$ also. 

\vskip 5pt

(44) $S_{227}$, 
$\be_{227}=\tfrac {1} {140} (-26,-6,4,4,24,-15,-5,5,15)$ 

We identify the element 
$(\diag(t_{11},t_{12},g_1,t_{13}),\diag(t_{21}\ccd t_{24}))
\in M_{[1,2,4],[1,2,3]}=M_{\be_{227}}$ 
with the element  
$g=(g_1,t_{11}\ccd t_{24})\in \gl_2\times \gl_1^7$. 
On $M^{\text{st}}_{\be_{227}}$, 
\begin{equation*}
\chi_{{227}}(g) 
= t_{11}^{-26}t_{12}^{-6}(\det g_1)^4t_{13}^{24}
t_{21}^{-15}t_{22}^{-5}t_{23}^5t_{24}^{15}
= t_{11}^{20}(\det g_1)^{30}t_{13}^{50}
t_{22}^{10}t_{23}^{20}t_{24}^{30}.
\end{equation*}

For $x\in Z_{227}$, let 
\begin{math}
A(x) = 
\left(
\begin{smallmatrix}
x_{351} & x_{234} \\
x_{451} & x_{244}
\end{smallmatrix}
\right).
\end{math}
We identify $Z_{227}\cong \m_2\oplus 1^{3\oplus}$ 
by the map $Z_{227}\ni x\mapsto (A(x),x_{252},x_{343},x_{154})$. 

It is easy to see that 
\begin{math}
A(gx) = 
g_1 A(x) 
\diag(t_{13}t_{21},t_{12}t_{24}). 
\end{math}
We put $P_1(x)=\det A(x)$ and 
\begin{math}
P(x) = P_1(x)^2x_{252}^3x_{343}^4x_{154}^3. 
\end{math}
Then on $M^{\text{st}}_{\be_{227}}$, 
\begin{align*}
P_1(gx) & = t_{12}(\det g_1)t_{13}t_{21}t_{24}P_1(x), \\
P(gx) & = (t_{12}(\det g_1)t_{13}t_{21}t_{24})^2
(t_{12}t_{13}t_{22})^3 ((\det g_1)t_{23})^4  
(t_{11}t_{13}t_{24})^3 P(x) \\
& = t_{11}^3t_{12}^5(\det g_1)^6t_{13}^8
t_{21}^2t_{22}^3t_{23}^4t_{24}^5 P(x)
= t_{12}^2(\det g_1)^3t_{13}^5
t_{22}t_{23}^2t_{24}^3P(x).
\end{align*}
Therefore, $P(x)$ is 
invariant under the action of $G_{\text{st},\be_{227}}$.

Let $R(227)\in Z_{227}$ be element such that 
\begin{math}
A(R(227)) = I_2
\end{math}
and the $x_{252},x_{343},x_{154}$-coordinates are $1$. 
Explicitly, 
$R(227)=e_{351}+e_{252}+e_{343}+e_{154}+e_{244}$. 
Then $P(R(227))=1$ and so $R(227)\in Z^{\sst}_{227}$. 

We show that  $Z^{\sst}_{227\,k}=M_{\be_{227}\,k}R(227)$. 
Suppose that $x\in Z^{\sst}_{227\,k}$. 
It is easy to see that there exists $g\in M_{\be_{227}\,k}$ 
such that $A(gx)=I_2$. So we may assume that $A(x)=I_2$. 
Let $t=(t_{12}^{-1}I_2,t_{11},t_{12},1,t_{12},1,t_{23},1)$.
Then $A(tx)=I_2$ and the $x_{252},x_{343},x_{154}$-coordinates 
of $tx$ are $t_{12}x_{252},t_{12}^{-2}t_{23}x_{343},t_{11}x_{154}$. 
Therefore, $Z^{\sst}_{227\,k}=M_{\be_{227}\,k}R(227)$. 

We assume that $u_{143}=0$. 
Then the four components of $n(u)R(227)$ 
are as follows: 
\begin{align*}
& \begin{pmatrix}
0 & 0 & 0 & 0 & 0 \\
0 & 0 & 0 & 0 & 1 \\
0 & 0 & 0 & 0 & u_{132}+u_{221} \\
0 & 0 & 0 & 0 & u_{142} \\
0 & -1 & * & * & 0 
\end{pmatrix}, \;
\begin{pmatrix}
0 & 0 & 0 & 0 & 0 \\
0 & 0 & 0 & 0 & u_{232} \\
0 & 0 & 0 & 1 & Q_1(u)+u_{154}+u_{231} \\
0 & 0 & -1 & 0 & Q_2(u)-u_{153} \\
0 & * & * & * & 0 \\
\end{pmatrix}, \\
& \begin{pmatrix}
0 & 0 & 0 & 0 & 1 \\
0 & 0 & 0 & 1 & u_{121}+u_{154}+u_{242} \\
0 & 0 & 0 & u_{132}+u_{243} & Q_3(u)+u_{131}+u_{241} \\
0 & -1 & * & 0 & Q_4(u)+u_{141}-u_{152} \\
-1 & * & * & * & 0 
\end{pmatrix}
\end{align*}
where $Q_1(u),Q_2(u)$ are polynomials of $u_{132},u_{142},u_{232}$
and $Q_3(u),Q_4(u)$ do not depend on $u_{131},u_{141},u_{152},u_{241}$.

We can apply Lemma \ref{lem:connected-criterion} 
to the map $\aff^{15}\to \aff^9$ defined by the sequence
\begin{align*}
& u_{142},u_{232},u_{132}+u_{221},u_{243}+u_{132},
u_{121}+u_{154}+u_{242}, \\
& u_{231}+u_{154}+Q_1(u),
u_{153}-Q_2(u), 
u_{131}+u_{241}+Q_3(u),
u_{141}-u_{152}+Q_4(u), 
\end{align*}
where $u_{151},u_{152},u_{154},u_{221},u_{241},u_{242}$ are extra variables.
So by Proposition \ref{prop:W-eliminate}, 
Property \ref{property:W-eliminate} holds for 
any $x\in Z^{\sst}_{227\,k}$. Therefore, 
$Y^{\sst}_{227\,k}=P_{\be_{227}\,k}R(227)$. 

\vskip 5pt

(45) $S_{232}$, 
$\be_{232} = \tfrac {1} {140} (-16,-6,-6,4,24,-15,-5,-5,25)$ 

We identify the element 
$(\diag(t_{11},g_1,t_{12},t_{13}),\diag(t_{21},g_2,t_{22}))
\in M_{[1,3,4],[1,3]}=M_{\be_{232}}$ 
with the element 
$g=(g_1,g_2,t_{11}\ccd t_{22})\in \gl_2^2\times \gl_1^5$. 
Then on $M^{\text{st}}_{\be_{232}}$, 
\begin{equation*}
\chi_{{232}}(g)  
= t_{11}^{-16}(\det g_1)^{-6}t_{12}^4t_{13}^{24}
t_{21}^{-15}(\det g_2)^{-5} t_{22}^{25}
= (\det g_1)^{10}t_{12}^{20}t_{13}^{40}
(\det g_2)^{10} t_{22}^{40}. 
\end{equation*}

For $x\in Z_{232}$, let 
\begin{math}
A(x) = 
\left(
\begin{smallmatrix}
x_{252} & x_{253} \\
x_{352} & x_{353}
\end{smallmatrix}
\right).
\end{math}
We identify $Z_{232}\cong \m_2\oplus 1^{3\oplus}$ 
by the map $Z_{232}\ni x\mapsto (A(x),x_{451},x_{144},x_{234})$. 

It is easy to see that 
\begin{math}
A(gx)  = t_{13} g_1 A(x) {}^tg_2.
\end{math}
We put $P_1(x)=\det A(x)$ and 
\begin{math}
P(x) = P_1(x)^3x_{451}^2x_{144}^4x_{234}^2. 
\end{math}
On $M^{\text{st}}_{\be_{232}}$, 
\begin{align*}
P_1(gx) & = (\det g_1)t_{13}^2(\det g_2)P_1(x), \\
P(gx) & = ((\det g_1)t_{13}^2(\det g_2))^3
(t_{12}t_{13}t_{21})^2 
(t_{11}t_{12}t_{22})^4 
((\det g_1)t_{22})^2 P(x) \\
& = t_{11}^4(\det g_1)^5t_{12}^6t_{13}^8
t_{21}^2(\det g_2)^3t_{22}^6 P(x)
= (\det g_1)t_{12}^2t_{13}^4(\det g_2)t_{22}^4P(x).
\end{align*}
Therefore, $P(x)$ is 
invariant under the action of $G_{\text{st},\be_{232}}$.

Let $R(232)\in Z_{232}$ be element such that 
\begin{math}
A(R(232)) = I_2
\end{math}
and the $x_{451},x_{144},x_{234}$-entries are $1$. 
Explicitly, 
$R(232)=e_{451}+e_{252}+e_{353}+e_{144}+e_{234}$. 
Then $P(R(232))=1$ and so $R(232)\in Z^{\sst}_{232}$. 

We show that  $Z^{\sst}_{232\,k}=M_{\be_{232}\,k}R(232)$. 
Suppose that $x\in Z^{\sst}_{232\,k}$. 
It is easy to see that there exists $g\in M_{\be_{232}\,k}$ 
such that $A(gx)=I_2$. So we may assume that $A(x)=I_2$. 
Let $t=(I_2,I_2,t_{11},1,1,t_{21},t_{22})$.
Then $A(tx)=I_2$ and the $x_{451},x_{144},x_{234}$-coordinates 
of $tx$ are $t_{21}x_{451},t_{11}t_{22}x_{144},t_{22}x_{234}$. 
Therefore, $Z^{\sst}_{232\,k}=M_{\be_{232}\,k}R(232)$. 

We assume that $u_{132}=0$ and $u_{232}=0$. 
Then the four components of $n(u)R(232)$ 
are as follows: 
\begin{align*}
& \begin{pmatrix}
0 & 0 & 0 & 0 & 0 \\
0 & 0 & 0 & 0 & 0 \\
0 & 0 & 0 & 0 & 0 \\
0 & 0 & 0 & 0 & 1 \\
0 & 0 & 0 & -1 & 0 
\end{pmatrix},\;
\begin{pmatrix}
0 & 0 & 0 & 0 & 0 \\
0 & 0 & 0 & 0 & 1 \\
0 & 0 & 0 & 0 & 0 \\
0 & 0 & 0 & 0 & u_{142}+u_{221} \\
0 & -1 & 0 & * & 0 
\end{pmatrix}, \;
\begin{pmatrix}
0 & 0 & 0 & 0 & 0 \\
0 & 0 & 0 & 0 & 0 \\
0 & 0 & 0 & 0 & 1 \\
0 & 0 & 0 & 0 & u_{143}+u_{231} \\
0 & 0 & -1 & * & 0 \\
\end{pmatrix}, \\
& \begin{pmatrix}
0 & 0 & 0 & 1 & u_{154} \\
0 & 0 & 1 & u_{121}+u_{143} & Q_1(u)+u_{153}+u_{242} \\
0 & -1 & 0 & u_{131}-u_{142} & Q_2(u)-u_{152}+u_{243} \\
-1 & * & * & 0 & Q_3(u)-u_{151}+u_{241} \\
* & * & * & * & 0 
\end{pmatrix}
\end{align*}
where $Q_1(u),Q_2(u)$ are polynomials of 
$u_{121},u_{131},u_{154}$ and 
$Q_3(u)$ does not depend on $u_{151},u_{241}$. 

We can apply Lemma \ref{lem:connected-criterion} 
to the map $\aff^{14}\to \aff^8$ defined by the sequence
\begin{align*}
& u_{154},u_{142}+u_{221},u_{143}+u_{231},
u_{121}+u_{143},u_{131}-u_{142}, \\
& u_{153}+u_{242}+Q_1(u),
u_{152}-u_{243}-Q_2(u), 
u_{151}-u_{241}-Q_3(u), 
\end{align*}
where $u_{141},u_{221},u_{231},u_{241},u_{242},u_{243}$ are extra variables.
So by Proposition \ref{prop:W-eliminate}, 
Property \ref{property:W-eliminate} holds for 
any $x\in Z^{\sst}_{232\,k}$. Therefore, 
$Y^{\sst}_{232\,k}=P_{\be_{232}\,k}R(232)$.

\vskip 5pt

(46) $S_{254}$, 
$\be_{254} = \tfrac {1} {20} (-3,-3,2,2,2,-5,0,0,5)$

We identify the element 
$(\diag(g_{11},g_{12}),\diag(t_{21},g_2,t_{22}))
\in M_{[2],[1,3]}=M_{\be_{254}}$ 
with the element 
$g=(g_{12},g_{11},g_2,t_{21},t_{22})\in 
\gl_3\times \gl_2^2\times \gl_1^2$. 
Then on $M^{\text{st}}_{\be_{254}}$, 
\begin{equation*}
\chi_{{254}}(g) = 
(\det g_{11})^{-3}(\det g_{12})^2
t_{21}^{-5}t_{22}^5
= (\det g_{12})^5(\det g_2)^5t_{22}^{10}
\end{equation*}

Let $\bbmp_{3,i},p_{3,ij}$ be as before. 
For $x\in Z_{254}$, let 
\begin{align*}
A_1(x) & = x_{342}p_{3,12}+x_{352} p_{3,13}+x_{452} p_{3,23}, \;
A_2(x) = x_{343}p_{3,12}+x_{353}p_{3,13}+x_{453}p_{3,23}, \\
B_1(x) & = x_{134}\bbmp_{3,1}+x_{144}\bbmp_{3,2}+x_{154}\bbmp_{3,3},\;
B_2(x) = x_{234}\bbmp_{3,1}+x_{244}\bbmp_{3,2}+x_{254}\bbmp_{3,3}
\end{align*}
and $A(x)=(A_1(x),A_2(x))$, $B(x)=(B_1(x),B_2(x))$. 
We regard that $A(x),B(x)$ are elements of 
$\Lam^{3,2}_{1,[3,5]}\otimes \Lam^{2,1}_{2,[2,3]} 
\cong \wedge^2 \aff^3\otimes \aff^2$, 
$\Lam^{3,1}_{1,[3,5]} \otimes \Lam^{2,1}_{1,[1,2]}
\cong \aff^3 \otimes \aff^2$ respectively. 
We identify $Z_{254}\cong (\wedge^2\aff^3)\otimes \aff^2\oplus \aff^3\otimes \aff^2$
by the map $Z_{254}\ni x\mapsto (A(x),B(x))$. 

It is easy to see that 
\begin{equation*}
\begin{pmatrix}
A_1(gx) \\
A_2(gx)
\end{pmatrix}
= g_2\begin{pmatrix}
(\wedge^2 g_{12})A_1(x) \\
(\wedge^2 g_{12})A_2(x)
\end{pmatrix},\; 
\begin{pmatrix}
B_1(gx) \\
B_2(gx)
\end{pmatrix}
= t_{22}g_{11}\begin{pmatrix}
g_{12}B_1(x) \\
g_{12}B_2(x)
\end{pmatrix}.
\end{equation*}

We define $\Phi(x) = A(x)\wedge B(x)\in \m_2$ 
regarding $A(x),B(x)$ as column, row vectors respectively 
and taking $\wedge$ of entries. Then 
$\Phi(gx)=(\det g_{12})t_{22}g_2\Phi(x){}^tg_{11}$.
Let $P(x)=\det \Phi(x)$.  Then on $M^{\text{st}}_{\be_{254}}$, 
\begin{align*}
P(gx) & = (\det g_{11})(\det g_{12})^2 (\det g_2)t_{22}^2 P(x)
= (\det g_{12}) (\det g_2)t_{22}^2 P(x). 
\end{align*}
Therefore, $P(x)$ is 
invariant under the action of $G_{\text{st},\be_{254}}$.

Let $R(254)\in Z_{254}$ be the element such that 
\begin{equation*}
A(R(254)) = 
[p_{3,23},-p_{3,13}],\;
B(R(254)) = 
[\bbmp_{3,1},\bbmp_{3,2}].
\end{equation*}
Explicitly, 
$R(254)=e_{452}-e_{353}+e_{134}+e_{244}$. 
Then $\Phi(R(254))=I_2$, 
$P(R(254))=1$ and so $R(254)\in Z^{\sst}_{254}$. 

We show that $Z^{\sst}_{254\,k}=M_{\be_{254}\,k}R(254)$. 
Suppose that $x\in Z^{\sst}_{254\,k}$. 
Since $A(x)\wedge B(x)$ is a non-singular matrix, the entries of 
$B(x)$ are linearly independent.  
So there exists $g\in M_{\be_{254}\,k}$ 
such that $B(gx)=B(R(254))$. 
If $g=(I_3,I_2,g_2,1,1)$ then 
$B(gx)=B(x)$ and $\Phi(gx)=g_2\Phi(x)$. 
Therefore, we may assume that $\Phi(x)=I_2$. 
This implies that $x_{452}=1,x_{453}=0,x_{352}=0,x_{353}=-1$ 
Let $g=(g_{12},I_2,I_2,1,1)$ where 
\begin{equation*}
g_{12} = 
\begin{pmatrix}
1 & 0 & x_{342} \\
0 & 1 & x_{343} \\
0 & 0 & 1
\end{pmatrix}.
\end{equation*}
Then $B(gR(254))=B(R(254))$ and 
$A(gx)=A(R(254))$. 
This implies that
$Z^{\sst}_{254\,k} = M_{\be_{254}\,k}R(254)$.  

We assume that $u_{1ij}=0$ unless 
$i=3,4,5,j=1,2$ and $u_{232}=0$. 
Then the first component of $n(u)R(254)$ 
is $0$ and the remaining components
are as follows: 
\begin{align*}
& \begin{pmatrix}
0 & 0 & 0 & 0 & 0 \\
0 & 0 & 0 & 0 & 0 \\
0 & 0 & 0 & 0 & 0 \\
0 & 0 & 0 & 0 & 1 \\
0 & 0 & 0 & -1 & 0 
\end{pmatrix},\;
\begin{pmatrix}
0 & 0 & 0 & 0 & 0 \\
0 & 0 & 0 & 0 & 0 \\
0 & 0 & 0 & 0 & -1 \\
0 & 0 & 0 & 0 & 0 \\
0 & 0 & 1 & 0 & 0 
\end{pmatrix}, \;
\begin{pmatrix}
0 & 0 & 1 & 0 & 0 \\
0 & 0 & 0 & 1 & 0 \\
-1 & 0 & 0 & u_{132}-u_{141} & -u_{151}-u_{243} \\
0 & -1 & * & 0 & -u_{152}+u_{242} \\
0 & 0 & * & * & 0 \\
\end{pmatrix}.
\end{align*}

We can apply Lemma \ref{lem:connected-criterion} 
to the map $\aff^{11}\to \aff^3$ defined by the sequence
\begin{math}
u_{132}-u_{141},u_{151}+u_{243},
u_{152}-u_{242}
\end{math}
where $u_{131},u_{141},u_{142},u_{221},u_{231},u_{241},u_{242},u_{243}$ 
are extra variables.
So by Proposition \ref{prop:W-eliminate}, 
Property \ref{property:W-eliminate} holds for 
any $x\in Z^{\sst}_{254\,k}$. Therefore, 
$Y^{\sst}_{254\,k}=P_{\be_{254}\,k}R(254)$ also. 

\vskip 5pt

(47) $S_{256}$, 
$\be_{256} = \tfrac {1} {20} (-8,2,2,2,2,-5,-5,5,5)$ 

We identify the element 
$(\diag(t_1,g_1),\diag(g_{21},g_{22}))
\in M_{[1],[2]}=M_{\be_{256}}$ 
with the element 
$g=(g_1,g_{21},g_{22},t_1)\in \gl_4\times \gl_2^2\times \gl_1$. 
On $M^{\text{st}}_{\be_{256}}$, 
\begin{equation*}
\chi_{{256}}(g) = 
t_1^{-8}(\det g_1)^2(\det g_{21})^{-5}(\det g_{22})^5
= (\det g_1)^{10}(\det g_{22})^{10}. 
\end{equation*}

For $x\in Z_{256}$, let 
\begin{equation*}
A_1(x) = 
\begin{pmatrix}
0 & x_{233} & x_{243} & x_{253} \\
- x_{233} & 0 & x_{343} & x_{353} \\
-x_{243} & -x_{343} & 0 & x_{453} \\
-x_{253} & -x_{353} & -x_{453} & 0  
\end{pmatrix},\;
A_2(x) = 
\begin{pmatrix}
0 & x_{234} & x_{244} & x_{254} \\
- x_{234} & 0 & x_{344} & x_{354} \\
-x_{244} & -x_{344} & 0 & x_{454} \\
-x_{254} & -x_{354} & -x_{454} & 0  
\end{pmatrix}
\end{equation*}
and $A(x)=(A_1(x),A_2(x))$. 
We identify $Z_{256}\cong \wedge^2 \aff^4\otimes \aff^2$
by the map $x\mapsto A(x)$. 
Note that $t_1,g_{21}$ acts trivially on $Z_{256}$. 
Let $P(x)$ be the degree $4$ polynomial on 
$Z_{\be_{256}}$ 
obtained by Proposition \ref{prop:42-invariant}. 
Then on $M^{\text{st}}_{\be_{256}}$, 
\begin{equation*}
P(gx) = (\det g_1)^2(\det g_{22})^2 P(x).
\end{equation*}
Therefore, $P(x)$ is 
invariant under the action of $G_{\text{st},\be_{256}}$.


By Proposition \ref{prop:Rational-orbits-3-cases}, 
$M_{\be_{256}\,k}\backslash Z^{\sst}_{256\,k}$ 
is in bijective correspondence with $\Ex_2(k)$. 
Since $W_{256}=\{0\}$, 
$P_{\be_{256}\,k}\backslash Y^{\sst}_{256\,k}$
is in bijective correspondence with $\Ex_2(k)$ also.

\vskip 5pt

(48) $S_{258}$, 
$\be_{258} = \frac {1} {20} (-8,-3,2,2,7,-5,0,0,5)$. 

We identify the element 
$(\diag(t_{11},t_{12},g_1,t_{13}),\diag(t_{21},g_2,t_{22}))
\in M_{[1,2,4],[1,3]}=M_{\be_{258}}$ 
with the element 
$g=(g_1,g_2,t_{11}\ccd t_{22})\in \gl_2^2\times \gl_1^5$. 
On $M^{\text{st}}_{\be_{258}}$, 
\begin{equation*}
\chi_{{258}}(g) = 
t_{11}^{-8}t_{12}^{-3}(\det g_1)^2t_{13}^7
t_{21}^{-5}t_{22}^5
= t_{12}^5(\det g_1)^{10}t_{13}^{15}
(\det g_2)^5t_{22}^{10}. 
\end{equation*}

For $x\in Z_{258}$, let 
\begin{math}
A(x) = 
\left(
\begin{smallmatrix}
x_{352} & x_{353} \\
x_{452} & x_{453} 
\end{smallmatrix}
\right).
\end{math}
We identify $Z_{258}\cong \m_2\oplus 1^{2\oplus}$ 
by the map $Z_{258}\ni x\mapsto (A(x),x_{254},x_{344})$. 

It is easy to see that 
$A(gx)= t_{13} g_1 A(x) {}^tg_2$. 
We put $P_1(x)=\det A(x)$ and 
$P(x) = P_1(x)x_{254}x_{344}$.
Then on $M^{\text{st}}_{\be_{258}}$, 
\begin{align*}
P_1(gx) & = (\det g_1)t_{13}^2(\det g_2)P_1(x), \\
P(gx) & = ((\det g_1)t_{13}^2(\det g_2))
(t_{12}t_{13}t_{22})
((\det g_1)t_{22}) P(x) \\
& = t_{12}(\det g_1)^2t_{13}^3
(\det g_2)t_{22}^2P(x).
\end{align*}
Therefore, $P(x)$ is 
invariant under the action of $G_{\text{st},\be_{258}}$.

Let $R(258)\in Z_{258}$ be the element such that 
$A(R(258))=I_2$ and the 
$x_{254},x_{344}$-coordinates are $1$.
Explicitly, 
$R(258)=e_{352}+e_{453}+e_{254}+e_{344}$. 
Then $P(R(258))=1$ and so $R(258)\in Z^{\sst}_{258}$. 

We show that $Z^{\sst}_{258\,k}=M_{\be_{258}\,k}R(258)$. 
Suppose that $x\in Z^{\sst}_{258\,k}$. 
It is easy to see that there exists $g\in M_{\be_{258}\,k}$ 
such that $A(gx)=I_2$. So we may assume that $A(x)=I_2$. 
Let $t=(I_2,I_2,1,t_{12},1,1,t_{22})$.
Then $A(tx)=I_2$ and the $x_{254},x_{344}$-coordinates 
of $tx$ are $t_{12}t_{22}x_{254},t_{22}x_{344}$. 
Therefore, $Z^{\sst}_{258\,k}=M_{\be_{258}\,k}R(258)$. 

We assume that $u_{143}=0$ and $u_{232}=0$. 
Then the first three components of $n(u)R(258)$ 
are the same as those of $R(258)$ and 
the last component is as follows: 
\begin{align*}
& 
\begin{pmatrix}
0 & 0 & 0 & 0 & 0 \\
0 & 0 & 0 & 0 & 1 \\
0 & 0 & 0 & 1 & u_{132}+u_{154}+u_{242} \\
0 & 0 & -1 & 0 & u_{142}-u_{153}+u_{243} \\
0 & -1 & * & * & 0 \\
\end{pmatrix}.
\end{align*}

We can apply Lemma \ref{lem:connected-criterion} 
to the map $\aff^{14}\to \aff^2$ defined by the sequence
\begin{math}
u_{132}+u_{154}+u_{242},u_{142}-u_{153}+u_{243}
\end{math}
where variables other than $u_{132},u_{142}$ 
are extra variables.
So by Proposition \ref{prop:W-eliminate}, 
Property \ref{property:W-eliminate} holds for 
any $x\in Z^{\sst}_{258\,k}$. Therefore, 
$Y^{\sst}_{258\,k}=P_{\be_{258}\,k}R(258)$ also. 

\vskip 5pt

(49) $S_{259}$, 
$\be_{259} = \tfrac {1} {60}(-14,-4,-4,6,16,-15,-5,5,15)$ 

We identify the element 
\begin{equation*}
(\diag(t_{11},g_1,t_{12},t_{13}),\diag(t_{21}\ccd t_{24}))
\in M_{[1,3,4],[1,2,3]}=M_{\be_{259}}
\end{equation*}
with the element 
$g=(g_1,t_{11}\ccd t_{24})\in \gl_2\times \gl_1^7$. 
On $M^{\text{st}}_{\be_{259}}$, 
\begin{equation*}
\chi_{{259}}(g) = 
t_{11}^{-14}(\det g_1)^{-4}t_{12}^6t_{13}^{16}
t_{21}^{-15}t_{22}^{-5}t_{23}^5t_{24}^{15} \\
= (\det g_1)^{10}t_{12}^{20}t_{13}^{30}
t_{22}^{10}t_{23}^{20}t_{24}^{30}. 
\end{equation*}

For $x\in Z_{259}$, let 
\begin{math}
A(x) = 
\left(
\begin{smallmatrix}
x_{253} & x_{244} \\
x_{353} & x_{344}
\end{smallmatrix}
\right).
\end{math}
We identify $Z_{259}\cong \m_2\oplus 1^{2\oplus}$ 
by the map $Z_{259}\ni x\mapsto (A(x),x_{452},x_{154})$. 

It is easy to see that 
\begin{math}
A(gx) = g_1 A(x) \diag(t_{13}t_{23},t_{12}t_{24}),
\end{math}
We put $P_1(x)=\det A(x)$ and $P(x) = P_1(x)^2x_{452}x_{154}$.  
Then on $M^{\text{st}}_{\be_{259}}$, 
\begin{align*}
P_1(gx) & = (\det g_1)t_{12}t_{13}t_{23}t_{24} P_1(x), \\
P(gx) & = ((\det g_1)t_{12}t_{13}t_{23}t_{24})^2
(t_{12}t_{13}t_{22})
(t_{11}t_{13}t_{24}) P(x) \\
& = t_{11}(\det g_1)^2t_{12}^3t_{13}^4t_{22}t_{23}^2t_{24}^3
= (\det g_1)t_{12}^2t_{13}^3t_{22}t_{23}^2t_{24}^3.
\end{align*}
Therefore, $P(x)$ is 
invariant under the action of $G_{\text{st},\be_{259}}$.

Let $R(259)\in Z_{259}$ be the element such that 
$A(R(259))=I_2$ and the $x_{452},x_{154}$-coordinates are $1$.
Explicitly, 
$R(259)=e_{452}+e_{253}+e_{154}+e_{344}$. 
Then $P(R(259))=1$ and so $R(259)\in Z^{\sst}_{259}$. 

We show that $Z^{\sst}_{259\,k}=M_{\be_{259}\,k}R(259)$. 
Suppose that $x\in Z^{\sst}_{259\,k}$. 
It is easy to see that there exists $g\in M_{\be_{259}\,k}$ 
such that $A(gx)=I_2$. So we may assume that $A(x)=I_2$. 
Let $t=(I_2,t_{11},1,1,t_{22},1,1)$.
Then $A(tx)=I_2$ and the $x_{452},x_{154}$-coordinates 
of $tx$ are $t_{22}x_{452},t_{11}x_{154}$. 
Therefore, $Z^{\sst}_{259\,k}=M_{\be_{259}\,k}R(259)$. 

We assume that $u_{132}=0$. 
Then the first two components of $n(u)R(232)$ 
are the same as those of $R(259)$ and 
the remaining components are as follows: 
\begin{align*}
& \begin{pmatrix}
0 & 0 & 0 & 0 & 0 \\
0 & 0 & 0 & 0 & 1 \\
0 & 0 & 0 & 0 & 0 \\
0 & 0 & 0 & 0 & u_{142}+u_{232} \\
0 & -1 & 0 & * & 0 
\end{pmatrix}, \;
\begin{pmatrix}
0 & 0 & 0 & 0 & 1 \\
0 & 0 & 0 & 0 & u_{121}+u_{243} \\
0 & 0 & 0 & 1 & u_{131}+u_{154} \\
0 & 0 & -1 & 0 & Q(u)+u_{141}-u_{153}+u_{242} \\
-1 & * & * & * & 0 \\
\end{pmatrix}
\end{align*}
where $Q(u)$ does not depend on $u_{141},u_{153},u_{242}$. 

We can apply Lemma \ref{lem:connected-criterion} 
to the map $\aff^{15}\to \aff^4$ defined by the sequence
\begin{align*}
& u_{142}+u_{232},u_{121}+u_{243},u_{131}+u_{154},
u_{141}-u_{153}+u_{242}+Q(u)
\end{align*}
where variables other than $u_{121},u_{131},u_{141},u_{142}$ 
are extra variables.
So by Proposition \ref{prop:W-eliminate}, 
Property \ref{property:W-eliminate} holds for 
any $x\in Z^{\sst}_{259\,k}$. Therefore, 
$Y^{\sst}_{259\,k}=P_{\be_{259}\,k}R(259)$ also. 

\vskip 5pt

(50) $S_{270}$, 
$\be_{270}=\tfrac {3} {20}(-1,-1,-1,-1,4,0,0,0,0)$ 

Since $G_{\text{st},\be_{270}}$ is semi-simple, 
any relative invariant polynomial is invariant. 
We identify
\begin{math} 
Z_{{270}}\cong \Lam^{4,1}_{1,[1,4]}
\otimes \Lam^{4,1}_{2,[1,4]}
\cong \m_4
\end{math}
and  put $P(x) = \det x$. 
Since $P(x)$ is a relative invariant 
polynomial, it is invariant under the action 
of $G_{\text{st},\be_{270}}$.

Let $R(270)\in Z_{270}$ be the element which corresponds
to $I_4$.  Explicitly, 
$R(270)=e_{151}+e_{252}+e_{353}+e_{454}$. 
Then $P(R(270))=1$ and so $R(270)\in Z^{\sst}_{270}$. 
It is easy to see that 
$Z^{\sst}_{270\,k}=M_{\be_{270}\,k}R(270)$. 
Since $W_{270}=\{0\}$, 
$Y^{\sst}_{270\,k}=P_{\be_{270}\,k}R(270)$. 

\vskip 5pt

(51) $S_{271}$, 
$\be_{271} = \tfrac {1} {60}(-24,-4,6,6,16,-5,-5,5,5)$ 

We identify the element 
$(\diag(t_{11},t_{12},g_1,t_{13}),\diag(g_{21},g_{22}))
\in M_{[1,2,4],[2]}=M_{\be_{271}}$ 
with the element 
$g=(g_1,g_{21},g_{22},t_{11},t_{12},t_{13})\in \gl_2^3\times \gl_1^3$. 
On $M^{\text{st}}_{\be_{271}}$, 
\begin{equation*}
\chi_{{271}}(g) 
= t_{11}^{-24}t_{12}^{-4}(\det g_1)^6t_{13}^{16}
(\det g_{11})^{-5}(\det g_{12})^5
= t_{12}^{20}(\det g_1)^{30}t_{13}^{40}
(\det g_{12})^{10}. 
\end{equation*}

For $x\in Z_{271}$, let 
\begin{equation*}
A(x) = \begin{pmatrix}
x_{351} & x_{352} \\
x_{451} & x_{452}
\end{pmatrix},\; 
B(x) = \begin{pmatrix}
x_{253} & x_{254} \\
x_{343} & x_{344} 
\end{pmatrix}.
\end{equation*}
We identify $Z_{271}\cong \m_2\oplus \m_2$ 
by the map $Z_{271}\ni x\mapsto (A(x),B(x))$. 
It is easy to see that 
\begin{align*}
& A(gx) = t_{13}g_1 A(x) {}^tg_{21},\;
B(gx) = 
\begin{pmatrix}
t_{12}t_{13} & 0 \\
0 & \det g_1
\end{pmatrix}
B(x) {}^t g_{22}.
\end{align*}

We put $P_1(x)=\det A(x),\; P_2(x)=\det B(x)$
and $P(x)=P_1(x)P_2(x)^2$. Then on $M^{\text{st}}_{\be_{271}}$, 
\begin{align*}
P_1(gx) & = (\det g_1)t_{13}^2(\det g_{21})P_1(x),\;
P_2(gx) = t_{12}(\det g_1)t_{13} (\det g_{22}) P_2(x), \\
P(gx) & = ((\det g_1)t_{13}^2(\det g_{21}))
(t_{12}(\det g_1)t_{13} (\det g_{22}))^2 P(x) \\
& = t_{12}^2(\det g_1)^3t_{13}^4 
(\det g_{21})(\det g_{22})^2P(x)
= t_{12}^2(\det g_1)^3t_{13}^4 (\det g_{22})P(x).
\end{align*}
Therefore, $P(x)$ is 
invariant under the action of $G_{\text{st},\be_{271}}$.

Let $R(271)\in Z_{271}$ be the element such that 
$A(R(271))=B(R(271))=I_2$.
Explicitly, 
$R(271)=e_{351}+e_{452}+e_{253}+e_{344}$. 
Then $P_1(R(271))=P_2(R(271))=1$. 
So $P(R(271))=1$ 
and $R(271)\in Z^{\sst}_{271}$. 
It is easy to see that 
$Z^{\sst}_{271\,k}=M_{\be_{271}\,k}R(271)$. 

We assume that $u_{143}=0$ 
and $u_{221}=u_{243}=0$.  
Then the first two components of $n(u)R(271)$ 
are the same as those of $R(271)$ 
and the remaining components are as follows: 
\begin{align*}
& \begin{pmatrix}
0 & 0 & 0 & 0 & 0 \\
0 & 0 & 0 & 0 & 1 \\
0 & 0 & 0 & 0 & u_{132}+u_{231} \\
0 & 0 & 0 & 0 & u_{142}+u_{232} \\
0 & -1 & * & * & 0 
\end{pmatrix},
\begin{pmatrix}
0 & 0 & 0 & 0 & 0 \\
0 & 0 & 0 & 0 & 0 \\
0 & 0 & 0 & 1 & u_{154}+u_{241} \\
0 & 0 & -1 & 0 & -u_{153}+u_{242} \\
0 & 0 & * & * & 0 
\end{pmatrix}.
\end{align*}

We can apply Lemma \ref{lem:connected-criterion} 
to the map $\aff^{13}\to \aff^4$ defined by the sequence
\begin{equation*}
u_{132}+u_{231},u_{142}+u_{232},u_{154}+u_{241},
u_{153}-u_{242}
\end{equation*}
where variables other than $u_{132},u_{142},u_{153},u_{154}$ 
are are extra variables.
So by Proposition \ref{prop:W-eliminate}, 
Property \ref{property:W-eliminate} holds for 
any $x\in Z^{\sst}_{271\,k}$. Therefore, 
$Y^{\sst}_{271\,k}=P_{\be_{271}\,k}R(271)$ also. 

\vskip 5pt

(52) $S_{272}$, 
$\be_{272} = \tfrac {1} {220}(-28,-28,-8,12,52,-35,5,5,25)$ 

We identify the element 
$(\diag(g_1,t_{11},t_{12},t_{13}),\diag(t_{21},g_2,t_{22}))
\in M_{[2,3,4],[1,3]}=M_{\be_{272}}$ 
with the element 
$g=(g_1,g_2,t_{11}\ccd t_{22})\in \gl_2^2\times \gl_1^5$. 
On $M^{\text{st}}_{\be_{272}}$, 
\begin{equation*}
\chi_{{272}}(g) 
= (\det g_1)^{-28}t_{11}^{-8}t_{12}^{12}t_{13}^{52}
t_{21}^{-35}(\det g_2)^5t_{22}^{25}
= t_{11}^{20}t_{12}^{40}t_{13}^{80}
(\det g_2)^{40}t_{22}^{60}. 
\end{equation*}

For $x\in Z_{272}$, let 
\begin{math}
A(x) = 
\left(
\begin{smallmatrix}
x_{152} & x_{153} \\
x_{252} & x_{253} 
\end{smallmatrix}
\right). 
\end{math}
We identify $Z_{272}\cong \m_2\oplus 1^{2\oplus}$ 
by the map $Z_{272}\ni x\mapsto (A(x),x_{451},x_{344})$. 
It is easy to see that $A(gx) = t_{13}g_1 A(x) {}^t g_2$. 

We put $P_1(x)=\det A(x)$ and $P(x) = P_1(x)^3x_{451}x_{344}^4$.
Then on $M^{\text{st}}_{\be_{272}}$, 
\begin{align*}
P_1(gx) & = (\det g_1)t_{13}^2(\det g_2) P_1(x), \\
P(gx) & = ((\det g_1)t_{13}^2(\det g_2))^3
(t_{12}t_{13}t_{21})
(t_{11}t_{12}t_{22})^4 P(x) \\
& = (\det g_1)^3t_{11}^4t_{12}^5t_{13}^7
t_{21}(\det g_2)^3t_{22}^4
= t_{11}t_{12}^2t_{13}^4(\det g_1)^2t_{22}^3.
\end{align*}
Therefore, $P(x)$ is 
invariant under the action of $G_{\text{st},\be_{272}}$.

Let $R(272)\in Z_{272}$ be the element such that 
$A(R(272))=I_2$ and the 
$x_{451},x_{344}$-coordinates are $1$. 
Explicitly, 
$R(272)=e_{451}+e_{152}+e_{253}+e_{344}$. 
Then $P(R(272))=1$ and so $R(272)\in Z^{\sst}_{272}$. 

We show that $Z^{\sst}_{272\,k}=M_{\be_{272}\,k}R(272)$. 
Suppose that $x\in Z^{\sst}_{272\,k}$. 
It is easy to see that there exists $g\in M_{\be_{272}\,k}$ 
such that $A(gx)=I_2$. So we may assume that $A(x)=I_2$. 
We put $t=(I_2,I_2,1,t_{12},1,1,t_{22})$.
Then $A(tx)=I_2$ and the $x_{451},x_{344}$-coordinates 
of $tx$ are $t_{12}x_{451},t_{12}t_{22}x_{344}$. 
Therefore, $Z^{\sst}_{272\,k}=M_{\be_{272}\,k}R(272)$. 

We assume that $u_{121}=0$ and $u_{232}=0$.  
Then the four components of $n(u)R(272)$ 
are as follows: 
\begin{align*}
& \begin{pmatrix}
0 & 0 & 0 & 0 & 0 \\
0 & 0 & 0 & 0 & 0 \\
0 & 0 & 0 & 0 & 0 \\
0 & 0 & 0 & 0 & 1 \\
0 & 0 & 0 & -1 & 0 
\end{pmatrix},\;
\begin{pmatrix}
0 & 0 & 0 & 0 & 1 \\
0 & 0 & 0 & 0 & 0 \\
0 & 0 & 0 & 0 & u_{131} \\
0 & 0 & 0 & 0 & u_{141}+u_{221} \\
-1 & 0 & * & * & 0 
\end{pmatrix}, \\
& \begin{pmatrix}
0 & 0 & 0 & 0 & 0 \\
0 & 0 & 0 & 0 & 1 \\
0 & 0 & 0 & 0 & u_{132} \\
0 & 0 & 0 & 0 & u_{142}+u_{231} \\
0 & -1 & * & * & 0 
\end{pmatrix},
\begin{pmatrix}
0 & 0 & 0 & 0 & u_{242} \\
0 & 0 & 0 & 0 & u_{243} \\
0 & 0 & 0 & 1 & Q_1(u)+u_{154} \\
0 & 0 & -1 & 0 & Q_2(u)-u_{153}+u_{241} \\
* & * & * & * & 0 
\end{pmatrix}
\end{align*}
where $Q_1(u)$ does not depend on 
$u_{153},u_{154},u_{241}$ 
and $Q_2(u)$ does not depend on 
$u_{153},u_{241}$.  

We can apply Lemma \ref{lem:connected-criterion} 
to the map $\aff^{14}\to \aff^8$ defined by the sequence
\begin{align*}
& u_{131},u_{132},u_{242},u_{243},
u_{141}+u_{221},u_{142}+u_{231}, 
u_{154}+Q_1(u),u_{153}-u_{241}-Q_2(u)
\end{align*}
where $u_{143},u_{151},u_{152},u_{221},u_{231},u_{241}$
are are extra variables.
So by Proposition \ref{prop:W-eliminate}, 
Property \ref{property:W-eliminate} holds for 
any $x\in Z^{\sst}_{272\,k}$. Therefore, 
$Y^{\sst}_{272\,k}=P_{\be_{272}\,k}R(272)$ also.

\vskip 5pt

(53) $S_{273}$, 
$\be_{273} = \tfrac {1} {140} (-6,-6,4,4,4,-5,-5,-5,15)$ 

We identify the element 
$(\diag(g_{11},g_{12}),\diag(g_2,t_2))
\in M_{[2],[3]}=M_{\be_{273}}$ 
with the element 
$g=(g_{12},g_2,g_{11},t_2)\in \gl_3^2\times \gl_2\times \gl_1$. 
On $M^{\text{st}}_{\be_{273}}$, 
\begin{equation*}
\chi_{{273}}(g) = 
(\det g_{11})^{-6}(\det g_{12})^4(\det g_2)^{-5}t_2^{15}
= (\det g_{12})^{10}t_2^{20}.
\end{equation*}

For $x\in Z_{273}$, let 
\begin{equation*}
A(x) = 
\begin{pmatrix}
x_{341} & x_{342} & x_{343} \\
x_{351} & x_{352} & x_{353} \\
x_{451} & x_{452} & x_{453} 
\end{pmatrix}.
\end{equation*}
We identify $Z_{273}\cong \m_3\oplus 1$ 
by the map $Z_{273}\ni x\mapsto (A(x),x_{124})$. 
It is easy to see that 
$A(gx) = (\wedge^2 g_{12}) A(x) {}^t g_2$.  

We put $P_1(x) = \det A(x)$ and 
$P(x) = P_1(x)^3x_{124}^5$. 
Then on $M^{\text{st}}_{\be_{273}}$, 
\begin{align*}
P_1(gx) & = (\det g_{12})^2(\det g_2)P_1(x), \\
P(gx) & = ((\det g_{12})^2(\det g_2))^3
((\det g_{11})t_2)^5 P(x) \\
& = (\det g_{11})^5(\det g_{12})^6(\det g_2)^3t_2^5P(x)
= (\det g_{12})t_2^2P(x). 
\end{align*}
Therefore, $P(x)$ is 
invariant under the action of $G_{\text{st},\be_{273}}$.

Let $R(273)\in Z_{273}$ be the element such that 
$A(R(273))=I_3$ and the $x_{124}$-coordinate is $1$. 
Explicitly, 
$R(273)=e_{341}+e_{352}+e_{453}+e_{124}$. 
Then $P(R(273))=1$ and so $R(273)\in Z^{\sst}_{273}$. 

We show that $Z^{\sst}_{273\,k}=M_{\be_{273}\,k}R(273)$. 
Suppose that $x\in Z^{\sst}_{273\,k}$. 
It is easy to see that there exists $g\in M_{\be_{273}\,k}$ 
such that $A(gx)=I_3$. So we may assume that 
$A(x)=I_3$. Let $t=(I_3,I_3,I_2,t_2)$. Then 
$A(tx)=I_3$ and the $x_{124}$-coordinate of $tx$ 
is $t_2x_{124}$.  Therefore, 
$Z^{\sst}_{273\,k}=M_{\be_{273}\,k}R(273)$. 

We assume that $u_{1ij}=0$ unless 
$i=3,4,5,j=1,2$ and $u_{2ij}=0$ unless $i=4$.   
Then the first three components of $n(u)R(273)$ 
are the same as those of $R(273)$ and 
the last component is as follows: 
\begin{align*}
& \begin{pmatrix}
0 & 1 & u_{132} & u_{142} & u_{152} \\
-1 & 0 & -u_{131} & -u_{141} & -u_{151} \\
* & * & 0 & Q_1(u)+u_{241} & Q_2(u)+u_{242} \\
* & * & * & 0 & Q_3(u)+u_{243} \\
* & * & * & * & 0 
\end{pmatrix}
\end{align*}
where $Q_1(u),Q_2(u),Q_3(u)$ do not depend on 
$u_{241},u_{242},u_{243}$. 

We can apply Lemma \ref{lem:connected-criterion} 
to the map $\aff^9\to \aff^9$ defined by the sequence
\begin{align*}
& u_{131},u_{132},u_{141},u_{142},
u_{151},u_{152},
u_{241}+Q_1(u),
u_{242}+Q_2(u),
u_{243}+Q_3(u)
\end{align*}
with no extra variables.
So by Proposition \ref{prop:W-eliminate}, 
Property \ref{property:W-eliminate} holds for 
any $x\in Z^{\sst}_{273\,k}$. Therefore, 
$Y^{\sst}_{273\,k}=P_{\be_{273}\,k}R(273)$ also. 

\vskip 5pt

(54) $S_{280}$, 
$\be_{280}=\tfrac {1} {60} (-24,-4,-4,16,16,-15,-15,5,25)$ 

We identify the element 
$(\diag(t_1,g_{11},g_{12}),\diag(g_2,t_{21},t_{22}))
\in M_{[1,3],[2,3]}=M_{\be_{280}}$ 
with the element 
$g=(g_{11},g_{12},g_2,t_1,t_{21},t_{22})\in \gl_2^3\times \gl_1^3$. 
On $M^{\text{st}}_{\be_{280}}$, 
\begin{equation*}
\chi_{{280}}(g) 
= t_1^{-24}(\det g_{11})^{-4}(\det g_{12})^{16}
(\det g_2)^{-15}t_{21}^5t_{22}^{25}
= (\det g_{11})^{20}(\det g_{12})^{40}
t_{21}^{20}t_{22}^{40}.
\end{equation*}

For $x\in Z_{280}$, let
\begin{math}
A(x) = 
\left(
\begin{smallmatrix}
x_{244} & x_{254} \\
x_{344} & x_{354} 
\end{smallmatrix}
\right).
\end{math}
We identify $Z_{280}\cong \m_2\oplus 1$ 
by the map $Z_{280}\ni x\mapsto (A(x),x_{453})$. 
It is easy to see that 
\begin{math}
A(gx) = t_{22}g_{11} A(x) {}^t g_{12}.
\end{math}

We put  $P_1(x) = \det A(x)$ and $P(x) = P_1(x)x_{453}$. 
Then on $M^{\text{st}}_{\be_{280}}$, 
\begin{align*}
P_1(gx) & = (\det g_{11})(\det g_{12})t_{22}^2 P_1(x), \\
P(gx) & = ((\det g_{11})(\det g_{12})t_{22}^2)
((\det g_{12})t_{21}) P(x) 
= (\det g_{11})(\det g_{12})^2t_{21}t_{22}^2. 
\end{align*}
Therefore, $P(x)$ is 
invariant under the action of $G_{\text{st},\be_{280}}$.

Let $R(280)\in Z_{280}$ be the element such that 
$A(R(280))=I_2$ and the $x_{453}$-coordinate is $1$. 
Explicitly, 
$R(280)=e_{453}+e_{244}+e_{354}$. 
Then $P(R(280))=1$ and so $R(280)\in Z^{\sst}_{280}$. 

We show that $Z^{\sst}_{280\,k}=M_{\be_{280}\,k}R(280)$. 
Suppose that $x\in Z^{\sst}_{280\,k}$. 
It is easy to see that there exists $g\in M_{\be_{280}\,k}$
such that $A(gx)=I_2$. So we may assume that $A(x)=I_2$. 
Let $t=(I_2,I_2,I_2,1,t_{21},1)$. 
Then the $x_{453}$-coordinate of $tx$ is 
$t_{21}x_{453}$. Therefore,  
$Z^{\sst}_{280\,k}=M_{\be_{280}\,k}R(280)$. 

We assume that $u_{132}=0$ and $u_{221}=0$. 
Then the first three components of $n(u)R(280)$ 
are the same as those of $R(280)$ 
and the last component is as follows: 
\begin{align*}
& 
\begin{pmatrix}
0 & 0 & 0 & 0 & 0 \\
0 & 0 & 0 & 1 & 0 \\
0 & 0 & 0 & 0 & 1 \\
0 & -1 & 0 & 0 & u_{143}-u_{152}+u_{243} \\
0 & 0 & -1 & * & 0 
\end{pmatrix}.
\end{align*}

Obviously we can apply 
Lemma \ref{lem:connected-criterion} 
and by Proposition \ref{prop:W-eliminate}, 
Property \ref{property:W-eliminate} holds for 
any $x\in Z^{\sst}_{280\,k}$. Therefore, 
$Y^{\sst}_{280\,k}=P_{\be_{280}\,k}R(280)$ also.

\vskip 5pt

(55) $S_{281}$, 
$\be_{281} = \tfrac {1} {20} (-4,0,0,0,4,-5,-5,3,7)$ 

We identify the element 
$(\diag(t_{11},g_1,t_{12}),\diag(g_2,t_{21},t_{22}))
\in M_{[1,4],[2,3]}=M_{\be_{281}}$ 
with the element 
$g=(g_1,g_2,t_{11}\ccd t_{22})\in \gl_3 \times \gl_2\times \gl_1^4$. 
On $M^{\text{st}}_{\be_{281}}$, 
\begin{equation*}
\chi_{{281}}(g) 
= t_{11}^{-4}t_{12}^4
(\det g_2)^{-5}t_{21}^3t_{22}^7
= (\det g_1)^4t_{12}^8
t_{21}^8t_{22}^{12}.
\end{equation*}

For $x\in Z_{281}$, let 
\begin{align*}
A(x) & = x_{234}p_{3,12}+x_{244}p_{3,13}+x_{344}p_{3,23}, \;
B(x) = x_{253}\bbmp_{3,1}+x_{353}\bbmp_{3,2}+x_{453}\bbmp_{3,3}. 
\end{align*}
We identify $Z_{281}\cong \wedge^2 \aff^3 \oplus \aff^3\oplus 1$ 
by the map $Z_{281}\ni x\mapsto (A(x),B(x),x_{154})$. 
It is easy to see that 
\begin{math}
A(gx) = t_{22}(\wedge^2 g_1)A(x), \;
B(gx) = t_{12}t_{21}g_1B(x).
\end{math}
%

We put $P_1(x)=A(x)\wedge B(x)$  
and $P(x)=P_1(x)^2 x_{154}$. Then on $M^{\text{st}}_{\be_{281}}$, 
\begin{align*}
P_1(gx) & = (\det g_1)t_{12}t_{21}t_{22}P_1(x), \\
P(gx) & = ((\det g_1)t_{12}t_{21}t_{22})^2 
(t_{11}t_{12}t_{22})P(x) \\
& = t_{11}(\det g_1)^2t_{12}^3 t_{21}^2t_{22}^3P(x) 
= (\det g_1)t_{12}^2 t_{21}^2t_{22}^3P(x).  
\end{align*}
Therefore, $P(x)$ is 
invariant under the action of $G_{\text{st},\be_{281}}$.

Let $R(281)\in Z_{281}$ be the element such that 
$A(R(281))=p_{3,12}$, $B(R(281))=\bbmp_{3,3}$ 
and that the $x_{154}$-coordinate is $1$. 
Explicitly, $R(281)=e_{453}+e_{154}+e_{234}$. 
Then $P(R(281))=1$ and so $R(281)\in Z^{\sst}_{281}$. 

We show that $Z^{\sst}_{281\,k}= M_{\be_{281}\,k}R(281)$. 
Suppose that $x\in Z^{\sst}_{281\,k}$.
By assumption, $A(x)\not=0$. So 
by Lemma II--4.6, there exists $g\in M_{\be_{281}\,k}$ 
such that $A(gx)=p_{3,12}$. 
So we may assume that $A(x)=p_{3,12}$. Then 
$x_{453}\not=0$ by assumption. Let 
\begin{equation*}
g_1 = \begin{pmatrix}
1 & 0 & -x_{453}^{-1}x_{253} \\
0 & 1 & -x_{453}^{-1}x_{353} \\
0 & 0 & x_{453}^{-1} 
\end{pmatrix}. 
\end{equation*}
Then $gx = R(281)$. Therefore, 
$Z^{\sst}_{281\,k}=M_{\be_{281}\,k}R(281)$.

We assume that $u_{1ij}=0$ unless $i=5$ 
or $j=1$ and $u_{221}=0$. 
Then the first three components of $n(u)R(281)$ 
are the same as those of $R(281)$ and the 
last component is as follows: 
\begin{align*}
\begin{pmatrix}
0 & 0 & 0 & 0 & 1 \\
0 & 0 & 1 & 0 & u_{121}+u_{153} \\
0 & -1 & 0 & 0 & u_{131}-u_{152} \\
0 & 0 & 0 & 0 & u_{141}+u_{243} \\
-1 & * & * & * & 0 
\end{pmatrix}.
\end{align*}

Obviously we can apply 
Lemma \ref{lem:connected-criterion} 
and by Proposition \ref{prop:W-eliminate}, 
Property \ref{property:W-eliminate} holds for 
any $x\in Z^{\sst}_{281\,k}$. Therefore, 
$Y^{\sst}_{281\,k}=P_{\be_{281}\,k}R(281)$ also.

\vskip 5pt

(56) $S_{285}$, 
$\be_{285} = \tfrac {1} {60} (-24,-4,-4,-4,36,-15,5,5,5)$ 

We identify the element 
$(\diag(t_{11},g_1,t_{12}),\diag(t_2,g_2))
\in M_{[1,4],[1]}=M_{\be_{285}}$ 
with the element 
$g=(g_1,g_2,t_{11},t_{12},t_2)\in \gl_3^2\times \gl_1^3$. 
On $M^{\text{st}}_{\be_{285}}$, 
\begin{equation*}
\chi_{{285}}(g) 
= t_{11}^{-24}(\det g_1)^{-4}t_{12}^{36}
t_2^{-15} (\det g_2)^5
= (\det g_1)^{20}t_{12}^{60}
(\det g_2)^{20}. 
\end{equation*}

Let 
\begin{equation*}
A(x) = \begin{pmatrix}
x_{252} & x_{253} & x_{254} \\
x_{352} & x_{353} & x_{354} \\
x_{452} & x_{453} & x_{454}
\end{pmatrix}.
\end{equation*}
We identify $Z_{285}\cong \m_3$ by the map 
$Z_{285}\ni x\mapsto A(x)$.  
Let $P(x) = \det A(x)$. Then 
\begin{equation*}
A(gx) = t_{12}g_1 A(x) {}^t g_2,\;
P(gx) = (\det g_1)t_{12}^3(\det g_2)P(x).
\end{equation*}
Therefore, $P(x)$ is 
invariant under the action of $G_{\text{st},\be_{285}}$.

Let $R(285)\in Z_{285}$ be the element such that 
$A(R(285))=I_3$. 
Explicitly, 
$R(285)=e_{252}+e_{353}+e_{454}$. 
Then $P(R(285))=1$ and so $R(285)\in Z^{\sst}_{285}$. 
It is easy to see that 
$Z^{\sst}_{285\,k}=M_{\be_{285}\,k}R(285)$. 
Since $W_{285}=\{0\}$, 
$Y^{\sst}_{285\,k}=P_{\be_{285}\,k}R(285)$. 

\vskip 5pt

(57) $S_{286}$, 	
$\be_{286} = \tfrac {1} {60} (-24,-24,16,16,16,-15,5,5,5)$ 

We identify the element 
$(\diag(g_{11},g_{12}),\diag(t_2,g_2))
\in M_{[2],[1]}=M_{\be_{286}}$ 
with the element 
$g=(g_{12},g_2,g_{11},t_2)\in \gl_3^2\times \gl_2\times \gl_1$. 
On $M^{\text{st}}_{\be_{286}}$, 
\begin{equation*}
\chi_{{286}}(g)  
= (\det g_{11})^{-24}(\det g_{12})^{16}
t_2^{-15}(\det g_2)^5
= (\det g_{12})^{40}(\det g_2)^{20}. 
\end{equation*}

For $x\in Z_{286}$, let 
\begin{equation*}
A(x) = \begin{pmatrix}
x_{342} & x_{343} & x_{344} \\
x_{352} & x_{353} & x_{354} \\
x_{452} & x_{453} & x_{454}
\end{pmatrix}.
\end{equation*}
We identify $Z_{286}\cong \m_3$ by the map 
$Z_{286}\ni x\mapsto A(x)$. Let $P(x) = \det A(x)$.
Then 
\begin{equation*}
A(gx) = (\wedge^2 g_{12}) A(x) {}^t g_2,\;
P(gx)= (\det g_{12})^2(\det g_2)P(x).
\end{equation*}
Therefore, $P(x)$ is 
invariant under the action of $G_{\text{st},\be_{286}}$.

Let $R(286)\in Z_{286}$ be the element such that 
$A(R(286))=I_3$. 
Explicitly, 
$R(285)=e_{342}+e_{353}+e_{454}$. 
Then $P(R(286))=1$ and so $R(286)\in Z^{\sst}_{286}$. 
It is easy to see that 
$Z^{\sst}_{286\,k}=M_{\be_{286}\,k}R(286)$. 
Since $W_{286}=\{0\}$, 
$Y^{\sst}_{286\,k}=P_{\be_{286}\,k}R(286)$. 

\vskip 5pt

(58) $S_{287}$, 
$\be_{287} = \tfrac {1} {20} (-2,-2,0,0,4,-5,1,1,3)$ 

We identify the element 
$(\diag(g_{11},g_{12},t_1),\diag(t_{21},g_2,t_{22}))
\in M_{[2,4],[1,3]}=M_{\be_{287}}$ 
with the element 
$g=(g_{11},g_{12},g_2,t_1,t_{21},t_{22})\in \gl_2^3\times \gl_1^3$. 
On $M^{\text{st}}_{\be_{287}}$, 
\begin{equation*}
\chi_{{287}}(g) 
= (\det g_{11})^{-2}t_1^4
t_{21}^{-5}(\det g_2)t_{22}^3
= (\det g_{12})^2t_1^6
(\det g_2)^6t_{22}^8. 
\end{equation*}

For $x\in Z_{287}$, let
\begin{math}
A(x) = 
\left(
\begin{smallmatrix}
x_{152} & x_{153} \\
x_{252} & x_{253} 
\end{smallmatrix}
\right).
\end{math}
We identify $Z_{287}\cong \m_2\oplus 1$ 
by the map $Z_{287}\ni x\mapsto (A(x),x_{344})$. 
It is easy to see that
\begin{math}
A(gx) = t_1g_{11} A(x) {}^t g_2.
\end{math}

We put $P_1(x)=\det A(x)$ and $P(x) = P_1(x)^3x_{344}^4$.
Then on $M^{\text{st}}_{\be_{287}}$, 
\begin{align*}
P_1(gx) & = (\det g_{11})t_1^2(\det g_2) P_1(x), \\
P(gx) & = ((\det g_{11})t_1^2(\det g_2))^3
((\det g_{12})t_{22})^4 P(x) \\
& = (\det g_{11})^3(\det g_{12})^4t_1^6
(\det g_2)^3t_{22}^4. 
= (\det g_{12})t_1^3 (\det g_2)^3t_{22}^4. 
\end{align*}
Therefore, $P(x)$ is 
invariant under the action of $G_{\text{st},\be_{287}}$.

Let $R(287)\in Z_{287}$ be the element such that 
$A(R(287))=I_2$ and the $x_{344}$-coordinate is $1$. 
Explicitly, 
$R(287)=e_{152}+e_{253}+e_{344}$. 
Then $P(R(287))=1$ and so $R(287)\in Z^{\sst}_{287}$. 

We show that $Z^{\sst}_{287\,k}=M_{\be_{287}\,k}R(287)$. 
Suppose that $x\in Z^{\sst}_{287\,k}$. 
It is easy to see that there exists $g\in M_{\be_{287}\,k}$
such that $A(gx)=I_2$. So we may assume that $A(x)=I_2$. 
Let $t=(I_2,I_2,I_2,1,1,t_{22})$. 
Then $A(tx)=I_2$ and the $x_{344}$-coordinate of $tx$ 
is $t_{22}x_{344}$. Therefore, 
$Z^{\sst}_{287\,k}=M_{\be_{287}\,k}R(287)$. 

We assume that $u_{121}=u_{143}=0$ and 
$u_{232}=0$. 
Then the first  component of $n(u)R(287)$ 
is $0$ and the remaining components are as follows: 
\begin{align*}
& \begin{pmatrix}
0 & 0 & 0 & 0 & 1 \\
0 & 0 & 0 & 0 & 0 \\
0 & 0 & 0 & 0 & u_{131} \\
0 & 0 & 0 & 0 & u_{141} \\
-1 & 0 & * & * & 0 
\end{pmatrix},\;
\begin{pmatrix}
0 & 0 & 0 & 0 & 0 \\
0 & 0 & 0 & 0 & 1 \\
0 & 0 & 0 & 0 & u_{132} \\
0 & 0 & 0 & 0 & u_{142} \\
0 & -1 & * & * & 0 
\end{pmatrix},\;
\begin{pmatrix}
0 & 0 & 0 & 0 & u_{242} \\
0 & 0 & 0 & 0 & u_{243} \\
0 & 0 & 0 & 1 & Q_1(u)+u_{154} \\
0 & 0 & -1 & 0 & Q_2(u)-u_{153} \\
* & * & * & * & 0 
\end{pmatrix}
\end{align*}
where $Q_1(u),Q_2(u)$ do not depend on 
$u_{153},u_{154}$. 

We can apply Lemma \ref{lem:connected-criterion} 
to the map $\aff^{13}\to \aff^8$ defined by the sequence
\begin{align*}
& u_{131},u_{132},u_{141},u_{142},
u_{242},u_{243},
u_{154}+Q_1(u),
u_{153}-Q_2(u)
\end{align*}
where $u_{151},u_{152},u_{221},u_{231},u_{241}$
are extra variables.
So by Proposition \ref{prop:W-eliminate}, 
Property \ref{property:W-eliminate} holds for 
any $x\in Z^{\sst}_{287\,k}$. Therefore, 
$Y^{\sst}_{287\,k}=P_{\be_{287}\,k}R(287)$ also.

\vskip 5pt

(59) $S_{289}$, 
$\be_{289} = \tfrac {1} {20} (-8,2,2,2,2,-5,-5,-5,15)$ 

We identify the element 
$(\diag(t_1,g_1),\diag(g_2,t_2))
\in M_{[1],[3]}=M_{\be_{289}}$ 
with the element 
$g=(g_1,g_2,t_1,t_2)\in \gl_4 \times \gl_3\times \gl_1^2$. 
On $M^{\text{st}}_{\be_{289}}$, 
\begin{equation*}
\chi_{{289}}(g) 
= t_1^{-8}(\det g_1)^2(\det g_2)^{-5}t_2^{15}
= (\det g_1)^{10}t_2^{20}.
\end{equation*}

For $x\in Z_{289}$, let 
\begin{equation*}
A(x) = \begin{pmatrix}
0 & x_{234} & x_{244} & x_{254} \\
-x_{234} & 0 & x_{344} & x_{354} \\
-x_{244} & -x_{344} & 0 & x_{454} \\
-x_{254} & -x_{354} & -x_{454} & 0
\end{pmatrix}.
\end{equation*}
We identify $Z_{289}\cong \wedge^2 \aff^4$
by the map $Z_{289}\ni x\mapsto A(x)\in\wedge^2 \aff^4$. 
It is easy to see that $A(gx)=t_2g_1A(x){}^tg_1$. 

Let $P(x)$ be the Pfaffian of $A(x)$. Then 
$P(gx) = (\det g_1)t_2^2P(x)$. 
Therefore, $P(x)$ is 
invariant under the action of $G_{\text{st},\be_{289}}$.

Let $R(289)\in Z_{289}$ be the element such that 
\begin{math}
A(R(289)) = \left(
\begin{smallmatrix}
J & 0 \\ 0 & J  
\end{smallmatrix}
\right)
\end{math}
(see (\ref{eq:J-defn})). 
Then $P(R(289))=1$ and so $R(289)\in Z^{\sst}_{289}$. 
By Lemma II--4.6, $Z^{\sst}_{{289}\, k} = M_{\be_{289}\, k}R(289)$. 
Since $W_{289}=\{0\}$, 
$Y^{\sst}_{{289}\, k} = P_{\be_{289}\, k}R(289)$. 

\vskip 5pt

(60) $S_{291}$, 
$\be_{291} = \tfrac {1} {20} (-8,-8,2,2,12,-5,-5,5,5)$ 

We identify the element 
$(\diag(g_{11},g_{12},t_1),\diag(g_{21},g_{22}))
\in M_{[2,4],[2]}=M_{\be_{291}}$ 
with the element 
$g=(g_{11},g_{12},g_{21},g_{22},t_1)\in \gl_2^4\times \gl_1$. 
On $M^{\text{st}}_{\be_{291}}$, 
\begin{equation*}
\chi_{291}(g) 
= (\det g_{11})^{-8}(\det g_{12})^2t_1^{12}
(\det g_{21})^{-5}(\det g_{22})^5
= (\det g_{12})^{10}t_1^{20}
(\det g_{22})^{10}. 
\end{equation*}

For $x\in Z_{291}$, let 
\begin{math}
A(x) = 
\left(
\begin{smallmatrix}
x_{353} & x_{354} \\
x_{453} & x_{454}  
\end{smallmatrix}
\right).
\end{math}
We identify $Z_{291}\cong \m_2$ 
by the map $Z_{291}\ni x\mapsto A(x)\in\m_2$. 
It is easy to see that 
$A(gx) = t_1g_{12}A(x){}^t g_{22}$. 
Let $P(x) = \det A(x)$. Then 
$P(gx) = (\det g_{12})t_1^2 (\det g_{22})P(x)$. 
Therefore, $P(x)$ is 
invariant under the action of $G_{\text{st},\be_{291}}$.

Let $R(291)\in Z_{291}$ be the element such
that $A(R(291))=I_2$. Then $P(R(291))=1$ and 
$R(291)\in Z^{\sst}_{291}$. 
It is easy to see that 
$Z^{\sst}_{291} = M_{\be_{291}\, k} R(291)$. 
Since $W_{291}=\{0\}$, 
$Y^{\sst}_{{291}\, k} = P_{\be_{291}\, k}R(291)$. 

\vskip 5pt 

(61) $S_{292}$, 
$\be_{292} = \tfrac {1} {20} (-8,-8,-8,12,12,-5,-5,-5,15)$ 

Let $R(292)\in Z_{292}$ be the element such that 
$x_{451}=1$. It is easy to see that $P(x) = x_{451}$
is invariant under the action of $G_{\text{st},\be_{292}}$, 
$R(292)\in Z^{\sst}_{292}$ and that 
$Z^{\sst}_{292\, k} =M_{\be_{292}\, k} R(292)$. 
Since $W_{292}=\{0\}$, 
$Y^{\sst}_{{292}\, k} = P_{\be_{292}\, k}R(292)$. 

\vskip 5pt
This completes the proof of Theorem \ref{thm:non-empty}.

\section{Empty strata}
\label{sec:empty-strata}

In this section we prove that $S_{\be_i}=\emptyset$ 
for $i$ not in the list (\ref{eq:list-non-empty-54}) for 
the \pv{} (\ref{eq:PV}). 
We use the notation $S_i$, etc., in this section also.  
We assume that elements of $Z_i$ are in the form 
$x=\sum_j y_j \mathbbm a_j\in Z_i$ (see the end of Section \ref{sec:notation})
unless otherwise stated.  The reason why  
we use this coordinate system instead of $x_{ijk}$ is that 
it is convenient this way to check the computation by 
MAPLE.

We proceed as in Sections 4,6 of \cite{tajima-yukie-GIT2}. 
Since stability does not change by replacing $k$ 
by $\overline k$, we assume that $k=\overline k$ 
throughout this section.

We prove three easy lemmas 
before considering individual cases.
We shall also use Lemmas II--4.1\ccd II--4,6 very often. 
\begin{lem}
\label{lem:empty-criterion}
Let $\be\in\gB$. Suppose that $x\in Z_{\be}$, 
$M_{\be}x\sub Z_{\be}$ is Zariski open and that $x$ is 
unstable with respect to the action of {\upshape $G_{\text{st},\be}$}. 
Then $Z^{\sst}_{\be}=\emptyset$. 
\end{lem}
\begin{proof}
Let $T_0$ be as in (\ref{eq:T0-defn}).
Then $M_{\be}=M^{\text{st}}_{\be}\cdot T_0$. 
Let $T_{\be,0}\sub G$ be the subgroup  generated by 
$T_0$ and $\{\lam_{\be}(t)\mid t\in \gl_1\}$ 
(see Section 2 \cite{tajima-yukie-GIT1}). 
Then $M_{\be} = G_{\text{st},\be}\cdot T_{\be,0}$ 
and $T_{\be,0}$ acts on $Z_{\be}$ by scalar multiplication.

Suppose that $Z^{\sst}_{\be}\not=\emptyset$. 
Since $Z^{\sst}_{\be},M_{\be}x\sub Z_{\be}$ are Zariski open, 
there exists $g\in G_{\text{st},\be}$, $t\in T_{\be,0}$ 
such that $gtx \in Z^{\sst}_{\be}$. Since 
$Z^{\sst}_{\be}$ is $G_{\text{st},\be}$-invariant, 
$tx\in Z^{\sst}_{\be}$. Since $tx$ is a scalar 
multiple of $x$, $x\in Z^{\sst}_{\be}$, 
which is a contradiction. 
\end{proof}

Let $G=\gl_2^3$, 
$V_1=\aff^2\otimes \aff^2 \oplus \aff^2\otimes \aff^2\oplus \aff^2$, 
$V_2=\aff^2\otimes \aff^2 \oplus \aff^2\oplus \aff^2$.
We define an action of $(g_1,g_2,g_3)\in G$ on $V_1,V_2$ so that 
\begin{align*}
& V_1\ni (v_1\otimes v_2,v_3\otimes v_4,v_5) \mapsto 
(g_1v_1\otimes g_2v_2,g_2v_3\otimes g_3v_4,g_3v_5)\in V_1, \\
& V_2\ni (v_1\otimes v_2,v_3,v_4) \mapsto 
(g_1v_1\otimes g_2v_2,g_2v_3,g_3v_4)\in V_2.
\end{align*}

Let $\{\bbmp_{2,1},\bbmp_{2,2}\}$, 
be the standard basis of $\aff^2$ as before.  
Let $q_{12}=\bbmp_{2,1}\otimes \bbmp_{2,2}$, etc.  

\begin{lem}
\label{lem:222-eliminate}
In the above situation, let $x\in V_1$ or $x\in V_2$. 
Then there exists $g\in G$ 
such that if $y=(y_1,y_2,y_3)=gx$ then the coefficients of 
$q_{11}$ in $y_1$, 
$q_{11}$ or $\bbmp_{2,1}$ in $y_2$ 
and $\bbmp_{2,1}$ in $y_3$ are $0$. 
\end{lem}
\begin{proof}
Since the consideration is similar, 
we only consider $V_1$. 
By Lemma II--4.1, we may assume that 
the coefficient of $\bbmp_{2,1}$ in $y_3$ is $0$. 
Then elements of the form $(g_1,g_2,I_2)\in G$ do not change 
this condition. By the action of $\gl_2$ on $\m_2$, 
we can always make the $(1,1)$-entry $0$. 
So by replacing $y$ by $(I_2,g_2,I_2)y$ ($g_2\in\gl_2$) 
if necessary, we can make the coefficient of $q_{11}$ in $y_2$ $0$. 
Then  elements of the form $(g_1,I_2,I_2)\in G$ do not 
change these conditions. By the same argument, 
we can make the coefficient of $q_{11}$ in $y_1$ $0$. 
\end{proof}

Let $G=\spl_3\times \spl_2$ and $V=\m_{3,2} \oplus \aff^2$.
We express elements of $V$ as $x=(A(x),v(x))$ where 
$A(x)\in \m_{3,2},v(x)\in\aff^2$. If $g=(g_1,g_2)\in G$ 
and $x=(A(x),v(x))\in V$ then we define 
$gx=(g_1A(x){}^tg_2,g_2v(x))$.

\begin{lem}
\label{lem:case85}
We fix $i=2$ or $3$, $j=1$ or $2$ and $l=1$ or $2$. 
In the above situation, if $x=(A(x),v(x))\in V$, 
then there exists $g\in G$ such that $gx=(B,w)$ where 
the first row of $B$ is $0$ and the $(i,j)$-entry of 
$B$ and the $l$-th entry of $w$ are $0$. 
\end{lem}
\begin{proof}
By applying Lemma II--4.2 to the action of $\spl_3$ on 
$\m_{3,2}$, we may assume that the first row of $A(x)$ is $0$. 
Then $\spl_2$ in the $2$--$3$ block of $\spl_3$ and the second component 
of $G$ do not change this condition. By applying Lemma II--4.3 
to the action of $\spl_2\times \spl_2$ on the subspace consisting
of elements $(A(x),v(x))$ such that the first row of $A(x)$ is $0$, 
we obtain the statement of the lemma. Note that because of the 
symmetry, $i,j,l$ do not have to be $1$. 
\end{proof}

In the following, we shall show either 
(i) some coordinates of $x\in Z_{\be}$ 
can be eliminated by the action of the semi-simple part 
$M^s_{\be}$ or (ii) there exists $x\in Z_{\be}$ such that 
$M_{\be}x\sub Z_{\be}$ is Zariski open. 
We then list a 1PS for each case 
with the property that weights of non-zero
coordinates of $x$ are all positive. 
Note that this proves that $Z^{\sst}_{\be}=\emptyset$
(using Lemma \ref{lem:empty-criterion} for the case (ii)). 

For example, for the case $i=2$, the top row in the following table 
shows coordinates in $Z_2$. The first small table starts with  
$4,\ldots, x_{151},\ldots$. This means that 
$\bbma_4=e_{151}\in Z_2$.  
These numbers are grouped according as \rep s of $M^s_{\be_i}$. 
The second row shows that $M_{\be_2}=M_{[4],[3]}$, 
we can make the coefficients of $\bbma_4,\bbma_9\ccd \bbma_{36}$ 
zero, and that the 1PS we have to consider is 
$\gl_1\ni t\mapsto (\diag(t^{-39}\ccd t^{-12}),\diag(t^{-32}\ccd t^{12}))$. 
We found these 1PS's in the same way as in Section 4 
\cite[pp.21,22]{tajima-yukie-GIT2}.  
The last row shows that 
\begin{math}
Z_2\cong \Lam^{4,1}_{1,[1,4]}\otimes \Lam^{3,1}_{2,[1,3]}
\oplus \Lam^{4,2}_{1,[1,4]}
\end{math}
and weights with respect to the above 1PS of $\bbma_7\ccd \bbma_{38}$ 
are $1\ccd 18$ respectively. 
Note that the image of the above 1PS is in $G_{\mathrm{st},\be_7}$ since 
$-39+45-22+28-12=-32+35-15+12=0$ and the inner product of 
$(-39,45,-22,28,-12,-32,35,-15,12)$ and 
$(-4,-4,-4,-4,16,-5,-5,-5,15)=(620/7)\be_2$ is $0$. 
As was the case in Part II, this can be verified by MAPLE commands
similarly as in the comments before the table in Section 4 
\cite[p.22]{tajima-yukie-GIT2}.  

For later $S_i$'s, less information is required
and we may change the format of the table to save space. 

\vskip 10pt

\begin{center}


\end{center}

We change the format of the table to save space. 

\begin{center}
%

\end{center}

We change the format of the table again. 

(1) coordinates of $Z_{\be_i}$, \;
(2) $Z_{\be_i}$ as a representation of $M_{\be_i}$, \;
(3) $M_{\be_i}$, \;  
(4) zero coordinates of $Z_{\be_i}$, \; 
(5) 1PS, non-zero coordinates of $Z_{\be_i}$ 
and their weights 

\begin{center}
%
\\
\hline 
\end{tabular}

\end{center}

\normalsize
\vskip 10pt

We verify the information (which coordinates we can eliminate) 
in the above table in the following, which proves the following
theorem. 

\begin{thm}
\label{thm:empty-verify}
For $i$ not in the list (\ref{eq:list-non-empty-54}), 
$S_{\be_i}=\emptyset$.
\end{thm}

We either use Lemmas II--4.1--II--4.6, \ref{lem:empty-criterion}, 
\ref{lem:222-eliminate} with possible extra arguments to 
show that we can assume some coordinates are $0$. 
The easiest cases are the following. Each factor of $M_{\be_i}$ 
except for $\gl_1$ have at most one standard \rep{} in 
$Z_i$ and it is enough to apply Lemma II--4.1 to all such 
standard \rep s to show that $S_{\be_i}=\emptyset$.  
The applicable cases are the following. 
The numbering is according to the order in the table. 

\vskip 10pt

\small

(4) $S_7$, $\be_7 =\tfrac {1} {44} (-4,0,0,0,4,-3,1,1,1)$  

(10) $S_{17}$, $\be_{17}=\tfrac {7} {220}(-4,-4,-4,6,6,-5,-5,5,5)$

(12) $S_{19}$, $\be_{19}=\tfrac {3} {1420} (-16,-16,4,4,24,-5,-5,-5,15)$ 

(15) $S_{24}$, $\be_{24}=\tfrac {23} {1420}(-16,-16,4,4,24,-5,-5,-5,15)$ 

(16) $S_{25}$, $\be_{25}=\tfrac {7} {3820}(-24,-4,-4,16,16,-25,-25,15,35)$
 
(19) $S_{28}$, $\be_{28}=\tfrac {19} {4780}(-28,-8,-8,12,32,-25,-25,15,35)$ 

(39) $S_{56}$, $\be_{56}=\tfrac {1} {60} (-14,-14,-14,6,36,-5,-5,-5,15)$ 

(62) $S_{83}$, $\be_{83}=\tfrac {1} {460} (-44,-24,-24,36,56,-15,5,5,5)$ 

(63) $S_{84}$, $\be_{84}=\tfrac {1} {560} (-4,-4,1,1,6,-140,40,45,55)$  

(64) $S_{85}$, $\be_{85}=\tfrac {1} {740} (-56,-16,-16,24,64,-185,-5,75,115)$ 

(67) $S_{88}$, $\be_{88}=\tfrac {1} {460} (-104,-104,-64,116,156,-55,-55,-55,165)$  

(68) $S_{89}$, $\be_{89}=\tfrac {1} {140} (-36,-36,4,14,54,-15,-15,-5,35)$ 

(69) $S_{90}$, $\be_{90}=\tfrac {1} {560} (-89,-4,-4,6,91,-50,-50,45,55)$  

(71) $S_{92}$, $\be_{92}=\tfrac {1} {580} (-82,-42,28,28,68,-25,-25,5,45)$ 

(75) $S_{97}$, $\be_{97}=\tfrac {11} {580} (-12,-2,-2,8,8,-5,-5,-5,15)$  

(92) $S_{122}$, $\be_{122}=\tfrac {1} {140} (-26,-26,14,14,24,-15,-5,-5,25)$  

(94) $S_{124}$, $\be_{124}=\tfrac {1} {940} (-256,-16,4,4,264,-75,-55,-55,185)$  

(96) $S_{126}$, $\be_{126}=\tfrac {1} {1660} (-124,-24,-24,76,96,-135,-15,65,85)$  

(100) $S_{130}$, $\be_{130}=\tfrac {1} {1260} (-64,-4,-4,16,56,-55,-55,25,85)$  

(107) $S_{138}$, $\be_{138}=\tfrac {1} {220} (-38,-8,-8,22,32,-25,-25,5,45)$  

(116) $S_{147}$, $\be_{147}=\tfrac {1} {40} (-16,-1,-1,9,9,-5,-5,5,5)$, 

(117) $S_{148}$, $\be_{148}=\tfrac {1} {780} (-82,-52,-2,28,108,-75,-75,35,115)$ 

(118) $S_{153}$, $\be_{153}=\tfrac {1} {420} (-88,-28,12,32,72,-45,-45,15,75)$  

(121) $S_{156}$, $\be_{156}=\tfrac {1} {340} (-76,-76,-76,24,204,-85,-5,-5,95)$ 

(124) $S_{159}$, $\be_{159}=\tfrac {1} {260} (-4,-4,-4,-4,16,-65,-65,55,75)$ 

(131) $S_{167}$, $\be_{167}=\tfrac {1} {220} (-48,-48,12,12,72,-55,-55,25,85)$

(132) $S_{168}$, $\be_{168}=\tfrac {1} {380} (-32,-12,8,8,28,-95,-95,85,105)$ 

(134) $S_{170}$, $\be_{170}=\tfrac {1} {260} (-64,-64,16,16,96,-65,-5,-5,75)$

(135) $S_{171}$, $\be_{171}=\tfrac {1} {340} (-46,-6,4,4,44,-85,5,35,45)$ 

(136) $S_{172}$, $\be_{172}=\tfrac {1} {820} (-168,-28,-28,112,112,-205,-25,-25,255)$ 

(150) $S_{187}$, $\be_{187}=\tfrac {1} {820} (-148,-148,52,52,192,-205,-45,95,155)$  

(151) $S_{188}$, $\be_{188}=\tfrac {1} {380} (-92,-32,8,8,108,-95,-15,25,85)$ 

(152) $S_{189}$, $\be_{189}=\tfrac {1} {1460} (-64,-4,16,16,36,-365,95,115,155)$ 

(155) $S_{192}$, $\be_{192}=\tfrac {1} {340} (-136,-56,-56,44,204,-25,-25,-25,75)$ 

(156) $S_{193}$, $\be_{193}=\tfrac {1} {220} (-48,-48,-28,-8,132,-15,-15,5,25)$  

(157) $S_{194}$, $\be_{194} = \tfrac {1} {820} (-328,32,52,52,192,-45,-25,-25,95)$ 

(158) $S_{195}$, $\be_{195} = \tfrac {1} {820} (-148,-148,-28,132,192,-125,-125,95,155)$ 

(161) $S_{198}$, $\be_{198}=\tfrac {1} {380} (-52,-12,-12,28,48,-95,-15,25,85)$ 

(162) $S_{199}$, $\be_{199}=\tfrac {1} {820} (-328,-28,112,112,132,-45,-25,-25,95)$ 

(163) $S_{200}$, $\be_{200}=\tfrac {1} {40} (-4,0,0,1,3,-3,0,1,2)$ 

(166) $S_{204}$, $\be_{204}=\tfrac {1} {340} (-136,-136,24,124,124,-25,-25,-25,75)$ 

(167) $S_{205}$, $\be_{205}=\tfrac {1} {380} (-32,-12,-12,28,28,-35,-35,25,45)$ 

(169) $S_{207}$, $\be_{207}=\tfrac {1} {1220} (-88,12,12,32,32,-45,-45,-5,95)$ 

(172) $S_{210}$, $\be_{210}=\tfrac {1} {580} (-92,-92,48,68,68,-45,-45,-45,135)$ 

(178) $S_{218}$, $\be_{218}=\tfrac {1} {140} (-56,-16,-16,44,44,-15,-15,-15,45)$  

(179) $S_{219}$, $\be_{219}=\tfrac {1} {260} (-104,-44,16,36,96,-25,-25,-5,55)$ 

(181) $S_{221}$, $\be_{221}=\tfrac {1} {60} (-24,4,4,8,8,-3,-3,1,5)$  

(182) $S_{222}$, $\be_{222}=\tfrac {1} {1460} (-224,-64,-4,36,256,-125,-125,95,155)$ 

(183) $S_{225}$, $\be_{225}=\tfrac {1} {60} (-16,-4,0,4,16,-7,-3,-3,13)$ 

(184) $S_{228}$, $\be_{228}=\tfrac {1} {380} (-52,-32,8,28,48,-15,-15,5,25)$ 

(185) $S_{229}$, $\be_{229}=\tfrac {1} {740} (-196,-16,-16,64,164,-65,-65,-65,195)$ 

(186) $S_{230}$, $\be_{230}=\tfrac {1} {1460} (-124,-64,-64,-4,256,-125,-125,-5,255)$

(188) $S_{233}$, $\be_{233}=\tfrac {1} {20} (-3,-3,-3,-3,12,-5,-5,-5,15)$ 

(189) $S_{234}$, $\be_{234}=\tfrac {1} {20} (-4,-4,-4,0,12,-5,-5,3,7)$ 

(193) $S_{238}$, $\be_{238}=\tfrac {1} {20} (-4,-4,-2,4,6,-5,-5,1,9)$ 

(194) $S_{239}$, $\be_{239}=\tfrac {1} {40} (-6,-1,-1,4,4,-10,-10,5,15)$  

(201) $S_{246}$, $\be_{246}=\tfrac {1} {20} (-8,-2,2,2,6,-5,-5,3,7)$ 

(202) $S_{247}$, $\be_{247}=\tfrac {1} {20} (-8,-3,-3,2,12,-5,0,0,5)$ 

(203) $S_{248}$, $\be_{248}=\tfrac {1} {20} (-4,-4,-2,-2,12,-5,1,1,3)$ 

(205) $S_{250}$, $\be_{250}=\tfrac {1} {260} (-44,-44,-4,36,56,-65,-25,35,55)$ 

(206) $S_{251}$, $\be_{251}=\tfrac {1} {20} (-8,-8,2,7,7,-5,0,0,5)$ 

(207) $S_{252}$, $\be_{252}=\tfrac {1} {340} (-16,-6,-6,14,14,-85,5,35,45)$ 

(211) $S_{260}$, $\be_{260}=\tfrac {1} {260} (-64,-4,-4,16,56,-65,-5,-5,75)$ 

(212) $S_{261}$, $\be_{261}=\tfrac {1} {5} (-2,-2,-2,3,3,0,0,0,0)$

(213) $S_{262}$, $\be_{262}=\tfrac {1} {20} (-8,-8,0,4,12,-1,-1,-1,3)$ 

(214) $S_{263}$, $\be_{263}=\tfrac {1} {80} (-32,3,3,13,13,-5,-5,-5,15)$ 

(216) $S_{265}$, $\be_{265}=\tfrac {1} {20} (-8,-2,-2,0,12,-1,-1,1,1)$ 

(217) $S_{266}$, $\be_{266}=\tfrac {1} {20} (-8,-8,4,6,6,-1,-1,1,1)$ 

(218) $S_{267}$, $\be_{267}=\tfrac {1} {260} (-104,-4,16,36,56,-25,-5,-5,35)$ 

(219) $S_{268}$, $\be_{268}=\tfrac {1} {340} (-46,-16,-6,-6,74,-25,-25,5,45)$ 

(221) $S_{274}$, $\be_{274}=\tfrac {1} {60} (-24,-4,-4,-4,36,-15,-15,-15,45)$ 

(223) $S_{276}$, $\be_{276}=\tfrac {1} {140} (-16,-16,4,4,24,-35,-35,-35,105)$ 

(224) $S_{277}$, $\be_{277}=\tfrac {1} {140} (-56,-16,-16,4,84,-35,-35,25,45)$ 

(225) $S_{278}$, $\be_{278}=\tfrac {1} {140} (-56,-56,24,44,44,-35,-35,25,45)$ 

(226) $S_{279}$, $\be_{279}=\tfrac {1} {220} (-28,-28,12,12,32,-55,-55,45,65)$  

(227) $S_{282}$, $\be_{282}=\tfrac {1} {60} (-24,-24,-24,36,36,-15,5,5,5)$ 

(228) $S_{283}$, $\be_{283}=\tfrac {1} {140} (-56,-56,4,24,84,-35,5,5,25)$ 

(229) $S_{284}$, $\be_{284}=\tfrac {1} {220} (-88,12,12,32,32,-55,5,5,45)$ 

(230) $S_{288}$, $\be_{288}=\tfrac {1} {20} (-8,-8,2,2,12,-5,-5,-5,15)$ 

(231) $S_{290}$, $\be_{290}=\tfrac {1} {20} (-8,-8,-8,12,12,-5,-5,5,5)$ 

\normalsize

\vskip 10pt

The next easy cases are those where we can apply 
Lemma II--4.1, but either it is enough to consider 
a part of standard \rep s or we have to choose 
a standard \rep{} for a particular 
component of $M_{\be_i}$. Note that there are cases 
where a $\gl_2$-component of $M_{\be_i}$ 
has two standard \rep s in $Z_i$ and 
eliminating a coordinate for only one \rep{} 
works.  In the following, we list such cases 
and the subspaces to which we apply 
Lemma II--4.1 to show that $S_{\be_i}=\emptyset$. 

\vskip 10pt

\small
(18) $S_{27}$, $\be_{27}=\tfrac {11} {2380} (-32,-12,8,8,28,-5,-5,-5,15)$,  
$\lan \mathbbm a_7,\mathbbm a_{17},\mathbbm a_{27}\ran$, 
$\lan \mathbbm a_{35},\mathbbm a_{36}\ran$

(20) $S_{29}$, $\be_{29}=\tfrac {1} {280} (-7,-2,-2,3,8,-10,0,5,5)$, 
$\lan \mathbbm a_{16},\mathbbm a_{18}\ran$ 

(25) $S_{38}$, $\be_{38}=\tfrac {13} {1420} (-16,-16,4,4,24,-15,5,5,5)$ , 
$\lan \mathbbm a_9,\mathbbm a_{10}\ran$ 

(29) $S_{44}$, $\be_{44}=\tfrac {1} {5580} (-32,-12,8,8,28,-35,-15,5,45)$,   
$\lan \mathbbm a_9,\mathbbm a_{10}\ran$

(31) $S_{46}$, $\be_{46}=\tfrac {7} {1020} (-14,-4,-4,6,16,-15,-5,5,15)$,  
$\lan \mathbbm a_{17},\mathbbm a_{19}\ran$ 

(32) $S_{47}$, $\be_{47}=\tfrac {1} {3180} (-32,-12,8,8,28,-15,-15,5,25)$, 
$\lan \mathbbm a_7,\mathbbm a_{17}\ran$

(33) $S_{48}$, $\be_{48}=\tfrac {1} {204} (-8,-4,0,4,8,-3,-3,1,5)$, 
$\lan \mathbbm a_8,\mathbbm a_{18}\ran$ 

(49) $S_{67}$, $\be_{67}=\tfrac {1} {540} (-26,-26,-26,-16,94,-35,-35,-25,95)$, 
$\lan \mathbbm a_{10},\mathbbm a_{20}\ran$ 

(55) $S_{76}$, $\be_{76}=\tfrac {1} {260} (-44,-44,-14,51,51,-40,-40,25,55)$, 
$\lan \mathbbm a_{10},\mathbbm a_{20}\ran$ 

(65) $S_{86}$, $\be_{86}=\tfrac {1} {440} (-56,-21,14,14,49,-110,25,25,60)$, 
$\lan \mathbbm a_{35},\mathbbm a_{36}\ran$

(82) $S_{109}$, $\be_{109}=\tfrac {1} {260} (-44,-44,11,11,66,-65,-15,40,40)$, 
$\lan \mathbbm a_{19},\mathbbm a_{20}\ran$ 

(98) $S_{128}$, $\be_{128}=\tfrac {1} {340} (-36,-16,9,9,34,-30,-5,15,20)$, 
$\lan \mathbbm a_{35},\mathbbm a_{36}\ran$

(99) $S_{129}$, $\be_{129}=\tfrac {1} {310} (-9,-4,-4,1,16,-10,-5,5,10)$, 
$\lan \mathbbm a_{26},\mathbbm a_{28}\ran$

(101) $S_{132}$, $\be_{132}=\tfrac {1} {620} (-68,-68,-38,72,102,-105,5,35,65)$, 
$\lan \mathbbm a_{33},\mathbbm a_{36}\ran$ 

(102) $S_{133}$, $\be_{133}=\tfrac {1} {180} (-7,-2,-2,3,8,-45,10,15,20)$, 
$\lan \mathbbm a_{17},\mathbbm a_{19}\ran$ 

(103) $S_{134}$, $\be_{134}=\tfrac {1} {680} (-12,-7,3,8,8,-15,0,5,10)$, 
$\lan \mathbbm a_{16},\mathbbm a_{17}\ran$

(104) $S_{135}$, $\be_{135}=\tfrac {1} {940} (-96,-96,4,44,144,-75,-75,25,125)$, 
$\lan \mathbbm a_9,\mathbbm a_{19}\ran$, 
$\lan \mathbbm a_{24},\mathbbm a_{27}\ran$

(106) $S_{137}$, $\be_{137}=\tfrac {3} {340} (-7,-7,-2,8,8,-10,-10,5,15)$, 
$\lan \mathbbm a_{10},\mathbbm a_{20}\ran$ 

(108) $S_{139}$, $\be_{139}=\tfrac {1} {90} (-16,-1,-1,4,14,-10,-5,5,10)$, 
$\lan \mathbbm a_{26},\mathbbm a_{28}\ran$

(109) $S_{140}$, $\be_{140}=\tfrac {1} {940} (-86,-6,-6,24,74,-75,5,35,35)$, 
$\lan \mathbbm a_{16},\mathbbm a_{18}\ran$ 

(112) $S_{143}$, $\be_{143}=\tfrac {1} {620} (-23,-18,2,7,32,-10,-5,-5,20)$, 
$\lan \mathbbm a_{18},\mathbbm a_{28}\ran$

(113) $S_{144}$, $\be_{144}=\tfrac {1} {620} (-88,-58,-28,72,102,-105,25,25,55)$, 
$\lan \mathbbm a_{18},\mathbbm a_{28}\ran$

(114) $S_{145}$, $\be_{145}=\tfrac {1} {340} (-36,-16,-1,19,34,-20,-5,-5,30)$, 
$\lan \mathbbm a_{17},\mathbbm a_{27}\ran$ 

(115) $S_{146}$, $\be_{146}=\tfrac {1} {45} (-8,-3,2,2,7,-5,0,0,5)$, 
$\lan \mathbbm a_{18},\mathbbm a_{28}\ran$

(133) $S_{169}$, $\be_{169}=\tfrac {1} {220} (-48,-48,-28,52,72,-55,-15,-15,85)$, 
$\lan \mathbbm a_{34},\mathbbm a_{37}\ran$ 

(153) $S_{190}$, $\be_{190}=\tfrac {1} {1140} (-96,-96,-36,84,144,-285,35,95,155)$, 
$\lan \mathbbm a_{33},\mathbbm a_{36}\ran$

(164) $S_{201}$, $\be_{201}=\tfrac {1} {1140} (-136,-96,-96,84,244,-185,-5,35,155)$, 
$\lan \mathbbm a_{36},\mathbbm a_{38}\ran$

(165) $S_{203}$, $\be_{203}=\tfrac {1} {1140} (-136,-16,4,4,144,-285,-5,135,155)$, 
$\lan \mathbbm a_{35},\mathbbm a_{36}\ran$ 

(168) $S_{206}$, $\be_{206}=\tfrac {1} {260} (-24,-20,12,12,20,-21,-13,15,19)$, 
$\lan \mathbbm a_{35},\mathbbm a_{36}\ran$

(171) $S_{209}$, $\be_{209}=\tfrac {1} {220} (-18,-18,-8,2,42,-15,-15,-5,35)$, 
$\lan \mathbbm a_9,\mathbbm a_{19}\ran$

(173) $S_{211}$, $\be_{211}=\tfrac {1} {1140} (-136,-76,-16,84,144,-285,75,75,135)$, 
$\lan \mathbbm a_{18},\mathbbm a_{28}\ran$ 

(174) $S_{212}$, $\be_{212}=\tfrac {1} {1140} (-176,-136,-16,84,244,-185,35,75,75)$, 
$\lan \mathbbm a_{28},\mathbbm a_{38}\ran$

(175) $S_{213}$, $\be_{213}=\tfrac {1} {260} (-24,-20,-8,20,32,-13,-1,-1,15)$, 
$\lan \mathbbm a_{17},\mathbbm a_{27}\ran$

(177) $S_{215}$, $\be_{215}=\tfrac {1} {220} (-38,-8,2,2,42,-15,-15,15,15)$, 
$\lan \mathbbm a_{28},\mathbbm a_{38}\ran$ 

(180) $S_{220}$, $\be_{220}=\tfrac {1} {380} (-92,-32,-32,48,108,-55,-55,25,85)$, 
$\lan \mathbbm a_{10},\mathbbm a_{20}\ran$

(187) $S_{231}$, $\be_{231}=\tfrac {1} {380} (-72,-52,28,48,48,-35,-35,-15,85)$, 
$\lan \mathbbm a_{10},\mathbbm a_{20}\ran$

(204) $S_{249}$, $\be_{249}=\tfrac {1} {40} (-16,-1,4,4,9,-10,0,5,5)$, 
$\lan \mathbbm a_{19},\mathbbm a_{20}\ran$ 

(208) $S_{253}$, $\be_{253}=\tfrac {1} {260} (-24,-14,-14,6,46,-65,5,15,45)$, 
$\lan \mathbbm a_{36},\mathbbm a_{38}\ran$ 

(209) $S_{255}$, $\be_{255}=\tfrac {1} {260} (-34,-24,6,6,46,-65,15,25,25)$, 
$\lan \mathbbm a_{28},\mathbbm a_{38}\ran$ 

\normalsize
\vskip10pt

There are cases where it is enough to apply 
Lemma II--4.3 once. 
For that purpose there has to be a factor 
$\gl_2\times \gl_2$ in $M_{\be_i}$ 
and a component $\m_2\oplus \aff^2$ in $Z_i$. 
If this pair is unique then we do not have to 
worry about to which component we apply Lemma II--4.3.  
We list such cases with the coordinates we can eliminate 
to show that $S_{\be_i}=\emptyset$.

\begin{rem}
\label{remark:II43-remark}
By symmetry, we can eliminate any 
one entry of $\m_2$ and either entry of $\aff^2$.  
\end{rem}

\small 

(21) $S_{30}$, $\be_{30}=\tfrac {9} {3980} (-28,-8,-8,12,32,-35,5,5,25)$, 
$y_{14},y_{16}=0$. 

(23) $S_{32}$, $\be_{32}=\tfrac {7} {1060} (-32,-12,8,8,28,-15,-15,5,25)$, 
$y_9,y_{35}=0$. 

(24) $S_{34}$, $\be_{34}=\tfrac {3} {1060} (-28,-8,-8,12,32,-15,-15,5,25)$, 
$y_7,y_{26}=0$. 

(54) $S_{74}$, $\be_{74}=\tfrac {1} {120} (-18,-3,-3,12,12,-30,-5,10,25)$,
$y_{26},y_{33}=0$. 

(66) $S_{87}$, $\be_{87}=\tfrac {1} {220} (-8,-3,2,2,7,-55,15,20,20)$, 
$y_{24},y_{25}=0$. 

(70) $S_{91}$, $\be_{91}=\tfrac {1} {110} (-4,-4,1,1,6,-5,-5,0,10)$, 
$y_9,y_{24}=0$.

(72) $S_{93}$, $\be_{93}=\tfrac {1} {480} (-42,-2,-2,3,43,-35,10,10,15)$, 
$y_{14},y_{16}=0$.

(73) $S_{94}$, $\be_{94}=\tfrac {3} {380} (-34,-4,-4,6,36,-15,-5,-5,25)$, 
$y_{17},y_{36}=0$.

(74) $S_{96}$, $\be_{96}=\tfrac {1} {940} (-76,-56,-56,84,104,-15,-15,5,25)$, 
$y_7,y_{26}=0$.

(77) $S_{99}$, $\be_{99}=\tfrac {1} {120} (-23,-18,12,12,17,-10,-10,-5,25)$, 
$y_9,y_{35}=0$.  

(90) $S_{119}$, $\be_{119}=\tfrac {1} {380} (-152,18,38,38,58,-15,-15,5,25)$, 
$y_9,y_{35}=0$. 

(91) $S_{120}$, $\be_{120}=\tfrac {1} {940} (-116,-16,-16,64,84,-35,-15,-15,65)$, 
$y_7,y_{16}=0$. 

(105) $S_{136}$, $\be_{136}=\tfrac {1} {480} (-42,-2,-2,23,23,-35,-10,15,30)$, 
$y_{16},y_{33}=0$.

(110) $S_{141}$, $\be_{141}=\tfrac {1} {220} (-33,-8,2,2,37,-20,-20,15,25)$, 
$y_9,y_{35}=0$.

(129) $S_{165}$, $\be_{165}=\tfrac {1} {660} (-24,-4,-4,16,16,-165,-165,155,175)$, 
$y_{26},y_{33}=0$. 

(130) $S_{166}$, $\be_{166}=\tfrac {1} {660} (-104,-104,-24,116,116,-165,-65,75,155)$, 
$y_{28},y_{33}=0$. 

(137) $S_{173}$, $\be_{173}=\tfrac {1} {820} (-48,-8,-8,32,32,-205,15,95,95)$, 
$y_{23},y_{25}=0$. 

(138) $S_{174}$, $\be_{174}=\tfrac {1} {660} (-164,-24,16,16,156,-165,-5,-5,175)$, 
$y_{19},y_{35}=0$. 

(141) $S_{177}$, $\be_{177}=\tfrac {1} {660} (-104,-104,-24,76,156,-165,-65,115,115)$, 
$y_{24},y_{28}=0$. 

(147) $S_{184}$, $\be_{184}=\tfrac {1} {660} (-264,16,16,76,156,-65,-5,-5,75)$, 
$y_{17},y_{36}=0$. 

(148) $S_{185}$, $\be_{185}=\tfrac {1} {660} (-264,-24,56,116,116,-65,-5,-5,75)$, 
$y_{18},y_{36}=0$. 

(149) $S_{186}$, $\be_{186}=\tfrac {1} {180} (-20,-8,-8,16,20,-5,-1,-1,7)$, 
$y_7,y_{16}=0$. 

(154) $S_{191}$, $\be_{191}=\tfrac {1} {820} (-28,-28,-8,32,32,-205,15,75,115)$, 
$y_{23},y_{32}=0$. 

(170) $S_{208}$, $\be_{208}=\tfrac {1} {820} (-108,-28,-28,-8,172,-65,-65,15,115)$, 
$y_7,y_{36}=0$. 

(192) $S_{237}$, $\be_{237}=\tfrac {1} {140} (-16,-16,4,14,14,-35,-35,25,45)$,  
$y_{28},y_{33}=0$

(197) $S_{242}$, 
$\be_{242} = \tfrac {1} {140} (-16,-16,4,4,24,-35,-35,35,35)$
$y_{24},y_{28}=0$. 

(199) $S_{244}$, 
$\be_{244} = \tfrac {1} {140} (-56,4,14,14,24,-35,5,5,25)$ 
$y_{19},y_{35}=0$. 

(210) $S_{257}$, $\be_{257}=\tfrac {1} {20} (-8,-2,-2,6,6,-5,-1,-1,7)$, 
$y_{20},y_{36}=0$.

\vskip 10pt

There are cases where it is enough to apply 
Lemma II--4.2 once. 
For that purpose there has to be a factor 
$\gl_n$ in $M_{\be_i}$
and a component $\m_{n,m}$ with $n>m$ in $Z_i$. 
In each of the following cases, there is a 
unique such component in $Z_i$ and applying 
Lemma II--4.2 once is enough 
to show that $S_{\be_i}=\emptyset$. 

\vskip 10pt

(37) $S_{54}$, 
$\be_{54} = \tfrac {1} {15} (-6,-1,-1,-1,9,0,0,0,0)$

(44) $S_{61}$, 
$\be_{61} = \tfrac {2} {15} (-3,-3,2,2,2,0,0,0,0)$ 

(58) $S_{79}$, 
$\be_{79} = \tfrac {1} {340} (-56,-56,34,34,44,-25,-25,-25,75)$

(79) $S_{102}$, 
$\be_{102} = \tfrac {1} {60} (-9,-9,-9,-9,36,-15,5,5,5)$

(89) $S_{118}$,
$\be_{118} = \tfrac {1} {220} (-88,-38,22,22,82,-15,-15,-15,45)$  

(111) $S_{142}$, 
$\be_{142} = \tfrac {1} {780} (-112,-52,8,48,108,-15,-15,-15,45)$ 

(123) $S_{158}$, 
$\be_{158} = \tfrac {1} {340} (-136,4,44,44,44,-85,15,15,55)$

(128) $S_{163}$, 
$\be_{163} = \tfrac {1} {340} (-56,-56,-56,84,84,-85,-85,15,155)$ 

(139) $S_{175}$, 
$\be_{175} = \tfrac {1} {340} (-56,-56,-56,-36,204,-85,15,35,35)$

(140) $S_{176}$, 
$\be_{176} = \tfrac {1} {580} (-12,-12,8,8,8,-145,35,35,75)$

(144) $S_{181}$, 
$\be_{181} = \tfrac {1} {340} (-136,-36,-16,-16,204,-5,-5,-5,15)$  

(145) $S_{182}$, 
$\be_{182} = \tfrac {1} {340} (-136,-136,84,84,104,-5,-5,-5,15)$ 

(146) $S_{183}$, 
$\be_{183} = \tfrac {1} {580} (-52,-52,-12,-12,128,-25,-25,-25,75)$ 

(191) $S_{236}$, 
$\be_{236} = \tfrac {1} {60} (-4,-4,-4,6,6,-15,-15,-15,45)$ 

(195) $S_{240}$, 
$S_{240} = \tfrac {1} {60} (-24,-4,-4,-4,36,-15,-15,15,15)$ 

(196) $S_{241}$, 
$S_{241} = \tfrac {1} {60} (-24,-24,16,16,16,-15,-15,15,15)$ 

(198) $S_{243}$, 
$\be_{243} = \tfrac {1} {60} (-24,-24,6,6,36,-15,5,5,5)$ 

(200) $S_{245}$, 
$\be_{245} = \tfrac {1} {20} (-3,-3,-3,-3,12,-5,-5,5,5)$ 

(215) $S_{264}$, 
$\be_{264} = \tfrac {1} {10}(-4,-4,1,1,6,0,0,0,0)$ 

\normalsize

\vskip 10pt

We consider the remaining cases one by one. 
We assume that $x=\sum_j y_j\bbma_j\in Z_i$. 

\vskip 10pt

(1) $S_2$,  
$\be_2= \tfrac {7} {620} (-4,-4,-4,-4,16,-5,-5,-5,15)$ 

We show that $S_2=\emptyset$ using 
Lemma \ref{lem:empty-criterion}. 

Let $G=\gl_4\times \gl_3$ and 
$V=\aff^4\otimes \aff^3\oplus \wedge^2 \aff^4$.  
we consider the natural action of $G$ on $V$. 
For $x\in Z_2$, let 
\begin{equation*}
M_1(x) = 
\begin{pmatrix}
x_{151} & x_{152} & x_{153} \\
x_{251} & x_{252} & x_{253} \\
x_{351} & x_{352} & x_{353} \\
x_{451} & x_{452} & x_{453} 
\end{pmatrix},\; 
M_2(x) = 
\begin{pmatrix}
0 & x_{124} & x_{134} & x_{144} \\
-x_{124} & 0 & x_{234} & x_{244} \\
-x_{134} & -x_{234} & 0 & x_{344} \\
-x_{144} & -x_{244} & -x_{344} & 0  
\end{pmatrix}.
\end{equation*}
We can identify $Z_2$ with $V$ by 
the map $Z_2\ni x \mapsto (M_1(x),M_2(x))\in V$.
It is easy to 

\noindent
see that $M_{\be_2}$ contains the above $G$ and 
the action on $Z_2$ corresponds to the natural action 
of $G$ on $V$. 

Let 
\begin{equation*}
w_1 = 
\begin{pmatrix}
0 & 0 & 0 \\
1 & 0 & 0 \\
0 & 1 & 0 \\
0 & 0 & 1 
\end{pmatrix},\;
w_2 = 
\begin{pmatrix}
0 & 1 & 0 & 1 \\
-1 & 0 & 0 & 0 \\
0 & 0 & 0 & 1 \\
-1 & 0 & -1 & 0 
\end{pmatrix}
\end{equation*}
and $w=(w_1,w_2)$. 

Let $k[\vep]/(\vep^2)$ be the ring of dual numbers, 
$e_M=(I_4,I_3)$ and $X=(x_{ij})\in \m_4,Y=(y_{ij})\in \m_3$.
Then elements of the form $e_M+\vep(X,Y)$   
belong to ${\mathrm T}_{e_G}(G_w)$ 
(the tangent space of the stabilizer $G_w$)
if and only if 
\begin{align*}
Xw_1+w_1{}^tY=0,\; 
Xw_2+w_2{}^tX=0.
\end{align*}
By long but straightforward computations, 
$X,Y$ are in the following form:
\begin{align*}
X & = 
\begin{pmatrix}
x_{11} & 0 & 0 & 0 \\
x_{21} & -x_{11}-x_{24} & x_{23} & x_{24} \\
-x_{24} & x_{32} & x_{33} & -x_{32} \\
x_{21}+x_{23} & -x_{11}-x_{24}+x_{33} & x_{43} & x_{24}-x_{33} 
\end{pmatrix}, \\
Y & = 
\begin{pmatrix}
x_{11}+x_{24} & -x_{32} &  x_{11}+x_{24}-x_{33} \\
-x_{23} & -x_{33} & -x_{43} \\
-x_{24} & x_{32} & -x_{24}+x_{33}
\end{pmatrix}.
\end{align*}
Therefore, $\dim {\mathrm T}_{e_G}(G_w)=7=25-18=\dim G-\dim V$. 
By Proposition \ref{prop:open-orbit}, 
$Gw\sub V$ is Zariski open and $G_w$ is smooth over $k$ 
(we do not need this latter part). 
This implies that $M_{\be_2}w\sub Z_2$ is Zariski open. 
For the point in $Z_2$ corresponding to $w$, 
the $y_i$ coordinate (the coefficient of 
$\bbma_i$) is $0$ unless  $i=7,19,30,31,33,38$. 
The 1PS with the required property is listed in the above table. 
Therefore, $Z^{\sst}_2=\emptyset$. 
We shall not repeat this comment in the following. 


\vskip 5pt

(2) $S_4$,  
$\be_4=\tfrac {3} {44} (-4,0,0,0,4,-1,-1,-1,3)$ 

Note that $\gl_3$ is contained in the $2$--$4$ block 
(resp. $1$--$3$ block) of the first (resp. second) 
component of $M_{\be_4}$. 
Let 
\begin{equation*}
A(x) = 
\begin{pmatrix}
x_{251} & x_{252} & x_{253} \\
x_{351} & x_{352} & x_{353} \\
x_{451} & x_{452} & x_{453} 
\end{pmatrix},\; 
B(x) = 
\begin{pmatrix}
0 & x_{234} & x_{244} \\
-x_{234} & 0 & x_{344} \\
-x_{244} & -x_{344} & 0 
\end{pmatrix}. 
\end{equation*}
Then $Z_4$ can be identified with $\m_3\oplus \wedge^2 \aff^3\oplus 1$
by the map $Z_4\ni x\mapsto (A(x),B(x),x_{154})$. 
The element $(g_1,g_2)\in \gl_3\times \gl_3\sub M_{\be_4}$ acts on 
$Z_4$ by 
\begin{equation*}
(A,B,c)\mapsto (g_1 A {}^t g_2,g_1 B {}^t g_1,c). 
\end{equation*}
By Lemma II--4.6, we can eliminate the $(1,2),(1,3)$-entries of $B$. 
Then we can eliminate the $(1,1),(1,2),(2,1)$-entries of $A$
by using the following lemma for the action of elements
of the form $(I_3,g_2)$. 

\begin{lem}
\label{lem:33-eliminate}
If $A\in\m_3(k)$ then there exists 
$g\in\gl_3(k)$ such that the 
$(1,1),(1,2)$, $(2,1)$-coordinates of $gA$ are $0$. 
We obtain the same statement replacing the action 
to $A\mapsto A{}^tg$. 
\end{lem}
\begin{proof} 
Since the argument is similar, we only consider the 
first statement. 
It is a standard fact in linear algebra that 
there exists $g\in\gl_3$ such that 
the $(2,1),(3,1),(3,2)$-coordinates of $gA$ are $0$. 
Then it is enough to apply the permutation matrix
corresponding to the transposition $(1\, 3)$. 
\end{proof}

By the above consideration, 
we can make the $y_i$-coordinate $0$ 
for $i=7,9,17,35,36$.

\vskip 5pt

(3) $S_6$,  
$\be_6= \tfrac {1} {12} (-2,0,0,0,2,-1,-1,1,1)$ 

We apply the argument of Lemma II--4.4 
to the subspace 
$\lan \mathbbm a_{24},\mathbbm a_{34}\ran 
\oplus \lan \mathbbm a_{25}\ccd \mathbbm a_{38}\ran$. 
The latter subspace is $(\wedge^2 \aff^3)\otimes \aff^2$ 
instead of $\aff^3\otimes \aff^2$ as a \rep{} of 
$\gl_3\times \gl_2\sub M_{\be_6}$, but the 
consideration is similar and we  
may assume that $y_{24},y_{25},y_{35}=0$. 

%

(5) $S_8$, 
$\be_8 = \tfrac {1} {76} (-4,0,0,0,4,-3,-3,1,5)$ 

We first apply Lemma II--4.2 to the subspace 
$\lan \mathbbm a_7,\mathbbm a_9,\mathbbm a_{10}, 
\mathbbm a_{17},\mathbbm a_{19},\mathbbm a_{20}\ran$  
and may assume that $y_7,y_{17}=0$.  Then 
$\gl_2^2\sub  M_{[1,2,4],[2,3]}\sub M_{\be_8}$
does not change this condition. 
We apply Lemma II--4.3 to the action of this $\gl_2^2$ on the subspace 
$\lan \bbma_9,\bbma_{10},\bbma_{19},\bbma_{20}\ran 
\oplus \lan \bbma_{25},\bbma_{26}\ran$ 
and eliminate $y_{20},y_{26}$.

(6) $S_{10}$, 
$\be_{10} = \tfrac {11} {780} (-12,-12,8,8,8,-5,-5,-5,15)$

By applying Lemma II--4.2 to the subspace 
$\lan \mathbbm a_{32}\ccd \mathbbm a_{37}\ran$, 
we may assume that $y_{32},y_{35}=0$. 
Then $\gl_2\sub M_{1,[1,2,3]}$ does not 
change this condition. So by applying Lemma II--4.1 
to the subspace $\lan \mathbbm a_{36},\mathbbm a_{37}\ran$,  
we may assume that $y_{36}=0$. 
Note that $\gl_3\sub M_{2,[3,4]}$ does not change these
conditions. So by applying Lemma II--4.6 to the subspace 
$\lan \mathbbm a_8,\mathbbm a_9,\mathbbm a_{10}\ran$, 
we may assume that $y_8,y_9=0$.

(7) $S_{11}$, 
$\be_{11} = \tfrac {1} {1580} (-12,-12,8,8,8,-15,-15,5,25)$

We apply the consideration of Lemma II--4.2 to the subspace
$\lan \bbma_8,\bbma_9,\bbma_{10}\ccd \bbma_{20}\ran$. 
Even though this subspace is $(\wedge^2 \aff^3)\otimes \aff^2$ 
instead of $\aff^3\otimes \aff^2$ as a \rep{} of 
$\gl_3\sub M_{1,[1,2]}$, the consideration is similar 
and we may assume that $y_8,y_{18}=0$.

(8) $S_{12}$, 
$\be_{12} = \tfrac {9} {1580} (-8,-8,-8,12,12,-15,-15,5,25)$ 

By applying Lemma II--4.4 to the subspace 
$\lan \mathbbm a_{10},\mathbbm a_{20}\ran 
\oplus \lan \mathbbm a_{23}\ccd \mathbbm a_{29}\ran$, 
we may assume that $y_{10},y_{23},y_{24}=0$.

(9) $S_{13}$, 
$\be_{13} = \tfrac {19} {780} (-8,-8,-8,12,12,-5,-5,-5,15)$ 

By applying Lemmas II--4.1, II--4.2 to the subspaces 
$\lan \bbma_{10},\bbma_{20},\bbma_{30}\ran$, 
$\lan \bbma_{33}\ccd \bbma_{39}\ran$ respectively, 
we may assume that $y_{10},y_{20}=0$, $y_{33},y_{34}=0$.

%

(11) $S_{18}$, 
$\be_{18} = \tfrac {7} {1420} (-24,-4,-4,16,16,-5,-5,-5,15)$

Let $G=\gl_3\times \gl_2^2$, $W=\aff^3\otimes \m_2$, $V=W\oplus \aff^2$. 
We define an action of $(g_2,g_3)\in \gl_2^2$ on $\m_2$ by 
$\m_2\ni A\mapsto g_2A{}^tg_3\in \m_2$. With the standard \rep{} of
$\gl_3$, it induces a linear action of $G$ on $W$.  We define 
a linear action of $g=(g_1,g_2,g_3)\in M$ on $V$ by  
$V\ni (w,v)\mapsto (gw,g_3v)\in V$. 
Let $B=(B_1,B_2,B_3)\in W$ where $B_1,B_2,B_3$ are as in (\ref{eq:B123-defn}) 
and $v_0=[0,1]$. 

\begin{lem}
\label{lem:S18-Zopen}
In the above situation, $G(B,v_0)\sub V$ is Zariski open. 
\end{lem}
\begin{proof}
By Proposition \ref{prop:322-invariant}, $GB\sub W$ is Zariski dense. 
Let $S=\{(B,v)\mid v\in \aff^2\setminus \{[0,0]\}\}$. 
Then $GS\sub V$ is Zariski open. We show that $G(B,v_0)=GS$, which 
proves that $G(B,v_0)\sub V$ is Zariski open. 

For $h\in\gl_2,t\in\gl_1$, 
let $\rho(h,t),\rho_1(h,t)$ be as in Lemmas \ref{lem:S14-L-invariant}, 
\ref{lem:S14-w-fix}. 
By Lemma \ref{lem:S14-w-fix}, if $g\in\gl_2$ then 
$({}^t\rho_1(g,1)^{-1},g,{}^tg^{-1})$ fixes $B$. 
For any $v\in \aff^2\setminus \{[0,0]\}$, 
there exists $g\in \gl_2$ such that ${}^tg^{-1}v=v_0$. 
Then $({}^t\rho_1(g)^{-1},g,{}^tg^{-1})(B,v)=(B,v_0)$. 
This proves that $G(B,v_0)=GS$. 
\end{proof}

For $x\in Z_{18}$, let 
\begin{equation*}
A_i(x) = 
\begin{pmatrix}
x_{24i} & x_{25i} \\
x_{34i} & x_{35i} 
\end{pmatrix} \; (i=1,2,3),\; 
v(x) = 
\begin{pmatrix}
x_{144} \\ x_{154}
\end{pmatrix}.
\end{equation*}
Then by the map 
$x\mapsto ((A_1(x),A_2(x),A_3(x)),v(x),x_{234})\in V\oplus 1$,  
we can identify $Z_{18}$ with $V\oplus 1$.  
We use Lemma \ref{lem:empty-criterion} to show that 
$S_{18}=\emptyset$. By the above consideration, we only have 
to consider $x=\sum_i y_i\mathbbm a_i$ such that $A(x)=B,v(x)=v_0$.  
So we may assume that $y_i=0$ for $i=7,8,16,18,19,26,27,29,33$. 

%

(13) $S_{22}$, 
$\be_{22} = \tfrac {13} {2220} (-16,-16,4,4,24,-15,-15,5,25)$ 

By Lemma \ref{lem:222-eliminate}, 
we may assume that $y_9,y_{24},y_{32}=0$. 

(14) $S_{23}$, 5
$\be_{23} = \tfrac {17} {2220} (-24,-4,-4,16,16,-15,-15,5,25)$ 

By Lemma \ref{lem:222-eliminate}, 
we may assume that $y_{10},y_{26},y_{33}=0$. 

%

%

(17) $S_{26}$, 
$\be_{26}=\tfrac {3} {3020} (-16,-16,4,4,24,-25,-5,-5,35)$ 

By applying Lemmas II--4.1, II--4.3 to the subspaces 
$\lan \mathbbm a_9,\mathbbm a_{10}\ran$, 
$\lan \mathbbm a_{14}\ccd \mathbbm a_{27}\ran \oplus
\lan \mathbbm a_{18},\mathbbm a_{28}\ran$ respectively, 
we may assume that $y_9,y_{14},y_{18}=0$.  

%

%

%

%

(22) $S_{31}$, 
$\be_{31} = \tfrac {3} {620} (-16,-6,4,4,14,-5,-5,5,5)$ 

By applying Lemmas II--4.1, II--4.3 to the subspaces 
$\lan \mathbbm a_7,\mathbbm a_{17}\ran$, 
$\lan \mathbbm a_{24},\mathbbm a_{34}\ran \oplus 
\lan\mathbbm a_{25}\ccd \mathbbm a_{36}\ran$ 
respectively, we may assume that $y_7,y_{24},y_{25}=0$.

%

%

%

(26) $S_{39}$, 
$\be_{39} = \tfrac {11} {3180} (-32,-12,8,8,28,-25,-5,15,15)$ 

By applying Lemma II--4.3 to the subspace 
$\lan \mathbbm a_{24},\mathbbm a_{34}\ran \oplus 
\lan \mathbbm a_{25}\ccd \mathbbm a_{36}\ran$, 
we may assume that $y_{24},y_{25}=0$.

(27) $S_{41}$, 
$\be_{41} = \tfrac {1} {2380} (-32,-12,8,8,28,-15,5,5,5)$ 

By applying Lemma II--4.5 to the subspace  
$\lan \mathbbm e_{14},\mathbbm e_{24},\mathbbm e_{34}\ran \oplus 
\lan \mathbbm e_{15}\ccd \mathbbm e_{36}\ran$, 
we may assume that $y_{14}=y_{24}=y_{15}=0$.

(28) $S_{43}$, 
$\be_{43} = \tfrac {7} {2220} (-24,-4,-4,16,16,-25,-5,15,15)$ 

By applying II--4.3 to the subspace 
$\lan \mathbbm a_{23}\ccd \mathbbm a_{34}\ran \oplus 
\lan \mathbbm a_{25},\mathbbm a_{35}\ran$
and then Lemmas II--4.1 to the subspace 
$\lan \mathbbm a_{16},\mathbbm a_{18}\ran$,  
we may assume that $y_{16},y_{23},y_{25}=0$. 

%

(30) $S_{45}$, 
$\be_{45} = \tfrac {1} {780} (-12,-2,-2,8,8,-15,-5,5,15)$

By applying Lemma II--4.3 to the subspace 
$\lan \mathbbm a_{16}\ccd \mathbbm a_{19}\ran 
\oplus \lan \mathbbm a_{31},\mathbbm a_{32}\ran$, 
we may assume that $y_{16},y_{31}=0$. 

%

%
%
%
%

(34) $S_{51}$, 
$\be_{51} = \tfrac {1} {220} (-8,2,2,2,2,-55,15,15,25)$ 

We identify $\lan \mathbbm a_{15}\ccd \mathbbm a_{30}\ran$ with 
with the space of $4\times 4$ alternating matrices 
with coefficients in the space 
of linear forms in two variables $u=(u_1,u_2)$. 
Taking the Pfaffian, we obtain a quadratic form $Q_x$ 
of $u$ for each $x\in Z_{51}$. 
Since $k=\overline k$, by the action of 
$\gl(2)\sub M_{2,[1,3]}\sub M_{\be_{51}}$,  
we may assume that $Q_x(1,0)=0$, which implies that 
the first  $4\times 4$ alternating matrix is degenerate 
(rank $<4$).  By Lemma II--4.6, we may assume that 
$y_{15}\ccd y_{19}=0$.  
Then $\gl_2\sub M_{1,[1,3,4]}\sub M_{\be_{51}}$ 
does not change this condition. 
By applying Lemma II--4.1 to the action of this $\gl_2$ 
on the subspace $\lan \bbma_{31},\bbma_{32}\ran$,  
we may assume that $y_{31}=0$.

(35) $S_{52}$, 
$S_{52} = \tfrac {1} {4} (0,0,0,0,0,-1,-1,1,1)$ 
 
By applying Lemma II--4.6 to the subspace 
$\lan \bbma_{21}\ccd \bbma_{30}\ran$, 
we may assume that 
$y_{21}\ccd y_{24},y_{26}\ccd y_{29}=0$. 
Then $\gl_2^2\sub M_{1,[1,3]}\sub M_{\be_{52}}$  
does not change this condition. So by 
applying Lemma II--4.1 to the subspaces 
$\lan \bbma_{31},\bbma_{32}\ran$, $\lan \bbma_{33},\bbma_{34}\ran$, 
we may assume that $y_{31},y_{33}=0$.

(36) $S_{53}$, 
$\be_{53} = \tfrac {1} {120} (-48,7,7,7,27,-5,-5,-5,15)$ 

By applying Lemma II--4.6 to the subspace 
$\lan \mathbbm a_{35},\mathbbm a_{36},\mathbbm a_{38}\ran$, 
we may assume that $y_{35},y_{36}=0$. 
Then $\gl_3\sub M_{2,[3]}\sub M_{\be_{53}}$ 
does not change this condition. 
By applying Lemma \ref{lem:33-eliminate} to 
the action of this $\gl_3$ on 
the subspace 
$\lan \mathbbm a_7\ccd \mathbbm a_{30}\ran$, 
we may assume that $y_7,y_9,y_{17}=0$.

%

(38) $S_{55}$, 
$\be_{55} = \tfrac {1} {4} (-1,0,0,0,1,-1,0,0,1)$ 

By applying Lemma II--4.2 to the subspace 
$\lan \bbma_{17}\ccd \bbma_{30}\ran$, 
we may assume that $y_{17},y_{27}=0$. 
Then $\gl_2\sub M_{1,[1,2,4]}\sub M_{\be_{55}}$ 
does not change this condition. 
By applying Lemma II--4.1 to the subspace 
$\lan \mathbbm a_{35},\mathbbm a_{36}\ran$, 
we may assume that $y_{35}=0$. 

%

(40) $S_{57}$, 
$\be_{57} = \tfrac {1} {460} (-44,-44,-44,36,96,-75,5,5,65)$ 

By Lemma II--4.5, 
we may assume that $y_{14},y_{33},y_{36}=0$.

(41) $S_{58}$, 
$\be_{58} = \tfrac {1} {120} (-48,-3,17,17,17,-5,-5,-5,15)$ 

By applying Lemma II--4.1 to the subspace
$\lan \mathbbm a_{35},\mathbbm a_{36},\mathbbm a_{37}\ran$,
we may assume that $y_{35},y_{36}=0$. 
Then $\gl_3\sub M_{2,[3]}\sub M_{\be_{58}}$ 
does not change this condition. 
By applying Lemma \ref{lem:33-eliminate} 
to the action of this $\gl_3$ on the subspace 
$\lan \mathbbm a_8\ccd \mathbbm a_{30}\ran$, 
we may assume that $y_8,y_9,y_{18}=0$.

(42) $S_{59}$, 
$\be_{59} = \tfrac {1} {60} (-4,-4,-4,0,12,-3,-3,-3,9)$ 

Similarly as in the above case, 
we may assume that $y_4,y_7,y_{14},y_{33},y_{36}=0$. 

(43) $S_{60}$, 
$\be_{60} = \tfrac {1} {8} (-1,0,0,0,1,-2,0,1,1)$ 

By applying Lemma II--4.4 to the subspace 
$\lan \mathbbm a_{24},\mathbbm a_{34} \ran \oplus 
\lan \mathbbm a_{25}\ccd \mathbbm a_{38}\ran$, 
we may assume that  $y_{24},y_{25},y_{35}=0$. 
Note that this subspace is $\aff^2\oplus (\wedge^2 \aff^3)\otimes \aff^2$
instead of $\aff^2\oplus \aff^3\otimes \aff^2$ as a \rep{} of 
$\gl_3\times \gl_2\sub M_{\be_{60}}$, but the 
consideration is similar. 

%

(45) $S_{62}$, 
$\be_{62} = \tfrac {1} {540} (-56,4,4,4,44,-35,5,5,25)$

By the consideration of Lemma II--4.5, 
we may assume that $y_7,y_9,y_{15}=0$.

(46) $S_{63}$, 
$\be_{63} = \tfrac {1} {240} (-6,-6,-6,9,9,-60,5,20,35)$ 

By the consideration of Lemma II--4.5, 
we may assume that $y_{23},y_{31},y_{32}=0$.

(47) $S_{65}$, 
$\be_{65} = \tfrac {1} {240} (-31,-6,-6,-6,49,-20,-20,5,35)$

By Lemma II--4.6, we may assume that $y_{35},y_{36}=0$.

(48) $S_{66}$, 
$\be_{66} = \tfrac {1} {460} (-34,6,6,6,16,-15,-15,-5,35)$ 

By applying Lemma II--4.2 to the subspace 
$\lan \mathbbm a_7\ccd \mathbbm a_{20}\ran$, 
we may assume that $y_7,y_{17}=0$. Then 
$\gl_2$'s in $M_{1,[1,2,4]},M_{2,[2,3]}\sub M_{\be_{66}}$ 
do not change this condition. 
By applying Lemma II--4.3 to the subspace 
$\lan \bbma_9,\bbma_{10},\bbma_{19},\bbma_{20} \ran \oplus 
\lan \bbma_{25},\bbma_{26}\ran$, 
we may assume that $y_{20},y_{26}=0$ (see Remark \ref{remark:II43-remark}).

%

(50) $S_{68}$, 
$\be_{68} = \tfrac {1} {60} (-4,-4,-4,6,6,-5,-5,-5,15)$ 

By applying Lemmas II--4.1, II--4.6 to the subspaces 
$\lan \bbma_{10},\bbma_{20},\bbma_{30}\ran$, 
$\lan \bbma_{31},\bbma_{32},\bbma_{35}\ran$ respectively,  
we may assume that $y_{10},y_{20}=0,y_{31},y_{32}=0$.

(51) $S_{69}$, 
$\be_{69} = \tfrac {1} {140} (-26,-26,-26,39,39,-35,-10,-10,55)$ 

By applying Lemmas II--4.1, II--4.2 to the subspaces 
$\lan \bbma_{10},\bbma_{20}\ran$, 
$\lan \bbma_{33}\ccd \bbma_{39}\ran$ respectively,  
we may assume that $y_{20}=0,y_{33},y_{34}=0$. 

(52) $S_{72}$, 
$\be_{72} = \tfrac {1} {340} (-36,-16,-16,34,34,-5,-5,-5,15)$ 

Let $G=\gl_3\times \gl_2^2$, 
$W=\aff^3\otimes \m_2$ and 
$V=W\oplus \aff^2$. 

We define an action of $G$ on $V$ in the same way as the case 
(11) $S_{18}$. Let $B=(B_1,B_2,B_3)$ be as in (\ref{eq:B123-defn}) 
and $v_0=[0,1]$. By Lemma \ref{lem:S18-Zopen}, 
$G(B,v_0)\sub V$ is Zariski open. 
In the same way as in the case (11) $S_{18}$, 
we can identify $Z_{72}$ with $V$. 
The element $-\bbma_6+\bbma_9+\bbma_{17}+\bbma_{28}+\bbma_{34}$
corresponds to the element $(B,v_0)$. 
We use Lemma \ref{lem:empty-criterion} to show that 
$S_{72}=\emptyset$. By the above consideration, we may 
assume that $y_i=0$ for $i=7,8,16,18,19,26,27,29,33$. 

(53) $S_{73}$, 
$\be_{73} = \tfrac {1} {260} (-4,-4,1,1,6,-65,20,20,25)$ 

By Lemma \ref{lem:222-eliminate}, 
we may assume that $y_{14},y_{18},y_{32}=0$. 

%

%

(56) $S_{77}$, 
$\be_{77} = \tfrac {1} {190} (-16,-16,9,9,14,-15,-10,10,15)$ 

By applying Lemma II--4.3, to the subspace 
$\lan \mathbbm a_{24},\mathbbm a_{27}\ran \oplus 
\lan \mathbbm a_{32}\ccd \mathbbm a_{36}\ran$, 
we may assume that $y_{24},y_{32}=0$. 

(57) $S_{78}$, 
$\be_{78} = \tfrac {1} {260} (-19,-19,-4,-4,46,-20,-20,-5,45)$ 

By Lemma \ref{lem:222-eliminate}, 
we may assume that $y_9,y_{24},y_{32}=0$. 

%

(59) $S_{80}$, 
$\be_{80} = \tfrac {1} {35} (-4,-4,1,1,6,0,0,0,0)$ 

By applying Lemma II--4.2 to the subspace 
$\lan \mathbbm a_4\ccd \mathbbm a_{37}\ran$, 
we may assume that $y_4,y_7,y_{14},y_{17}=0$. 
Then $\gl_2\sub M_{2,[2,3]}\sub M_{\be_{80}}$ 
does not change this condition. 
By applying Lemma II--4.1 to the subspace 
$\lan \mathbbm a_8,\mathbbm a_{18}\ran$, 
we may assume that $y_8=0$. 

(60) $S_{81}$, 
$\be_{81} = \tfrac {1} {95} (-8,-8,-3,7,12,-5,0,0,5)$ 

By applying Lemma II--4.3 to the subspace 
$\lan \mathbbm a_{14}\ccd \mathbbm a_{27} \ran \oplus 
\lan \mathbbm a_{33},\mathbbm a_{36}\ran$, 
we may assume that $y_{14},y_{33}=0$. 

(61) $S_{82}$, 
$\be_{82} = \tfrac {1} {130} (-22,-2,-2,3,23,-10,-10,10,10)$ 

By Lemma \ref{lem:222-eliminate}
we may assume that $y_7,y_{24},y_{26}=0$.

(76) $S_{98}$, 
$\be_{98} = \tfrac {3} {170} (-6,-1,-1,4,4,-5,-5,5,5)$ 

We apply Lemmas II--4.1, II--4.3 to the subspaces 
$\lan \mathbbm a_{10},\mathbbm a_{20}\ran$, 
$\lan \mathbbm a_{23}\ccd \mathbbm a_{34} \ran 
\oplus \lan \mathbbm a_{25},\mathbbm a_{35}\ran$ respectively  
and may assume that $y_{10}=0,y_{23},y_{25}=0$.

%

(78) $S_{100}$, 
$\be_{100} = \tfrac {1} {460} (-124,-24,16,16,116,-35,-35,-35,105)$ 

By Lemma II--4.4, we may assume that $y_9,y_{10},y_{35}=0$.

%

(80) $S_{103}$, 
$\be_{103} = \tfrac {1} {70} (-13,-13,-13,-3,42,-5,-5,5,5)$

We apply Lemmas II--4.1, II--4.2 to the subspaces 
$\lan \mathbbm a_{10},\mathbbm a_{20}\ran$, 
$\lan \mathbbm a_{24},\mathbbm a_{27},
\mathbbm a_{29},\mathbbm a_{34},\mathbbm a_{37},\mathbbm a_{39}\ran$ respectively  
and may assume that $y_{10},y_{24},y_{34}=0$.

(81) $S_{104}$, 
$\be_{104} = \tfrac {1} {130} (-2,-2,-2,3,3,-5,-5,5,5)$ 

We apply Lemmas II--4.1, II--4.2 to the subspaces 
$\lan \mathbbm a_{10},\mathbbm a_{20}\ran$, 
$\lan \mathbbm a_{21},\mathbbm a_{22},
\mathbbm a_{25},\mathbbm a_{31},\mathbbm a_{32},\mathbbm a_{35}\ran$ respectively  
and may assume that $y_{10},y_{21},y_{31}=0$.
The second \rep{} is $\wedge^2 \aff^3 \otimes \aff^2$, 
but the consideration is similar to that of 
Lemma II--4.2.

%

(83) $S_{111}$, 
$\be_{111} = \tfrac {1} {130} (-22,-22,-7,18,33,-20,-20,20,20)$ 

By applying Lemma II--4.3 to the subspace
$\lan \mathbbm a_{24}\ccd \mathbbm a_{37} \ran 
\oplus \lan \mathbbm a_{28},\mathbbm a_{38}\ran$, 
we may assume that $y_{24},y_{28}=0$. 
 
(84) $S_{112}$, 
$\be_{112} = \tfrac {9} {380} (-8,-8,2,2,12,-5,-5,5,5)$ 

By applying Lemma II--4.3 to the subspace
$\lan \mathbbm a_{24}\ccd \mathbbm a_{37} \ran 
\oplus \lan \mathbbm a_{28},\mathbbm a_{38}\ran$, 
we may assume that $y_{24},y_{28}=0$.

(85) $S_{114}$, 
$\be_{114} = \tfrac {1} {460} (-64,-64,-4,36,96,-75,25,25,25)$ 

We first apply Lemma II--4.2 to the subspace 
$\lan \mathbbm a_{14},\mathbbm a_{17},
\mathbbm a_{24},\mathbbm a_{27},
\mathbbm a_{34},\mathbbm a_{37}\ran$ 
and may assume that $y_{14},y_{17} =0$. 
Then $\gl_2^2\sub \m_{[2,3,4],[1,2]}$ 
does not change this condition. 
By applying Lemma II--4.3 to the action of this 
$\gl_2^2$ on the subspace 
$\lan \mathbbm a_{24},\mathbbm a_{27},
\mathbbm a_{34},\mathbbm a_{37}\ran \oplus
\lan \mathbbm a_{28},\mathbbm a_{38}\ran$,  
we may assume that $y_{24},y_{28}=0$.

(86) $S_{115}$, 
$\be_{115} = \tfrac {1} {380} (-37,-27,18,18,28,-30,-20,25,25)$ 

By applying Lemma II--4.3 to the subspace  
$\lan \mathbbm a_{24},\mathbbm a_{34}\ran \oplus
\lan \mathbbm a_{25}\ccd \mathbbm a_{36}\ran$, 
we may assume that $y_{24}=y_{25}=0$. 

(87) $S_{116}$, 
$\be_{116} = \tfrac {1} {580} (-12,-2,-2,8,8,-15,5,5,5)$ 

We first apply Lemma II--4.2 to the subspace 
$\lan \mathbbm a_{13},\mathbbm a_{14},
\mathbbm a_{23},\mathbbm a_{24},
\mathbbm a_{33},\mathbbm a_{34}\ran$ 
and may assume that $y_{13},y_{14} =0$.   
Then $\gl_2^2\sub \m_{[1,2,3],[1,2]}$ 
does not change this condition. 
By applying Lemma II--4.3 to the action of this 
$\gl_2^2$ on the subspace 
$\lan \mathbbm a_{23},\mathbbm a_{24},
\mathbbm a_{33},\mathbbm a_{34}\ran \oplus
\lan \mathbbm a_{25},\mathbbm a_{35}\ran$,  
we may assume that $y_{23},y_{25}=0$.

(88) $S_{117}$, 
$\be_{117} = \tfrac {1} {140} (-31,-31,-11,-11,84,-5,-5,-5,15)$ 

We apply Lemmas II--4.2, II--4.1 to the subspaces 
$\lan \mathbbm a_{9},\mathbbm a_{10},
\mathbbm a_{19},\mathbbm a_{20},\mathbbm a_{29},\mathbbm a_{30}\ran$,
$\lan \mathbbm a_{34},\mathbbm a_{37}\ran$ respectively  
and may assume that $y_{9},y_{10},y_{34}=0$.

%

%

%

%

(93) $S_{123}$, 
$\be_{123} = \tfrac {1} {220} (-88,2,22,22,42,-15,5,5,5)$ 

We identify $\lan \mathbbm a_{17}\ccd \mathbbm a_{38}\ran$
with $\m_{3,2}$. By applying Lemmas II--4.1, II--4.2 
to the subspaces $\lan \bbma_9,\bbma_{10}\ran$, 
$\lan \bbma_{17}\ccd \bbma_{38}\ran$ respectively, we may assume that 
$y_9=0,y_{17},y_{18}=0$.

%

(95) $S_{125}$, 
$\be_{125} = \tfrac {1} {440} (-71,-56,14,14,99,-60,10,25,25)$ 

By applying Lemma II--4.1 to the subspaces
$\lan \mathbbm a_9,\mathbbm a_{10}\ran$
$\lan \mathbbm a_{24},\mathbbm a_{34}\ran$, 
we may assume that $y_9=0,y_{24}=0.$

%

(97) $S_{127}$, 
$\be_{127} = \tfrac {3} {520} (-16,-1,4,4,9,-5,-5,0,10)$ 

By applying Lemma II--4.1 to the subspaces 
$\lan \mathbbm a_8,\mathbbm a_{18}\ran$
$\lan \mathbbm a_{25},\mathbbm a_{26}\ran$, 
we may assume that $y_8=0,y_{25}=0$.

(119) $S_{154}$,  
$\be_{154} = \tfrac {1} {340} (-36,-36,24,24,24,-85,-85,55,115)$ 

By applying Lemma II--4.2 to the subspace 
$\lan \mathbbm a_{32}\ccd \mathbbm a_{37}\ran$, 
we may assume that $y_{32},y_{35}=0$. 
Then $\gl_2^2\sub M_{[2,3],[1,2,3]}$ does not change this condition. 
By applying Lemma II--4.3 to the action of this 
$\gl_2^2$ on the subspace 
$\lan \mathbbm a_{28},\mathbbm a_{29}\ran \oplus
\lan \mathbbm a_{33},\mathbbm a_{34},\mathbbm a_{36},\mathbbm a_{37}\ran$,  
we may assume that $y_{28},y_{33}=0$.

(120) $S_{155}$, 
$\be_{155} = \tfrac {1} {340} (-136,24,24,24,64,-85,15,15,55)$ 

Similarly as in the above case, we may assume that 
$y_{17},y_{27}=0,y_{19},y_{35}=0$.

%

(122) $S_{157}$, 
$\be_{157} = \tfrac {1} {420} (-28,-28,-28,12,72,-105,15,15,75)$ 

We first apply Lemma II--4.2 to the subspace 
$\lan \mathbbm a_{14},\mathbbm a_{17},
\mathbbm a_{19},\mathbbm a_{24},
\mathbbm a_{27},\mathbbm a_{29}\ran$ 
and may assume that $y_{14},y_{24} =0$.   
Then $\gl_2^2 \sub \m_{[1,3,4],[1,3]}$ 
does not change this condition. 
By applying Lemma II--4.3 to the action of this 
$\gl_2^2$ on the subspace 
$\lan \mathbbm a_{17},\mathbbm a_{19},
\mathbbm a_{27},\mathbbm a_{29}\ran \oplus
\lan \mathbbm a_{36},\mathbbm a_{38}\ran$,  
we may assume that $y_{17},y_{36}=0$.

%

%

(125) $S_{160}$, 
$\be_{160} = \tfrac {1} {80} (-7,-2,-2,-2,13,-20,0,5,15)$ 

By applying Lemma II--4.1 to the subspace 
$\lan \bbma_{17},\bbma_{19},\bbma_{20}\ran$, 
we may assume that $y_{17},y_{19}=0$. 
Then $\gl_2\sub \m_{1,[1,3,4]}\sub \m_{\be_{160}}$ 
does not change this condition. 
By applying Lemma II--4.1 to the subspace 
$\lan \bbma_{36},\bbma_{38}\ran$, 
we may assume that $y_{36}=0$. 

(126) $S_{161}$, 
$\be_{161} = \tfrac {1} {580} (-32,-32,-32,48,48,-145,-5,-5,155)$ 

By applying Lemmas II--4.1, II--4.6 to the subspaces 
$\lan \bbma_{20},\bbma_{30}\ran$, 
$\lan \bbma_{31},\bbma_{32},\bbma_{35}\ran$ 
respectively, we may assume that 
$y_{20}=0,y_{31},y_{32}=0$.

(127) $S_{162}$, 
$\be_{162} = \tfrac {1} {4} (0,0,0,0,0,-1,-1,-1,3)$

By Lemma II--4.6, we may assume that $y_{31},y_{32},y_{33},y_{34}=0$.

(142) $S_{179}$, 
$\be_{179} = \tfrac {1} {420} (-48,-48,12,12,72,-105,35,35,35)$ 

We first apply Lemma II--4.2 to the subspace 
$\lan \mathbbm a_{14},\mathbbm a_{17},
\mathbbm a_{24},\mathbbm a_{27},
\mathbbm a_{34},\mathbbm a_{37}\ran$ 
and may assume that $y_{14},y_{17} =0$.   
Then $\gl_2^2\sub \m_{[2,3,4],[1,2]}\sub M_{\be_{179}}$ 
does not change this condition. 
By applying Lemma II--4.3 to the action of this 
$\gl_2^2$ on the subspace 
$\lan \mathbbm a_{24},\mathbbm a_{27},
\mathbbm a_{34},\mathbbm a_{37}\ran \oplus
\lan \mathbbm a_{28},\mathbbm a_{38}\ran$,  
we may assume that $y_{24},y_{28}=0$.

(143) $S_{180}$, 
$\be_{180} = \tfrac {1} {20} (-8,0,2,2,4,-1,-1,-1,3)$ 

By Lemma II--4.4, we may assume that $y_9,y_{10},y_{35}=0$.

(159) $S_{196}$, 
$\be_{196} = \tfrac {1} {340} (-46,-46,4,4,84,-45,5,5,35)$ 

By applying Lemma II--4.1 to the action of 
$\gl_2^2\sub M_{[2,4],[1,2,3]}\sub M_{\be_{196}}$
on the subspaces
$\lan \mathbbm a_9,\mathbbm a_{10}\ran$, 
$\lan \mathbbm a_{14},\mathbbm a_{17}\ran$, 
we may assume that $y_9=0,y_{14}=0$.

(160) $S_{197}$, $\be_{197}=\tfrac {1} {60} (-12,-12,4,4,16,-7,-7,5,9)$, 

By applying Lemma II--4.1 to the action of 
$\gl_2^2\sub M_{[2,4],[1,2,3]}\sub M_{\be_{197}}$
on the subspaces $\lan \mathbbm a_9,\mathbbm a_{10}\ran$, 
$\lan \mathbbm a_{34},\mathbbm a_{37}\ran$
we may assume that $y_9=0,y_{34}=0$.

(176) $S_{214}$, 
$\be_{214} = \tfrac {1} {740} (-96,-76,24,24,124,-5,-5,-5,15)$ 

We identify $\lan \bbma_7\ccd \bbma_{28}\ran$ with 
$\m_{3,2}$. By Lemma II--4.2, we may assume that $y_7,y_8=0$.  

(190) $S_{235}$, 
$\be_{235} = \tfrac {1} {80} (-2,-2,-2,3,3,-20,-20,15,25)$ 

By Lemma II--4.6, we may assume that $y_{31},y_{32}=0$.

(220) $S_{269}$, $\be_{269}=\tfrac {1} {180} (-22,-22,8,8,28,-5,-5,5,5)$, 

By applying Lemma II--4.1 to the action of 
$\gl_2^2\sub M_{[2,3,4],[2,3]}\sub M_{\be_{269}}$
on the subspaces 
$\lan \mathbbm a_8,\mathbbm a_{18}\ran$, 
$\lan \mathbbm a_{24},\mathbbm a_{27}\ran$, 
we may assume that $y_8=0,y_{24}=0$.

(222) $S_{275}$, 
$\be_{275} = \tfrac {1} {60}(-24,-24,16,16,16,-15,-15,-15,45)$ 

By Lemma II--4.6, we may assume that $y_{38},y_{39}=0$. 

This completes the proof of Theorem \ref{thm:empty-verify}.

%
%
%
%
%
%
%
%
%
%
%
%
%
%
%
%
%

\section{Smaller \pv s}
\label{sec:smaller-pv-s}

In this section we show that if 
the GIT stratification of $(G,V)$ 
is determined and $S_{\be}$ is a non-empty stratum 
of the GIT stratification of $(G,V)$,  
then it requires much less work 
to determine the GIT stratification of 
$(M_{\be},Z_{\be})$. We shall prove two 
propositions and explain what has to be done
to determine the GIT stratification of 
$(M_{\be},Z_{\be})$.

Let $G,G_{\mathrm{st}},V,G_{\mathrm{st},\be},T_0,T,T_{\text{st}},\gB$, etc.,  
be as in Section \ref{sec:notation-related-git}. 
We still assume that $G$ is split. 
Note that $T_0$ for (\ref{eq:PV}) is defined in (\ref{eq:T0-defn}).

Suppose that $\be\in\gB$, $S_{\be}\not=\emptyset$. 
We consider the \rep{} of $M_{\be}$ on $Z_{\be}$. 
Let $T_{\be,0}\sub T$ be the subgroup generated by 
$T_0$ and $\im(\lam_{\be})$. Then elements of 
$T_{\be,0}$ act on $Z_{\be}$ by scalar multiplication 
and $M_{\be}=T_{\be,0}G_{\mathrm{st},\be}$ as algebraic groups
(i.e., $M_{\be\, \overline k} = T_{\be,0\,\overline k}G_{\mathrm{st},\be\,\overline k}$). 
For the \rep{} $(M_{\be},Z_{\be})$, we let $T_{\be,0}$ and 
$G_{\mathrm{st},\be}$ play the roles 
of $T_0$ and $G_{\mathrm{st}}$ in the general situation respectively. 
So $Z^{\sst}_{\be}$ is the pull back 
of $\p(Z_{\be})^{\sst}$ where the stability is with respect to the 
action of $G_{\mathrm{st},\be}$.

We put 
\begin{equation*}
T_{\text{st},\be} \stackrel{\rm{def}}=(T\cap G_{\mathrm{st},\be})^{\circ},\;
\gt^*_{\be} \stackrel{\rm{def}}=X^*(T_{\mathrm{st},\be})\otimes \R.
\end{equation*}
Then $T_{\text{st},\be}$ is a maximal split torus of $G_{\mathrm{st},\be}$ and 
we can identify $\gt^*_{\be}$
with the orthogonal complement in $\gt^*$ of $\R\be$ 
with respect to the inner product $(\;,\;)_*$ on $\gt^*$.
Let $\gt^*_+$ be a Weyl chamber. 
We use the restriction of $(\;,\;)_*$ to $\gt^*_{\be}$. 
Since the Weyl group of $M_{\be}$ is a subgroup of $\weyl$, 
we can choose a Weyl chamber 
$\gt^*_{\be+}$ so that it includes $\gt^*_+\cap \gt^*_{\be}$. 
Let $\gB_{\be}\sub \gt^*_{\be+}$ be the set which parametrizes 
the GIT stratification of $(M_{\be},Z_{\be})$ for the
above situation.

For $\be'\in \gB_{\be}$, let 
\begin{math}
\overline M_{\be'},\overline G_{\text{st},\be'}\sub M_{\be},
\overline Z_{\be'},\overline W_{\be'},\overline Y_{\be'},
\overline S_{\be'}\sub Z_{\be} 
\end{math}
be the subsets $G_{\be'}$, $G_{\text{st},\be'}$
$Z_{\be'}$, $W_{\be'}$, $Y_{\be'}$, $S_{\be'}$ 
defined in Section \ref{sec:notation-related-git} 
for $G$ and $V$ replaced by $M_{\be}$ and $Z_{\be}$ respectively.  
We denote the semi-simple part of $\overline M_{\be'}$ by 
$\overline M^s_{\be'}$. 
Also the torus $T_{\text{st},\be'}$ in this situation 
is denoted by $\overline T_{\text{st},\be'}$. 
We use this slightly different notations
to distinguish with  those for $(G,V)$. 
If $\be\in \gC$ (see Definition \ref{defn:gB-defn} (1)) 
then the 1PS $\lam_{\be}\in X_*(T_{\text{st}})$ 
and the character $\chi_{\be}$ of $M_{\be}$ were defined 
in Section \ref{sec:notation-related-git}.
If $\be\in\gB$ is fixed and $(G,V)$ is replaced by 
$(M_{\be},Z_{\be})$ and $\be'\in\gB_{\be}$, 
then $\lam_{\be'}\in X_*(T_{\text{st},\be})$, 
$\chi_{\be'}\in X^*(\overline M_{\be'})$ defined in this situation 
are denoted by $\overline \lam_{\be'},\overline \chi_{\be'}$ 
to distinguish with those of $(G,V)$. 

We summarize corresponding notations as below. 
Note that $\be'$ in the first (resp. second) row
belongs to $\gB\sub \gt^*$ 
(resp. $\gB_{\be}\sub \gt^*_{\be}$).  

\vskip 10pt

\begin{center}
\begin{tabular}{|c|c|c|c|c|c|c|c|c|c|c|c|c|c|}
\hline
\rule[-8pt]{-0.2cm}{24pt}
$G$ & $V$ & $T_0$ & $G_{\text{st}}$ & $T_{\text{st}}$ 
& $\gt^*$ & $G_{\be'}$ &  $G_{\text{st},\be'}$ 
& $Z_{\be'}$ & $W_{\be'}$ & $Y_{\be'}$ & $S_{\be'}$ 
& $\lam_{\be'}$ & $\chi_{\be'}$ \\
\hline
\rule[-8pt]{-0.2cm}{24pt}
$M_{\be}$ & $Z_{\be}$ & $T_{\be,0}$ & $G_{\text{st},\be}$ 
& $T_{\text{st},\be}$ & $\gt^*_{\be}$ 
& $\overline M_{\be'}$ & $\overline G_{\text{st},\be'}$  
& $\overline Z_{\be'}$ & $\overline W_{\be'}$ 
& $\overline Y_{\be'}$ & $\overline S_{\be'}$ 
& $\overline \lam_{\be'}$ & $\overline \chi_{\be'}$ \\
\hline
\end{tabular}
\end{center}

\vskip 10pt

We still assume that the action of $T$ on $V$ is diagonalized, 
$\coorda_i$ ($i=1\ccd N$) is the coordinate vector corresponding 
to the $i$-th coordinate and $\gam_i\in \gt^*$ is its weight. 
Let $I_{\be}=\{i\mid \coorda_i\in Z_{\be}\}$, 
$J_{\be}=\{i\mid \coorda_i\in W_{\be}\}$.
If $i\in I_{\be}$ then $(\gam_i,\be)_*=(\be,\be)_*$. 
So $\gam'_i\stackrel{\rm{def}}=\gam_i-\be$ is orthogonal to $\be$ and 
$\gam'_i$ can be regarded as the weight of $\coorda_i$ 
in $\gt^*_{\be}$.

Since $S_{\be}\not=\emptyset$, $Z^{\sst}_{\be\, k}\not=\emptyset$. 
Let $x\in Z^{\sst}_{\be\, k}\not=\emptyset$. If 
$\gI=\{\gam_i'\mid x_i\not=0\}$ then $\Conv \gI$ 
contains the origin. So, 
$\operatorname{Conv} \{\gam'_i\mid i\in I_{\be}\}$
contains the origin.  
Therefore, there exists a rational number $0\leq a_i\leq 1$ 
for each $i\in I_{\be}$
such that $\sum_{i\in I_{\be}} a_i=1$ and that 
$0=\sum_{i\in I_{\be}} a_i\gam'_i$. 
This implies that $\be=\sum_{i\in I_{\be}} a_i \gam_i$.

Suppose that $\be'\in\gB_{\be}$. 
We assume that $\overline S_{\be'}\not=\emptyset$. 
Then there exist a finite subset $X\sub I_{\be}$ and 
a rational number $0\leq a_i\leq 1$ 
for each $i\in X$ such that $\sum_{i\in X}a_i=1$, 
$\be'=\sum_{i\in X}a_i \gam'_i = \sum_{i\in X} a_i\gam_i-\be$. 
We may assume that $(\gam_i',\be')_*=(\be',\be')_*$ for all $i\in X$.
Then $\be''\stackrel{\rm{def}}=\be'+\be$ is in the convex hull of 
$\{\gam_i\mid i\in X\}$. If $i\in X$ then $\gam_i'$ is orthogonal
to $\be$. So $\be'$ is the closest point of 
the convex hull of $\{\gam_i'\mid i\in X\}$ to 
the origin if and only if $\be''$ is the closest 
point of the convex hull of 
$\{\gam_i\mid i\in X\}$ to the origin. 
Therefore, $\be''\in \gC$ (see Definition \ref{defn:gB-defn} (1)) 
and we can consider $Z_{\be''}$, etc. 

\begin{lem}
\label{lem:Zbe''}
$Z_{\be''}\cap Z_{\be}=\overline Z_{\be'}$. 
\end{lem}
\begin{proof}
Let $0\leq i\leq N$. Suppose that $\coorda_i\in Z_{\be''}\cap Z_{\be}$. 
Since $(\be,\be')_*=0$ and $(\gam_i,\be)_*=(\be,\be)$, 
\begin{align*}
(\be'',\be'')_* & = (\be'+\be,\be'+\be)_*
= (\be',\be')_*+(\be,\be)_*, \\
(\gam_i,\be'')_* & =  (\gam_i,\be')_* +(\be,\be)_*
= (\gam'_i,\be')_*+(\be,\be)_*.
\end{align*}
Since $(\gam_i,\be'')_*=(\be'',\be'')_*$, 
we have $(\gam'_i,\be')_*=(\be',\be')_*$. Therefore, 
$\coorda_i\in \overline Z_{\be'}$. Since $Z_{\be''}\cap Z_{\be}$, 
$\overline Z_{\be'}$ are spanned by coordinate vectors, 
$Z_{\be''}\cap Z_{\be}\sub \overline Z_{\be'}$.  

Conversely, suppose that $\coorda_i\in \overline Z_{\be'}$.  
Then $\coorda_i\in Z_{\be}$ of course. 
Since $(\gam'_i,\be)_*=(\be,\be')_*=0$, 
\begin{align*}
(\gam_i,\be'')_* & = (\gam'_i+\be,\be'+\be)_*
= (\gam'_i,\be')_* + (\gam'_i,\be)_* + (\be,\be')_* + (\be,\be)_* \\
& = (\be',\be')_* + (\be,\be)_*
= (\be'',\be''). 
\end{align*}
So $\coorda_i\in Z_{\be''}$. 
This implies that $\overline Z_{\be'} = Z_{\be''}\cap Z_{\be}$. 
\end{proof}

There exists $w\in \weyl$ 
(the Weyl group of $G$) 
such that $\be'''\stackrel{\rm{def}}=w^{-1}\be''\in \gt^*_+$. Since 
$\be''=\sum_{i\in X} a_i\gam_i$, we have 
$\be'''=\sum_{i\in X} a_iw^{-1}\gam_i$.
Since $\be''$ is the closest point of 
$\Conv \{\gam_i\mid i\in X\}$ to the origin, 
$\be'''$ is the closest point of 
$\Conv \{w^{-1}\gam_i\mid i\in X\}$ to the origin.
Therefore, $\be'''\in \gB$.

\begin{prop}
\label{prop:Zbe-stratificiation} 
Suppose that $Z_{\be''}\sub Z_{\be}$
and that $M_{\be''}\sub M_{\be}$. Then
\begin{itemize}
\item[(1)]
$Z^{\sst}_{\be''}=\overline Z^{\sst}_{\be'}$.
\item[(2)]
$M_{\be''\,k}\backslash Z^{\sst}_{\be''\,k}\cong 
\overline M_{\be'\,k}\backslash \overline Z^{\sst}_{\be'\,k}$
\item[(3)]
$\overline S_{\be'}\not=\emptyset$ if and only if 
$S_{\be''}\not=\emptyset$.
\end{itemize}
\end{prop}
\begin{proof}
(1) If the action of $g\in G$ fixes $\be,\be'$ then 
it fixes $\be''$. So 
$\overline M_{\be'}\sub M_{\be''}$. 
If $g\in M_{\be''}$ 
then $g\in M_{\be}$ by assumption. So $g$ fixes $\be,\be''$. 
Therefore, $g$ fixes $\be'$. 
This implies that $M_{\be''}=\overline M_{\be'}$. 

We may assume that $k=\overline k$. 
Let $\psi_{\be},\psi_{\be''}$ be the restrictions of 
$\chi_{\be},\chi_{\be''}$ to $T_{\text{st}}$. 
Note that $T_{\text{st},\be}=\kernel(\psi_{\be})^{\circ}$, 
$T_{\text{st},\be''}=\kernel(\psi_{\be''})^{\circ}$.
Let $\overline \psi_{\be'}$ be the restriction of 
$\overline{\chi}_{\be'}$ to $T_{\text{st},\be}$. 
Then $\overline T_{\text{st},\be'}=\kernel(\overline \psi_{\be'})^{\circ}$.

It is easy to see that 
$G_{\mathrm{st},\be''}$ (resp. $\overline G_{\mathrm{st},\be'}$) 
is generated by $M^s_{\be''}$ (the semi-simple part) 
and $T_{\text{st},\be''}$ 
(resp. $\overline M^s_{\be'}=M^s_{\be''}$ and $\overline T_{\text{st},\be'}$). 
So the only difference between $G_{\mathrm{st},\be''}$ and 
$\overline G_{\mathrm{st},\be'}$ is the difference  
between $\overline T_{\be'}$ and $T_{\be''}$.

Since $\be''=\be+\be'$, 
$\overline T_{\text{st},\be'}=(\kernel(\psi_{\be})\cap \kernel(\psi_{\be''}))^{\circ}$. 
By assumption and Lemma \ref{lem:Zbe''}, 
$Z_{\be''}=\overline Z_{\be'}$. 
Note that $\lam_{\be}$ (resp. $\lam_{\be'}$) acts on 
$Z_{\be''}=\overline Z_{\be'}$ by scalar multiplication
by weight $(\be,\be)_*$ (resp. $(\be',\be')_*$). 
So for $a,b\in\Z$, $\lam_{\be}^a\lam_{\be'}^b$ acts on 
$Z_{\be''}=\overline Z_{\be'}$ by scalar multiplication
by weight $a(\be,\be)_*+b(\be',\be')_*$. 
Since $\be''=\be+\be'$, 
the weight of $\psi_{\be''}(\lam_{\be}^a\lam_{\be'}^b)$
is proportional to $a(\be,\be)_*+b(\be',\be')_*$. 
Since $k=\overline k$, $k^{\times} \ni t\mapsto t^n \in k^{\times}$ 
is surjective for all $n>0$. 
Therefore, $T_{\be''}/\overline T_{\be'}$ is 
represented by $\lam_{\be}^a\lam_{\be'}^b$ such that 
$a(\be,\be)_*+b(\be',\be')_*=0$ and this acts on 
$Z_{\be''}=\overline Z_{\be'}$ trivially. 
This implies that a polynomial on 
$Z_{\be''}=\overline Z_{\be'}$ is invariant with respect to 
$G_{\mathrm{st},\be''}$ if and only if it is invariant with respect to 
$\overline G_{\mathrm{st},\be'}$. Therefore, 
$Z^{\sst}_{\be''}=\overline Z^{\sst}_{\be'}$. 

(2) $\overline M_{\be'}$ consists of $g\in G$ which fixes (by conjugation)
$\be,\be'$. So $\overline M_{\be'}\sub M_{\be''}$. 
By assumption, $M_{\be''}\sub M_{\be}$. Since $\be'=\be''-\be$, 
$M_{\be''}\sub \overline M_{\be'}$. Therefore, 
$M_{\be''}=\overline M_{\be'}$. So  
(2) follows even though $G_{\mathrm{st},\be''}\not=\overline G_{\mathrm{st},\be'}$.

(3) follows from (1) (2). 
\end{proof}

In the situation of the above proposition, 
we still have to determine if 
$\overline P_{\be'\,k}\backslash \overline Y^{\sst}_{\be'\,k}\cong 
\overline M_{\be'\,k}\backslash \overline Z^{\sst}_{\be'\,k}$.

We review the property of the action of $U_{\be}$ on $Z_{\be},Y_{\be}$. 
For $\gam\in \gt^*,\lam\in\gt$, 
let $\lan \lam,\gam\ran$ be the natural paring. 
The element $\lam(\be)\in\gt$ dual to $\be$ is characterized
by the property that 
$(\be,\gam)_*=\lan \lam(\be),\gam\ran$. 
There is a positive rational number $a>0$ such that 
the 1PS $\lam_{\be}$ is $a\lam(\be)$. 
So $a(\be,\gam)_*=\lan \lam_{\be},\gam\ran$. 
This implies that if $x\in V$ is a weight vector 
with weight $\gam$ with respect to the action of $T_{\text{st}}$ 
then $\lam_{\be}(t)x = t^{a(\be,\gam)_*}x$ for $t\in\gl_1$. 

\begin{lem}
\label{lem:u-action}
If $x\in Y_{\be}$ is a weight vector with weight $\gam$
and $u\in U_{\be}$ then there exist weight vectors 
$y_1\ccd y_n\in Y_{\be}$ with weights $\del_1\ccd \del_n$
such that $ux=x+\sum_{i=1}^ny_i$ and $(\be,\gam)_*<(\be,\del_i)_*$ 
for all $i$. 
\end{lem}
\begin{proof}
Since $Y_{\be}$ is spanned by weight vectors, there exist 
weight vectors $y_0\ccd y_n\in Y_{\be}\mzer$ with 
distinct weights 
$\gam_0\ccd \gam_n$ such that $ux=y_0+\cdots+y_n$. 
Then 
\begin{align*}
\lam_{\be}(t) u \lam_{\be}(t)^{-1} x 
& =  t^{-a(\be,\gam)_*}\lam_{\be}(t)ux
= \sum_{i=0}^n  t^{-a(\be,\gam)_*}\lam_{\be}(t) y_i
= \sum_{i=0}^n  t^{a(\be,\del_i)_*-a(\be,\gam)_*}y_i
\end{align*}
By definition, $U_{\be}$ consists of elements $u\in P_{\be}$ 
such that 
\begin{equation*}
\lim_{t\to 0}\lam_{\be}(t) u \lam_{\be}(t)^{-1} = e_G.
\end{equation*}
So $(\be,\del_i)_*\geq (\be,\gam)_*$ for all $i$ and 
\begin{equation*}
x = \sum_{i:(\be,\del_i)_*=(\be,\gam)_*}y_i.
\end{equation*}
Since $x$ is a weight vector, we may assume that $y_0=x$ 
and $(\be,\del_i)_*>(\be,\gam)_*$ for $i=1\ccd n$. 
\end{proof}

\begin{cor}
\label{cor:u-action-Zbeta}
If $x\in Z_{\be}$ and $u\in U_{\be}$ then 
$ux=x+y$ with $y\in W_{\be}$.  
\end{cor}

Condition \ref{cond:unipotent-eliminate-condition} is stated for 
strata of (\ref{eq:PV}). However, 
we can consider Condition \ref{cond:unipotent-eliminate-condition} 
for general \pv s and by replacing $\be_i$ by 
arbitrary $\be\in\gC$. 

\begin{prop}
\label{prop:PY=MZ}
In the situation of Proposition \ref{prop:Zbe-stratificiation}, 
suppose that Condition \ref{cond:unipotent-eliminate-condition} 
holds for $\be,\be''$ (with respect to $G,V$, etc.). Then 
Condition \ref{cond:unipotent-eliminate-condition} 
holds for $\be'$ (with respect to $M_{\be},Z_{\be}$, etc.). 
\end{prop}
\begin{proof}
Condition \ref{cond:unipotent-eliminate-condition} (1) 
follows from Proposition \ref{prop:Zbe-stratificiation} (2)
(by replacing $k$ by $k^{\sep}$).  

Suppose that $R\in Z^{\sst}_{\be''\,k}$ satisfies 
Condition \ref{cond:unipotent-eliminate-condition} (2).  
Note that $U_{\be''}=\overline U_{\be'} \ltimes U_{\be}$ 
and $\overline U_{\be'}\sub M_{\be}$. 
We show that $G_R\cap U_{\be''}=(G_R\cap \overline U_{\be'})\ltimes (G_R\cap U_{\be})$. 
 
$G_R\cap U_{\be''}\supset (G_R\cap \overline U_{\be'})\ltimes (G_R\cap U_{\be})$ 
is obvious. Suppose that $u_1\in \overline U_{\be'},u_2\in U_{\be}$, 
$u=u_2u_1$ and $uR=R$. By Corollary \ref{cor:u-action-Zbeta}, 
there exists $y\in W_{\be}$ such that $u_2R=R+y$.  
Since $u_1\in \overline U_{\be'}\sub M_{\be}$, 
$u_1$ commutes with $\lam_{\be}(t)$. So $u_1$ fixes  weight spaces of 
$\lam_{\be}$. This implies that $uR=u_1R+u_1y=R$ and 
$u_1R\in Z_{\be},u_1y\in W_{\be}$. Therefore, $u_1R=R,u_1y=0$. 
Since $\overline U_{\be'}$ is a group, $y=0$, and so $u_2R=R$. 
Hence, 
\begin{equation*}
G_R\cap U_{\be''}=(G_R\cap \overline U_{\be'})\ltimes (G_R\cap U_{\be}).
\end{equation*}

By assumption, $G_R\cap U_{\be''}$ is connected. So 
$G_R\cap \overline U_{\be'}=M_{\be\, R} \cap \overline U_{\be'}$ must be connected
($M_{\be\,R}$ is the stabilizer of $R$ in $M_{\be}$)
and  Condition \ref{cond:unipotent-eliminate-condition} (3) is satisfied for 
$\be'$.   

Suppose that $y\in \overline W_{\be'}$. By assumption, 
there exists $u_1\in \overline U_{\be'\,k^{\sep}}$, $u_2\in U_{\be\,k^{\sep}}$
such that $u_1u_2R = R+y$.  By Corollary \ref{cor:u-action-Zbeta}, 
$u_2R = R+y_1$ with $y_1\in W_{\be}$, $u_1u_2R = u_1R+u_1y_1=R+y$ and 
$u_1y\in W_{\be}$. Since $u_1R,R+y\in Z_{\be}$, 
$u_1R=R+y$ and $u_1y_1=0$.  So Condition \ref{cond:unipotent-eliminate-condition} (2) 
is satisfied for $\be'$. 
\end{proof}

We now explain what has to be done to determine the GIT stratification 
of $(M_{\be},Z_{\be})$.  
Since $(M_{\be},Z_{\be})$ is smaller than $(G,V)$, 
it is very likely that if $\gB$ can be determined 
by computer computation then we can determine $\gB_{\be}$ 
by computer computation also. If $\be'\in\gB_{\be}$ then 
it is in the form 
$\be'=w \be'''-\be$ where $w\in \weyl$ and $\be'''\in\gB$. 
Of course $\be'$ has to be in the Weyl chamber $\gt^*_{\be+}$.
For example, if $(G,V)$ is the \pv{} (\ref{eq:PV}) then 
we have determined the set $\gB$ which consists of $292$ elements. 
The possibilities for $w$ are at most $2880$.
So the possibilities for $(\be''',w)$ are at most 
$292\times 2880$.  What we have to do is to compare 
$\be'$ and $w \be'''-\be$. This is not such a 
big task for computer computation. If the 
condition of Proposition \ref{prop:Zbe-stratificiation} 
is satisfied for $\be,\be''=\be'+\be$ then 
we do not have to worry about whether or not 
$\overline S_{\be'}$ is empty. We expect that 
Proposition \ref{prop:Zbe-stratificiation} is applicable 
in most cases and we only have to determine 
whether or not $\overline S_{\be'}$ is empty 
for exceptional cases.

\bibliographystyle{plain} 
\bibliography{ref4} 

\begin{thebibliography}{10}

\bibitem{bhargava4}
M.~Bhargava.
\newblock The density of discriminants of quartic rings and fields.
\newblock {\em Ann. of Math. (2)}, 162(2):1031--1063, 2005.

\bibitem{borelb}
A.~Borel.
\newblock {\em Linear algebraic groups}.
\newblock Springer-{V}erlag, Berlin, Heidelberg, New York, 2nd edition, 1991.

\bibitem{dolgachev-invariant}
I.~Dolgachev.
\newblock {\em Lectures on invariant theory}, volume 296 of {\em London
  Mathematical Society Lecture Note Series}.
\newblock Cambridge University Press, Cambridge, 2003.

\bibitem{hayasaka-yukie-tamagawaI}
N.~Hayasaka and A.~Yukie.
\newblock On the density of unnormalized tamagawa numbers of orthogonal groups
  {I}.
\newblock {\em Publ. Res. Inst. Math. Sci.}, 44(2):545--607, 2008.

\bibitem{kayu}
A.C. Kable and A.~Yukie.
\newblock Prehomogeneous vector spaces and field extensions {II}.
\newblock {\em Invent. Math.}, 130(2):315--344, 1997.

\bibitem{kable-yukie-quinary}
A.C. Kable and A.~Yukie.
\newblock On the space of quadruples of quinary alternating forms.
\newblock {\em J. Pure Appl. Algebra}, 186(3):277--295, 2004.

\bibitem{kato-yukie-jordan}
R.~Kato and A.~Yukie.
\newblock Rational orbits of the space of pairs of exceptional {J}ordan
  algebras.
\newblock {\em J. Number Theory}, 189:304--352, 2018.

\bibitem{kimura-2simple-typeI}
T.~Kimura, S-I. Kasai, M.~Inozuka, and O.~Yasukura.
\newblock A classification of 2-simple prehomogeneous vector spaces of type
  {I}.
\newblock {\em J. Algebra}, 114(2):369--400, 1988.

\bibitem{kimura-2simple-typeII}
T.~Kimura, S-I. Kasai, M.~Taguchi, and M.~Inozuka.
\newblock Some {P}.{V}.-equivalences and a classification of {$2$}-simple
  prehomogeneous vector spaces of type {${\rm II}$}.
\newblock {\em Trans. Amer. Math. Soc.}, 308(2):433--494, 1988.

\bibitem{kimura-2simple-typeI-inv}
T.~Kimura, T.~Kogiso, and K.~Sugiyama.
\newblock Relative invariants of 2-simple prehomogeneous vector spaces of type
  {I}.
\newblock {\em J. Algebra}, 308(2):445--483, 2007.

\bibitem{kimura-3simple}
T.~Kimura, K.~Ueda, and T.~Yoshigaki.
\newblock A classification of {$3$}-simple prehomogeneous vector spaces of
  nontrivial type.
\newblock {\em Japan. J. Math. (N.S.)}, 22(1):159--198, 1996.

\bibitem{kimura-2simple-some-inv}
T.~Kogiso, G.~Miyabe, M.~Kobayashi, and T.~Kimura.
\newblock Relative invariants of some 2-simple prehomogeneous vector spaces.
\newblock {\em Math. Comp.}, 72(242):865--889, 2003.

\bibitem{mum}
D.~Mumford.
\newblock {\em Lectures on curves on an algebraic surface}, volume~59 of {\em
  Annales of Mathematical Studies}.
\newblock Princeton University Press, Princeton, New Jersey, 1966.

\bibitem{mufoki}
D.~Mumford, J.~Fogarty, and F.~Kirwan.
\newblock {\em Geometric invariant theory}.
\newblock Springer-{V}erlag, Berlin, Heidelberg, New York, 3rd edition, 1994.

\bibitem{ochiai}
H.~Ochiai.
\newblock Quotients of some prehomogeneous vector spaces.
\newblock {\em J. Algebra}, 192(1):61--73, 1997.

\bibitem{ozekic}
I.~Ozeki.
\newblock On the micro-local structure of the regular prehomogeneous vector
  spaces associated with {$\text{SL}(5) \times \gl(4)$} {I}.
\newblock {\em Publ. Res. Inst. Math. Sci.}, 26:no3, 539--584, 1990.

\bibitem{saki}
M.~Sato and T.~Kimura.
\newblock A classification of irreducible prehomogeneous vector spaces and
  their relative invariants.
\newblock {\em Nagoya Math. J.}, 65:1--155, 1977.

\bibitem{Springer-LAG}
T.A. Springer.
\newblock {\em Linear algebraic groups}, volume~9 of {\em Progress in
  Mathematics}.
\newblock Birkh\"auser Boston, Inc., Boston, MA, second edition, 1998.

\bibitem{tajima-yukie}
K.~Tajima and A.~Yukie.
\newblock Stratification of the null cone in the non-split case.
\newblock {\em Comment. Math. Univ. St. Pauli}, 63(1-2):261--276, 2014.

\bibitem{tajima-yukie-GIT1}
K.~Tajima and A.~Yukie.
\newblock On the {GIT} stratification of prehomogeneous vector spaces {I}.
\newblock {\em Comment. Math. Univ. St. Pauli}, 68(1):1--48, 2020.

\bibitem{tajima-yukie-GIT2}
K.~Tajima and A.~Yukie.
\newblock On the {GIT} stratification of prehomogeneous vector spaces {II}.
\newblock {\em Tsukuba J. Math.}, 44(1):1--62, 2020.

\bibitem{wryu}
D.J. Wright and A.~Yukie.
\newblock Prehomogeneous vector spaces and field extensions.
\newblock {\em Invent. Math.}, 110(2):283--314, 1992.

\bibitem{yukiel}
A.~Yukie.
\newblock Prehomogeneous vector spaces and field extensions {III}.
\newblock {\em J. Number Theory}, 67(1):115--137, 1997.

\bibitem{yukiem}
A.~Yukie.
\newblock Prehomogeneous vector spaces and ergodic theory {III}.
\newblock {\em J. Number Theory}, 70(2):160--183, 1998.

\end{thebibliography}

\end{document}